\documentclass[11pt]{article}

\usepackage[T1]{fontenc}
\usepackage{lmodern}
\usepackage[a4paper,margin=25mm]{geometry}
\usepackage{amsmath,amssymb,amsthm,mathtools}
\usepackage{booktabs,longtable,array,tabularx}
\usepackage{enumitem}
\usepackage{xcolor}
\usepackage{float}
\usepackage{placeins}
\usepackage{tikz}
\usetikzlibrary{arrows.meta,decorations.markings,positioning,shapes.geometric}
\usepackage[hidelinks]{hyperref}
\hypersetup{
  pdftitle={Local Uniform Finite Cyclicity of the H14-3 Semihyperbolic Hemicycle in Quadratic Systems},
  pdfauthor={Haibo Lu},
  pdfsubject={Local uniform finite cyclicity for the H14-3 hemicycle in the full quadratic coefficient space},
  pdfkeywords={quadratic vector fields, finite cyclicity, hemicycle, H14-3}
}

\setlist[itemize]{leftmargin=2em,itemsep=0.25em,topsep=0.35em}

\newtheorem{theorem}{Theorem}
\newtheorem{proposition}[theorem]{Proposition}
\newtheorem{lemma}[theorem]{Lemma}
\newtheorem{corollary}[theorem]{Corollary}
\theoremstyle{definition}
\newtheorem{definition}[theorem]{Definition}
\theoremstyle{remark}
\newtheorem{remark}[theorem]{Remark}
\numberwithin{equation}{section}

\newcommand{\Hfourteen}{H_{14}^{3}}
\newcommand{\GammaH}{\Gamma_{H14^3}}
\newcommand{\NH}{N_{H14^3}}
\newcommand{\BH}{B_{H14^3}}
\newcommand{\R}{\mathbb R}
\newcommand{\subjclass}[2][2020]{%
  \par\smallskip\noindent\textit{#1 Mathematics Subject Classification.} #2\par}
\newcommand{\keywords}[1]{%
  \par\smallskip\noindent\textit{Keywords.} #1\par}

\title{Local Uniform Finite Cyclicity of the
$H_{14}^{3}$ Semihyperbolic Hemicycle in Quadratic Systems}
\author{Haibo Lu\\
Shanghai Institute of Technology\\
\texttt{luhaibo1985@gmail.com}}
\date{}

\begin{document}

\maketitle

\begin{abstract}
We prove local uniform finite cyclicity of the labelled $H_{14}^{3}$
hemicycle in a full neighborhood of its base field in the twelve-dimensional
space of planar quadratic vector fields.  Thus there is one fixed two-sided
annular neighborhood of the compactified graphic in which the number of
isolated limit cycles is bounded uniformly for every sufficiently close
quadratic field.  A local analytic slice theorem is part of the result: the
displayed five-parameter source-normalized family is transverse at the
$B=0$ field to the seven-dimensional action of affine phase changes and
positive constant time rescalings.  This removes the normalization and
transfers the same cyclicity bound to the full quadratic coefficient space.

The normalized problem combines a noncompact period annulus, two
semihyperbolic endpoints at infinity, and a degeneration at the upper
equatorial point.  We replace the unavailable global Poincar\'e map by finitely
many stopped transition maps and prove an exact multiplicity-preserving
correspondence between collar cycles and displacement zeros.  At the
noncompact source a matched nonlinear return on a common physical domain,
controlled through six derivatives, gives a two-zero bound by Rolle's
theorem.  At the two saddle-nodes an exhaustive physical incidence and scale
decomposition reduces all limits to source, mixed, hyperbolic, central, and
lips estimates.  A finite specialization argument then extends these bounds
across coefficient and identity strata.  The resulting bound is existential
and is not claimed to be sharp.  In the terminology of the quadratic
finite-cyclicity program, this completes the labelled $H_{14}^{3}$ open case.
\end{abstract}


\setcounter{tocdepth}{1}
\tableofcontents

\section{Introduction and main result}
\label{sec:introduction}

The cyclicity of a graphic is the maximal number of isolated periodic orbits
that can bifurcate from it in a prescribed family.  In the local form of
Hilbert's sixteenth problem one asks not merely for finiteness at each fixed
parameter, but for one bound that is uniform in a whole parameter
neighborhood; see Roussarie~\cite{Roussarie1998} and
Ilyashenko~\cite{Ilyashenko2002} for the geometric and historical framework,
and Ilyashenko--Yakovenko~\cite{IlyashenkoYakovenko2008} for analytic
differential-equation background.  This distinction is essential near a noncompact graphic:
periodic orbits may approach different singular pieces while the parameters
approach the base field, so a collection of pointwise estimates need not have
a common bound.

The object considered here is the semihyperbolic hemicycle labelled
$H_{14}^{3}$ in the quadratic classification.  On the Poincar\'e sphere it is
the outer boundary of a period annulus and contains two saddle-nodes at the
horizontal points at infinity, together with a degenerate point on the upper
equator.  The simultaneous presence of these three pieces is precisely what
prevents one from writing down a single return map at the outset.  The main
result proves finite cyclicity on an open neighborhood in the full quadratic
coefficient space, without replacing the physical collar by a
parameter-dependent neighborhood.  The proof first keeps all five
source-normalized parameters and then removes that normalization by a local
analytic slice theorem.

\subsection{The family and the theorem}

The starting example is the source-normalized field
\begin{equation}
 X_0:\qquad \dot x=-y,\qquad \dot y=x+xy .
 \label{eq:source}
\end{equation}
On the half-plane $y>-1$ it has the first integral
\begin{equation}
 H(x,y)=\frac{x^2}{2}+y-\log(1+y).
 \label{eq:first-integral}
\end{equation}
Its regular positive levels form the period annulus around the origin.  Their
outer compactified limit is the labelled set $\GammaH$, consisting of the
finite invariant line $y=-1$, two horizontal points at infinity, the oriented
upper equatorial arcs, and the positive vertical point at infinity.  The
labels and orientation specify the graphic throughout the paper.

The unfolding studied here is
\begin{equation}
\begin{aligned}
\dot x&=-y+Bx^2+\mu_2y^2+(\mu_4+B\mu_5)x,\\
\dot y&=x+xy+\mu_3y^2+(1-2B)\mu_5y,
\end{aligned}
\qquad
\lambda=(B,\mu_2,\mu_3,\mu_4,\mu_5).
\label{eq:family}
\end{equation}
This is the five-parameter source-normalized family displayed in
Roussarie--Rousseau~\cite[Theorem~3.1]{RoussarieRousseau2015}; in that
classification the value $B=0$ is precisely the labelled
$\Hfourteen$ case.  The earlier DRR94 work supplies the quadratic
classification program and the nomenclature~\cite{DRR1994}.  Although the
normal form in Roussarie--Rousseau is introduced from a base value $B_0>0$,
the coefficient family \eqref{eq:family} is defined at the boundary value
$B=0$.  We therefore verify directly, at that boundary field, that it is a
local slice for the full quadratic coefficient space.

Let $\mathcal Q_2$ be the twelve-dimensional real vector space of planar
polynomial vector fields
\[
 Y=P(x,y)\partial_x+Q(x,y)\partial_y,
 \qquad \deg P,\deg Q\le2.
\]
For an affine map $\Phi$ close to the identity, use the pullback convention
\[
 (\Phi^*Y)(z)=D\Phi(z)^{-1}Y(\Phi(z)),
 \qquad z=(x,y),
\]
and let $\rho>0$ multiply time by a positive constant.

\begin{proposition}[Local analytic slice for quadratic vector fields]
\label{prop:quadratic-local-slice}
There are neighborhoods $\mathcal G$ of $(\operatorname{id},1)$ in
$\operatorname{Aff}(2)\times\R_{>0}$, $\Lambda_0$ of $0$ in $\R^5$, and
$\mathcal V_0$ of $X_0$ in $\mathcal Q_2$ such that
\begin{equation}
 \mathcal S(\Phi,\rho,\lambda)
   =\rho\,\Phi^*X_\lambda
 \label{eq:quadratic-slice-map}
\end{equation}
is a real-analytic diffeomorphism from
$\mathcal G\times\Lambda_0$ onto $\mathcal V_0$.  In particular, every
quadratic vector field sufficiently close to $X_0$ has a unique
representation \eqref{eq:quadratic-slice-map} with $\Phi$ close to the
identity, $\rho>0$ close to one, and all five coordinates of $\lambda$ free.
The representation preserves oriented periodic orbits and their Poincar\'e
multiplicities.
\end{proposition}

\begin{proof}
The map \eqref{eq:quadratic-slice-map} is analytic because an affine pullback
preserves polynomial degree.  It remains to check its differential at
$(\operatorname{id},1,0)$.  For
$\Phi_\varepsilon=\operatorname{id}+\varepsilon h$, put
\[
 \mathcal L(h)=
 \left.\frac{d}{d\varepsilon}\right|_{\varepsilon=0}
 \Phi_\varepsilon^*X_0
 =DX_0h-Dh\,X_0.
\]
Use the six affine generators
\[
 h_1=(1,0),\ h_2=(0,1),\ h_3=(x,0),\ h_4=(y,0),\
 h_5=(0,x),\ h_6=(0,y).
\]
Their tangent vectors, followed by the positive time-scaling direction, are
\begin{equation}
\begin{aligned}
 \mathcal L(h_1)&=(0,1+y),&
 \mathcal L(h_2)&=(-1,x),&
 \mathcal L(h_3)&=(y,x+xy),\\
 \mathcal L(h_4)&=(-x-xy,y+y^2),&
 \mathcal L(h_5)&=(-x,x^2+y),&
 \mathcal L(h_6)&=(-y,-x),\\
 \partial_\rho\mathcal S&=X_0=(-y,x+xy).
\end{aligned}
\label{eq:quadratic-orbit-tangents}
\end{equation}
At the origin of the parameter space, the five family directions are
\begin{equation}
 \partial_BX_\lambda=(x^2,0),\quad
 \partial_{\mu_2}X_\lambda=(y^2,0),\quad
 \partial_{\mu_3}X_\lambda=(0,y^2),\quad
 \partial_{\mu_4}X_\lambda=(x,0),\quad
 \partial_{\mu_5}X_\lambda=(0,y).
\label{eq:quadratic-slice-tangents}
\end{equation}
In the ordered coefficient basis
\begin{equation}
 \bigl((1,0),(x,0),(y,0),(x^2,0),(xy,0),(y^2,0),
       (0,1),(0,x),(0,y),(0,x^2),(0,xy),(0,y^2)\bigr),
\label{eq:quadratic-coefficient-basis}
\end{equation}
the twelve columns in
\eqref{eq:quadratic-orbit-tangents}--\eqref{eq:quadratic-slice-tangents}
have determinant $2$.  Thus $D\mathcal S$ is invertible.  The analytic inverse
function theorem proves the local diffeomorphism statement.

Finally, $D\Phi\,(\rho\Phi^*X_\lambda)=\rho X_\lambda\circ\Phi$, so $\Phi$
maps oriented trajectories of $\rho\Phi^*X_\lambda$ to those of $X_\lambda$.
On a transverse section it is an analytic coordinate change with nonzero
derivative.  Conjugating a Poincar\'e map by this coordinate change preserves
the order of its displacement zero, and hence the cycle multiplicity.
\end{proof}

For $Y\in\mathcal Q_2$ and an annular neighborhood $U$ of $\GammaH$ on the
Poincar\'e sphere, let $N_{\mathcal Q_2}(Y;U)$ denote the number of isolated
limit cycles of $Y$ contained in $U$.

\begin{theorem}[Local uniform finite cyclicity in quadratic systems]
\label{thm:main}
There exist a fixed two-sided annular neighborhood $U$ of $\GammaH$, an open
neighborhood $\mathcal V$ of $X_0$ in the full coefficient space
$\mathcal Q_2$, and a finite constant $\BH$ such that
\begin{equation}
 N_{\mathcal Q_2}(Y;U)\le B_{H14^3}
 \label{eq:main-bound}
\end{equation}
for every $Y\in\mathcal V$.
\end{theorem}

The analytic work is carried out in the following local slice.  It retains all
five parameters and is stronger than what is needed after quotienting by
affine orbital equivalence.

\begin{theorem}[Source-normalized core theorem]
\label{thm:normalized-main}
There exist a fixed two-sided annular neighborhood $U_*$ of $\GammaH$, a
neighborhood $\Lambda\subset\R^5$ of the origin, and a finite constant $\BH$
such that
\begin{equation}
 \NH(\lambda;U_*)\le B_{H14^3}
 \label{eq:normalized-main-bound}
\end{equation}
for every $\lambda\in\Lambda$, where the left-hand side counts isolated limit
cycles of \eqref{eq:family} contained in $U_*$.  The neighborhood $U_*$ and
the bound are uniform in all five original parameters, including coefficient
and identity strata.
\end{theorem}

The quantifiers in these two theorems are distinct and complementary.
Theorem~\ref{thm:normalized-main} fixes the physical annulus before any of the
five source-normalized parameters varies.  Proposition~\ref{prop:quadratic-local-slice}
then supplies every remaining quadratic coefficient direction, while a pair
of nested collars makes the transfer uniform under all near-identity affine
maps.  Thus Theorem~\ref{thm:main} is a statement on an open set in the full
twelve-dimensional coefficient space, not merely a five-parameter
restriction.  No explicit or optimal value of $B_{H14^3}$ is asserted.

\subsection{Why the semihyperbolic case is not a perturbation of the
hyperbolic one}

Three difficulties interact.  First, a trajectory issued from a fixed
section can hit the next transverse section, a saddle-node separatrix, a side
of a singular box, or the boundary of the collar.  Which event occurs first
changes with the parameters.  Hence a global Poincar\'e map has no common
domain until the competing first hits have been classified.  No global return
map is assumed in advance; full-return maps are formed only on the domains
supplied by that classification.  Second, in the
source regime a fixed point of the return may escape through either end of
the period annulus.  An expansion on a bounded action interval cannot control
such zeros, and differentiation with respect to the parameters moves both
entry and exit points.  Third, when the two central saddle-node passages
coalesce, the usual lips configuration meets two distinct zero-scale faces:
the source face and a persistent-endpoint face.  An estimate valid for a
positive merger scale need not remain valid on either face.

These are failures of uniformity, not merely longer local calculations.  A
proof must construct the physical domains before differentiating the return,
control the two noncompact ends simultaneously, and cover every approach to
the corner of parameter space by theorem neighborhoods with compatible
boundaries.

\subsection{Main contributions}

The proof introduces four mathematical mechanisms and then combines them
with established local finiteness theorems.

\paragraph{Removal of the source normalization.}
The five displayed parameters are not merely a convenient subfamily.
Proposition~\ref{prop:quadratic-local-slice} computes the infinitesimal
affine--orbital action at the boundary field $B=0$ and proves that its seven
directions are transverse to the five parameter directions.  The analytic
inverse function theorem therefore identifies the normalized family with a
local slice of the full quadratic coefficient space.  A nested-collar
argument transfers the normalized cyclicity bound uniformly under the
resulting near-identity affine conjugacies.

\paragraph{Stopped-return reduction.}
We cut a fixed annular collar by transverse sections and follow each orbit
only until its first competing boundary contact.  The resulting transition
maps form a finite acyclic incidence graph after the radial section is cut
open.  Restoring that section gives finitely many full-return itineraries.
Theorem~\ref{thm:finite-words} and
Proposition~\ref{prop:exact-once} show that every isolated cycle in $U$
corresponds to exactly one zero, with the same multiplicity, of one itinerary
displacement.  Thus no cycle is lost at a changing domain boundary and no
cycle is counted twice.  This geometric statement is independent of the
subsequent analytic estimate.

\paragraph{A matched return at the noncompact source.}
For a source itinerary we first construct a common physical domain containing
the forced return, both center restrictions, and every Hadamard path used in
coefficient division.  On that domain the lower tails, upper passage, moving
cuts, and inverse hits have a uniformly controlled six-jet.  A global action
localization theorem then prevents zeros from escaping the matched cells.
Here six is a transparent proof budget rather than a number inferred from a
computer search: four local phase/parameter derivatives are combined with at
most two derivatives that keep the five original coefficients fixed.  The
finite enumeration verifies closure of this budget but does not choose it.
The exact nonlinear return is reduced, after two legitimate differentiations,
to a strictly signed curvature term.  This yields the source estimate
Theorem~\ref{thm:part-ii-source-zero}; it is not a compact-energy Melnikov
argument.

\paragraph{An exhaustive two-saddle-node decomposition.}
For an itinerary that traverses both central sectors, the existence of an
actual hyperbolic--hyperbolic connection and of a complete family of
parabolic--parabolic connections is proved in the original vector field, not
inferred from a normal form.  The geometric routing
Theorem~\ref{thm:part-iii-two-central-exhaustion} separates the no-connection case, a
positive-margin lips case, an intermediate coalescing scale, a positive
root-merger scale, and the two zero-scale faces.  After all local estimates
have been proved, Theorem~\ref{thm:two-central-physical-completion} assembles
the one-way chain ``physical completion, nonaffinity, scale assignment, zero
estimate.''  The parameter partition is half-open and fixed in advance, so
no diagonal approach is omitted and no boundary is controlled merely by
continuity from a punctured region.

The local slice identifies the coefficient space in which the final theorem
holds.  The stopped-return mechanism then converts the dynamical problem into
finitely many scalar zero problems, and the source and two-saddle-node
mechanisms supply the estimates that were missing
at the noncompact and coalescing parts of the graphic.  Bounded analytic
itineraries are handled by Weierstrass preparation, separated hyperbolic
itineraries by Mourtada's quasi-regular Hilbert theorem, and certified central
and lips configurations by the results of
Dumortier--Ilyashenko--Rousseau.  A finite specialization theorem then carries
these estimates through all proper parameter faces, and compactness is used
only after an open zero-bound neighborhood has been obtained at every point.

The order of these steps is part of the argument.  The stopped atlas first
gives the domains on which cycles are represented by displacement zeros;
the local theorems then bound those zeros; specialization carries the bounds
to the proper parameter faces; and only after open zero-bound neighborhoods
have been obtained does compactness give a finite sum.  The local slice is
used last to pass from the normalized family to the full quadratic
coefficient space.

\subsection{Relation to previous work}

Roussarie--Rousseau identify the displayed family and the $B=0$
$H_{14}^{3}$ case, and explicitly leave this semihyperbolic case outside the
finite-cyclicity result proved there
\cite[pp.~1--4 and Theorem~3.1]{RoussarieRousseau2015}.  The classification
and notation originate in the DRR program~\cite{DRR1994}.  Their normal-form
statement starts from $B_0>0$; Proposition~\ref{prop:quadratic-local-slice}
supplies the missing boundary transversality directly at $B=0$.  Consequently
Theorem~\ref{thm:main} proves finite cyclicity of the labelled $H_{14}^{3}$
graphic in the full quadratic coefficient space, rather than only along the
displayed normalization.  Recent work of
Mar\'in and Villadelprat proves precise cyclicity results for hemicycles whose
two points at infinity are hyperbolic saddles~\cite{MarinVilladelprat2025};
those hypotheses exclude the two saddle-nodes present here.

Two external local theories are used at clearly separated interfaces.
Mourtada's theorem applies only after a retained all-hyperbolic itinerary has
been realized as one admissible quasi-regular Hilbert system on common
physical sections~\cite{Mourtada2009}.  The saddle-node and lips results of
Dumortier--Ilyashenko--Rousseau apply only after the required
hyperbolic--hyperbolic and parabolic--parabolic connections and their boundary
incidence have been verified for the actual family~\cite{DIR2002}.  Neither
theory constructs the collar or decides which transition exists.  Those
applicability statements are proved here before either result is invoked.

\subsection{What is reusable, and what is specific to $H_{14}^{3}$}

Three outputs are reusable as statements rather than as a chronology of this
calculation.  First, Proposition~\ref{prop:stopped-first-hit},
Theorem~\ref{thm:finite-words}, and Proposition~\ref{prop:exact-once} form a
finite stopped-return criterion: competing first exits may replace a missing
global Poincar\'e map, provided the cut-open incidence graph is finite and
strictly ordered.  The output is an exact, multiplicity-preserving reduction
to finitely many scalar displacement germs.  Second,
Lemma~\ref{lem:matched-curvature-criterion} is the analytic core of the source
argument.  Once a common-domain matched equation has a monotone coefficient
ratio and strictly signed normalized curvature, two Rolle steps give a
two-zero bound without compactifying the phase interval.  Third,
Proposition~\ref{prop:specialization} and
Proposition~\ref{prop:conditional-assembly} give a finite-face gluing
criterion: terminating, confluent boundary assignments plus open local
zero-bound neighborhoods yield one uniform bound by a finite subcover.

What is not automatic in another family is precisely the verification of
these interfaces.  One must construct the common physical domains before
coefficient division, prove the signed curvature with constants uniform at
the open gates, and certify the complete saddle-node incidence before invoking
a lips theorem.  The half-open scale partition is then a reusable device for
assigning zero-scale faces to independent ambient estimates, not a substitute
for those estimates.

The explicit center ideal, the signed compactification formulas, the
Li\'enard--Dulac multiplier on the mixed face, and the middle/root curvature
identities are specific to family~\eqref{eq:family}.  We therefore state the
geometric and analytic interfaces in the main text and place the finite
algebraic certificates in the appendices.  Computer algebra verifies only
finite identities and derivative enumerations; it is not used to infer an
orbit, a first-hit domain, an incidence relation, or a zero bound.

\subsection{Organization and reading guide}

Sections~\ref{sec:part-i-graphic}--\ref{sec:part-i-exact-once} construct the
fixed collar and prove the exact cycle--zero reduction.  Sections
\ref{sec:regime-table}--\ref{sec:two-central-exhaustion} are a roadmap: they
introduce the regimes and scales and preview the only nontrivial geometric
routing, but do not yet invoke an unproved local estimate.  The source and
exact mixed estimates are proved in Sections
\ref{sec:part-ii-target}--\ref{sec:part-ii-handoff}; once the center ideal is
available, Section~\ref{sec:compact-analytic} treats bounded analytic
itineraries.  The hyperbolic, central, lips, middle, and root-scale estimates
follow in Sections
\ref{sec:part-iii-hyperbolic}--\ref{sec:part-iii-root}.  Only after all these
inputs are available does
Theorem~\ref{thm:two-central-physical-completion} collect the entire
two-central chain in one theorem-level interface.  Sections
\ref{sec:part-iii-package}--\ref{sec:part-i-assembly} prove the disjoint
regime assignment, boundary specialization, compact assembly, and the main
theorem.  The complete boundary audit is concentrated in
Table~\ref{tab:specialization-edges} and
Proposition~\ref{prop:ambient-edge-neighborhoods}: they distinguish the
unique owner of a face from the finite family of estimates that makes its
neighborhood open in the original parameters.

The appendices are proof supplements, not part of the first reading path.
Every fact needed to know what a theorem says, why it applies, and how it
enters the main proof is stated in the body.  The appendices retain signed
chart formulas, long finite recurrences, joint endpoint-clock estimates, and
exact algebraic certificates.  A first reading may therefore follow Sections
\ref{sec:introduction}, \ref{sec:three-box-model},
\ref{sec:part-i-exact-once}, \ref{sec:regime-table}, the statements of the
local zero theorems, and finally Sections
\ref{sec:part-iii-package}--\ref{sec:part-i-assembly}, without consulting an
appendix.

\section{The compactified hemicycle and a fixed counting collar}
\label{sec:part-i-graphic}

We first identify the oriented compactified graphic and fix a physical collar,
radial cut, and finite face skeleton.  These parameter-independent objects
provide the common domain for all stopped first-hit relations below.  We use
the standard Poincar\'e compactification and planar-sector conventions of
Dumortier--Llibre--Art\'es~\cite[Chapters~3 and~6]{DumortierLlibreArtes2006}.

The first integral \eqref{eq:first-integral} is strictly convex on $y>-1$:
its Hessian is
\[
 D^2H=\begin{pmatrix}1&0\\0&(1+y)^{-2}\end{pmatrix},
\]
and its only critical point is the origin.  Since $H$ tends to $+\infty$
both at $y=-1$ and at affine infinity, every level $H=h>0$ is one periodic
oval.  As $h\to\infty$, these ovals converge on the Poincar\'e sphere to
\[
 \GammaH=L_-\cup\{p_+\}\cup E_+\cup\{q\}\cup E_-\cup\{p_-\},
\]
where $L_-=\{y=-1\}$, $p_\pm$ are the horizontal endpoints, $q$ is the
positive vertical point, and $E_\pm$ are the two upper equatorial arcs.

The orientation is part of the label.  Along $L_-$ one has $\dot x=1$, so
$p_-\to p_+$.  In the upper chart $y=1/r$, $x=w/r$, after the positive
degree-two desingularization,
\[
 \dot r=-rw-r^2w,\qquad \dot w=-w^2-r(1+w^2).
\]
Thus the upper arcs are oriented $p_+\to q\to p_-$.  Equivalently, in the
positive endpoint coordinates $x=1/r$, $z=1+y$,
\[
 r'=-r^3(1-z),\qquad z'=z,
\]
whereas at the negative endpoint $x=-1/r$,
\[
 r'=r^3(1-z),\qquad z'=-z.
\]
These formulas fix the cyclic order
\begin{equation}
 L_-\longrightarrow p_+\longrightarrow E_+\longrightarrow q
 \longrightarrow E_-\longrightarrow p_-\longrightarrow L_- .
\label{eq:legacy-3-1}
\end{equation}

\begin{figure}[htbp]
\centering
\begin{tikzpicture}[
  scale=0.95,
  every node/.style={font=\small},
  flow/.style={-{Latex[length=2.1mm]},thick},
  point/.style={circle,fill=black,inner sep=1.7pt},
  collar/.style={gray!70,dashed,thin}
]
\coordinate (pm) at (-4,0);
\coordinate (pp) at (4,0);
\coordinate (q) at (0,3.2);

\path[fill=gray!9,even odd rule]
  (-4.48,-0.58) -- (4.48,-0.58)
  .. controls (4.28,2.30) and (2.25,3.76) .. (0,3.82)
  .. controls (-2.25,3.76) and (-4.28,2.30) .. (-4.48,-0.58) -- cycle
  (-3.56,0.58) -- (3.56,0.58)
  .. controls (3.00,1.78) and (1.46,2.53) .. (0,2.58)
  .. controls (-1.46,2.53) and (-3.00,1.78) .. (-3.56,0.58) -- cycle;
\draw[collar]
  (-4.48,-0.58) -- (4.48,-0.58)
  .. controls (4.28,2.30) and (2.25,3.76) .. (0,3.82)
  .. controls (-2.25,3.76) and (-4.28,2.30) .. (-4.48,-0.58) -- cycle;
\draw[collar]
  (-3.56,0.58) -- (3.56,0.58)
  .. controls (3.00,1.78) and (1.46,2.53) .. (0,2.58)
  .. controls (-1.46,2.53) and (-3.00,1.78) .. (-3.56,0.58) -- cycle;
\node[font=\footnotesize,fill=gray!9,inner sep=1.2pt] at (3.72,3.05)
  {fixed collar $U$};

\draw[very thick] (pm) -- (pp);
\draw[flow] (-2.8,0) -- (-1.4,0);
\draw[flow] (1.4,0) -- (2.8,0);
\draw[very thick] (pp) .. controls (3.4,2.1) and (1.8,3.2) .. (q);
\draw[very thick] (q) .. controls (-1.8,3.2) and (-3.4,2.1) .. (pm);
\draw[flow] (3.05,1.85) -- (2.45,2.37);
\draw[flow] (-2.45,2.37) -- (-3.05,1.85);

\draw[gray!62,thin]
  (-3.43,0.20) .. controls (-3.02,2.02) and (-1.58,2.91) .. (0,2.96)
  .. controls (1.58,2.91) and (3.02,2.02) .. (3.43,0.20)
  .. controls (1.75,0.36) and (-1.75,0.36) .. cycle;
\draw[gray!58,thin]
  (-3.18,0.38) .. controls (-2.72,1.86) and (-1.35,2.70) .. (0,2.74)
  .. controls (1.35,2.70) and (2.72,1.86) .. (3.18,0.38)
  .. controls (1.58,0.54) and (-1.58,0.54) .. cycle;
\draw[gray!62,-{Latex[length=1.5mm]},thin] (-0.55,0.31) -- (0.55,0.31);
\draw[gray!62,-{Latex[length=1.5mm]},thin] (2.55,1.88) -- (2.12,2.20);

\node[point,label=below left:$p_-$] at (pm) {};
\node[point,label=below right:$p_+$] at (pp) {};
\node[point,label=above:$q$] at (q) {};
\node[below=4pt] at (0,0) {$L_- = \{y=-1\}$};
\node[right] at (2.8,2.55) {$E_+$};
\node[left] at (-2.8,2.55) {$E_-$};
\node[align=center,fill=white,inner sep=1.5pt] at (0,1.32)
  {representative high-energy\\periodic ovals of $H$};
\end{tikzpicture}
\caption{The labelled source graphic $\GammaH$ and its orientation.  The
heavy curve is the limiting graphic.  The gray band between the two dashed
curves is the fixed two-sided annular collar $U$; in particular, the finite
center lies in the complementary component inside the inner collar side.
The two light closed curves are representative high-energy levels of the
unperturbed first integral on that side of the graphic.  Labels and flow
orientation are exact; distances, curvature, and collar width are schematic.
The endpoint and upper-chart coordinates used in the proof are given in
Sections~\ref{sec:part-i-graphic} and \ref{sec:part-i-boxes}.}
\label{fig:source-graphic}
\end{figure}
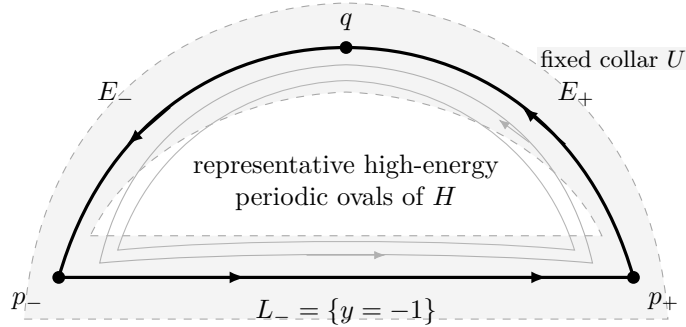

We now choose the counting region before perturbing.  Put $z=1+y$ and fix
$X_L>1$.  For $z_L>0$ and the parameter ball sufficiently small,
\begin{equation}
 B_L=\{|x|\le X_L,\ |z|\le z_L\},\qquad
 \dot x=1-z+Bx^2+\mu_2(z-1)^2+ax>\frac12,
\label{eq:legacy-3-2}
\end{equation}
where
\begin{equation}
 a=\mu_4+B\mu_5,\qquad c=(1-2B)\mu_5.
\label{eq:legacy-3-3}
\end{equation}
The proper radial cut is
\begin{equation}
 \Sigma=\{x=-X_L,\ |z|<z_L\}.
\label{eq:legacy-3-4}
\end{equation}
The finite-line strip is two-sided: it includes both $z>0$ and $z<0$.
Indeed the former invariant line may be crossed because
\begin{equation}
 \dot z\big|_{z=0}=\mu_3-c
   =\mu_3-(1-2B)\mu_5.
\label{eq:legacy-3-5}
\end{equation}
The strict coordinate is $x$; the source energy is never evaluated on a
perturbed orbit below $z=0$.

Choose disjoint fixed neighborhoods of $p_+$, $q$, and $p_-$, and cover the
remaining compact regular arcs by finitely many analytic flow boxes.  Their
constant flow-coordinate levels are transverse cuts; the two transverse
levels of each flow box are collar sides and are not promoted to passage
sections.  Gluing these finitely many pieces gives a fixed two-sided open
annulus $U$, disjoint from the finite center and the lower equator.  All five
parameters in \eqref{eq:family} remain active on this same $U$.

\begin{proposition}[Fixed collar and transverse sections]
\label{prop:fixed-skeleton}
There is a finite list $\mathcal F_0$ of cross-cuts, collar faces, box faces,
and chart-overlap faces, fixed before gate classification.  Every cross-cut
has a uniform signed transverse normal for small $\lambda$.  Every collar
face retains its complete tangency equation $X_\lambda f=0$ and is treated as
an exit side.  Pullback through the finite endpoint and upper-vertical
resolutions produces another finite face list $\mathcal F$, enlarged by its
nonempty face intersections and tangency strata.  Its members are exactly:
the radial and regular overlap cuts; the incoming, outgoing, previous-side,
and collar faces of the three singular boxes; the prepared endpoint root
lines and their strong and central gates; the upper faces
$r=0,r=r_*,w=\pm\delta$, the curves $F=0,Q=0$, and their root gates; and the
corners, invariant axes, and backward corner orbits determined by these
faces.  No later zero theorem adds a stopping face.
\end{proposition}

\begin{remark}
The skeleton is chosen once for a sufficiently small ball in all five
parameters.  Cross-cut normals retain their sign throughout that ball, while
every collar tangency remains an explicit exit equation.  In particular,
collar sides are never silently promoted to passage sections.
\end{remark}

\begin{proof}
Take $B_L$, fixed endpoint rectangles, a fixed upper rectangle, and finitely
many flow-box overlap levels as above.  Compactness of each selected
cross-cut and the strict source normal give a positive margin that persists
for small parameters.  No such claim is made for a collar side; for example
$X_0(z-z_L)=xz_L$ vanishes at $x=0$, so this tangency is kept explicitly.
The endpoint incidence, source-core, and parameter-dominated blow-ups are
finite maps with finitely many coordinate faces.  Pulling back the already
fixed equations therefore preserves finiteness.  Later invariant axes are
introduced only inside their named normal-form boxes and meet another box
through a preassigned member of $\mathcal F_0$.
\end{proof}

\section{Stopped transition maps: a model and the general principle}
\label{sec:three-box-model}

The following elementary model isolates the order and no-Zeno mechanism later
verified uniformly in the $H_{14}^{3}$ charts.  It starts with a cut-open
annular collar, finitely many singular boxes, and their physical stopping faces.

We begin with the small model that governs the full construction.  Cut an
annular collar at one proper radial section $\Sigma$.  Suppose the cut-open
collar contains three disjoint singular boxes $B_+$, $B_q$, and $B_-$,
joined in this order by regular flow strips.  Each box has one incoming arc,
one next-cut arc, finitely many singular gate arcs, and collar sides.  The
question is not initially whether a return map has a zero.  It is whether a
point on an incoming arc reaches the next cut before any competing boundary
component.

For a point $s$ on an incoming interval $I$, follow its orbit only until its
first physical boundary contact.  The outcome is either a transverse next
cut, a named singular gate, the preceding side of the box, or a collar exit.  Boundaries
between outcome intervals are called \emph{divider points}.  They are backward
first hits of fixed corners or tangencies, or intersections with a stable or
center half-branch.  When two divider labels coincide, the intervening
interval is a collapsed interval and carries no phase point.

\begin{proposition}[Three-box stopped-return principle]
\label{prop:three-box-model}
Assume that every regular strip has a strict flow coordinate, each singular
box has finitely many equilibria and local sectors, and every nonexiting orbit
in a singular box tends to one of those sectors.  Assume also that each
labelled corner orbit and invariant half-branch meets the incoming arc at most
once.  Then every incoming arc is divided by finitely many labelled points
into open intervals with one first outcome.  Every through outcome is an
order-preserving first-hit diffeomorphism.  After $\Sigma$ is deleted, through
edges strictly increase the box/cut index; hence there are finitely many
cut-open paths and finitely many full-return itineraries after $\Sigma$ is restored.
\end{proposition}

\begin{proof}
For an open stopping face $J$, the set of points whose first contact lies in $J$ is
open by transversality and continuous dependence.  A boundary point of this
set cannot first hit the interior of a transverse stopping face.  If its contact time
is finite, it therefore lies on a listed corner, tangency, or invariant side.
If the contact time diverges, the planar limit set lies in the finite
equilibrium/sector list.  The one-intersection hypothesis assigns at most one
divider to each label.  Removing these finitely many points leaves intervals
on which the outcome label is locally constant, hence constant.  Planar orbit
uniqueness makes a through map order preserving.

Number the successive physical cuts in the positive collar orientation.  A
through map goes to the next cut; all other outcomes are terminal.  The cut
number is therefore a strict integer rank, so the cut-open directed graph is
finite and acyclic.  Restoring the deleted cut closes only complete laps.  The
full H14 proof below verifies these hypotheses uniformly through every root
and gate collision; the model explains why stopping must precede return-map
formation.
\end{proof}

\subsection{From stopped paths to finite sums}
\label{sec:proof-architecture}

For an isolated cycle in the fixed collar, the reductions occur in a fixed
order: stopped first hits, a retained full-return itinerary, a local zero theorem,
and finally the compact finite sum.  In particular, compactness and full-return
equations are used only after the stopped atlas has supplied their domains.

Figure~\ref{fig:stopped-itinerary} shows the elementary mechanism.  The solid
chain becomes a return equation only after all competing first contacts have
been excluded.

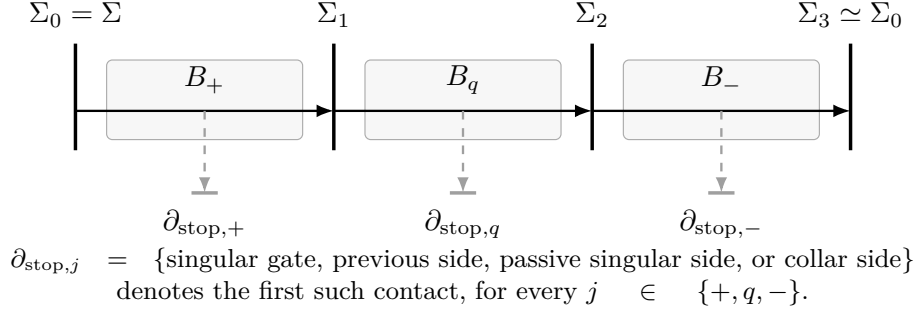
\begin{figure}[htbp]
\centering
\resizebox{.98\textwidth}{!}{%
\begin{tikzpicture}[
  every node/.style={font=\small,align=center},
  section/.style={very thick},
  box/.style={draw=gray!65,rounded corners=2pt,fill=gray!6},
  flow/.style={-{Latex[length=2.1mm]},thick},
  stopflow/.style={-{Latex[length=1.9mm]},thick,dashed,gray!75}
]
\coordinate (s0) at (0,1.20);
\coordinate (s1) at (3.15,1.20);
\coordinate (s2) at (6.30,1.20);
\coordinate (s3) at (9.45,1.20);

\draw[section] (0,0.72) -- (0,2.02);
\draw[section] (3.15,0.72) -- (3.15,2.02);
\draw[section] (6.30,0.72) -- (6.30,2.02);
\draw[section] (9.45,0.72) -- (9.45,2.02);
\node[above=2pt] at (0,2.02) {$\Sigma_0=\Sigma$};
\node[above=2pt] at (3.15,2.02) {$\Sigma_1$};
\node[above=2pt] at (6.30,2.02) {$\Sigma_2$};
\node[above=2pt] at (9.45,2.02) {$\Sigma_3\simeq\Sigma_0$};

\draw[box] (0.38,0.85) rectangle (2.77,1.82);
\draw[box] (3.53,0.85) rectangle (5.92,1.82);
\draw[box] (6.68,0.85) rectangle (9.07,1.82);
\node at (1.575,1.58) {$B_+$};
\node at (4.725,1.58) {$B_q$};
\node at (7.875,1.58) {$B_-$};

\draw[flow] (s0) -- (s1);
\draw[flow] (s1) -- (s2);
\draw[flow] (s2) -- (s3);

\foreach \x/\j in {1.575/+,4.725/q,7.875/-}{
  \draw[stopflow] (\x,1.20) -- (\x,0.20);
  \draw[very thick,gray!75] (\x-0.16,0.20) -- (\x+0.16,0.20);
  \node[below=2pt] at (\x,0.20) {$\partial_{\mathrm{stop},\j}$};
}
\node[text width=14.2cm,font=\footnotesize] at (4.725,-0.82)
  {$\partial_{\mathrm{stop},j}
   =\{\text{singular gate, previous side, passive singular side,
   or collar side}\}$ denotes the first such contact, for every
   $j\in\{+,q,-\}$.};
\end{tikzpicture}%
}
\caption{A stopped itinerary through the three singular boxes.  The diagram
records the common first-contact rule, not the detailed local phase portraits:
each solid arrow is retained only for points reaching the next section before
the corresponding $\partial_{\mathrm{stop},j}$.  Thus the alternatives listed
below the diagram apply to every singular box rather than being assigned
one-to-one to particular boxes.  A dashed arrow stops the current local map
but need not terminate the physical orbit.  A divider collision deletes the
intervening interval, while overlap charts identify the same physical state.
Regular strips are suppressed, and all flow arrows follow the positive
orientation of the hemicycle.}
\label{fig:stopped-itinerary}
\end{figure}

There are three logically distinct classification stages.  First, the stopped
transition analysis rejects candidates that do not carry a retained full-return itinerary.  Second,
each retained itinerary enters one relative-interior analytic regime.  Third,
proper specializations descend inside a finite specialization graph.  In
particular, an ordinary coefficient face need not be an identity; it is sent
to its terminal minimal-face regime.  Only an identity interval, including an all-zero
coefficient stratum that produces such an identity, contributes no isolated
member.

\FloatBarrier
\section{Local dynamics at the singular points at infinity}
\label{sec:part-i-boxes}

The fixed skeleton leaves three singular neighborhoods unresolved.  We derive
representative equations at the horizontal endpoints and the upper vertical
point, classify their retained passages, and obtain a finite list of local
passage types with physical sections and descent in the original parameter space.

The local equations are derived from the full family, not from a reduced
representative.  At the positive horizontal endpoint, $x=1/r$, $y=z-1$,
\begin{equation}
 \dot r=-r\{B+ar+r^2[1-z+\mu_2(z-1)^2]\},\qquad
 \dot z=z+r[\mu_3(z-1)^2+c(z-1)].
\label{eq:legacy-3-6}
\end{equation}
At the negative endpoint, $x=-1/r$,
\begin{equation}
 \dot r=r\{B-ar+r^2[1-z+\mu_2(z-1)^2]\},\qquad
 \dot z=-z+r[\mu_3(z-1)^2+c(z-1)].
\label{eq:legacy-3-7}
\end{equation}
Solving the transverse zero graph analytically and restricting the radial
numerator to it gives
\begin{equation}
 P_+(r,\lambda)=B+ar+r^2U_+(r,\lambda),\qquad
 P_-(r,\lambda)=B-ar+r^2U_-(r,\lambda),
\label{eq:legacy-3-8}
\end{equation}
with $U_\pm(0,0)\ne0$.  Degree-two preparation therefore gives only a
simple hyperbolic root, a double saddle-node root, or the critical
multiplicity-three layer $B=a=0$ after the persistent factor $r$ is restored.
There is no uniform factorization as $r^3$ away from that layer.

For a prescribed finite differentiability order, the joint endpoint
normal form may be written, after the required local time reversal, as
\begin{equation}
 x'=q_\alpha(x),\qquad y'=-y,
\label{eq:legacy-3-9}
\end{equation}
with at most three weak roots.  The function
\begin{equation}
 \mathcal L_\alpha(x,y)
   =\int_0^x q_\alpha(v)\,dv-\frac{y^2}{2}
\label{eq:legacy-3-10}
\end{equation}
satisfies
\begin{equation}
 \dot{\mathcal L}_\alpha=q_\alpha(x)^2+y^2.
\label{eq:legacy-3-11}
\end{equation}
It excludes recurrence inside the fixed endpoint slab.  Root lines, the
strong axis, and backward corner orbits form its finite list of dividers.
The normalizer transports fixed model sections; it is not used to define a
global analytic stratification.

At the upper vertical point use $y=1/r$, $x=w/r$.  Set
\begin{equation}
 F(r,w)=(1+r)w+\mu_3+rc,
\label{eq:legacy-3-12}
\end{equation}
\begin{equation}
 Q(r,w)=\mu_2-\mu_3w+(B-1)w^2
       +r[-1+(a-c)w-w^2].
\label{eq:legacy-3-13}
\end{equation}
Then $\dot r=-rF$ and $\dot w=Q$.  Choose fixed $\delta,r_*>0$ so that
\begin{equation}
 Q(r,\pm\delta)<0,\qquad
 F(r,\delta)>0,\qquad F(r,-\delta)<0
\quad(0\le r\le r_*).
\label{eq:legacy-3-14}
\end{equation}
Thus $w=\delta$ is the physical entry, $w=-\delta$ the next cut, and
$r=r_*$ the collar side.  On $F=0$ the graph
\[
 w_s(r)=-\frac{\mu_3+rc}{1+r}
\]
has $S(r):=Q(r,w_s(r))$ with $S'(r)\le-1/2$.  Moreover
\begin{equation}
 F_w=1+r>0,\qquad Q_r<-\frac12,
\quad \dot F|_{F=0}=(1+r)Q,
\quad \operatorname{sgn}\dot Q|_{Q=0}=\operatorname{sgn}F.
\label{eq:legacy-3-15}
\end{equation}
These signs give a directed four-cell graph.  Every nonexiting orbit is
coordinatewise monotone until it reaches one of finitely many equilibria;
there is no internal recurrent set.

On the equator the gate equation is
\begin{equation}
 E(w,\lambda)=\mu_2-\mu_3w+(B-1)w^2=0,
\label{eq:legacy-3-16}
\end{equation}
so there are at most two equatorial roots.  The graph $F=0$ supplies at most
one interior root, and its determinant is negative.  The only joint
collisions are the equatorial discriminant collision, the equatorial/interior
collision, and the source core.  In the source core,
\begin{equation}
 r=\theta^2,\quad w=\theta W,
\quad \mu_2=\theta^2\bar m_2,
\quad \mu_3=\theta\bar m_3,
\label{eq:legacy-3-17}
\end{equation}
and the exceptional equation has $\dot W=-1-W^2/2$ at the source point.
Thus $W$ is a direct passage coordinate, while
$\theta^2\bar m_2=\mu_2$ and $\theta\bar m_3=\mu_3$ keep the lifted flow on
the original five-parameter fiber.

The signed substitutions, source cancellations, and auxiliary endpoint-scale
identities used in this section are collected in
Appendix~\ref{app:chart-endpoint-variants}.  The next proposition is the
body-level interface for the simultaneous degeneration of an endpoint root
and its section coordinate.  Appendix~\ref{app:endpoint-joint-uniformity}
contains the cell-by-cell fixed-fiber clock estimates behind it.

\begin{proposition}[Joint endpoint first-hit uniformity]
\label{prop:endpoint-joint-uniformity}
Fix a finite differentiability order \(K\).  After shrinking the two fixed
endpoint slabs, there is a finite signed resolved cover, joint in the section
coordinate \(s\) and the prepared root coefficients
\((B_{1,\xi},A_\xi)\), with the following properties.
\begin{enumerate}[label=\textup{(\roman*)},leftmargin=2.8em]
\item Every labelled endpoint root is assigned to a simple, \(D_1\), or
  \(D_2\) cell, while a positive cofactor root reached by the orbit is a
  named gate or no-passage face.
\item On each retained through cell the physical first hit has the
  fixed-fiber factorization
  \[
    R_{\rm out}\circ D_\xi\circ R_{\rm in};
  \]
  the two regular factors have uniformly bounded \(C^K\) jets, and
  overlapping cells represent the same physical map.
\item If \(\alpha\) denotes all resolved endpoint parameters, then for every
  \(L\ge0\) the singular factor satisfies
  \begin{equation}
   \lim_{\varepsilon\downarrow0}
   \sup_{\substack{\alpha\ \mathrm{in\ the\ finite\ cover}\\0<s<\varepsilon}}
   s^{-L}\bigl|(s\partial_s)^jD_\xi(s,\alpha)\bigr|=0,
   \qquad 0\le j\le K.
   \label{eq:endpoint-joint-flatness}
  \end{equation}
\end{enumerate}
Consequently the retained endpoint passages have uniform \(C^K\) graph
representatives through simultaneous root--section collisions on the same
original parameter fiber.  A stopped gate is not crossed, and no argument
requires differentiation of the inverse of an exponentially flat map.
\end{proposition}

\begin{proof}
The physical radial numerator and the DIR model numerator are related by a
triangular matrix of analytic units, so their labelled simple, \(D_1\), and
\(D_2\) roots have the same finite resolved incidence.  In the exact endpoint
clock, the joint variables \((s,B_{1,\xi},A_\xi)\) have weights \((1,1,2)\).
The phase-dominant, balanced cofactor, discriminant-gap, and hierarchical
\(D_1\)-to-\(D_2\) cells give respectively exponential lower bounds of order
\(s^{-2}\), \((\delta^2(|u|+g))^{-1}\), and \((ds)^{-1}\), with only
algebraic losses under fixed-fiber Euler differentiation.  Exponential
domination gives \eqref{eq:endpoint-joint-flatness}; regular section changes
are unit-equivalent to \(s\).  The cell estimates and overlap calculation are
given in Appendix~\ref{app:endpoint-joint-uniformity}.
\end{proof}

\begin{proposition}[Finite local phase-portrait alternatives]
\label{prop:finite-singular-alphabet}
After a finite resolved cover, every retained vertex in the endpoint and
upper slabs is a direct regular passage, a separated hyperbolic saddle, or a
simple saddle-node with fixed sector type.  Same-sign node cores,
stable-center sectors, wrong orientations, and collar exits are terminal
labels.  The source core is the direct passage \eqref{eq:legacy-3-17}; no unresolved
double-zero vertex remains.
\end{proposition}

\begin{remark}
The classification is uniform on the finite endpoint-incidence and
upper-vertical resolved cover.  Root and gate collisions remain in their
joint boxes; wrong orientations, stable-center sides, nodes, and collar
contacts are terminal first outcomes.
\end{remark}

\begin{proof}
Equation \eqref{eq:legacy-3-8} gives the complete endpoint root list.  For the upper
box, $F=0$ has at most one interior singular point because $S'<0$, and its
linear determinant is negative.  On $r=0$, equation \eqref{eq:legacy-3-16} is
quadratic.  A single zero eigenvalue has a nonzero quadratic center term;
simultaneous vanishing forces $\mu_2=\mu_3=0$ in the normalized chart and
therefore belongs to the source core \eqref{eq:legacy-3-17}.  The sign of the two
nonzero eigenvalues separates saddles from nodes.  The fixed model sections
and the monotonicity identities above then give the stated retained and
terminal sector list.
\end{proof}

\section{Uniform first-hit maps on fixed sections}
\label{sec:part-i-atlas}

We next decide the first outcome for every point of every incoming interval,
including limits with unbounded flight time.  The fixed face skeleton and the
finite local phase portraits yield finitely many physical first-hit relations
on common domains; no full-return map is formed at this stage.

\begin{definition}[Stopped first-hit relation]
Fix one block and one incoming physical section.  For an entry point, follow
the positive orbit until its first contact with the union of all transverse
exit sections, competing sides, invariant sides, and collar boundaries.  If
the orbit remains in the block and converges to a specified equilibrium,
invariant axis, saddle-node sector, or node sector, that limiting sector is
recorded as its terminal outcome.  The resulting partially defined relation
is the \emph{stopped first-hit relation}.  It is never continued through a
singular limit merely to manufacture a remote landing coordinate.
\end{definition}

The stopping set contains the next transverse cut, every competing transverse
side, invariant root/axis/center sides inside their owning box, previous-side
faces, fixed corners, and collar sides.  Resolved chart overlaps identify the
same physical point and are not competing outcomes.  This convention is what
makes the first-hit relation stable when flight time tends to infinity.

\begin{definition}[First-contact ownership]
\label{def:first-contact-owner}
The face list $\mathcal F$ is partially ordered by inclusion.  If an orbit
first reaches several incident faces at the same time, its owner is the
unique minimal face of $\mathcal F$ containing that contact.  Thus a corner
or tangency stratum owns the contact before either adjacent open face.  A
resolved chart overlap is an identification of this same physical owner,
not a second outcome.  If two backward divider labels coincide, the open
interval between them is empty and has no owner.
\end{definition}

The finite outcome alphabet is therefore completely geometric:
\begin{equation}
 \mathcal A_{\rm stop}=
 \{\mathrm{next},\mathrm{gate},\mathrm{previous},\mathrm{center},
   \mathrm{node},\mathrm{wrong},\mathrm{exit},\mathrm{collapse}\}.
\label{eq:stopped-outcome-alphabet}
\end{equation}
Here \emph{next} is the only retained through outcome.  The labels
\emph{gate} and \emph{center} include their prepared simple-root,
saddle-node, invariant-axis, and stable-center variants; the finite local
phase portrait decides whether the corresponding physical candidate is
subsequently represented by a new next-cut component or is terminal.
The remaining five labels are terminal for the current candidate, and
\emph{collapse} carries no phase point.

\begin{proposition}[Uniform first-hit decomposition]
\label{prop:stopped-first-hit}
On every fixed regular strip, endpoint slab, and upper slab, and on every
resolved parameter face attached to it, the incoming interval has a uniformly
bounded finite divider set.  Every complementary open interval has exactly
one first-hit outcome.  A next-cut interval carries an order-preserving
$C^K$ first-hit diffeomorphism for any prescribed finite $K$.  A gate,
previous-side, stable-center, node-core, wrong-orientation, or collar label is
terminal unless the finite local portrait has already supplied a separately
labelled next-cut component.  Simultaneous first contacts are assigned by
Definition~\ref{def:first-contact-owner}; coincident divider labels delete the
interval between them.  Consequently every entry point has one physical
owner and chart overlaps never duplicate an outcome.
\end{proposition}

\begin{remark}
For each prescribed finite differentiability order $K$, the conclusion is
uniform on every closed resolved endpoint, upper, source-core, and transported
normal-form face.  Unbounded flights stop at named singular outer sections,
label collisions delete the interval between them, and collar contacts are
exits.
\end{remark}

\begin{proof}
The ordinary strips have an analytic coordinate $\theta$ with
$X_\lambda\theta\ge c>0$.  Their only outcome boundaries are backward
orbits of fixed cut/collar corners.  In an endpoint slab, \eqref{eq:legacy-3-11} is a strict
Lyapunov function away from the finitely many roots, and the invariant root
lines and strong axis give a finite sector list.  In the upper slab,
\eqref{eq:legacy-3-14}--\eqref{eq:legacy-3-15} give strict outer normals and a directed four-cell graph.  The
only tangency on $r=r_*$ is a strict local minimum of $r$, hence cannot be a
first outer contact from the interior.  The source core has
$-\dot W\ge1/2$ after shrinking.  Separated saddle, node, and simple
saddle-node boxes use fixed model sections with certified normal or Lyapunov
margins.

For one open stopping face $J$, let $E_J$ be the entry points whose first contact is in
$J$.  This set is relatively open.  If $s_*$ lies on its boundary and its
contact time stays bounded, the limiting contact is a fixed corner,
tangency, or outer section of a named singular box.  If the time diverges,
the strict coordinates above force the orbit into a named singular sector;
the regular flight has already been stopped at that box's outer section.
Thus $s_*$ is the stopped backward intersection of one element of a finite
label list.  Each such orbit meets the one-way entry face at most once.
There are therefore finitely many divider points, uniformly through the
joint root and gate boxes.  At an endpoint, the uniformity in the simultaneous
root--section limit and the separation of through, gate, and no-passage cells
are precisely Proposition~\ref{prop:endpoint-joint-uniformity}.

At a simultaneous contact, closure of $\mathcal F$ under intersections gives
the minimal face in Definition~\ref{def:first-contact-owner}.  Hence the two
adjacent open outcomes meet at one boundary owner rather than producing two
first hits.  The finite local alphabet
\eqref{eq:stopped-outcome-alphabet} contains every such face: endpoint
collisions are among the prepared root/strong-axis strata, upper collisions
are intersections of $F=0$, $Q=0$, and the four fixed outer faces, and an
ordinary-strip collision is a fixed corner or collar tangency.

On a complementary interval the outcome is locally constant, hence constant.
Uniqueness of planar trajectories preserves order, and the implicit hit
equation gives the claimed $C^K$ map on regular strips; the endpoint graph
statement follows from Proposition~\ref{prop:endpoint-joint-uniformity}.
The resolved phase--parameter flow
preserves $\theta^2\bar m_2=\mu_2$ and
$\theta\bar m_3=\mu_3$ in the source core, so the lifted relation descends to
the same original parameter fiber.  On an overlap, orbit uniqueness gives
\begin{equation}
 P_{\chi'}=\psi_{\rm out}\circ P_\chi\circ\psi_{\rm in}^{-1};
\label{eq:legacy-4-1}
\end{equation}
both sides stop at the same physical contact and therefore carry the same
first-hit outcome.
\end{proof}

\section{Finite full-return itineraries}
\label{sec:part-i-words}

The stopped relations are local and may share divider points.  Gluing them on
fixed cuts, rejecting terminal sectors, and proving a strict global order
produces the finite closed family of retained full-return itineraries used by
the later analytic estimates.  An itinerary is an ordered composition of
local first-hit maps between consecutive sections.  In the derivative
estimates below, a finite ordered composition of differential operators is
occasionally called a word; this carries no symbolic-dynamics meaning.

Let $C_0=\Sigma,C_1,\ldots,C_m$ be the fixed cuts in the oriented order
\eqref{eq:legacy-3-1}.  Refine each $C_i$ by the union of the divider sets from its two
adjacent stopped relations.  On each finite ordering chart write
\begin{equation}
 e_{\pi(1)}\le\cdots\le e_{\pi(n)},\qquad \pi\in\mathfrak S_n.
\label{eq:legacy-4-2}
\end{equation}
Equality in \eqref{eq:legacy-4-2} records a collapsed interval.  Images and inverse images of
divider points under an order-preserving through map add only finitely many
labels.  This finite sorting operation is performed on fixed one-dimensional
cuts; it does not assert that a zero set of arbitrary transported
finite-smooth faces has finitely many components.

Before full-return itineraries are formed, we remove the terminal outcomes.  A wrong-oriented
upper cell cannot reach the next larger $-w$ level.  A stable-center sector
tends to its singular point or returns to a previous face.  A same-sign node
has a Lyapunov core which can be entered only for a sink, or left only for a
source; the core is not an incoming-to-next-cut passage.  A regular corridor
outside that node core is retained precisely when its strict coordinate joins
consecutive cuts.  Collar exits leave $U$.  These are geometric rejections,
not zero-count conclusions.

\begin{theorem}[Finite full-return itinerary theorem]
\label{thm:finite-words}
The retained primitive section components form a finite directed graph.  With
$C_0$ deleted, every retained edge goes from a subinterval of $C_i$ to a
subinterval of $C_{i+1}$ and therefore strictly increases
\begin{equation}
 I(v)=i\qquad(v\subset C_i).
\label{eq:legacy-4-3}
\end{equation}
The cut-open graph is acyclic and has finitely many paths.  Restoring
$\Sigma$ gives a finite family of stopped full-return itineraries, each with
at most $m+1$ primitive first-hit factors.  Every degeneration of a retained
itinerary is exactly one of: a minimal corner or tangency owner; a prepared
root, gate, or sector owner from $\mathcal A_{\rm stop}$; another labelled
boundary itinerary; a terminal previous-side, node, wrong-orientation, or
collar outcome; or a collapsed interval.  There is no infinite-switching or
Zeno alternative, and overlap charts do not create additional paths.
\end{theorem}

\begin{remark}
The same finite directed graph works on the complete resolved parameter ball,
including root, divider, gate, and overlap limits.  Every proper limit is an
adjacent labelled itinerary, a no-passage face, or a collapsed interval; no
remote boundary itinerary is introduced.
\end{remark}

\begin{proof}
Proposition~\ref{prop:stopped-first-hit} gives a uniformly finite vertex and
edge list.  By construction, only a next-cut outcome is retained, so \eqref{eq:legacy-4-3}
is strict.  A directed path therefore visits at most the number of retained
vertices and, more sharply, at most one component on each of the successive
cuts $C_0,\ldots,C_m$.  Thus a restored full lap has at most $m+1$ primitive
factors.  Root or divider collisions do not enlarge the graph: they identify
labels or delete an interval, and Definition~\ref{def:first-contact-owner}
assigns an incident corner once.

For closure, take a convergent sequence of segments.  Fix a box, ordering,
and first-hit outcome after passing to a subsequence.  Bounded flight times give
the same transverse hit or a named corner limit.  Unbounded times enter a
named singular isolating block, where the regular relation was already
stopped.  At a nonsingular accumulation point one flow box contains the whole
tail; at a singular point one of finitely many sectors contains it.  Hence
there is no infinite switching, or Zeno, alternative.  The finite incidence
complex is exhaustive.
\end{proof}

\section{Exact reduction from limit cycles to displacement zeros}
\label{sec:part-i-exact-once}

To count physical cycles without omission or duplication, we combine the
stopped itinerary family with the positive radial cut.  Annular intersection
then gives an exact cycle--zero correspondence.  For a retained itinerary
\(\omega\), intersect the domains of its successive physical first hits
before composing them, denote the resulting interval on \(\Sigma\) by
\(I_{\omega,\lambda}\), and write
\begin{equation}
 D_\omega(s,\lambda)=P_\omega(s,\lambda)-s,
 \qquad s\in I_{\omega,\lambda}.
\label{eq:itinerary-displacement}
\end{equation}
No factor in \(P_\omega\) is continued through an exit or an infinite-time
singular contact.

At a divider met by a closed regular trajectory, choose a transverse arc at
another point of that trajectory.  The ordinary analytic Poincar\'e germ on
that arc transports to an open germ \(\widehat P\) on \(\Sigma\).  All
incident half-open itinerary maps are restrictions of this one germ.  In the
statement below, isolation at a divider always means isolation as a zero of
\(\widehat P-\mathrm{id}\), not isolation relative to one arbitrary
one-sided interval.

\begin{proposition}[Exact cycle--zero correspondence, including multiplicity]
\label{prop:exact-once}
After $U$ and the parameter ball are sufficiently small, every isolated limit
cycle contained in $U$ is essential in that annulus, crosses $\Sigma$ exactly once
in positive orientation, and is represented by exactly one retained stopped
full-return itinerary.  Its intersection point is an isolated zero of
\eqref{eq:itinerary-displacement} in the common ambient germ, with the same
multiplicity as the cycle.  Conversely, every retained itinerary zero that is
isolated in this ambient germ closes one isolated limit cycle contained in
$U$, with the same multiplicity.  Minimal-face and half-open conventions
identify overlap descriptions and assign every boundary lift once.
\end{proposition}

\begin{remark}
The ambient-germ clause is necessary.  A zero on the endpoint of one
half-open itinerary interval can fail to be isolated relative to a neighboring
label even though all labels describe the same regular Poincar\'e germ.  The
proposition compares isolation and multiplicity with that physical germ,
whereas the minimal-face rule decides only which label records it.
\end{remark}

\begin{proof}
By \eqref{eq:legacy-3-2}, every crossing of $\Sigma$ has positive sign.  A periodic orbit
cannot miss $\Sigma$: otherwise its retained cut index would increase
strictly around a loop, or the orbit would lie in one regular strip or
singular slab, where the strict coordinates and Lyapunov functions above
exclude recurrence.

A limit cycle is an embedded closed curve.  If it were contractible in the
annulus, its algebraic intersection with a proper radial cut would be zero,
contradicting the existence and common positive sign of its intersections
with $\Sigma$.  Thus it is essential.  An embedded essential circle
represents a primitive generator of the annulus, so its algebraic intersection
with $\Sigma$ is $+1$.  Since every geometric intersection is positive,
\begin{equation}
 \#(\gamma\cap\Sigma)=1.
\label{eq:legacy-4-4}
\end{equation}

Starting at this unique point, the cycle chooses one next-cut interval at
every stage.  A gate, previous-side, node-core, stable-center, wrong-oriented,
or collar outcome would be terminal and is impossible on the cycle.  Hence
the itinerary is one retained stopped full-return itinerary and the starting
point is its fixed point.

Conversely, equality \(P_\omega(s,\lambda)=s\) concatenates actual first-hit
arcs.  Consecutive factors meet at the same physical section point, and the
first-contact rule excludes an intervening exit or singular limit.  Flow
uniqueness therefore closes one trajectory after one lap in $U$.

Suppose the cycle or zero lies on a finite divider.  The transported germ
\(\widehat P\) introduced above exists because the closed orbit itself is
regular.  Flow uniqueness and \eqref{eq:legacy-4-1} show that every incident
itinerary map is its restriction after analytic changes of section.  Since
only finitely many labels are incident, a sequence of distinct nearby cycles
approaching the divider has a subsequence in one fixed label.  Isolation in
the physical Poincar\'e germ and isolation in the union of the incident word
germs are therefore equivalent.

It remains to compare multiplicities.  Fix a regular transverse section with
analytic coordinate $u$ and transport it to $\Sigma$ along the cycle.  This
gives an analytic diffeomorphism $h$, with $h'(u_0)\ne0$.  If
\(\widetilde P=h^{-1}\circ P_\omega\circ h\), then
\begin{equation}
\begin{aligned}
 P_\omega(h(u))-h(u)
  &=h(\widetilde P(u))-h(u)\\
  &=(\widetilde P(u)-u)
    \int_0^1h'\!\left(u+t(\widetilde P(u)-u)\right)\,dt .
\end{aligned}
\label{eq:multiplicity-conjugacy}
\end{equation}
The integral is a nonzero analytic unit at $u_0$.  Thus the order of the
Poincar\'e displacement---the multiplicity of the cycle---equals the order of
the zero of $D_\omega$.  The same calculation proves the converse and applies
to every later analytic section gauge.  When the source, middle, or root
argument makes an additional monotone phase change, its displayed positive
Jacobian and nonvanishing factors give the identical conclusion.

Finally, overlap charts describe the same physical orbit by
\eqref{eq:legacy-4-1}.  At a resolved boundary choose the unique minimal face
and the fixed half-open priority; a collapsed interval has no point and a
terminal face cannot carry a closed orbit.  These conventions assign the
physical zero once without creating a proper-subarc return.
\end{proof}

\section{Roadmap of the parameter regimes and scales}
\label{sec:regime-table}

This section is a navigation device, not the proof of exhaustion.  It names
the finite analytic regimes that will be proved later and records in advance
which theorem owns every equality face.  The formal disjoint assignment is
established only in Proposition~\ref{prop:regime-assignment}, after the local
estimates and the two-central classification have been proved.

The regimes are distinguished by quantities intrinsic to the physical flow:
which singular sectors are traversed, whether the relevant connections are
present, and which normalized scale remains positive.  Table~\ref{tab:regimes}
is therefore a map of the phase--parameter carrier, not a chronological list
of proof tasks.

The recognition data in Table~\ref{tab:regimes} are section normals,
eigenvalue margins, actual first-hit outcomes, and complete pp/hh incidence.
The relevant transition maps are therefore defined before an analytic
finiteness theorem is applied.

We use the standard saddle-node sector letters throughout.  A lower-case
\(h\) or \(p\) denotes, respectively, a local hyperbolic (strong) or
parabolic (central) separatrix sector.  Thus an hh connection is an actual
strong-separatrix connection between the two selected saddle-nodes, whereas a
pp strip is a nonempty interval of complete orbits with those saddle-nodes as
its alpha- and omega-limits.  PP and BP denote the two physical boundary
graphics admitted by the DIR lips theorem: the principal-endpoint boundary
and the center-side boundary, respectively.  ``Attractivity'' is the sign of
the nonzero transverse eigenvalue after the one common local time orientation.
On the intermediate scale we use two overlapping coordinate cells: a
buffered finite-phase cell and a hyperbolic-corner cell.  The technical
abbreviations QBF and QHH, retained in a few formulas and appendix labels,
refer only to these cells; they are defined in
Section~\ref{sec:part-iii-middle}.

The five symbols that organize the coalescing analysis have the following
fixed meanings.  This dictionary is used unchanged in all later theorems.
\begin{table}[htbp]
\centering
\caption{Scale variables and their boundary roles.}
\label{tab:scale-dictionary}
\small
\begin{tabularx}{\textwidth}{@{}l
  >{\raggedright\arraybackslash}X
  >{\raggedright\arraybackslash}X@{}}
\toprule
Symbol&Geometric meaning&Boundary assignment\\
\midrule
\(t\)&defining function of the source-coalescing tube&
\(t=0\) is source; \(0<t<t_{\rm str}\) is the middle/root comparison;
\(t=t_{\rm str}\) belongs to strict\\
\(\varrho_w\)&weighted distance of
\((B/t^2,a/t)\) from the root corner&\(\varrho_w=0\), with \(t>0\), is
exact mixed; \(\varrho_w=\varrho_\#\) belongs to middle\\
\(\kappa\)&radial coordinate of the weight-\((2,1)\) root blow-up&
\(0<\kappa<\kappa_\#(\text{angle})\) is root; \(\kappa=0\) is mixed\\
\(\gamma\)&the unscaled remaining coefficient \(c\) in middle/root charts&
\(|\gamma|\le\gamma_0\), with no hidden restriction on \(\gamma/t\)\\
QBF/QHH&overlapping finite-phase and hyperbolic-corner cells&
coordinate cells inside the middle theorem, not additional regimes\\
\bottomrule
\end{tabularx}
\end{table}

The strict/coalescing cutoff is fixed before any theorem neighborhood is
chosen.  Select nested finite tubular covers of the source-coalescing face.
In every signed chart of the inner cover let $t\ge0$ be its defining
function, and choose one regular value $0<t_{\rm str}<t_0$ in the overlap.
The strict regime contains the exterior of the inner tube and the equality
$t=t_{\rm str}$; only $0<t<t_{\rm str}$ enters the middle/root
comparison.

The latter comparison is half-open at the source/root corner.  On the finite
signed cover of the selected upper $D$-double root write
$w=-q$, $q=\sigma t$, with $\sigma\in\{\pm1\}$ away from $t=0$, and
set
\begin{equation}
 b_m=\frac{B}{t^2},\qquad A_m=\frac{a}{t},\qquad
 \varrho_w=(b_m^2+A_m^4)^{1/4}.
\label{eq:legacy-5-2a}
\end{equation}
Under the weight-$(2,1)$ root blow-up
$b_m=\kappa^2b$, $A_m=\sigma\kappa A$,
\begin{equation}
 \varrho_w=\kappa(b^2+A^4)^{1/4}.
\label{eq:legacy-5-2b}
\end{equation}
The angular factor is bounded above and away from zero.  Choose a regular
value $\varrho_\#$ in the doubled middle/root overlap and a fixed positive
root-chart bound $\kappa_0$.  After terminal first-hit outcomes and the
no-pp alternative have been removed, the complete-lips carrier is assigned
by the single half-open rule
\begin{equation}
\begin{array}{rcl}
 t=0&\longrightarrow&\text{matched source},\\
 t>0,\ \kappa=0&\longrightarrow&\text{exact mixed},\\
 t\ge t_{\rm str}\ \text{with positive lips margins}
   &\longrightarrow&\text{strict},\\
 0<t<t_{\rm str},\ 0<\varrho_w<\varrho_\#
   &\longrightarrow&\text{root merger},\\
 0<t<t_{\rm str},\ \varrho_w\ge\varrho_\#
   &\longrightarrow&\text{middle}.
\end{array}
\label{eq:legacy-5-2}
\end{equation}
Equality at $t=t_{\rm str}$ therefore belongs to strict, whereas equality at
$\varrho_w=\varrho_\#$ belongs to middle.  In a fixed angular root chart the
root-assigned interval is $0<\kappa<\kappa_\#(\mathrm{angle})$, although the
positive-root theorem remains valid on the larger interval
$0<\kappa\le\kappa_0$, with
$\kappa_\#(\mathrm{angle})\le\kappa_0$.  Thus theorem validity and regime
assignment are not being identified.  These definitions make every scale in
Table~\ref{tab:regimes} available before the decomposition is stated.

\begin{table}[p]
\centering
\caption{Zero estimates assigned to the retained physical regimes.}
\label{tab:regimes}
\begingroup
\hbadness=10000
\small
\begin{tabularx}{\textwidth}{@{}
 >{\hsize=.75\hsize\raggedright\arraybackslash}X
 >{\hsize=1.55\hsize\raggedright\arraybackslash}X
 >{\hsize=.95\hsize\raggedright\arraybackslash}X@{}}
\toprule
Regime & Physical recognition data & Applicable zero estimate\\
\midrule
Matched noncompact source & Complete center-compatible source itinerary after
the action refinement, including the face $t=0$ &
Theorem~\ref{thm:part-ii-source-zero}\\
Compact regular & Bounded jointly analytic transverse itinerary &
Theorem~\ref{thm:compact-analytic} (Weierstrass preparation)\\
Separated hyperbolic & Positive eigenvalue, connector, section, and itinerary
margins & Theorem~\ref{thm:part-iii-hyperbolic-zero} (Mourtada QRH)\\
One central block & Exactly one retained internal saddle-node and hyperbolic
complement & Theorem~\ref{thm:part-iii-one-central-zero}\\
Two central blocks, no pp & Both possible internal central blocks,
same attractivity, and no complete pp connection &
Theorem~\ref{thm:part-iii-two-central-no-pp}\\
Positive-margin lips & Actual hh connection, complete pp strip, PP/BP boundary,
positive margins, and the strict assignment range in \eqref{eq:legacy-5-2} &
Theorem~\ref{thm:part-iii-strict-zero}\\
Intermediate coalescing scale & $0<t<t_{\rm str}$, the
finite-phase/hyperbolic-corner cover, and $\varrho_w\ge\varrho_\#$ in the
weighted overlap & Theorem~\ref{thm:part-iii-middle-zero}\\
Positive root merger & $0<t<t_{\rm str}$, $0<\varrho_w<\varrho_\#$, represented by
finite relative-interior root charts
$0<\kappa<\kappa_{\#}(\mathrm{angle})$ &
Theorem~\ref{thm:part-iii-root-zero}\\
Exact mixed face & $t>0,\ \kappa=0$, equivalently $B=a=0$ off the source face
& Theorem~\ref{thm:part-ii-mixed-zero}\\
No full-return itinerary & Passive sector, collar exit, node core, wrong
orientation, or collapsed interval & No displacement equation;
Proposition~\ref{prop:stopped-first-hit} and
Theorem~\ref{thm:finite-words}\\
\bottomrule
\end{tabularx}
\endgroup
\end{table}

Here ``parameter-dominated'' names a resolved coordinate regime, not an
additional analytic case.  At the first failed effective source
threshold, the finite parameter-dominated gate classification is evaluated on
the same physical lift.  A transverse bounded itinerary goes to the compact row;
separated nonzero eigenvalues go to the hyperbolic row; one or two central
gates go to the corresponding central/two-central routing; a persistent
\(B=a=0\) endpoint goes to the mixed row; and a sink, wrong orientation,
stable-center side, or exit admits no full-return itinerary.  These are
exactly rows already present in the table.
The table becomes exhaustive through one routing result that is not itself an
additional analytic case.  Theorem~\ref{thm:part-iii-two-central-exhaustion}
sends every retained itinerary
carrying both possible central blocks to the two-central no-pp, strict lips,
middle, positive-root, exact-mixed, or source row.  Its first-hit trichotomy
and the nonaffineness argument leave no residual affine regime.  This routing occurs
after terminal outcomes have been removed and before the analytic estimate is
chosen.

Equation~\eqref{eq:legacy-5-2} gives the assignment rule exactly, while
Figure~\ref{fig:parameter-handoff} makes visible the two different cutoff
equalities and the larger theorem-validity radius $\kappa_0$.  Neither the
root theorem nor the diagram is extrapolated to a zero-scale face.  A loss of
any other strict-lips margin is assigned by its first physical contact before
an analytic theorem is selected.

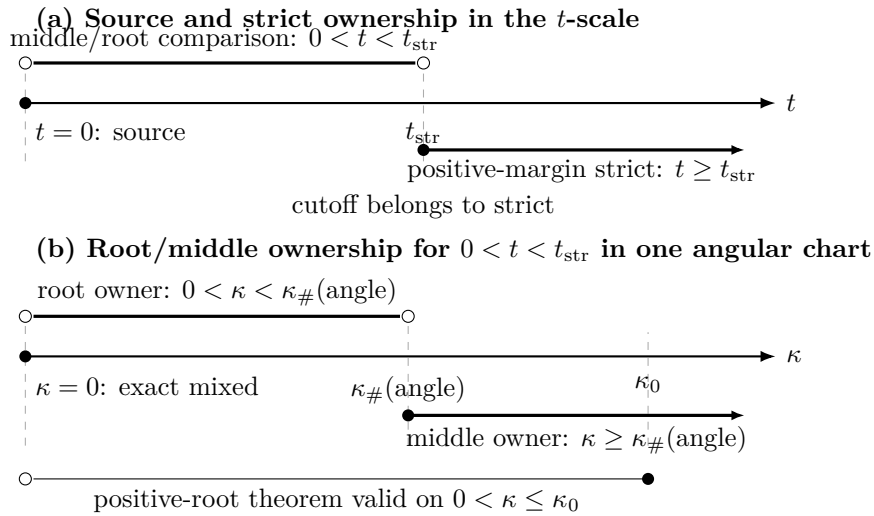
\begin{figure}[htbp]
\centering
\begin{tikzpicture}[
  x=1.35cm,y=1cm,
  every node/.style={font=\small,align=center},
  axis/.style={-{Latex[length=2mm]},thick},
  owner/.style={very thick},
  validity/.style={thin},
  guide/.style={gray!65,dashed},
  openpt/.style={circle,draw,fill=white,inner sep=1.6pt},
  closedpt/.style={circle,fill=black,inner sep=1.6pt}
]
\node[anchor=west,font=\small\bfseries] at (0,2.35)
  {(a) Source and strict ownership in the $t$-scale};
\draw[axis] (0,1.25) -- (7.35,1.25) node[right] {$t$};
\draw[guide] (0,0.48) -- (0,1.93);
\draw[guide] (3.90,0.48) -- (3.90,1.93);
\node[closedpt] at (0,1.25) {};
\node[anchor=north west] at (0,1.12) {$t=0$: source};
\node[anchor=north] at (3.90,1.12) {$t_{\rm str}$};

\draw[owner] (0.08,1.78) -- (3.82,1.78);
\node[openpt] at (0,1.78) {};
\node[openpt] at (3.90,1.78) {};
\node[above] at (1.95,1.78)
  {middle/root comparison: $0<t<t_{\rm str}$};

\draw[owner,-{Latex[length=1.8mm]}] (3.90,0.63) -- (7.05,0.63);
\node[closedpt] at (3.90,0.63) {};
\node[below] at (5.45,0.63)
  {positive-margin strict: $t\ge t_{\rm str}$};
\node[anchor=north] at (3.90,0.18)
  {cutoff belongs to strict};

\node[anchor=west,font=\small\bfseries] at (0,-0.70)
  {(b) Root/middle ownership for $0<t<t_{\rm str}$ in one angular chart};
\draw[axis] (0,-2.10) -- (7.35,-2.10) node[right] {$\kappa$};
\draw[guide] (0,-3.28) -- (0,-1.42);
\draw[guide] (3.75,-3.28) -- (3.75,-1.42);
\draw[guide] (6.10,-3.28) -- (6.10,-1.78);
\node[closedpt] at (0,-2.10) {};
\node[anchor=north west] at (0,-2.23) {$\kappa=0$: exact mixed};
\node[anchor=north] at (3.75,-2.23) {$\kappa_{\#}(\mathrm{angle})$};
\node[anchor=north] at (6.10,-2.23) {$\kappa_0$};

\draw[owner] (0.08,-1.57) -- (3.67,-1.57);
\node[openpt] at (0,-1.57) {};
\node[openpt] at (3.75,-1.57) {};
\node[above] at (1.88,-1.57)
  {root owner: $0<\kappa<\kappa_{\#}(\mathrm{angle})$};

\draw[owner,-{Latex[length=1.8mm]}] (3.75,-2.88) -- (7.05,-2.88);
\node[closedpt] at (3.75,-2.88) {};
\node[below] at (5.40,-2.88)
  {middle owner: $\kappa\ge\kappa_{\#}(\mathrm{angle})$};

\draw[validity] (0.08,-3.72) -- (6.10,-3.72);
\node[openpt] at (0,-3.72) {};
\node[closedpt] at (6.10,-3.72) {};
\node[below] at (3.05,-3.72)
  {positive-root theorem valid on $0<\kappa\le\kappa_0$};
\end{tikzpicture}
\caption{The exact half-open ownership rules.  Separate horizontal bars make
the open and closed endpoints explicit.  Panel~(a) assigns $t=0$ to the
source face, the open comparison interval to the middle/root analysis, and
$t=t_{\rm str}$ to the strict regime.  Conditional on
$0<t<t_{\rm str}$, panel~(b) fixes one angular root chart:
$\kappa=0$ is the exact mixed face, the root regime owns
$0<\kappa<\kappa_{\#}(\mathrm{angle})$, and the cutoff equality belongs to
the middle regime.  Thus the corner $t=\kappa=0$ remains on the source face.
The lower thin bar records the larger theorem-validity
interval $0<\kappa\le\kappa_0$; it is not a regime boundary.  Lengths and
spacing are schematic.}
\label{fig:parameter-handoff}
\end{figure}

The exact cycle--zero reduction therefore precedes every local estimate, and
the two-central theorem makes the resulting list exhaustive before
specialization and compactness are invoked.  The source and mixed faces are
independent theorems, not limiting values of the punctured root theorem;
bounded analytic itineraries are handled directly by
Theorem~\ref{thm:compact-analytic}.  These separations are what make
\eqref{eq:legacy-5-2} cover every diagonal approach.

\section{Preview of the two-central routing}
\label{sec:two-central-exhaustion}

Only itineraries carrying both possible internal central blocks require a
geometric routing before an analytic zero theorem can be selected.  This
section previews that routing so the regime table can be read; the physical
proof and formal statement are Theorem~\ref{thm:part-iii-two-central-exhaustion}
in Section~\ref{sec:part-iii-central}.

A retained two-central itinerary has no third internal central gate.  If no
complete pp connection exists, it belongs to the two-central no-pp theorem.
Otherwise the signed gate count and center intersection identify two actual
opposite saddle-nodes.  The physical strong landing equation must certify the
hh connection, and the stopped first-hit relation must either supply a
complete pp strip with PP/BP boundary or stop at a named first-contact
outcome.  On a complete strip the first-hit trichotomy leaves the hh chain and
the endpoint--upper chain as its two boundary graphics.

Figure~\ref{fig:two-central-incidence} distinguishes the missing-connection
case from a complete orbit family.  The no-pp panel asserts the absence of
even a single pp connection and therefore draws no connecting curve.  In the
opposite-attractivity panel, every light curve in the shaded family is a
complete physical pp orbit from $S_h$ to $S_v$; the lower PP/BP boundary has
the same orientation, whereas the upper hh boundary is the complementary
return leg from $S_v$ to $S_h$.
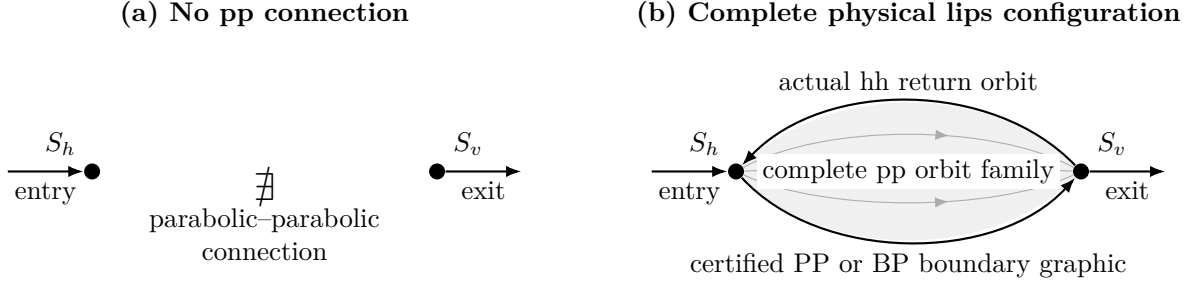
\begin{figure}[H]
\centering
\resizebox{.98\textwidth}{!}{%
\begin{tikzpicture}[
  every node/.style={font=\small,align=center},
  sn/.style={circle,fill=black,inner sep=2.1pt},
  flow/.style={-{Latex[length=2mm]},thick},
  ppflow/.style={gray!65,thin,
    postaction={decorate},
    decoration={markings,
      mark=at position 0.62 with {\arrow{Latex[length=1.8mm]}}}}
]
\begin{scope}[xshift=-4.2cm]
\node[font=\small\bfseries] at (0,2.05)
  {(a) No pp connection};
\node[sn,label=above left:$S_h$] (a1) at (-2.25,0) {};
\node[sn,label=above right:$S_v$] (a2) at (2.25,0) {};
\draw[flow] (-3.35,0) -- node[below]{entry} (a1);
\draw[flow] (a2) -- node[below]{exit} (3.35,0);
\node[font=\Large] at (0,-0.20) {$\nexists$};
\node[text width=43mm] at (0,-0.82)
  {parabolic--parabolic connection};
\end{scope}

\begin{scope}[xshift=4.2cm]
\node[font=\small\bfseries] at (0,2.05)
  {(b) Complete physical lips configuration};
\fill[gray!12]
  (-2.25,0) .. controls (-1.10,1.20) and (1.10,1.20) .. (2.25,0)
  .. controls (1.15,-1.20) and (-1.15,-1.20) .. cycle;
\node[sn,label=above left:$S_h$] (b1) at (-2.25,0) {};
\node[sn,label=above right:$S_v$] (b2) at (2.25,0) {};
\draw[flow] (-3.35,0) -- node[below]{entry} (b1);
\draw[flow] (b2) -- node[below]{exit} (3.35,0);
\draw[flow] (b2) .. controls (1.10,1.20) and (-1.10,1.20) ..
  node[above]{actual hh return orbit} (b1);
\draw[ppflow] (b1) .. controls (-1.10,0.63) and (1.10,0.63) .. (b2);
\draw[ppflow] (b1) .. controls (-1.10,0.05) and (1.10,0.05) .. (b2);
\draw[ppflow] (b1) .. controls (-1.10,-0.53) and (1.10,-0.53) .. (b2);
\draw[flow] (b1) .. controls (-1.15,-1.20) and (1.15,-1.20) ..
  node[below]{certified PP or BP boundary graphic} (b2);
\node[fill=white,inner sep=2pt] at (0,0)
  {complete pp orbit family};
\end{scope}
\end{tikzpicture}%
}
\caption{The two-central incidence after the physical first-hit refinement.
Panel~(a) records the same-attractivity alternative: no
parabolic--parabolic (pp) connection exists, so no fictitious trajectory is
drawn across the gap.  Panel~(b) records the opposite-attractivity lips
configuration.  Its shaded band is a nonempty interval of complete pp orbits
from $S_h$ to $S_v$, bounded below by the certified PP or BP graphic with the
same orientation.  The hyperbolic--hyperbolic (hh) boundary is the
complementary return orbit from $S_v$ to $S_h$.  Black points are the two
saddle-nodes, and arrows give physical flow orientation; the light-arrow
markers are placed inside the pp curves to keep their common endpoint
legible.  The topology, ordering, and orientation are asserted data;
distances and curvature are schematic.}
\label{fig:two-central-incidence}
\end{figure}

The potential residual class is a complete two-gate hh/pp itinerary whose pp
transition would be called affine before its second boundary was identified.
The proof later follows every point of the endpoint interval through the
upper cooperative squeeze, shows that the non-hh endpoint is the complete PP
chain, and invokes the certified boundary theorem to rule out affine
transition.  Thus the substantive point is physical completeness of the
second boundary, not a formal normal-form calculation.
Corollary~\ref{cor:two-central-nonaffine-boundary} records the resulting empty
residual class.  Once this routing is available, the half-open scale dictionary above
gives the disjoint assignment formalized at the end of the paper.

\section{The matched return equation near the noncompact source}
\label{sec:part-ii-target}

For a source-regime itinerary obtained in
Sections~\ref{sec:part-i-atlas}--\ref{sec:part-i-exact-once}, fixed points are zeros of
the exact nonlinear equation stated below.  Its derivation requires a common
physical domain, a justified center division, control of the moving first-hit
maps, and a two-step Rolle estimate; these ingredients occupy the following
sections.

Retain the parameter coordinates
\[
 a=\mu_4+B\mu_5,\qquad c=(1-2B)\mu_5,\qquad
 d=\mu_3,\qquad m=\mu_2
\]
and put
\begin{equation}
 \tau=a+c,\qquad \ell=d-c,\qquad t_c=B+m,\qquad
 \sigma_0=ct_c.
\label{eq:legacy-6-1}
\end{equation}
On the lower physical section
\(\Sigma_0=\{x=0,z=1+y=s\}\), set
\begin{equation}
 h=s-1-\log s,\qquad L=-\log s,\qquad
 \theta=h^{-1/2},\qquad k=\frac{\ell}{s}.
\label{eq:legacy-6-2}
\end{equation}
The variable \(k\), not \(\ell\) alone, records the forced lower
displacement.  It will remain an independent bounded variable on every
zero-carrying source cell.

Choose the physical strong sections
\[
 S_+=\{z=1,x>0\},\qquad S_-=\{z=1,x<0\}.
\]
The selected itinerary factors into actual first-hit maps
\begin{equation}
 P^\ell=J_-^\ell\circ M^\ell\circ J_+^\ell,
\label{eq:legacy-6-3}
\end{equation}
where the two \(J\)'s contain the complete lower tails and horizontal
endpoint passages, while \(M^\ell:S_+\to S_-\) is precisely the upper
passage.  No auxiliary endpoint clock is included in \(M^\ell\).

Compare each forced lower map with its unforced map on the same physical
sections:
\begin{equation}
 s_+=(J_+^0)^{-1}J_+^\ell(s),\qquad
 s_-=J_-^0(J_-^\ell)^{-1}(s).
\label{eq:legacy-6-4}
\end{equation}
On the open survivor interval these levels have the exact secant form
\begin{equation}
 C_+=\frac{s_+}{ s}=1+k\Gamma_+,\qquad
 C_-=\frac{s_-}{ s}=1-k\Gamma_-,
\label{eq:legacy-6-5}
\end{equation}
with \(\Gamma_\pm>0\).  The faces \(C_+=0\) and \(C_-=0\) are the
two named lower first-hit gates.  They are not values at which a return map
is continued.  Define
\begin{equation}
 \lambda_+=\log C_+-s(C_+-1),\qquad
 \lambda_-=-\log C_-+s(C_--1).
\label{eq:legacy-6-6}
\end{equation}
The elementary identity
\begin{equation}
 h(sC)-h(s)=s(C-1)-\log C
\label{eq:legacy-6-7}
\end{equation}
shows that \(1-\lambda_+/h\) and \(1+\lambda_-/h\) are the incoming
and outgoing normalized actions.

Use the unforced lower maps themselves to define the two action coordinates
\begin{equation}
 e_+(u)=\frac{h((J_+^0)^{-1}u)}{ h(s)},\qquad
 e_-(v)=\frac{h(J_-^0v)}{ h(s)}
\label{eq:legacy-6-8}
\end{equation}
and the physical middle action map
\begin{equation}
 T_{s,\beta,k}=e_-\circ M^\ell\circ e_+^{-1},\qquad
 \beta=(Bh,mh,a\sqrt h,c\sqrt h).
\label{eq:legacy-6-9}
\end{equation}
Then \(P^\ell(s)=s\) is exactly equivalent, with multiplicity, to
\begin{equation}
 T_{s,\beta,k}(1-\lambda_+/h)=1+\lambda_-/h.
\label{eq:legacy-6-10}
\end{equation}

On the unforced face \(k=0\), put \(P_0=P^0\) and
\(G_0=h(P_0)-h(s)\).  The two center components proved below give the exact
reduced division
\begin{equation}
 G_0=\tau L^{3/2}A+\sigma_0L^{5/2}C,
\label{eq:legacy-6-11}
\end{equation}
where \(A,C>0\) and, for \(j=0,1,2\),
\begin{equation}
 L^j\left|\mathfrak D^j
 \left(\frac{A}{4\sqrt2/3}-1\right)\right|
 +L^j\left|\mathfrak D^j
 \left(\frac{C}{8\sqrt2/15}-1\right)\right|
 \le C_0\epsilon+\eta(\epsilon,s).
\label{eq:legacy-6-12}
\end{equation}
Here \(\mathfrak D=d/dL\) keeps the five original parameters fixed.

For a scalar \(\lambda\), define the action secant and the forced middle
remainder by
\begin{equation}
 \mathcal T_{s,\beta}(\lambda)
 =h\{T_{s,\beta,0}(1)
      -T_{s,\beta,0}(1-\lambda/h)\},
\label{eq:legacy-6-13}
\end{equation}
\begin{equation}
 \mathcal R
 =h\{T_{s,\beta,0}-T_{s,\beta,k}\}
        (1-\lambda_+/h).
\label{eq:legacy-6-14}
\end{equation}
Subtracting \eqref{eq:legacy-6-10} from its unforced counterpart yields the decisive
identity.

\begin{theorem}[Matched source preparation]
\label{thm:part-ii-matched-preparation}
On every unforced-center-complete zero-carrying source itinerary, after the finite
first-hit and action refinement, the physical fixed-point equation is
exactly
\begin{equation}
 \boxed{\;
 \tau L^{3/2}A+\sigma_0L^{5/2}C
 =\mathcal T_{s,\beta}(\lambda_+)+\lambda_-+\mathcal R .
 \;}
\label{eq:part-ii-matched}
\end{equation}
All terms are defined on one common physical product tube.  The two lower
factors are exact, the only middle error is \(\mathcal R\), and the
equivalence preserves zero multiplicity.  The equation is not asserted on
either gate face \(C_\pm=0\).
\end{theorem}

Three inputs are still needed before the zero count can be made: a common
domain on which both center divisions are legitimate, differentiated control
of every moving hit and inverse, and a localization result placing every
possible zero in those domains.  We establish these inputs in that order and
then return to the exact equation in
Section~\ref{sec:part-ii-matched-proof}.  This order excludes both failed
shortcuts: division on an unproved outer return domain and replacement of
\eqref{eq:part-ii-matched} by a linear combination of source-energy
asymptotics.

\section{The center set, Bautin ideal, and common return domains}
\label{sec:part-ii-center}

We identify the two complete center components and prove the reduced division
used in \eqref{eq:legacy-6-11}.  The division is carried out for a physical stopped itinerary,
not for an unspecified return germ, on a three-contraction unforced domain
where both center identities hold throughout the section interval.

\begin{theorem}[Center set and reduced ideal]
\label{thm:part-ii-center-ideal}
Near the source the center set is the union
\begin{equation}
 \mathcal C_R=\{\tau=0,a=0,d=0\},\qquad
 \mathcal C_Q=\{\tau=0,t_c=B+m=0,d+a=0\},
\label{eq:legacy-6-17}
\end{equation}
and its reduced analytic ideal is
\begin{equation}
 \mathcal I_{\mathcal C}=(\tau,\ell,d(B+m)).
\label{eq:legacy-6-18}
\end{equation}
On the unforced face \(\ell=0\), let
\(\sigma_0=c(B+m)=(\tau-a)t_c\).  If a \(C^K\) displacement
\(F_0\), \(K\ge2\), is defined on a itinerary domain star-shaped under the
successive contractions \(\tau\mapsto0\), \(a\mapsto0\), and
\(t_c\mapsto0\), and vanishes on both complete center slices, then
\begin{equation}
 F_0=\tau F_{0,\tau}+\sigma_0F_{0,\sigma},
 \qquad F_{0,j}\in C^{K-2}.
\label{eq:legacy-6-19}
\end{equation}
No continuation in \(\ell\) and no division of a finite-smooth
saddle-node normalizer occurs.
\end{theorem}

\begin{proof}
A center first has zero trace, hence \(\tau=0\) and \(c=-a\).  Put
\[
 u=x-ay,\qquad v=\sqrt{1-a^2}\,y,\qquad
 \omega=\sqrt{1-a^2}.
\]
After a positive time division the system is
\begin{equation}
 u'=-v+Au^2+Cuv+Dv^2,\qquad
 v'=u+Euv+Fv^2,
\label{eq:legacy-6-20}
\end{equation}
where
\begin{equation}
 A=\frac{B}{\omega},\quad
 C=\frac{a(2B-1)}{\omega^2},\quad
 D=\frac{a^2(B-1)+m-ad}{\omega^3},\quad
 E=\frac{1}{\omega},\quad F=\frac{a+d}{\omega^2}.
\label{eq:legacy-6-21}
\end{equation}
The focal calculation is finite and can be reproduced without a black-box
center theorem.  Write
\(V_n=\sum_{j=0}^nc_{n,j}u^{n-j}v^j\) and
\(G_n=(Q_1\partial_u+Q_2\partial_v)V_{n-1}
     =\sum g_{n,j}u^{n-j}v^j\).  The coefficients obey
\begin{equation}
\begin{aligned}
g_{n,j}={}&[A(n-1-j)+Ej]c_{n-1,j}\\
&+[C(n-j)+F(j-1)]c_{n-1,j-1}
  +D(n-j+1)c_{n-1,j-2},
\end{aligned}
\label{eq:legacy-6-22}
\end{equation}
and the homological equation is
\begin{equation}
\begin{aligned}
0={}&(j+1)c_{n,j+1}-(n-j+1)c_{n,j-1}+g_{n,j}\\
&-\mathbf1_{\{n,j\ \mathrm{even}\}}
  \binom{n/2}{j/2}L_{n/2-1}.
\end{aligned}
\label{eq:legacy-6-23}
\end{equation}
For odd \(n\) this system is invertible; for even \(n\), the gauge
\(c_{n,0}=0\) determines the coefficients and the single obstruction.
At degree four,
\begin{equation}
 8L_1=AC+CD+2DF-EF.
\label{eq:legacy-6-24}
\end{equation}
After multiplication by the positive unit \(\omega^5\), this becomes
\begin{equation}
\begin{aligned}
\ell_1={}&2B^2a+2Bam-Ba-2a^2d+am-2ad^2-a+2md-d .
\end{aligned}
\label{eq:legacy-6-25}
\end{equation}
Since \(\partial_d\ell_1(0)=-1\), its zero set is one analytic graph.
Put \(e=d+a\).  On that graph the degree-six recurrence
\eqref{eq:legacy-6-22}--\eqref{eq:legacy-6-23}, equivalently angular averaging of \(G_6\), gives
\begin{equation}
 L_2=a(B+m)U(a,B,m),\qquad U(0)=\frac{1}{48}.
\label{eq:legacy-6-26}
\end{equation}
For completeness, exact divisibility in \eqref{eq:legacy-6-26} is not inferred from its
quadratic jet.  The first obstruction is the polynomial equation
\begin{equation}
 e(2a^2+2t_c-2B-1)-2ae^2+a(2B-1)t_c=0.
\label{eq:legacy-6-27}
\end{equation}
Its unique root \(e=\psi(a,t_c,B)\) vanishes identically on \(a=0\) and
on \(t_c=0\), hence \(\psi=at_cV\) by two Hadamard integrals.  Substitution
in the degree-six polynomial obtained from \eqref{eq:legacy-6-22}--\eqref{eq:legacy-6-23} vanishes on the
same two slices, so a second two-variable integral gives \eqref{eq:legacy-6-26}; the
degree-two term is \(a t_c/48\), proving that \(U\) is a unit.  The
finite rational recurrence is also checked electronically, but the proof is
the displayed recurrence together with the slice cancellations and the unit
coefficient.

Thus \(L_1=L_2=0\) gives either \(a=e=0\) or \(t_c=e=0\).  On the first
branch the transformed field is reversible.  On the second branch
\(m=-B,d=c=-a\), and
\begin{equation}
\begin{aligned}
K_Q={}&B^2x^2-B^2y^2+Baxy+2Bax-Bx^2-2By\\
&+a^2y+a^2-ax-1,
\qquad
\mathcal V_Q=\frac{(1+y)K_Q}{ a^2-1}
\end{aligned}
\label{eq:legacy-6-28}
\end{equation}
satisfies
\begin{equation}
 X(\mathcal V_Q)=(\operatorname{div}X)\mathcal V_Q .
\label{eq:legacy-6-29}
\end{equation}
Since \(\mathcal V_Q(0,0)=1\), division by it gives a closed analytic
one-form with definite quadratic part.  Both branches are therefore centers,
and the preceding necessity excludes a third branch.

In the coordinates \((\tau,\ell,a,t_c,B)\), the union \eqref{eq:legacy-6-17} is
\[
 \{\tau=\ell=a=0\}\cup\{\tau=\ell=t_c=0\},
\]
whose reduced ideal is \((\tau,\ell,a t_c)\).  Since
\(d=\ell+\tau-a\),
\[
 at_c=\tau t_c+\ell t_c-dt_c,
\]
which proves \eqref{eq:legacy-6-18}.  On \(\ell=0\), Hadamard division first in
\(\tau\) and then in \((a,t_c)\) gives
\(F_0=\tau A+a t_c C\).  Because \(c=\tau-a\),
\(at_c=\tau t_c-ct_c\), yielding \eqref{eq:legacy-6-19}.  The three integral segments are
exactly the contractions stated in the theorem.
\end{proof}

Appendix~\ref{app:center-bautin} records the finite focal recurrence, one
representative coefficient calculation, the two global center-domain
identities, and the exact boundary between the symbolic checks and the human
division argument.  The raw symbolic expansion is not part of the main
reading path.

The center identities must hold on the whole physical itinerary interval.
Their outer barriers can differ.  On \(\mathcal C_R\), with \(z=1+y\),
\begin{equation}
 H_R(x,z)=\frac{1}{2}z^{-2B}x^2+V_R(z),\qquad
 V_R'(z)=z^{-2B-1}\{(z-1)-m(z-1)^2\}.
\label{eq:legacy-6-30}
\end{equation}
For \(m>0\) the only additional finite critical point is the saddle
\(S_R=(0,1+1/m)\).  When \(B<0\), comparison with the boundary
\(z=0\) is exact:
\begin{equation}
 V_R(0)-V_R(1+1/m)
 =\frac{2(B+m)}{(-2B)(1-2B)(2-2B)}
 \left(\frac{1+m}{ m}\right)^{1-2B}.
\label{eq:legacy-6-31}
\end{equation}
On \(\mathcal C_Q\), the component of
\(\{\mathcal V_Q\ne0\}\cap\{y>-1\}\) containing the origin carries
the analytic first integral.  Its only extra finite singularity is
\begin{equation}
 S_Q=(-a/B,-1/B)\quad(B<0),
\label{eq:legacy-6-32}
\end{equation}
a saddle on \(K_Q=0\).  Hence a center return is the identity exactly on
the connected section interval inside its maximal period annulus; its
endpoint is one of
\begin{equation}
 y=-1,\qquad K_Q=0,\qquad S_R,\qquad S_Q
\label{eq:legacy-6-33}
\end{equation}
or a named compactification face.

\begin{proposition}[Common domain for the unforced center returns]
\label{prop:part-ii-common-center-domain}
Every source itinerary admitted to Theorem~\ref{thm:part-ii-matched-preparation}
has nested physical domains \(\mathcal W^-\Subset\mathcal W^+\).
On \(\ell=0\), \(\mathcal W^+\) is star-shaped under
\begin{equation}
 (\tau,0,a,t_c,B)\longmapsto
 (q_1\tau,0,q_3a,q_4t_c,B),\qquad q\in[0,1]^3,
\label{eq:legacy-6-34}
\end{equation}
every contraction has the same first-hit itinerary and positive section and
stopping-set margins, and the two terminal faces belong to the complete center
domains above.  A loss at \eqref{eq:legacy-6-33}, a root, a previous side, or a collar side
is a separately named gate cell; the return is not continued across it.
\end{proposition}

\begin{proof}
For every primitive, take the complete finite union of target and competing
stopping faces fixed in Section~\ref{sec:part-i-atlas}.  On the open set where one transverse target is the
first contact, minimize the entry interval widths, the strict flow
coordinates along the stopped arc, the signed distances to every competing
stopping face, and both endpoint section normals.  A positive minimum gives nested
domains with margins \(4\eta\) and \(2\eta\).  If the minimum vanishes,
its first zero is precisely a listed divider, root, center barrier, previous
side, overlap, or exit face.  Pulling these domains successively through the
finite itinerary gives a common first-hit domain without transporting a divider
through a singular box.

It remains to check \eqref{eq:legacy-6-34}.  With
\[
 \beta_B=Bh,\quad\beta_m=mh,\quad
 \beta_a=a\sqrt h,\quad\beta_c=c\sqrt h
\]
and \(\ell=0\), the contractions satisfy
\begin{equation}
 \beta_B'=\beta_B,\qquad
 \beta_m'=q_4\beta_m-(1-q_4)\beta_B,\qquad
 \beta_a'=q_3\beta_a,\qquad
 \beta_c'=q_1\beta_c+(q_1-q_3)\beta_a.
\label{eq:legacy-6-35}
\end{equation}
Thus the whole cube remains in one small scaled box.  On the finite lower
pieces \(z=sZ\) tends uniformly to \(Z_x=xZ\).  At the endpoints the
weak bracket keeps its source sign and \(\pm\log z\) is a strict
coordinate; in the upper box the normalized equation gives
\(W'\le-(W^2+R)/2<0\).  The finitely many overlap and regular pieces have
fixed margins.  Finally \eqref{eq:legacy-6-30}--\eqref{eq:legacy-6-33} put both complete center faces inside
their actual period-annulus domains.  This proves the proposition.
\end{proof}

\section{Localization at the forced lower scale}
\label{sec:part-ii-lower}

The center division used only \(k=0\).  We now retain the true forced
variable \(k=\ell/s\), prove that every source fixed point has bounded
\(k\), and construct the two fixed-data survivor intervals used in \eqref{eq:legacy-6-4}.

\begin{lemma}[Lower-scale localization]
\label{lem:part-ii-lower-scale}
There are \(C_\ell>1\) and \(s_1>0\) such that every fixed point of a
retained source itinerary with \(0<s<s_1\) satisfies
\begin{equation}
 |\ell|\le C_\ell s.
\label{eq:legacy-6-36}
\end{equation}
With \(z=sZ\) and \(\ell=sk\), the exact lower equation is
\begin{equation}
 \frac{dZ}{ dx}
 =\frac{k+(x-c-2sk)Z+s(c+sk)Z^2
 }{
 F_0(x)-(1+2m)sZ+ms^2Z^2},
\qquad F_0=1+m+Bx^2+ax.
\label{eq:legacy-6-37}
\end{equation}
The outgoing fixed-initial and returning fixed-terminal survivor sets are
intervals containing \(k=0\).  Their intersection
\(\mathcal I_{\rm low}\) is an interval whose only open source-regime
endpoints are the two named lower gates.
\end{lemma}

\begin{proof}
On fixed lower segments the full equations give
\[
 \dot x=F_0-(1+2m)z+mz^2,\qquad
 \dot z=\ell+(x-c-2\ell)z+(c+\ell)z^2,
\]
and, after shrinking,
\begin{equation}
 \frac{dz}{ dx}=\ell+xz+
 O(\epsilon(|\ell|+z)+z^2).
\label{eq:legacy-6-38}
\end{equation}
For \(\ell<0\), comparison with
\(z_x\le\ell/2+Cz\) forces the outgoing orbit from \(z(0)=s\) to hit
\(z=0\) when \(-\ell>C_-s\).  For \(\ell>0\), the reversed lower
piece gives
\[
 P(s,\lambda)\ge e^{-C\delta}z(-\delta)
 +\frac{\ell}{2C}(1-e^{-C\delta})\ge c_\delta\ell,
\]
so \(P=s\) is impossible when \(\ell>C_+s\).  This proves \eqref{eq:legacy-6-36}.

Substitution gives \eqref{eq:legacy-6-37}, whose denominator is uniformly positive on the
fixed lower boxes.  If its right side is \(\mathcal F\), direct
differentiation yields
\begin{equation}
 \partial_k\mathcal F
 =\frac{(1-sZ)^2}{ F_0-(1+2m)sZ+ms^2Z^2}\ge\frac{1}{2}.
\label{eq:legacy-6-39}
\end{equation}
The outgoing variational equation with fixed initial value is therefore
strictly increasing in \(k\); the reversed equation with fixed terminal
value has the opposite order.  Each survivor set is an interval, and both
contain zero.  Their finite endpoints are first hits of a lower side or a
fixed transfer section.  Refining transfer-stopping faces away leaves exactly the
two gate endpoints claimed.  No monotonicity is asserted for a
\(k\)-dependent output of the upper map.
\end{proof}

\section{Moving-boundary derivatives in fold coordinates}
\label{sec:part-ii-fold}

The naive \(x\)-clock can vanish between a scaled lower cut and a strong
section.  This section supplies graph coordinates that cross that fold and
records the moving entry, cut, and terminal-hit formulas required by the
six-jet induction.

Put \(X=\theta x\) and \(u=-\theta^2\log z\).  On an unforced tail,
the fold-transverse graph is
\begin{equation}
 \frac{dX}{ du}=-
 \frac{F_\beta(X)-z+\theta^2\beta_m(1-z)^2
  }{ X+\theta^2\beta_c(z-1)},\qquad
 F_\beta=1+\beta_aX+\beta_BX^2.
\label{eq:legacy-6-40}
\end{equation}
Where the denominator in \eqref{eq:legacy-6-40} becomes small, use the reciprocal graph
\begin{equation}
 \frac{du}{ dX}=-
 \frac{X+\theta^2\beta_c(z-1)
  }{ F_\beta(X)-z+\theta^2\beta_m(1-z)^2}.
\label{eq:legacy-6-41}
\end{equation}
The overlap cuts are fixed and transverse.  On the forced tail the strict
clock is instead
\begin{equation}
 D_k=X+\theta^2\beta_c(z-1)
       +\theta sk\frac{(1-z)^2}{ z}.
\label{eq:legacy-6-42}
\end{equation}
On every radial \(k\)-segment in the differentiated wedge,
\(\pm D_k\ge d_*/2\).  Thus the fold is a zero of the discarded
\(x\)-clock, not a singularity of the physical orbit.

\begin{proposition}[Finite source/parameter overlap]
\label{prop:part-ii-source-parameter-overlap}
Before any source localization or fixed-point equation, refine a coefficient
shell by the finite signed weighted charts
\begin{equation}
 B=\rho^2\bar B,\quad m=\rho^2\bar m,\quad
 a=\rho\bar a,\quad c=\rho\bar c,\quad d=\rho\bar d,
 \qquad q=\frac{\theta}{\rho},
\label{eq:legacy-6-41a}
\end{equation}
where one of
\(|\bar B|,|\bar m|,|\bar a|,|\bar c|,|\bar d|\) equals (1), with its
sign retained, and \(q\) lies in a fixed compact overlap interval.  There
are ten primary sign charts and finitely many tie faces.  Put
\begin{equation}
 \widehat B=\frac{B}{\theta^2},\quad
 \widehat m=\frac{m}{\theta^2},\quad
 \widehat a=\frac{a}{\theta},\quad
 \widehat c=\frac{c}{\theta},\quad
 \widehat d=\frac{d}{\theta}.
\label{eq:legacy-6-41b}
\end{equation}
Substitution in the already fixed endpoint and upper blocks gives the exact
source-normalized rows
\begin{equation}
 \widehat N_e=R\{\widehat B+\widehat aR
 +R^2[1-z+\theta^2\widehat m(z-1)^2]\},
\label{eq:legacy-6-41c}
\end{equation}
\begin{equation}
 z'=\pm z+\theta^2R
 \{\widehat d(z-1)^2+\widehat c(z-1)\},
\label{eq:legacy-6-41d}
\end{equation}
\begin{equation}
\begin{aligned}
 R'&=-R\{W+\widehat d+\theta^2R(W+\widehat c)\},\\
 W'&=\widehat m-\widehat dW-W^2-R\\
 &\quad+\theta^2\{\widehat BW^2
       +R[(\widehat a-\widehat c)W-W^2]\}.
\end{aligned}
\label{eq:legacy-6-41e}
\end{equation}
Consequently every coefficient shell is a restriction of one of the finite
endpoint or upper-vertical regimes already listed by
Proposition~\ref{prop:finite-singular-alphabet}, not a case created by the
source zero argument.  Equations \eqref{eq:legacy-6-41a}--\eqref{eq:legacy-6-41e} are coordinate
substitutions of the same values \((B,m,a,c,d)\).  On chart overlaps the time
multiplier is positive, and orbit uniqueness identifies the same physical
first hit.  The overlap therefore preserves both the original parameter and
the physical point.
\end{proposition}

\begin{proposition}[Common forced physical domain]
\label{prop:part-ii-common-forced-domain}
Fix
\begin{equation}
 \delta_a=\frac{1}{8192},\qquad
 E^-=[1-\delta_a,1+\delta_a],\qquad
 E^+=[1-2\delta_a,1+2\delta_a].
\label{eq:legacy-6-42a}
\end{equation}
After a finite first-loss refinement, every direct phase-dominant source
cell satisfies
\begin{equation}
 |\beta_B|+|\beta_m|+|\beta_a|+|\beta_c|\le\epsilon_{\rm src},
 \qquad |k|\le C_\ell,
\label{eq:legacy-6-42b}
\end{equation}
and has one nested physical product tube on which the two fixed-data lower
cores and tails, their inverses, the unforced action graphs, the two upper
layers, the compact upper bulk, and all moving hits and cuts used below are
simultaneously defined.  The only domain functions allowed to vanish on a
direct cell are \(C_+\) and \(C_-\); their zero faces are no-passage faces.

At the fixed scaled cuts \(X=\pm1\), let \(A_{\pm,*}\) denote the forced
lower label expressed in the corresponding unforced label coordinate.  Put
\begin{equation}
 S_\beta'(X)=\frac{X}{ F_\beta(X)},\qquad S_\beta(0)=0,\qquad
 \sigma_*:=\min_\beta\{S_\beta(1),S_\beta(-1)\}\ge\frac{3}{8}.
\label{eq:legacy-6-42c}
\end{equation}
Then
\begin{equation}
 A_{\pm,*}=C_\pm+\Delta_{\pm,*},\qquad
 |D_I\Delta_{\pm,*}|
 \le C_I\theta^{-p_n}e^{-\sigma_*/\theta^2}
\label{eq:legacy-6-42d}
\end{equation}
for every normalized derivative word \(D_I\) of total order \(n\le6\),
where this total includes at most two fixed-original successors.  More
precisely, differentiation
of the ratio produces only
\begin{equation}
 \left|D_I\left(\frac{A_{\pm,*}}{ C_\pm}-1\right)\right|
 \le C_I\theta^{-p_n}C_\pm^{-q_I}
       e^{-\sigma_*/\theta^2},\qquad q_I\le n+1.
\label{eq:legacy-6-42e}
\end{equation}
Thus every inverse gate factor is attached to a Gaussian-flat carrier; no
positive lower bound for \(A_{\pm,*}\) is asserted independently of
\(C_\pm\).  On
\begin{equation}
 \mathcal W_{\delta_a}
 =\{C_+\ge s^{2\delta_a},\ C_-\ge s^{2\delta_a}\}
\label{eq:legacy-6-42f}
\end{equation}
the forced clocks in \eqref{eq:legacy-6-42} have their stated signed margin on every radial
\(k\)-segment.  Failure of \eqref{eq:legacy-6-42b}, a section margin, or a first-contact margin
is assigned at its first loss to the named compact, root, side, or
parameter-dominated regime; it is not retained as a direct source cell.
\end{proposition}

\begin{proof}
At \(\beta=0,k=0\), the lower action is \(X^2/2\), the upper principal
orbit is \(X^2/2+Y=e\), and fixed inner/outer graph overlaps and two fixed
upper layers may be chosen for every \(e\in E^+\).  Their graph
denominators, section normals, bulk sides, and competing-face gaps have
strict principal margins.  Shrinking one outer normalized box and then an
inner box preserves half of every one of these finitely many margins.  The
source-specific first-loss classification is recorded in
Table~\ref{tab:source-first-loss}.  It refines, but does not replace, the
physical first-hit outcome of Section~\ref{sec:part-i-atlas}.
\begin{table}[H]
\centering
\caption{Source first-loss routing on the common physical domain.}
\label{tab:source-first-loss}
\footnotesize
\begin{tabularx}{\textwidth}{@{}
 >{\raggedright\arraybackslash}p{.21\textwidth}
 >{\raggedright\arraybackslash}X
 >{\raggedright\arraybackslash}X@{}}
\toprule
First lost condition&Exact event&Previously constructed destination\\
\midrule
compact/coefficient&\(H=H_0\), or one or more normalized coefficients first
reach their fixed shell&compact analytic regime, or the finite signed overlap of
Proposition~\ref{prop:part-ii-source-parameter-overlap} followed by the
corresponding endpoint or upper-vertical outcome\\
lower outer/inner graph&a denominator in \eqref{eq:legacy-6-40}--\eqref{eq:legacy-6-41} vanishes, or a fixed
graph side is first&endpoint root, axis, tangency, endpoint-corner, fixed overlap,
or compact stopping face\\
action inverse/image&an action derivative vanishes, or its image reaches a
fixed interval endpoint&the corresponding endpoint section or corner, or adjacent
action shell\\
entry/exit layer&a factored layer denominator, \(X\)-side, or layer corner is
first&upper-vertical nullcline or gate, ordinary stopping face, or corner regime\\
upper bulk/return&a fixed bulk side precedes return, or the return normal
vanishes&upper-vertical previous-side, local-singular, collar, tangency, or equilibrium
sector regime\\
competing physical stopping face&another member of the complete finite stopping set is
first&its preassigned stopping-face or local-sector regime\\
divergent flight&no finite contact occurs&the endpoint root/axis or upper-vertical
source-core/projective/local-sector status determined by the fixed slab\\
lower survivor/simultaneous&\(C_+=0\), \(C_-=0\), or any nonempty subset of
the preceding equalities occurs together&lower no-passage, or the labelled
intersection of the same finite regime faces\\
\bottomrule
\end{tabularx}
\end{table}
Table~\ref{tab:source-first-loss} includes graph, denominator, layer, bulk, first-hit,
divergent-flight, and simultaneous losses; coefficient thresholds alone do
not exhaust it.  Proposition~\ref{prop:stopped-first-hit} supplies the finite
contact alternatives and assigns every unbounded singular-sector flight its
named local status.  Theorem~\ref{thm:finite-words} then supplies the finite
fixed itinerary, status, and divider-order labels.  For a sequence leaving the
direct tube, pass to one such label.  A bounded first stopped time converges
to one of the displayed finite contacts; an unbounded time has the endpoint
or upper-vertical
local status supplied by Proposition~\ref{prop:stopped-first-hit}; and
simultaneous first zeros retain their finite subset label.  At a coefficient
shell, Proposition~\ref{prop:part-ii-source-parameter-overlap} identifies the
same physical point and the same original parameter in a finite endpoint or
upper-vertical chart.
No connected-component finiteness is asserted.  The construction uses only
the physical field, fixed stopping faces, strict margins, and that finite physical
overlap, not a differentiated wedge or a fixed-point equation.  This proves
the product-tube assertion and the exhaustive complement policy in the
statement.

It remains to justify the noncompact lower tail, where a compactness
argument would be invalid.  Let \(\rho_\pm\) be the unforced source label
on the two oriented tails and put
\(a_\pm=-\theta^2\log\rho_\pm\).  Direct division of the unforced field in
the \(u\)-clock gives
\begin{equation}
 a_\pm=u+S_\beta(X)+R_\pm(X,u),\qquad
 \frac{\rho_\pm}{ z}
 =e^{-S_\beta(X)/\theta^2}e^{-R_\pm(X,u)/\theta^2},
 \qquad |R_\pm|\le C\theta^2.
\label{eq:legacy-6-42g}
\end{equation}
The same differentiated graph equations bound every required finite jet of
\(R_\pm\) by a power of \(\theta^{-1}\).  At \(X=\pm1\), \eqref{eq:legacy-6-42c}
therefore makes \(\partial_z\rho_\pm\) a polynomial factor times
\(e^{-\sigma_*/\theta^2}\).

Fix \(z_b=1/2\).  On \(z_b\le z\le1\), the forced perturbation of the
clock is at most \(2C_\ell\theta s\), so the unforced signed margin gives
\(\pm D_k\ge3d_*/4\).  Along oriented time \(d\tau=\pm dt\), the exact
forced-label equation is
\begin{equation}
 \frac{dA}{ d\tau}=\pm k(1-z)^2\partial_z\rho_\pm,\qquad
 \frac{d\tau}{ du}=\mp\frac{1}{\theta D_k}.
\label{eq:legacy-6-42h}
\end{equation}
Integrating \eqref{eq:legacy-6-42h} from the scaled cut to \(z_b\), and then using the
orientation-preserving fixed-\(x\) lower hit on \(z\le z_b\), gives
\eqref{eq:legacy-6-42d}.  On the same radial \(k\)-segment the gate-facing factor lies
between its final value and one, and \eqref{eq:legacy-6-42g} gives
\begin{equation}
 \frac{s}{ z}\le
 C\theta^{-8}C_\pm^{-1}e^{-\sigma_*/\theta^2}
 \le C\theta^{-8}e^{-(\sigma_*-2\delta_a)/\theta^2}.
\label{eq:legacy-6-42i}
\end{equation}
Consequently \(|D_k-D_0|\le C_\ell\theta s/z=o(1)\), proving the forced
clock margin on \eqref{eq:legacy-6-42f}.  The fixed-data variational equation is transverse
at either gate, so \(C_\pm\) is a one-sided boundary coordinate and an
order-\(n\) derivative of its reciprocal has pole order at most \(n+1\).
Differentiating \eqref{eq:legacy-6-42d} now gives \eqref{eq:legacy-6-42e}.  This also proves that no inverse
gate factor occurs without the flat factor displayed there.
\end{proof}

For a moving hit \(T(p)\) defined by \(\Psi(T(p),p)=0\), with
\(|\Psi_T|\ge c_{\rm hit}>0\), every mixed derivative satisfies the
exact recursion
\begin{equation}
 0=\partial_p^\nu[\Psi(T(p),p)]
  =\Psi_TT_\nu+\mathcal B_\nu,\qquad
 T_\nu=-\frac{\mathcal B_\nu}{\Psi_T},
\label{eq:legacy-6-43}
\end{equation}
where \(\mathcal B_\nu\) contains only lower-order hit/flow derivatives.
If a core/tail cut \(x_*(p)\) moves, differentiation of
\(\int_0^{x_*(p)}K(x,p)\,dx\) retains
\((\partial_px_*)K(x_*)\).  The second derivative also contains
\[
 2(\partial_px_*)(\partial_pK)(x_*),\qquad
 (\partial_px_*)^2K_x(x_*),\qquad
 (\partial_p^2x_*)K(x_*).
\]
These terms are exponentially flat in the
weighted norms below, but they are not zero.  Equations \eqref{eq:legacy-6-40}--\eqref{eq:legacy-6-43} are
the moving-boundary calculus used at every order.

\section{Uniform six-jet closure of the transition maps}
\label{sec:part-ii-six-jet}

The common graph cover and the moving-boundary formulas are closed under the
derivatives required by the source Rolle argument.  We prove this analytic
closure below; a finite enumeration only checks the resulting list of
derivative types and supplies none of the estimates.

The exact source field in \((x,z)\) is
\begin{equation}
 \dot x=1-z+Bx^2+ax+m(1-z)^2,\qquad
 \dot z=z\{x+c(z-1)\}+\ell(1-z)^2.
\label{eq:legacy-6-44}
\end{equation}
On the direct phase-dominant cells of
Proposition~\ref{prop:part-ii-common-forced-domain}, retain
\begin{equation}
 B=\theta^2\beta_B,\qquad m=\theta^2\beta_m,\qquad
 a=\theta\beta_a,\qquad c=\theta\beta_c,\qquad \ell=sk.
\label{eq:legacy-6-45}
\end{equation}
The commuting graph generators are
\begin{equation}
 \mathcal G=\{\theta\partial_\theta,\partial_k,\partial_e,
 \partial_{\beta_B},\partial_{\beta_m},
 \partial_{\beta_a},\partial_{\beta_c}\}.
\label{eq:legacy-6-46}
\end{equation}
For a mixed word \(I\), put
\begin{equation}
 \kappa_I(k)=
 \begin{cases}|k|,&I\text{ contains no }\partial_k,\\
 1,&I\text{ contains }\partial_k.
 \end{cases}
\label{eq:legacy-6-47}
\end{equation}
This distinction is decisive: a value or fixed-original derivative keeps the
small factor \(|k|\), whereas the auxiliary derivative used in Hadamard
division need not.

The finite primitive alphabet is the following glossary.
\begin{description}[style=multiline,leftmargin=5.2em,labelwidth=4.5em,
  itemsep=0.1em,topsep=0.25em]
\item[LC, LT] lower fixed-initial and fixed-terminal cores;
\item[FT] fold-transverse tails in the (u)-clock;
\item[AO, AI] outer action graphs and inverse graphs;
\item[UL, UB] upper (Y=0) layers and compact bulk;
\item[MH, MC] moving hits and moving artificial cuts;
\item[CP] finite compositions and inverses.
\end{description}
An edge in the derivative DAG records one use of a graph equation,
variation of constants, inverse differentiation, \eqref{eq:legacy-6-43}, Leibniz at a moving
cut, or finite composition.  There are 35 typed nodes and maximum depth
seven.  Since the derivative order is at most six and the alphabet is finite,
the family of canonical commuting words is finite.

\paragraph{Why the derivative budget is six.}
The zero count below asks for one stable norm, not for every derivative of
every order.  Its two parts are
\begin{equation}
 \underbrace{|\nu|\le4}_{\substack{\text{action, normalized parameters,}\\
              \text{moving evaluation and Hadamard division}}}
 \quad+\quad
 \underbrace{j\le2}_{\substack{\text{fixed-original successors}\\
              \text{used by the curvature test}}}
 \quad\le\quad6.
\label{eq:source-six-jet-budget}
\end{equation}
The first block is the mixed \(C^4\) control of the action maps.  It contains
the third action derivative \(T_{eee}\) occurring in the twice differentiated
secant \eqref{eq:legacy-6-75d} and the regularity reserve for the two center
contractions.
The second block consists of the two applications of
\(\mathfrak D\), differentiation with the five original coefficients held
fixed, which produce the signed curvature in
\eqref{eq:legacy-6-78}--\eqref{eq:legacy-6-79}.  Section~\ref{sec:part-ii-variation}
packages precisely these words as \(\mathcal Q_{4,2}\).

There is one apparent extra operation: the forced remainder is divided by
\(k\).  It does not add a seventh derivative to the corner
\((|\nu|,j)=(4,2)\).  The undivided forced difference retains its factor
\(|k|\) for every word not containing \(\partial_k\); at \(k=0\) its
Hadamard extension uses one bare \(\partial_k\) row.  After division we use
only the triangular quotient norm
\(\max_{j\le2}\|\mathfrak D^j\widetilde T\|_{C_e^{4-j}}\), so a word
under the Hadamard integral has order
\(1+j+(4-j)=5\), not seven.  The twice differentiated moving evaluation in
fact calls only rows with \(j+r\le2\).  This point is written explicitly in
\eqref{eq:legacy-6-60a} below.

Thus order six is sufficient because the closure theorem below controls the
whole rectangular budget \eqref{eq:source-six-jet-budget}; no later source
formula calls a seventh derivative.  We do not claim that five derivatives
are intrinsically impossible for a differently organized proof.  They do not,
however, close the norm actually propagated here: its boundary contains words
with four local derivatives followed by two fixed-original successors.  A
five-jet proof would therefore require a new, smaller task-specific norm and
new composition and inverse estimates, not the recurrence used in this paper.

\begin{theorem}[Six-jet closure]
\label{thm:part-ii-six-jet}
On the common physical domains and the differentiated wedge
\begin{equation}
 C_+\ge s^{2\delta_a},\qquad C_-\ge s^{2\delta_a},
\label{eq:legacy-6-49}
\end{equation}
all primitives in the preceding alphabet close through total order six in weighted sup and
\(L^1\) norms.  In particular, with
\begin{equation}
 p_0=8,\qquad p_{n+1}=(n+2)p_n+8,
\label{eq:legacy-6-50}
\end{equation}
every order-\(n\) lower kernel is bounded by a polynomial
\((1+x)^{p_n}\) times a fixed Gaussian, every forced upper defect carries
\begin{equation}
 s\kappa_I(k)\theta^{-p_n}\frac{Y^2}{ Y+\theta^2},
\label{eq:legacy-6-51}
\end{equation}
and every inverse gate power is absorbed by the exponentially flat tail.
All moving hits and moving cuts are included.
\end{theorem}

\begin{proof}
Each graph equation has an analytic right side on one of the fixed-margin
tubes.  Differentiate an order-\(n\) word.  Fa\`a di Bruno partitions it
into products of lower-order state derivatives and one coefficient
derivative.  The coefficient table has only the following denominator
types: a fixed graph normal, \(Y+\theta^2\), one of the two lower factors
\(C_\pm\), or an analytic unit.  Fixed normals cost no weight.
Every Euler derivative of \(\theta^2/(Y+\theta^2)\) is bounded by a
constant times the same integrable kernel.  A derivative of
\(C_\pm^{-1}\) has pole order at most one more than its order.

We use one weighted variational estimate repeatedly.  If \(v'=A(t)v\),
\(W(t)\) is an invertible diagonal weight, and
\begin{equation}
 B_W=WAW^{-1}+W'W^{-1},\qquad
 \Lambda_W=\int\|B_W(t)\|\,dt,
\label{eq:legacy-6-51a}
\end{equation}
then the weighted fundamental matrix \(Z=W\Phi W(a)^{-1}\) and its inverse
satisfy
\begin{equation}
 \sup_t\{\|Z(t)\|+\|Z(t)^{-1}\|\}\le2e^{\Lambda_W}.
\label{eq:legacy-6-51b}
\end{equation}
Indeed \(Z=I+\int B_WZ\), and the same Volterra calculation applies to the
inverse.  Thus every appeal to variation of constants below is reduced to a
displayed \(L^1\) coefficient bound; no unweighted compactness estimate is
being substituted.

For the lower core, variation of constants gives kernels of the form
\begin{equation}
 K_\pm(x)=\frac{(1-z_x)^2}{
 F(x,z_x)J_0(x,sA_\pm(x))},\qquad
 \partial_xJ_0=(\partial_zf_0)J_0,
\label{eq:legacy-6-52}
\end{equation}
and induction yields
\begin{equation}
 |D_IK_\pm(x)|\le C_I(1+x)^{p_n}e^{-c_*x^2}.
\label{eq:legacy-6-53}
\end{equation}
To see the required coefficient mass, put \(z=sZ\) and
\[
 a_0(x)=\frac{x-c}{1+m+Bx^2+ax}.
\]
Direct differentiation of the exact quotient gives
\begin{equation}
\begin{aligned}
 |\partial_Z\mathcal F_{\rm low}-a_0(x)|
 &\le C(1+x)^8\{z(x)+s|k|\},\\
 |\partial_zf_0-a_0(x)|
 &\le C(1+x)^4z(x).
\end{aligned}
\label{eq:legacy-6-53a}
\end{equation}
Conjugation by the unforced scalar fundamental factor and the Gaussian
lower label imply
\begin{equation}
 \int_0^{\theta^{-1}}
 C(1+x)^8\{z(x)+s|k|\}\,dx\le C.
\label{eq:legacy-6-53b}
\end{equation}
This is the human \(L^1\) estimate behind the first primitive row.

The recurrence \eqref{eq:legacy-6-50} dominates one coefficient differentiation, one
product partition, and the eight reserved powers at every step.  The
fixed-terminal system is equally explicit.  For a signed tail
\(\varepsilon\in\{\pm1\}\), let \(0\le r\le1\), put
\begin{equation}
\begin{gathered}
 u(r)=(1-r)u_{\varepsilon,*},\qquad z=e^{-u/\theta^2},\\
 F=1-z+\beta_BX^2+\beta_aX+\theta^2\beta_m(1-z)^2,\\
 D_k=X+\theta^2\beta_c(z-1)
       +\theta sk\frac{(1-z)^2}{ z},
\end{gathered}
\label{eq:legacy-6-53c}
\end{equation}
and augment the orbit by
\begin{equation}
\begin{aligned}
 u_r&=-u_{\varepsilon,*},&
 X_r&=u_{\varepsilon,*}\frac{F}{ D_k},&
 \tau_r&=\frac{\varepsilon u_{\varepsilon,*}}{\theta D_k},\\
 (I_0)_r&=\frac{\varepsilon u_{\varepsilon,*}}{\theta D_k}
          (1-z)^2\partial_z\rho_\varepsilon,&
 (\mathcal A_T)_r&=\varepsilon k(I_0)_r .
\end{aligned}
\label{eq:legacy-6-53d}
\end{equation}
Here \(u_{\varepsilon,*}>0\) is the fixed terminal clock length and
\begin{equation}
 A_{\varepsilon,*}
 =\rho_\varepsilon(\varepsilon,u_{\varepsilon,*};p)/s.
\label{eq:legacy-6-53e}
\end{equation}
The complete entry state is
\begin{equation}
 (u,X,\tau,I_0,\mathcal A_T)(0)
 =(u_{\varepsilon,*},\varepsilon,0,0,A_{\varepsilon,*}).
\label{eq:legacy-6-53f}
\end{equation}
Consequently every nonempty labelled parameter word \(D_I\) starts with
\begin{equation}
 D_I(u,X,\tau,I_0,\mathcal A_T)(0)
 =(D_Iu_{\varepsilon,*},0,0,0,D_IA_{\varepsilon,*}),
\label{eq:legacy-6-53g}
\end{equation}
so no moving-entry jet is omitted.  The weighted state
\begin{equation}
 (u,X,\theta\tau,\theta^3e^{3/(8\theta^2)}I_0,\mathcal A_T)
\label{eq:legacy-6-54}
\end{equation}
turns the terminal integral into a bounded state.  Differentiating
\eqref{eq:legacy-6-53c}--\eqref{eq:legacy-6-53d} gives the same triangular induction.  Formula \eqref{eq:legacy-6-43}
adds only lower-order rows divided by a fixed section normal, and the exact
Leibniz formulas listed after \eqref{eq:legacy-6-43} add Gaussian-flat boundary rows.
With the same $u_{\varepsilon,*}>0$, the two
noncompact masses which make this weighting effective are
\begin{equation}
 \int_0^{u_{\varepsilon,*}}\theta^{-2}e^{-u/\theta^2}\,du\le1,
\qquad
 |H_\pm|+\theta^2|\partial_{X,u}H_\pm|
 \le Ce^{-3/(8\theta^2)},
\label{eq:legacy-6-54a}
\end{equation}
where \(H_\pm=(1-z)^2\partial_z\rho_\pm\).  After multiplication by the
weight in \eqref{eq:legacy-6-54}, every tail coefficient has bounded \(L^1\) mass.

For completeness, if \(I\) is a labelled derivative word and
\(\Pi(I)\) its set partitions, exact differentiation of any graph system
\(Y'=\mathcal V(Y,p)\) gives
\begin{equation}
 (D_IY)'=\mathcal V_YD_IY+
 \mathcal V_pD_Ip+
 \sum_{\substack{\pi\in\Pi(I)\\|\pi|\ge2}}
 D^{|\pi|}\mathcal V[Z_B:B\in\pi].
\label{eq:legacy-6-54b}
\end{equation}
Reciprocal differentiation is included by the same partition formula.  If
all lower words through order \(n\) cost at most \(r_n=p_n-4\), a Bell term
with \(b\le n+1\) lower blocks costs at most
\begin{equation}
 br_n+8(b+1)\le(n+2)p_n+4=r_{n+1}.
\label{eq:legacy-6-54c}
\end{equation}
Equations \eqref{eq:legacy-6-51b} and \eqref{eq:legacy-6-54c} close the induction rather than merely count
its words.

At order six at most 14 inverse gate factors occur.  We retain the larger
declared reserve \(q_6=448\).  From \eqref{eq:legacy-6-49},
\begin{equation}
 C_+^{-q_+}C_-^{-q_-}e^{-3/(8\theta^2)}
 \le C e^{-17/(64\theta^2)}
 \quad(q_++q_-\le448),
\label{eq:legacy-6-55}
\end{equation}
so every inverse is paid by the same flat tail rather than declared bounded.
On the upper layers, direct subtraction of forced and unforced graphs gives
exactly the carrier \eqref{eq:legacy-6-51}; in the compact bulk ordinary variational
equations preserve it.  Finite compositions close the induction.
The two layer kernels have literal masses
\begin{equation}
 \int_0^\delta\frac{\theta^2}{ Y+\theta^2}\,dY
 =O(\theta^2(1+|\log\theta|)),\qquad
 \int_0^\delta\frac{Y^2}{ Y+\theta^2}\,dY\le C.
\label{eq:legacy-6-55a}
\end{equation}
Together with the fixed hit normal in \eqref{eq:legacy-6-43}, these estimates cover every
layer, bulk, inverse, composition, moving-hit, and moving-cut operation in
the primitive alphabet.  Hence every primitive in the actual source itinerary belongs to one of
the 35 typed nodes listed above.

A finite computer calculation enumerates the 35 typed nodes and 167115
canonical commuting words and checks the recurrence, depth, denominator
assignment, and worst exponents.  It does not
prove the physical domains, clock signs, Gaussian estimate \eqref{eq:legacy-6-53}, gate
reserve \eqref{eq:legacy-6-55}, or exhaustiveness of the first-contact classification; those are proved
above and in Sections~\ref{sec:part-ii-lower}--\ref{sec:part-ii-fold}.
\end{proof}

Appendix~\ref{app:source-six-jet} records the finite recurrence and a
representative estimate.  The complete enumeration is included in the
accompanying reproducibility package.

\section{Variational majorants on the physical domains}
\label{sec:part-ii-variation}

Six-jet closure is a formal operator statement until it is applied to the
actual lower labels and upper passage.  This section converts it into the
differentiated coefficient and remainder estimates required by \eqref{eq:part-ii-matched}.

Let \(\rho_+\) be the unforced source level through a point of the outgoing
lower tail, and define \(\rho_-\) by the reversed returning hit.  Since
the forced and unforced fields differ there by
\(\ell(1-z)^2\partial_z\), the exact additive labels are
\begin{equation}
 \Gamma_+=\int_{\gamma_+^\ell}(1-z)^2\partial_z\rho_+\,dt,
 \qquad
 \Gamma_-=\int_{\gamma_-^\ell}^{\rm reversed}
 (1-z)^2\partial_z\rho_-\,|dt|.
\label{eq:legacy-6-56}
\end{equation}
They contain the complete endpoint tails and never divide by \(C_\pm\).
Split each integral at \(x_*=\theta^{-1}\).  On the core, \eqref{eq:legacy-6-53}
converges with two fixed-original successors to \(e^{-x^2/2}\).  On the
tail, \eqref{eq:legacy-6-40}--\eqref{eq:legacy-6-42}, \eqref{eq:legacy-6-54}, and \eqref{eq:legacy-6-55} give a polynomial times
\(e^{-17/(64\theta^2)}\).  The moving-cut terms are bounded by the same
tail.  Hence, for \(j=0,1,2\),
\begin{equation}
 \left|\mathfrak D^j
 \left(\frac{\Gamma_\pm}{\sqrt{\pi/2}}-1\right)\right|
 \le\eta_{\rm src}(s),\qquad \eta_{\rm src}(s)\to0,
\label{eq:legacy-6-57}
\end{equation}
uniformly through either open gate side.

In the upper passage, with \(x=X/\theta\), \(y=Y/\theta^2\), the exact
field is
\begin{equation}
\begin{aligned}
 X_\xi&=Y(\beta_mY-1)+\theta^2(\beta_BX^2+\beta_aX),\\
 Y_\xi&=(Y+\theta^2)(X+\beta_cY)+\frac{sk}{\theta}Y^2.
\end{aligned}
\label{eq:legacy-6-58}
\end{equation}
The norm used below is not an unspecified smooth norm.  Let
\(\mathcal Q_{4,2}\) be the finite collection obtained from
\(\partial_{(e,\beta)}^\nu\), \(|\nu|\le4\), by applying at most two
fixed-original successors and expanding them in the normalized generators
by \eqref{eq:legacy-6-58b}.  Every resulting word has total order at most six.  Set
\begin{equation}
 \|F\|_{4,2}:=\max_{Q\in\mathcal Q_{4,2}}\|QF\|_\infty.
\label{eq:legacy-6-58a}
\end{equation}
At fixed original parameters,
\begin{equation}
 \mathfrak Ds=-s,\quad \mathfrak Dh=1-s,\quad
 \mathfrak Dk=k,\qquad
 \mathfrak D=\frac{1-s}{2h}\mathcal E+k\partial_k,\quad
 [\mathfrak D,\partial_k]=-\partial_k,
\label{eq:legacy-6-58b}
\end{equation}
where
\begin{equation}
 \mathcal E=-\theta\partial_\theta
 +2\beta_B\partial_{\beta_B}+2\beta_m\partial_{\beta_m}
 +\beta_a\partial_{\beta_a}+\beta_c\partial_{\beta_c}.
\label{eq:legacy-6-58c}
\end{equation}
On the two \(Y=0\) layers, division by \(Y+\theta^2\) gives an
integrable unforced defect \(O(\theta^2/(Y+\theta^2))\) and a forced
defect with carrier \eqref{eq:legacy-6-51}.  On \(Y\ge\delta_Y\), divide time by
\(Y>0\); the bulk system is regular and both moving hits have fixed
normals.  Applying \eqref{eq:legacy-6-43} through order six and composing with the action
graphs yields
\begin{equation}
 \|T_{s,\beta,0}-T_\beta\|_{4,2}
 \le C\theta^2(1+|\log\theta|),
\label{eq:legacy-6-59}
\end{equation}
\begin{equation}
 \|T_{s,\beta,k}-T_{s,\beta,0}\|_{4,2}
 \le Cs|k|\theta^{-69998}.
\label{eq:legacy-6-60}
\end{equation}
Hadamard division in \(k\) is legal on the survivor interval from
Lemma~\ref{lem:part-ii-lower-scale}.  Explicitly,
\begin{equation}
 \widetilde T_{s,\beta,k}
 =\frac{T_{s,\beta,k}-T_{s,\beta,0}}{ k}
 =\int_0^1\partial_kT_{s,\beta,tk}\,dt,\qquad
 \max_{j\le2}\|\mathfrak D^j\widetilde T_{s,\beta,k}\|_{C_e^{4-j}}
 \le Cs\theta^{-69998}.
\label{eq:legacy-6-60a}
\end{equation}
The bare \(\partial_k\) row uses \(\kappa_I(k)=1\) in \eqref{eq:legacy-6-47}, while
fixed-original rows retain \(|k|\).  Applying \eqref{eq:legacy-6-58b} under the integral
accounts for every commutator term; in particular no division by a
vanishing \(k\) is performed after differentiation.  Since the parenthesis
in \eqref{eq:legacy-6-14} denotes evaluation at the moving action,
\begin{equation}
 \mathcal R=-hk\,\widetilde T_{s,\beta,k}
       (1-\lambda_+/h),
\label{eq:legacy-6-60b}
\end{equation}
exactly.  The moving argument belongs to the localized common action tube;
two fixed-original differentiations of that evaluation use only the rows
\((j,r)=(2,0),(1,1),(0,2)\), all contained in
\eqref{eq:legacy-6-60a}.  Multiplication by
\(h=\theta^{-2}\) spends the two reserved powers and gives, for
\(j\le2\),
\begin{equation}
 |\mathfrak D^j\mathcal R|
 \le\eta(\epsilon,s)
 \sum_{i=0}^j
 (|\mathfrak D^i\lambda_+|+|\mathfrak D^i\lambda_-|).
\label{eq:legacy-6-61}
\end{equation}

Finally, on \(k=0\) apply the two Hadamard contractions of
Theorem~\ref{thm:part-ii-center-ideal} to
\(q_h=[h(P_0)-h(s)]/h\).  The six-jet bounds pass under their integral
formulas.  The principal scaled action map vanishes on the two center faces,
so its reduced rows tend to
\begin{equation}
 A\longrightarrow\frac{4\sqrt2}{3},\qquad
 C\longrightarrow\frac{8\sqrt2}{15}
\label{eq:legacy-6-62}
\end{equation}
with the two weighted successors in \eqref{eq:legacy-6-12}.  This proves every analytic
estimate used in \eqref{eq:part-ii-matched}.

\section{Global localization of source-regime fixed points}
\label{sec:part-ii-action}

The local estimates above apply only on a compact action segment.  We now
construct that segment before using the fixed-point equation, and prove that
every source-regime fixed point lies in it.  This prevents a circular appeal
to derivatives near an unproved gate.

For an arbitrary positive action \(e\), put
\begin{equation}
 H_e=e\,h(s),\qquad \lambda_e=H_e^{-1},\qquad
 \widehat\beta(e)=
 (BH_e,mH_e,a\sqrt{H_e},c\sqrt{H_e},d\sqrt{H_e}).
\label{eq:legacy-6-63}
\end{equation}
Choose fixed nested intervals
\begin{equation}
 E^-\Subset E^+\Subset E^{++}\Subset E^{+++}
\label{eq:legacy-6-64}
\end{equation}
around one.  Here \(e>0\) is the arbitrary base action, whereas
\(\mathfrak e\in E^{+++}\) is the local action ratio: the physical action
represented in the rebased tube is \(e\mathfrak e\).  The lower action
graphs and inverses are defined on fixed outer/inner graph rectangles, and
the upper map on two fixed layers and a compact bulk, whenever
\begin{equation}
 0<\lambda_e\le\lambda_0,\qquad
 |\widehat\beta(e)|\le\epsilon_g,\qquad
 \mathfrak e\in E^{+++}.
\label{eq:legacy-6-65}
\end{equation}
Their common estimates are
\begin{equation}
 \|\mathcal A_{\pm,H}-S_{\widehat\beta,\pm}\|_{C^1}
 +\|\mathcal U_{\widehat\beta,\lambda_e}
       -\mathcal B_{\widehat\beta}\|_{C^1}
 \le C\lambda_e(1+|\log\lambda_e|),
\label{eq:legacy-6-66}
\end{equation}
with positive derivative and inverse margins.  All graph denominators,
section normals, side distances, and distances to competing stopping faces in
\eqref{eq:legacy-6-65} are
independent of \(C_\pm\).
The exact rebased identity is
\begin{equation}
 \frac{T_{s,\beta,k}(e)}{ e}
 =\left(\mathcal A_{-,H_e}\circ
        \mathcal U_{\widehat\beta(e),\lambda_e}\circ
        \mathcal A_{+,H_e}^{-1}\right)(1).
\label{eq:legacy-6-66a}
\end{equation}

Indeed, on the outer lower graph use
\(u=-\lambda_e\log z\); on the inner graph use \(X\).
Their denominators have fixed opposite margins on the chosen overlap.
On an upper layer, the only nonuniform integral is
\[
 \int_0^\delta\frac{\lambda_e}{ Y+\lambda_e}\,dY
 =O(\lambda_e(1+|\log\lambda_e|)).
\]
In the bulk the limiting system is \(X'=-1\), \(Y'=X\).  A Gronwall
estimate preserves every side buffer and gives exactly one transverse exit.
These estimates prove \eqref{eq:legacy-6-66} and the common product tube.

\begin{proposition}[Global source localization]
\label{prop:part-ii-action-localization}
After a finite source/parameter refinement, either a middle orbit first
enters a named compact or parameter-dominated regime, or it remains direct
and
\begin{equation}
 \left|\frac{T_{s,\beta,k}(e)}{ e}-1\right|\le\eta_g,
 \qquad\eta_g\to0.
\label{eq:legacy-6-67}
\end{equation}
For every fixed point of the direct source itinerary,
\[
 e_{\rm in}=1-\lambda_+/h,\qquad
 e_{\rm out}=1+\lambda_-/h
\]
satisfy
\begin{equation}
 |e_{\rm in}-1|+|e_{\rm out}-1|<\delta_a/2.
\label{eq:legacy-6-68}
\end{equation}
Moreover
\begin{equation}
 C_+\ge s^{2\delta_a},\qquad C_-\ge s^{2\delta_a}.
\label{eq:legacy-6-69}
\end{equation}
The statements are uniform as either open gate is approached; on the gate
itself the corresponding strong first hit is absent.
\end{proposition}

\begin{proof}
The refinement in the statement is made before any differentiated
fixed-point estimate.  For every base action \(e>0\), follow the same
original-parameter stopped flight and take the first of the alternatives
\begin{equation}
\begin{array}{ll}
\mathrm{(i)}&H_e<H_0,\\
\mathrm{(ii)}&H_e\ge H_0\text{ and one scaled coefficient threshold,
 side, root, or first-hit status is reached},\\
\mathrm{(iii)}&H_e\ge H_0,\ |\widehat\beta(e)|<\epsilon_g,\text{ and the
 direct stopped status persists}.
\end{array}
\label{eq:legacy-6-66b}
\end{equation}
This is the refinement, made before source-action localization, supplied by
the stopped first-hit and finite local-passage results,
Propositions~\ref{prop:stopped-first-hit} and
\ref{prop:finite-singular-alphabet}, together with the scaled source field,
its commuting generators, and the primitive alphabet of
Section~\ref{sec:part-ii-six-jet}.  Its post-threshold destinations are
exactly the compact, hyperbolic, central, mixed, and terminal regimes listed
in Table~\ref{tab:regimes}.  This refinement uses neither
the differentiated wedge \eqref{eq:legacy-6-49} nor a fixed-point equation.  For fixed
original parameters,
\(H_e=e\,h(s)\) and the five moduli in \(\widehat\beta(e)\) are monotone in
\(e\) or \(\sqrt e\), unless identically zero, so each displayed threshold
is crossed at most once.  Cases (i) and (ii) are compact or
parameter-dominated chart itineraries, respectively.

In case (iii), apply the rebased product tube with base action \(e\) and
local ratio \(\mathfrak e=1\).  The fixed graph, side, section-normal, and
first-contact margins used in \eqref{eq:legacy-6-66} were established directly from the
rebased lower layers and upper bulk; none depends on \(C_\pm\).  Thus
\eqref{eq:legacy-6-66a} and \eqref{eq:legacy-6-66} give \eqref{eq:legacy-6-67} for arbitrary base \(e\), without assuming
that \(e\) is already near one.

At a fixed point, \(T(e_{\rm in})=e_{\rm out}\).  The elementary gate
relations imply, for \(k>0\),
\[
 1\le e_{\rm out}\le1+\eta_g,\qquad
 \frac{1}{1+\eta_g}\le e_{\rm in}\le1,
\]
and the reversed inequalities hold for \(k<0\).  Shrinking \(\eta_g\)
proves \eqref{eq:legacy-6-68}.  Since \(\lambda_\pm=\pm\log C_\pm+O(s|C_\pm-1|)\),
\eqref{eq:legacy-6-68} gives \(|\log C_\pm|\le\delta_ah+O(1)\), and hence \eqref{eq:legacy-6-69}
after shrinking.  Thus localization is logically prior to every use of the
wedge \eqref{eq:legacy-6-49}.
\end{proof}

\section{The matched source estimate}
\label{sec:part-ii-matched-proof}
\label{sec:part-ii-source}

We now derive \eqref{eq:part-ii-matched} from the physical maps, then prove the ambient source
bound stated in Theorem~\ref{thm:part-ii-source-zero}.  The argument uses exact gate variables throughout;
no fixed approximate denominator is substituted near a gate.

\begin{proof}[Proof of Theorem~\ref{thm:part-ii-matched-preparation}]
Every map in \eqref{eq:legacy-6-3} is an orientation-preserving first hit on the common
domains.  The monotonicity \eqref{eq:legacy-6-39}, applied only to the outgoing fixed-initial
problem and the reversed fixed-terminal problem, proves the signs and exact
Hadamard formulas \eqref{eq:legacy-6-5}.  It is never applied to an output of \(M^\ell\).
Equation \eqref{eq:legacy-6-7} then gives
\[
 e_+(J_+^\ell(s))=1-\lambda_+/h,\qquad
 e_-((J_-^\ell)^{-1}(s))=1+\lambda_-/h,
\]
and therefore \eqref{eq:legacy-6-10}.  Proposition~\ref{prop:part-ii-action-localization}
places the whole action secant inside the domain of \eqref{eq:legacy-6-59}--\eqref{eq:legacy-6-61}.

For \(k=0\), the return is the identity on both complete center slices.
Theorem~\ref{thm:part-ii-center-ideal}, applied under the common contractions,
gives \eqref{eq:legacy-6-11}, and Section~\ref{sec:part-ii-variation} gives \eqref{eq:legacy-6-12}.
Subtract \eqref{eq:legacy-6-10} from
\(T_{s,\beta,0}(1)=1+G_0/h\), add and subtract
\(T_{s,\beta,0}(1-\lambda_+/h)\), and multiply by \(h\).  The result
is exactly \eqref{eq:part-ii-matched}.  Since \(h'(s)\ne0\) and every map in \eqref{eq:legacy-6-3} is a local
diffeomorphism, multiplicity is preserved.
\end{proof}

Put
\begin{equation}
 \Phi=\mathcal T_{s,\beta}(\lambda_+)+\lambda_-+\mathcal R.
\label{eq:legacy-6-70}
\end{equation}

\begin{proposition}[Uniform source curvature package]
\label{prop:source-curvature-package}
Fix the finite source itinerary cover and the nested physical product tubes
used above.  There is one order of choices
\begin{equation}
 \text{normalized parameter box}\ \longrightarrow\
 \text{Gaussian cutoff}\ \longrightarrow\ s_0\ \longrightarrow\ L_0
\label{eq:source-choice-order}
\end{equation}
and constants \(c_q>0\) and \(L_0<\infty\) such that the following statements
hold simultaneously on every retained noncompact source cell.  For all five
original parameters in that box, every \(L\ge L_0\), and every open-gate point
\(C_+C_->0\),
\begin{equation}
 \operatorname{sign}(\Phi,\mathfrak D\Phi,\mathfrak D^2\Phi)
   =\operatorname{sign}k,\qquad
 |\mathfrak D^2\Phi|\ge\frac13|\mathfrak D\Phi|
   \ge\frac19|k|,
\label{eq:source-uniform-ledger-phi}
\end{equation}
and, for \(q=LC/A\),
\begin{equation}
 \mathfrak Dq\ge c_q,\qquad
 \mathfrak D^2q=O((\epsilon+o(1))/L).
\label{eq:source-uniform-ledger-q}
\end{equation}
The implicit constant and \(c_q\) depend only on the fixed finite atlas,
nested tubes, and cutoffs.  They are independent of the current parameter,
of \(L\), and of \(C_+^{-1},C_-^{-1}\).  On \(k=0\), one has
\(\Phi=0\) exactly.  At \(C_+=0\) or \(C_-=0\) there is no retained full
return, rather than a limiting equation.  The compact range \(L\le L_0\)
lies in the same original parameter neighborhood and is a jointly analytic
physical tube.
\end{proposition}

\begin{proof}
The lower estimate \eqref{eq:legacy-6-57} must be used at the actual gates.  Set
\begin{equation}
 a_+=k\Gamma_+,\qquad a_-=k\Gamma_-,\qquad
 C_+=1+a_+>0,\qquad C_-=1-a_->0.
\label{eq:legacy-6-71}
\end{equation}
With
\(u_\pm=\mathfrak D\Gamma_\pm/\Gamma_\pm\) and
\(v_\pm=\mathfrak D^2\Gamma_\pm/\Gamma_\pm\), \eqref{eq:legacy-6-57} makes
\(|u_\pm|+|v_\pm|\) arbitrarily small.  Direct differentiation gives
\begin{equation}
 \mathfrak D\log C_+
 =\frac{a_+(1+u_+)}{ C_+},\qquad
 \mathfrak D(-\log C_-)
 =\frac{a_-(1+u_-)}{ C_-},
\label{eq:legacy-6-72}
\end{equation}
and
\begin{equation}
\mathfrak D^2\log C_+
 =\frac{a_+}{ C_+^2}\{1+2u_++v_++a_+(v_+-u_+^2)\},
\label{eq:legacy-6-73}
\end{equation}
\begin{equation}
\mathfrak D^2(-\log C_-)
 =\frac{a_-}{ C_-^2}\{1+2u_-+v_-+a_-(u_-^2-v_-)\}.
\label{eq:legacy-6-74}
\end{equation}
In exact gate variables,
\begin{equation}
 \lambda_+=\log(1+a_+)-sa_+,\qquad
 \lambda_-=-\log(1-a_-)-sa_-.
\label{eq:legacy-6-74a}
\end{equation}
If \(b_\pm=\mathfrak Da_\pm\) and
\(c_\pm=\mathfrak D^2a_\pm\), direct differentiation gives
\begin{equation}
\begin{aligned}
 \mathfrak D\lambda_+&=\frac{b_+}{1+a_+}+s(a_+-b_+),&
 \mathfrak D^2\lambda_+&=\frac{c_+}{1+a_+}
  -\frac{b_+^2}{(1+a_+)^2}+s(-a_++2b_+-c_+),\\
 \mathfrak D\lambda_-&=\frac{b_-}{1-a_-}+s(a_--b_-),&
 \mathfrak D^2\lambda_-&=\frac{c_-}{1-a_-}
  +\frac{b_-^2}{(1-a_-)^2}+s(-a_-+2b_--c_-).
\end{aligned}
\label{eq:legacy-6-74b}
\end{equation}
Because \(sk=\ell\) is fixed by \(\mathfrak D\), the corrections are
exactly
\begin{equation}
 \mathfrak D(-sa_\pm)=-sa_\pm u_\pm,\qquad
 \mathfrak D^2(-sa_\pm)=-sa_\pm v_\pm.
\label{eq:legacy-6-74c}
\end{equation}
Thus they spend the same actual denominator reserve as \eqref{eq:legacy-6-72}--\eqref{eq:legacy-6-74},
with the correct sign at both gates.  After the lower-kernel modulus in
\eqref{eq:legacy-6-57} is fixed sufficiently small, all three rows have the sign of \(k\)
and satisfy
\begin{equation}
 |\lambda_\pm|\le\frac{5}{2}|\mathfrak D\lambda_\pm|,\qquad
 |\mathfrak D\lambda_\pm|
 \le\frac{5}{2}|\mathfrak D^2\lambda_\pm|,\qquad
 \frac{ |\mathfrak D\lambda_\pm|^2
  }{|\mathfrak D^2\lambda_\pm|}\le\frac{5}{4}.
\label{eq:legacy-6-74d}
\end{equation}
These comparisons remain uniform as the allowed exact denominator tends to
zero.

Write
\begin{equation}
 \mathcal T_{s,\beta}(\lambda_+)=\lambda_+U,\qquad
 U=\int_0^1
 (T_{s,\beta,0})_e(1-t\lambda_+/h)\,dt.
\label{eq:legacy-6-75}
\end{equation}
Put \(q_+=\lambda_+/h\), and let a dot denote the total fixed-original
derivative of \(T=T_{s,\beta,0}\) at fixed action.  The moving evaluation
has the exact derivatives
\begin{equation}
 \mathfrak Dq_+=\frac{\mathfrak D\lambda_+}{ h}
 -\frac{(1-s)\lambda_+}{ h^2},
\label{eq:legacy-6-75a}
\end{equation}
\begin{equation}
 \mathfrak D^2q_+=\frac{\mathfrak D^2\lambda_+}{ h}
 -\frac{2(1-s)\mathfrak D\lambda_+}{ h^2}
 -\frac{s\lambda_+}{ h^2}
 +\frac{2(1-s)^2\lambda_+}{ h^3},
\label{eq:legacy-6-75b}
\end{equation}
and
\begin{equation}
 \mathfrak DU=\int_0^1
 \{\dot T_e-tT_{ee}\mathfrak Dq_+\}\,dt,
\label{eq:legacy-6-75c}
\end{equation}
\begin{equation}
 \mathfrak D^2U=\int_0^1
 \{\ddot T_e-2t\dot T_{ee}\mathfrak Dq_+
 -tT_{ee}\mathfrak D^2q_+
 +t^2T_{eee}(\mathfrak Dq_+)^2\}\,dt.
\label{eq:legacy-6-75d}
\end{equation}
In particular, the \(T_{eee}(\mathfrak Dq_+)^2\) term, both derivatives
of \(h^{-1}\), and both moving-evaluation terms are present.  Equations
\eqref{eq:legacy-6-59}, \eqref{eq:legacy-6-68}, and the prescribed limit order make
\begin{equation}
 \max\{|T_e-1|,|T_{ee}|,|T_{eee}|,
 L|\dot T_e|,L|\dot T_{ee}|,L^2|\ddot T_e|\}\le10^{-3}.
\label{eq:legacy-6-75e}
\end{equation}
Combining \eqref{eq:legacy-6-74d}--\eqref{eq:legacy-6-75e} in the product rules for
\(E_{\mathcal T}=\lambda_+(U-1)\) gives
\begin{equation}
 |E_{\mathcal T}|\le10^{-3}|\lambda_+|,\quad
 |\mathfrak DE_{\mathcal T}|\le2\cdot10^{-3}|\mathfrak D\lambda_+|,\quad
 |\mathfrak D^2E_{\mathcal T}|\le3\cdot10^{-3}
       |\mathfrak D^2\lambda_+|.
\label{eq:legacy-6-75f}
\end{equation}
Equation \eqref{eq:legacy-6-61} gives the same relative smallness for \(\mathcal R\).
Consequently, after the normalized parameter box is shrunk, the fixed
Gaussian cutoff is chosen, and only then \(s_0\) is chosen,
\begin{equation}
 \operatorname{sign}(\Phi,\mathfrak D\Phi,
 \mathfrak D^2\Phi)=\operatorname{sign}k,
\qquad
 |\mathfrak D^2\Phi|\ge\frac{1}{3}|\mathfrak D\Phi|
 \ge\frac{1}{9}|k|.
\label{eq:legacy-6-76}
\end{equation}
All constants are independent of \(C_+^{-1}\) and \(C_-^{-1}\).

The same parameter box is used in \eqref{eq:legacy-6-12}.  Since \(A\) and
\(C\) are positive and converge, with two rate derivatives, to the displayed
positive constants there, differentiation of \(q=LC/A\) gives
\begin{equation}
 \mathfrak Dq>c_q>0,\qquad
 \mathfrak D^2q=O((\epsilon+o(1))/L).
\label{eq:legacy-6-77}
\end{equation}
After the finite source-cell refinement, take the minimum of the finitely many
positive \(c_q\)'s and the maximum of the finitely many thresholds \(L_0\).
This does not alter the order \eqref{eq:source-choice-order}.  The exact
center identity gives \(\Phi=0\) on \(k=0\); the stopped construction removes
the gate faces.  The remaining bounded action interval is contained in the
common analytic tubes constructed before coefficient division.  These facts
prove all assertions of the proposition.
\end{proof}

\begin{lemma}[Matched-curvature zero criterion]
\label{lem:matched-curvature-criterion}
Let \(I\) be an open interval, let \(q\) and \(g\) be real analytic on
\(I\) and \(q(I)\), respectively, with \(q'>0\), and write a scalar closing
equation in the form
\begin{equation}
 a+bq-g(q)=0.
\label{eq:abstract-matched-curvature}
\end{equation}
Here \(a\) and \(b\) are constants on the phase interval.
If \(g''\) has one strict sign on \(q(I)\), then
\eqref{eq:abstract-matched-curvature} has at most two zeros in \(I\), counted
with multiplicity.  The conclusion is uniform for a family whenever the
monotonicity and curvature margins are uniform.
\end{lemma}

\begin{proof}
Since \(q'>0\), \(q\) is an analytic coordinate on \(I\).  In that coordinate,
three zeros counted with multiplicity would give, by the multiplicity form of
Rolle's theorem applied twice, a zero of the second derivative
\(-g''(q)\).  Every finite collection of zeros in the open interval lies in a
compact subinterval, so no endpoint regularity is required.
\end{proof}

\begin{theorem}[Uniform source-regime zero bound]
\label{thm:part-ii-source-zero}
Every unforced-center-complete direct source chart--itinerary has at most two
isolated fixed points, counted with multiplicity, on each noncompact
normalized-action cell.  This bound is locally uniform in all five original
parameters and as either lower factor tends to its open gate.  The gate face
has no full-return itinerary.  A bounded section piece has an independent analytic
Weierstrass bound, and an exact center face is an identity with no isolated
member.
\end{theorem}

\begin{proof}
Divide \eqref{eq:part-ii-matched} by the positive function
\(f_1=L^{3/2}A\), and put
\[
 q=\frac{L^{5/2}C}{ L^{3/2}A}=L\frac{C}{ A},\qquad
 g=\frac{\Phi}{ L^{3/2}A}.
\]
Proposition~\ref{prop:source-curvature-package} gives the monotonicity and
curvature ledger \eqref{eq:legacy-6-76}--\eqref{eq:legacy-6-77}.  Two
differentiations give
\begin{equation}
 \mathfrak D^2g=\frac{1}{ L^{7/2}A}
 \{L^2\mathfrak D^2\Phi-3L\mathfrak D\Phi
   +\tfrac{15}{4}\Phi+E_{\rm div}\},
\label{eq:legacy-6-78}
\end{equation}
where
\(|E_{\rm div}|\le C(\epsilon+o(1))
\{L|\mathfrak D\Phi|+|\Phi|\}\).  Hence
\begin{equation}
 \frac{d^2g}{ dq^2}
 =\frac{(\mathfrak D^2g)(\mathfrak Dq)
   -(\mathfrak Dg)(\mathfrak D^2q)
  }{(\mathfrak Dq)^3}
\label{eq:legacy-6-79}
\end{equation}
is nonzero and has the sign of \(k\) for large \(L\).
The common-domain results preceding the matched equation make this division
legitimate.  Figure~\ref{fig:matched-source-mechanism} shows the scalar
geometry of the remaining curvature argument, including the multiplicity-two
tangency allowed by the estimate.
\begin{figure}[htbp]
\centering
\begin{tikzpicture}[
  x=1.25cm,y=1cm,
  every node/.style={font=\small},
  axis/.style={-{Latex[length=1.8mm]},thick},
  curve/.style={very thick},
  affine/.style={thick,gray!70},
  root/.style={circle,fill=black,inner sep=1.8pt}
]
\begin{scope}[xshift=-3.40cm]
\node[font=\small\bfseries] at (0,1.62)
  {(a) Two simple zeros};
\draw[axis] (-2.40,0) -- (2.48,0) node[right] {$q$};
\draw[axis] (0,-1.17) -- (0,1.30);
\draw[curve,domain=-2.35:2.35,samples=100,smooth,variable=\x]
  plot ({\x},{0.32*\x*\x-0.80});
\draw[affine] (-2.35,-0.135) -- (2.35,0.335);
\node[anchor=north west,gray!80] at (0.72,0.18)
  {$\ell_1$};
\node[anchor=south east] at (-1.90,0.50)
  {$g$};
\node[root] at (-1.528,-0.0528) {};
\node[root] at (1.841,0.2841) {};
\node[below=3pt] at (-1.528,-0.0528) {$q_1$};
\node[above=3pt] at (1.841,0.2841) {$q_2$};
\end{scope}

\begin{scope}[xshift=3.40cm]
\node[font=\small\bfseries] at (0,1.62)
  {(b) One zero of multiplicity two};
\draw[axis] (-2.40,0) -- (2.48,0) node[right] {$q$};
\draw[axis] (0,-1.17) -- (0,1.30);
\draw[curve,domain=-2.35:2.35,samples=100,smooth,variable=\x]
  plot ({\x},{0.32*\x*\x-0.80});
\draw[affine] (-0.40,-1.136) -- (2.35,0.096);
\node[anchor=north west,gray!80] at (1.45,-0.30)
  {$\ell_2$};
\node[anchor=south east] at (-1.90,0.50)
  {$g$};
\node[root] at (0.70,-0.6432) {};
\node[anchor=north east] at (-0.10,-0.76)
  {double zero};
\end{scope}

\node[align=center] at (0,-1.62)
  {$F(q)=\ell(q)-g(q),\qquad
   F''(q)=-g''(q)\ne0
   \quad\Longrightarrow\quad
   \# Z(F)\le2\quad\text{(counted with multiplicity)}$.};
\end{tikzpicture}
\caption{The scalar geometry behind the matched-source estimate.  On each
noncompact action cell, $q$ is a monotone phase coordinate and
$\ell(q)=\tau+\sigma_0q$ is affine.  Both panels are drawn for $k>0$, when
$g$ is strictly convex, but they represent different affine coefficient
fibers.  A secant gives two simple zeros, while a tangent gives one zero of
multiplicity two.  A third zero counted with multiplicity would force a zero
of $F''=-g''$ by the multiplicity form of Rolle's theorem.  The case $k<0$
is the vertically reflected concave picture, and the affine case $k=0$ is
treated separately in the proof.  Equations and multiplicities are exact;
positions and slopes are schematic.}
\label{fig:matched-source-mechanism}
\end{figure}
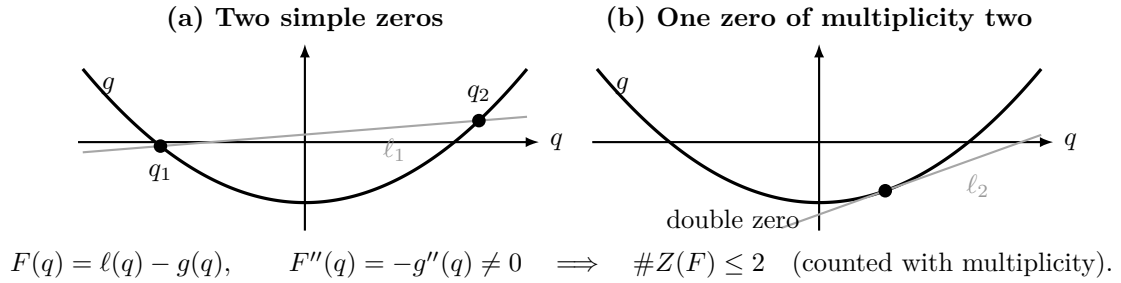

After division, \eqref{eq:part-ii-matched} is
\[
 \tau+\sigma_0q-g=0.
\]
Lemma~\ref{lem:matched-curvature-criterion}, with
\eqref{eq:legacy-6-79}, gives at most two zeros, counted with multiplicity,
on each open noncompact cell.  This remains uniform as a gate is approached;
the gate itself is not an endpoint zero because the stopped return is absent
there.  If \(k=0\), then \(\Phi=0\) and the equation is affine in \(q\),
unless \(\tau=\sigma_0=0\); the latter is one of the two center identities.
The compact range of \(L\) is a jointly analytic physical tube in the same
parameter neighborhood and is covered by Weierstrass preparation.  This
proves the theorem.
\end{proof}

\section{The persistent-endpoint face: a Li\'enard--Dulac estimate}
\label{sec:part-ii-mixed}

We consider a retained itinerary with a persistent nonhyperbolic horizontal
endpoint and at most the upper vertical \(D\)-collision as an internal
central label.  The result is uniform in the resolved parameters and named
specializations within the face \(B=a=0\), but does not assert a direct Dulac
estimate on an open five-parameter neighborhood.  Its proof is intrinsic to
that face; it is not obtained by taking
\(\kappa\downarrow0\) in the positive root theorem.

First consider \(B=0\), \(a\ne0\).  In the two horizontal endpoint
charts, with \(K=1-z+m(z-1)^2>0\), the exact radial equations are
\begin{equation}
\begin{array}{lll}
p_+:&r'=-ar^2-r^3K,&z'=z+O(r),\\
p_-:&r'=-ar^2+r^3K,&z'=-z+O(r).
\end{array}
\label{eq:legacy-6-80}
\end{equation}
For one sign of \(a\), the positive endpoint has the required weak entry
and strong exit while the negative endpoint points toward its corner; for
the other sign the roles reverse.  A source-order return uses both boxes, so no
retained full return exists.  We henceforth take \(B=a=0\).

Put \(m=-p\).  The exact family in \(z=1+y>0\) is
\begin{equation}
 \dot x=-y-py^2,\qquad
 \dot y=(1+y)(x+dy)-\ell y.
\label{eq:legacy-6-81}
\end{equation}
The following physical lemma supplies the zero count on the complete
collision/no-root cone.

\begin{lemma}[Dulac count on a line-punctured domain]
\label{lem:line-punctured-dulac}
Let \(\mathcal D\) be a planar domain homeomorphic to \(\mathbb R^2\).
Suppose
\[
 Z=\{o\}\sqcup\mathcal A
\]
is the disjoint union of an equilibrium \(o\) and a proper embedded \(C^1\)
line \(\mathcal A\), and that
\(\mathcal D\setminus\mathcal A\) has two simply connected components.
Assume periodic orbits cannot cross \(Z\).  If a \(C^1\) Dulac multiplier is
defined on \(\mathcal D\setminus Z\) and its weighted divergence has one
strict sign on each component, then \(\mathcal D\) contains at most one
periodic orbit.
\end{lemma}

\begin{proof}
Removing \(o\) punctures exactly one of the two components of
\(\mathcal D\setminus\mathcal A\).  In the unpunctured component, the Jordan
domain bounded by any periodic orbit stays in the component, so Green's
formula and the strict divergence sign exclude that orbit.  In the punctured
component, the same argument excludes every orbit contractible there.  Any
remaining orbit surrounds \(o\); two such orbits are nested, and the closed
annulus between them avoids \(Z\).  Green's formula on that annulus excludes
the pair.  Hence at most one periodic orbit remains.
\end{proof}

\begin{lemma}[Full-ratio Li\'enard--Dulac bound]
\label{lem:part-ii-mixed-dulac}
There are \(d_0>0\), \(0<p_0<1\), and
\begin{equation}
 \frac{2}{9}<q_0<\frac{1}{4},\qquad
 p_0+\frac{2d_0^2}{9}<1,
\label{eq:legacy-6-82}
\end{equation}
such that, whenever
\begin{equation}
 |d|<d_0,\qquad0<p<p_0,\qquad p\ge q_0d^2,
\label{eq:legacy-6-83}
\end{equation}
the field \eqref{eq:legacy-6-81} has at most one isolated periodic orbit in \(z>0\), for every
real \(\ell\).  The statement is uniform through
\(p=d^2/4\), both adjacent sides, and every value of \(\ell/d\).
\end{lemma}

\begin{proof}
Set
\[
 v=x+dy-\frac{\ell y}{1+y},\qquad
 d\tau=(1+y)dt,\qquad
 g(y)=\frac{y(1+py)}{1+y}.
\]
The time change is positive, and direct differentiation gives
\begin{equation}
 y_\tau=v,\qquad
 v_\tau=-g(y)+\left(d-\frac{\ell}{(1+y)^2}\right)v.
\label{eq:legacy-6-84}
\end{equation}
For \(E=v^2/2+\int_0^yg(r)\,dr\),
\begin{equation}
 E_\tau=\left(d-\frac{\ell}{(1+y)^2}\right)v^2.
\label{eq:legacy-6-85}
\end{equation}
The involution
\((x,t,d,\ell,p)\mapsto(-x,-t,-d,-\ell,p)\)
allows \(d\ge0\).  Equation \eqref{eq:legacy-6-85} excludes periodic orbits when
\(d\ell\le0\), except that \(d=\ell=0\) gives period annuli and no
isolated member.  It remains to take \(d,\ell>0\) and set
\begin{equation}
 \epsilon=\frac{\ell}{ d},\qquad\delta=d^2.
\label{eq:legacy-6-86}
\end{equation}

If \(\epsilon\ge1\), the positive multiplier
\(\mathcal B_0=e^{2dx}/(1+y)\) satisfies
\begin{equation}
 \frac{\operatorname{div}(\mathcal B_0X)}{\mathcal B_0}
 =d\left(1-y-\frac{\epsilon}{1+y}-2py^2\right)\le0,
\label{eq:legacy-6-87}
\end{equation}
and is not identically zero on any open set.  When \(\epsilon=1\), equality
holds on the line \(y=0\), but the divergence is strictly negative off that
line; Green's formula still excludes a periodic orbit.  Suppose
\(0<\epsilon<1\).  Put
\[
 H(y)=dy-\frac{\ell y}{1+y},\qquad u=v-H(y),\qquad F=-H.
\]
Then
\begin{equation}
 y_\tau=u-F(y),\qquad u_\tau=-g(y).
\label{eq:legacy-6-88}
\end{equation}
Writing \(\zeta=1+y\), define
\[
 \phi(\zeta)=(\zeta-1)\left(\frac{\epsilon}{\zeta}-1\right),
 \qquad F=d\phi,
\]
\begin{equation}
 G(y)=\int_0^yg(r)\,dr
 =(1-p)\{\zeta-1-\log\zeta\}
   +\frac{p}{2}(\zeta-1)^2.
\label{eq:legacy-6-89}
\end{equation}
Now set
\begin{equation}
 \mathcal V=u^2-\frac{2}{3}uF-\frac{1}{9}F^2+2G.
\label{eq:legacy-6-90}
\end{equation}
A direct expansion, retaining the divergence term, gives
\begin{equation}
 X\mathcal V-\frac{2}{3}(\operatorname{div}X)\mathcal V
 =\frac{4d}{3}\mathcal K,
\label{eq:legacy-6-91}
\end{equation}
where
\begin{equation}
 \mathcal K=W+\frac{\delta}{9}\phi^2\phi',\qquad
 W=G\phi'-\phi g,\qquad
 \phi'=\frac{\epsilon}{\zeta^2}-1.
\label{eq:legacy-6-92}
\end{equation}

The function \(W\) is affine separately in \(p\) and \(\epsilon\).
Its four corner values on the unit square are
\begin{equation}
\begin{array}{c|c|c}
 &\epsilon=0&\epsilon=1\\ \hline
p=0&\log\zeta-1+\zeta^{-1}
 &(\zeta-1)\{(\zeta+1)\log\zeta-2(\zeta-1)\}/\zeta^2\\[1mm]
p=1&(\zeta-1)^2/2&(\zeta-1)^4/(2\zeta^2).
\end{array}
\label{eq:legacy-6-93}
\end{equation}
All are positive away from \(\zeta=1\).  For the only nonimmediate entry,
differentiate
\(\log\zeta-2(\zeta-1)/(\zeta+1)\); its derivative is
\((\zeta-1)^2/[\zeta(\zeta+1)^2]\).  Multi-affine interpolation gives
\begin{equation}
 W\ge \frac{p(\zeta-1)^2}{2\zeta^2}
 \{\zeta^2-2\epsilon\zeta+\epsilon\}.
\label{eq:legacy-6-94}
\end{equation}
When \(\zeta^2>\epsilon\), use
\begin{equation}
\begin{aligned}
&\zeta^2(\zeta^2-2\epsilon\zeta+\epsilon)
 -(\zeta-\epsilon)^2(\zeta^2-\epsilon)\\
&\qquad=\epsilon\{(2-\epsilon)\zeta^2
 -2\epsilon\zeta+\epsilon^2\}\ge0.
\end{aligned}
\label{eq:legacy-6-95}
\end{equation}
Equations \eqref{eq:legacy-6-83}, \eqref{eq:legacy-6-94}, and \eqref{eq:legacy-6-95} imply
\(\mathcal K>0\) for \(\zeta\ne1\).  This is the full-ratio estimate;
no boundedness of \(\ell/d\) was used.

Completing the square,
\begin{equation}
 \mathcal V=(u-F/3)^2+Q(\zeta),\qquad
 Q=2G-\frac{2\delta}{9}\phi^2.
\label{eq:legacy-6-96}
\end{equation}
The inequalities in \eqref{eq:legacy-6-82}--\eqref{eq:legacy-6-83} give \(Q>0\) for
\(\zeta\ge\epsilon\), \(\zeta\ne1\).  On
\(0<\zeta<\epsilon\), \(Q\to-\infty\) at zero, and at a zero of
\(Q\) direct differentiation gives \(Q'>0\).  Indeed this reduces to
\[
 2G-(1-\zeta)^2\{1-p(1-\zeta)\}>0,
\]
which follows from
\begin{equation}
 (1-p)\{2(\zeta-1-\log\zeta)-(1-\zeta)^2\}
 +p(1-\zeta)^3>0.
\label{eq:legacy-6-97}
\end{equation}
Consequently \(Q\) has exactly one zero
\(\zeta_*\in(0,\epsilon)\).  It cannot have two: every zero is crossed with
positive derivative, whereas a second crossing would first require a
positive-to-negative crossing.  The nonisolated part of
\(\{\mathcal V=0\}\) is therefore the explicit arc
\begin{equation}
 \mathcal A=\left\{
 0<\zeta\le\zeta_*,\quad
 u=\frac{F(\zeta)}3\pm\sqrt{-Q(\zeta)}
 \right\}.
\label{eq:mixed-zero-arc}
\end{equation}
The two graphs meet smoothly at \(\zeta_*\), because \(Q'(\zeta_*)>0\).
As \(\zeta\downarrow0\), their \(u\)-coordinates tend respectively to
\(+\infty\) and \(-\infty\).  Hence \(\mathcal A\) is a proper embedded
line in
\(\mathcal D=\mathbb R\times(0,\infty)\).  Under
\((u,\zeta)\mapsto(u,\log\zeta)\), \(\mathcal D\) is the plane, and the
proper line theorem shows that
\(\mathcal D\setminus\mathcal A\) has two simply connected components.
The remaining zero of \(\mathcal V\) is the isolated equilibrium
\(o=(0,1)\); deleting it punctures exactly one of those components.

On \(\mathcal A\), equation \eqref{eq:legacy-6-91} gives
\(X\mathcal V>0\), so a periodic orbit cannot cross the arc; it cannot pass
through \(o\) either.  On \(\mathcal D\setminus(\mathcal A\cup\{o\})\) use
\(\widehat{\mathcal B}=|\mathcal V|^{-3/2}\):
\begin{equation}
 \operatorname{div}(\widehat{\mathcal B}X)
 =-\frac{3}{2}\operatorname{sign}(\mathcal V)
 |\mathcal V|^{-5/2}\,\frac{4d}{3}\mathcal K.
\label{eq:legacy-6-98}
\end{equation}
This multiplier is \(C^1\) on each component and its divergence has a strict
sign there because \(\mathcal K>0\).  Lemma~\ref{lem:line-punctured-dulac}
therefore gives at most one periodic orbit.  No limiting small-circle
integral around \(o\) is needed: the only Green domain in the punctured
component is the annulus between two hypothetical cycles.  At the integrable
corner every orbit belongs to a period annulus, so the conclusion for
isolated periodic orbits remains valid.
\end{proof}

\begin{theorem}[Uniform zero bound on the persistent-endpoint face]
\label{thm:part-ii-mixed-zero}
Every retained itinerary in the mixed class has a locally uniform finite number
of isolated fixed points.  More precisely, \(B=0,a\ne0\) carries no
source-order full return; on \(B=a=0\), the collision/no-root cone has at
most one isolated periodic orbit by
Lemma~\ref{lem:part-ii-mixed-dulac}, and the
complementary split cone has a fixed sink/no-passage outcome.  The
assignment is uniform through the \(D\)-collision, both adjacent sides,
the source overlap, and all lower first-hit faces.  If \(t_{\rm mer}\)
denotes the merger scale of the root chart, its subface \(t_{\rm mer}=0\),
including \(t_{\rm mer}=\kappa=0\), belongs to the source regime.
\end{theorem}

\begin{remark}
The theorem concerns a retained full-return itinerary in the fixed collar $U$
whose orbit arcs between certified first-hit sections remain in
$z=1+y>0$ until a named lower stopping face.  Its bound is uniform on $B=a=0$
for all values of $\ell/d$, throughout the no-root cone, adjacent split
collar, complementary split cases, source overlap, and lower first-hit
faces.  On the direct Dulac cone there is at most one isolated orbit.  No
open-neighborhood Dulac estimate in $B$ or $a$ is asserted.  The face
$t_{\rm mer}=0$ is treated by
Theorem~\ref{thm:part-ii-source-zero}; a split root reaches its first sink,
hyperbolic, passive, or regular stopping face, and loss of a lower hit is no-passage.
\end{remark}

\begin{proof}
The orientation calculation \eqref{eq:legacy-6-80} handles \(B=0,a\ne0\).  On
\(B=a=0\), choose for example \(q_0=31/128\) and split first by
\begin{equation}
 p=-m\ge q_0d^2.
\label{eq:legacy-6-99}
\end{equation}
This closed cone contains the collision \(p=d^2/4\), the adjacent split
collar \(q_0d^2\le p<d^2/4\), and the whole no-root side.  A periodic
orbit represented by the itinerary stays in \(z>0\): on \(z=0\),
\(\dot z=\ell\), so crossings have one orientation when \(\ell\ne0\),
while the line is invariant when \(\ell=0\).  Lemma
\ref{lem:part-ii-mixed-dulac} therefore applies to every possible orbit in
this cone.

The source overlap is not recognized from \(d/\theta\) alone.  Its full
effective vector contains
\begin{equation}
 \Xi_{\rm src}=
 \max\left\{\frac{p}{\theta^2},\frac{|d|}{\theta},
 \frac{|c|}{\theta}\right\}.
\label{eq:legacy-6-100}
\end{equation}
The direct Dulac estimate is independent of \(\Xi_{\rm src}\), so it
controls both sides of every finite source/parameter threshold.  At the
exact source subface, the assignment nevertheless follows the half-open priority
of \eqref{eq:legacy-5-2}; Theorem~\ref{thm:part-ii-source-zero} supplies
its ambient neighborhood.

On the complementary split cone use the same involution as in the lemma to
take \(d\ge0\), and set
\begin{equation}
 \Delta=d^2+4m>0,\qquad
 w_+=\frac{-d+\sqrt\Delta}{2}.
\label{eq:legacy-6-101}
\end{equation}
For \(m=-p<0\) with \(p<q_0d^2\),
\(\Delta>d^2/32\); for \(m\ge0\),
\(\sqrt\Delta/(d+\sqrt m)\) also has a fixed positive lower bound.  Define
\begin{equation}
 \rho=
 \begin{cases}
 d,&m<0,\\
 d+\sqrt m,&m\ge0.
 \end{cases}
\label{eq:legacy-6-101a}
\end{equation}
Use \(w=\rho W\), \(r=\rho^2R\), and divide time by \(\rho\).  The two
normalized eigenvalues at the first source-oriented root are
\begin{equation}
 \frac{\lambda_w}{\rho}=-\frac{\sqrt\Delta}{\rho}<0,\qquad
 \frac{\lambda_r}{\rho}=-\frac{d+\sqrt\Delta}{2\rho}<0.
\label{eq:legacy-6-102}
\end{equation}
If \(m=-p<0\), then \(p<q_0d^2\) gives
\(\sqrt\Delta/d\ge1/\sqrt{32}\), while the second quotient has modulus at
least \(1/2\).  If \(m\ge0\), then
\(\sqrt{d^2+4m}/(d+\sqrt m)\ge2/3\), and the second quotient again has
modulus at least \(1/2\).  Thus the normalized family has a fixed
two-negative-eigenvalue margin.
Moreover the equatorial polynomial is negative for \(w>w_+\), so this is
the first root met by the source-oriented flow.  To record the isolating
block quantitatively, translate \(W\) to the normalized root and write
\(\zeta=(R,U)\).  The compact normalized family of linear parts is uniformly
Hurwitz by \eqref{eq:legacy-6-102}, including the repeated-eigenvalue face \(m=0\).
The Lyapunov equation therefore supplies positive forms
\(\mathcal L_\alpha(\zeta)=\zeta^TP_\alpha\zeta\) and uniform constants
\(c_0,C_0,C_1>0\) such that
\begin{equation}
 c_0|\zeta|^2\le\mathcal L_\alpha(\zeta)\le C_0|\zeta|^2,\qquad
 \dot{\mathcal L}_\alpha\le-|\zeta|^2+C_1|\zeta|^3.
\label{eq:legacy-6-102a}
\end{equation}
One fixed sufficiently small positive level of \(\mathcal L_\alpha\),
intersected with \(R\ge0\), is thus inward pointing; the face \(R=0\) is
invariant.  This gives a uniform sink isolating block, in the form of a
half-block, without choosing eigenvectors across the repeated face.
An orbit arriving from \(w>w_+\) cannot leave it in forward time and cannot
belong to a periodic lap.  The case \(m=0\) is included after rescaling;
the sole corner \(d=m=0\) belongs to the source regime.

There are only finitely many threshold, root, and lower-first-hit charts.  Every
specialization either remains in the direct one-cycle estimate, enters the
matched source neighborhood, reaches the split sink/hyperbolic/regular
regime, or loses a first hit.  Compactness of this already named finite cover
gives the asserted locally uniform bound.
\end{proof}

\section{The two zero-scale faces}
\label{sec:part-ii-handoff}

The source analysis has established two estimates on open neighborhoods.
Theorem~\ref{thm:part-ii-source-zero}
holds on the common source action tube,
including its open-gate no-passage and center-identity conclusions.
Theorem~\ref{thm:part-ii-mixed-zero} holds on the retained physical full-return domain
in \(U\cap\{z>0\}\), including the collision/no-root cone, the split-cone
first-hit classification, and the lower first-hit faces.

The uniform assembly invokes these theorems only after the stopped-itinerary
and exact cycle--zero reductions: they treat the classes \(\mathfrak P_{\rm src}\) and
\(\mathfrak P_{\rm mix}\), respectively, and their boundary conclusions feed
the boundary specialization theorem.  The following sections treat the complementary
positive-scale regimes.  In its middle/root priority,
\(t_{\rm mer}=0\), including \(t_{\rm mer}=\kappa=0\), is treated by the source theorem,
whereas \(t_{\rm mer}>0,\ \kappa=0\) is treated by the mixed theorem.  These are ambient
boundary theorems proved here, not continuity limits of the positive-root
theorem.

\section{Bounded analytic itineraries}
\label{sec:compact-analytic}

Some retained itineraries remain in bounded transverse analytic tubes and need
none of the singular theorems above.  Finite coefficient division and
Weierstrass preparation give a local zero bound for each such jointly analytic
itinerary, including its identity coefficient fibers.

\begin{theorem}[Uniform zero bound for bounded analytic itineraries]
\label{thm:compact-analytic}
Let $\omega$ be a bounded stopped full-return itinerary whose primitive first hits are
jointly analytic on nested transverse physical tubes.  Then its displacement
\begin{equation}
 D_\omega(s,\lambda)=P_\omega(s,\lambda)-s
\label{eq:legacy-5-3}
\end{equation}
has a locally uniform finite number of isolated zeros.  The same ambient
neighborhood controls coefficient and identity fibers.
\end{theorem}

\begin{remark}
The bound is uniform on one analytic phase--parameter neighborhood and then
on a finite cover of the compact section closure.  Coefficient and identity
strata remain in the same preparation; an identity fiber has no isolated
member.
\end{remark}

\begin{proof}
The finite composition defining $P_\omega$ is analytic on a common bounded
phase--parameter tube.  Fix a phase origin and expand
\begin{equation}
 D_\omega(s,\lambda)=\sum_{n\ge0}a_n(\lambda)s^n.
\label{eq:analytic-coefficient-expansion}
\end{equation}
The coefficient germs generate an ideal $I$ in the local analytic parameter
ring.  By Noetherianity, there are indices
\(0\le n_0<\cdots<n_N\) such that
\(b_j:=a_{n_j}\), \(0\le j\le N\), generate every \(a_n\).  The analytic
Noetherian and preparation--division ingredients used here are recorded in
Herv\'e~\cite[Chapter~II, Theorems~1--2]{Herve1963}.
If $I=0$, the joint displacement germ is identically zero and there is no
isolated member.  We therefore assume $I\ne0$.

We now spell out the only coefficient resolution used in this argument.  Let
\begin{equation}
 \pi:\operatorname{Bl}_I\Lambda\longrightarrow\Lambda
\label{eq:coefficient-blowup}
\end{equation}
be the analytic blow-up of \(I\), realized as the closure of the graph of
\(\lambda\mapsto[b_0(\lambda):\cdots:b_N(\lambda)]\) in
\(\Lambda\times\mathbb P^N\).  Replacing this space by its normalization
would not change the argument.  The map is proper after the parameter
neighborhood is shrunk.  On the standard chart indexed by $i$, the pulled-back
ideal is the invertible ideal $(\varepsilon)$ and
\begin{equation}
 \pi^*a_n=\varepsilon\,\widetilde a_n,\qquad
 \widetilde b_i=1
\label{eq:coefficient-chart}
\end{equation}
for one of the chosen generators.  Here the quotients are analytic
coordinates on the blow-up; no division by a possibly vanishing function on
the original parameter space is being made.  Consequently
\begin{equation}
 \pi^*D_\omega=\varepsilon\,\widetilde D_\omega,
\label{eq:normalized-analytic-displacement}
\end{equation}
and $\widetilde D_\omega$ is not identically zero in $s$ at any projective
direction, because at least one normalized generator is a unit there.

Complexify the normalized germ in the phase and resolved parameter variables.
At every point of the compact resolved phase--parameter closure, holomorphic
Weierstrass preparation
\cite[Chapter~II, Theorem~1]{Herve1963} writes
$\widetilde D_\omega$ as a distinguished polynomial of finite degree times a
unit.  That degree bounds its real isolated zeros on a smaller product
neighborhood.  Uniqueness and conjugation invariance make the factors real on
the real locus.  Krantz--Parks
\cite[Theorem~6.1.3]{KrantzParks2002} gives the corresponding direct
real-analytic preparation statement.

Properness of \eqref{eq:coefficient-blowup}, compactness of the section
closure, and the finite standard projective cover now give finitely many such
neighborhoods and hence one bound.  At a parameter for which every generator
vanishes, all coefficients in \eqref{eq:analytic-coefficient-expansion}
vanish; the original displacement is identically zero and contributes no
isolated cycle.  Its neighboring nonidentity fibers remain controlled by the
resolved neighborhoods above.  When complete center slices are present, $I$
is the reduced center ideal proved in Section~\ref{sec:part-ii-center}; in all
other bounded analytic itineraries it is the local coefficient ideal just
constructed.  No global resolution or iterative principalization theorem is
being imported.
\end{proof}

\section{Separated hyperbolic itineraries and the QRH theorem}
\label{sec:part-iii-hyperbolic}

We begin with an exact-once itinerary already shown in
Sections~\ref{sec:part-i-words}--\ref{sec:part-i-exact-once} to contain
only separated hyperbolic saddles and regular first hits.  We first construct
its physical analytic system, then verify QRH membership and a common
admissible representative, and only then invoke Mourtada's local finiteness
theorem.  We use Mourtada's principal-orbit and integral-projection framework,
the local rational and quasi-resonant construction in Appendix~VA and the
proof of Theorem~0, the QRH inverse result VB4, and the local finiteness
theorem IVC1 along an admissible Hilbert derivation
\cite{Mourtada2009}.  None of these results constructs an H14 section,
connector, or physical itinerary.  For the central and lips cases we use
Theorem~3, Theorems~3.1--3.2, and Corollary~3.6 of
Dumortier--Ilyashenko--Rousseau~\cite{DIR2002}, only after their geometric
hypotheses have been verified in the physical family.

Table~\ref{tab:imported-theorem-interface} is the formal import interface.
Its third column contains the hypotheses proved for the H14 family; no entry
in the first column is used to manufacture those hypotheses.

The table is also a call-site ledger.  Appendix~VA, VA1, and VB4 are used
only in the construction leading to
Proposition~\ref{prop:part-iii-qrh-realization}; IVC1 is invoked only in
Theorem~\ref{thm:part-iii-hyperbolic-zero}.  DIR Theorem~3 supplies only the
finite-smooth saddle-node coordinates in the one-central and strict-lips
arguments.  The proof-level affine hh reduction from Theorem~3.1 is used
only in Lemma~\ref{lem:part-iii-physical-affine-closing}; the boundary
nonaffinity argument from Theorem~3.2 is used only after
Proposition~\ref{prop:complete-physical-lips}, in
Corollary~\ref{cor:two-central-nonaffine-boundary} and at the natural endpoint
of Theorem~\ref{thm:part-iii-strict-zero}.  Corollary~3.6 is invoked only in
Theorems~\ref{thm:part-iii-one-central-zero} and
\ref{thm:part-iii-two-central-no-pp}.  These restrictions prevent an imported
normal form or finiteness theorem from serving as a substitute for an H14
first-hit, connection, or boundary proof.

\begin{table}[p]
\centering
\caption{Assumption--verification--conclusion map for the imported local
theorems.  Page locators refer to the printed pages of the cited sources.}
\label{tab:imported-theorem-interface}
\begingroup
\scriptsize
\setlength{\tabcolsep}{3pt}
\begin{tabularx}{\textwidth}{@{}
 >{\raggedright\arraybackslash}p{.17\textwidth}
 >{\raggedright\arraybackslash}X
 >{\raggedright\arraybackslash}X
 >{\raggedright\arraybackslash}p{.19\textwidth}@{}}
\toprule
Source result&Imported hypotheses&Verification in this paper&Conclusion used\\
\midrule
Mourtada, Appendix VA and Theorem~VA1 (p.~80), with the proof of
Theorem~0 (p.~79) and VB4 (p.~88)&
Analytic hyperbolic saddle; one-sided positive sections; a local rational,
ratio-one, or quasi-resonant chart; VB4 only for a tangent-to-identity QRH
germ&
Proposition~\ref{prop:part-iii-hyperbolic-word} gives analytic saddles,
positive ratios, sections, and physical connectors; the local ratio cover is
proved below&
Each physical Dulac correction, and only the diagonal inverse that needs it,
belongs to \(QRH_1\)\\
\addlinespace
Mourtada, Definitions IB1--IB5 (pp.~11--13), IB7 (pp.~21--22), and
Section IVA.1 (p.~60)&
An admissible positive-corner representative, a principal orbit, analytic
integral projection, and diagonal Hilbert data&
Proposition~\ref{prop:part-iii-qrh-realization} constructs one connected
physical fiber, its exact splitting integrals, the invariant corner, and the
closing germ&
The H14 system is one
\(\chi_\omega\in\Xi H_k[QRH]\) system on the physical fibers\\
\addlinespace
Mourtada, Theorem IVC1 (p.~69), read with Theorem~0 (pp.~2, 79)&
A QRH closing germ along a derivation in \(\Xi H_k[QRH]\), on one common
admissible representative whose integral fibers are the relevant orbits&
Propositions~\ref{prop:part-iii-hyperbolic-word} and
\ref{prop:part-iii-qrh-realization}; exact-once is
Proposition~\ref{prop:exact-once}&
Finite local degree, hence the hyperbolic itinerary zero bound\\
\addlinespace
DIR, Theorem~3 (p.~6)&
A finite-multiplicity analytic saddle-node in a \(C^\infty\) unfolding&
The prepared weak quadratic, nonzero transverse eigenvalue, physical
sections, and retained central passage are verified before
Section~\ref{sec:part-iii-central}&
A prescribed finite-smooth orbital normal form; no analytic upgrade or
physical incidence is inferred\\
\addlinespace
DIR, Theorem~3.1 and its proof (pp.~12--13)&
Two saddle-nodes of opposite attractivity, an actual hh connection, a
continuum of complete pp orbits, and a nonzero transition derivative of some
order \(n\ge2\)&
Proposition~\ref{prop:complete-physical-lips},
Lemma~\ref{lem:two-central-landing-product-tube}, and
Lemma~\ref{lem:part-iii-physical-affine-closing}; nonaffinity is
Corollary~\ref{cor:two-central-nonaffine-boundary}, and
Theorem~\ref{thm:part-iii-strict-zero} selects the finite jet order before
requesting that normal-form order&
Cyclicity at most \(n\) for the selected lips graphic and an affine physical
closing equation\\
\addlinespace
DIR, Theorem~3.2 and its proof (pp.~13--14)&
An analytic lips ensemble with a PP boundary of ratio different from one
(or no analytic first integral), or a BP boundary&
Proposition~\ref{prop:complete-physical-lips} proves the complete boundary,
Lemma~\ref{lem:two-central-landing-product-tube} preserves the stopped word,
\eqref{eq:legacy-7-21} gives the PP ratio margin, and
Corollary~\ref{cor:two-central-nonaffine-boundary} applies the result&
Boundary finite cyclicity and the nonaffine transition jet required in the
previous row\\
\addlinespace
DIR, Corollary~3.6 (pp.~17--18)&
A finite central saddle-node graphic using the selected nodes centrally; for
the clause used here there is no pp connection and, with more than one node,
all selected nodes have the same attractivity&
Theorems~\ref{thm:part-iii-one-central-zero} and
\ref{thm:part-iii-two-central-exhaustion} certify the physical graphic,
absence of pp, and attractivity&
Cyclicity one in the one-central and two-central no-pp cases\\
\bottomrule
\end{tabularx}
\endgroup
\end{table}

The only subtle scope issue in the QRH rows is the characteristic ratio.
Near ratio one, VA1 treats an analytic saddle form
\begin{equation}
 u\,dv+v\{r(\alpha)+a(u,v,\alpha)\}\,du=0,\qquad
 a(0,0,\alpha)=0,\quad r(\alpha)=1+\mu(\alpha),
\label{eq:legacy-7-0a}
\end{equation}
and puts its Dulac correction in \(QRH_1\).  At a rational center
\(r_0=n/m\), positive double ramification \(x=X^m,\ y=Y^n\) changes the
characteristic ratio to
\begin{equation}
 \widetilde r=\frac{m r}{n}=1+\mu,\qquad
 \widetilde r(r_0)=1.
\label{eq:legacy-7-0b}
\end{equation}
The ramification is one-to-one on the positive quadrant and preserves
isolated positive solutions.  The same Appendix~VA scope paragraph treats a
quasi-resonant center directly.  Hence every ratio in the compact positive
set of a certified H14 cell has a source-supplied local chart, and a finite
subcover is taken only after this local cover exists.  No density of rational
ratios and no extension across a zero eigenvalue is used.

Mourtada's VB4 is reserved for the diagonal QRH inverse; ordinary analytic
connector inverses use the analytic inverse theorem.  On the DIR side,
Theorem~3.1's nonzero \(n\)-th derivative is supplied by the nonaffine
infinity-jet argument in the proof of Theorem~3.2 only after the physical
PP/BP boundary and its ratio alternative have been certified.  Thus a normal
form never creates an hh connection, a pp strip, or a boundary graphic in
this proof.

\begin{proposition}[Analytic realization of a hyperbolic itinerary]
\label{prop:part-iii-hyperbolic-word}
Every retained all-hyperbolic H14 itinerary has a finite cover by compact analytic
itinerary cells.  On each cell, every used saddle has analytic separatrix
coordinates and fixed one-sided analytic sections such that, after one fixed
local time orientation,
\begin{equation}
 \dot u_j=u_jA_j(u_j,v_j,\alpha),\qquad
 \dot v_j=-v_jB_j(u_j,v_j,\alpha),
 \qquad A_j,B_j\ge\lambda_*>0.
\label{eq:legacy-7-1}
\end{equation}
The ratios \(r_j=B_j(0,0,\alpha)/A_j(0,0,\alpha)\) remain in one compact
subset of \((0,\infty)\).  Every connector between consecutive sections is
the actual jointly analytic first hit
\begin{equation}
 s_{j+1}=T_j(t_j,\alpha),
 \qquad 0<c_T\le\partial_{t_j}T_j\le C_T,
\label{eq:legacy-7-2}
\end{equation}
on one common interval, with positive section-normal, stopping-face, and
itinerary-existence margins.  The same original five parameters are used at every
vertex and connector.

A zero eigenvalue, node, multiple endpoint root, missing connector, change in
the first-contact outcome, or collar exit is not a boundary of this analytic cell.  It is a
separately named adjacent regime reached at the first lost margin.  No DIR endpoint
coordinate and no finite-smooth saddle-node normalizer occurs here.
\end{proposition}

\begin{proof}
The finite root, sector, and stopping-face cover of
Sections~\ref{sec:part-i-boxes}--\ref{sec:part-i-atlas} first separates every used
saddle from the equatorial-discriminant and equatorial/interior joint-collision
boxes of \eqref{eq:legacy-3-16}, multiple horizontal roots, and same-sign node cells.  On a
closed separated cell both eigenvalues have a
fixed nonzero margin.  The analytic stable-manifold theorem straightens the
two invariant curves; analytic Hadamard division then gives \eqref{eq:legacy-7-1}.
Fixed small levels of \(u_j\) and \(v_j\) are transverse sections with fixed
normal signs.

Delete these saddle boxes from the stopped itinerary.  Every remaining arc is a
finite concatenation of regular flow boxes.  Analytic dependence of the ODE
and the implicit-hit theorem give \(T_j\), while planar order preservation in
the oriented section coordinates gives \(\partial T_j>0\).  Inner and outer
itinerary domains with positive competing-face margins give the two-sided bounds
in \eqref{eq:legacy-7-2}.  The first vanishing normal, interval width, or competing-face
distance is exactly one of the already numbered faces, proving the finite
cover and the stated passage to adjacent regimes.
\end{proof}

Fix one certified cell and retain every intermediate outgoing value
\(x_1,\ldots,x_k\).  Let \(\delta_j\) be the actual local Dulac map in the
sections of Proposition~\ref{prop:part-iii-hyperbolic-word}.  The local orbit
one-form is
\begin{equation}
 u_j\,dv_j+v_j\frac{B_j(u_j,v_j,\alpha)}{
                         A_j(u_j,v_j,\alpha)}\,du_j=0,
 \qquad \frac{B_j}{ A_j}=r_j+a_j,\quad a_j(0,0,\alpha)=0.
\label{eq:legacy-7-3}
\end{equation}
On each local ratio chart described above, Mourtada's Appendix~VA result
places the analytic Dulac correction in \(QRH_1\): by
descent from \eqref{eq:legacy-7-0b} at a rational center, and directly on a quasi-resonant
chart.  In the physical section coordinate this gives
\begin{equation}
 \delta_j(x_j,\alpha)
 =x_j^{r_j(\alpha)}\{1+D_j(x_j,\alpha)\},
 \qquad D_j\in QRH_1,\quad D_j(0,\alpha)=0.
\label{eq:legacy-7-4}
\end{equation}
This is an exact membership statement for the analytic saddle cell, not a
formal linearization and not a claim at a zero eigenvalue.

The regular connectors, including the closing connector, are retained before
normalization.  Orient every target coordinate so that \(T_j'>0\), and put
\begin{equation}
 f_{j+1}(u,\alpha)=T_j^{-1}(u,\alpha)-T_j^{-1}(0,\alpha),
 \qquad \lambda_j(\alpha)=T_j^{-1}(0,\alpha),
 \qquad 1\le j\le k.
\label{eq:legacy-7-5}
\end{equation}
Use the cyclic convention \(x_{k+1}=x_1\).  Here \(T_k\) is the actual
regular first hit from the last outgoing section to the first incoming
section; it is not an inferred translation.
The physical cycle equations are then
\begin{equation}
 \delta_j(x_j,\alpha)-f_{j+1}(x_{j+1},\alpha)
     =\lambda_j(\alpha),\qquad 1\le j\le k.
\label{eq:legacy-7-6}
\end{equation}
Thus no connector, including the closing one, is silently replaced by a
guessed translation.

\begin{proposition}[QRH realization on the physical fibers]
\label{prop:part-iii-qrh-realization}
After the permitted diagonal analytic section changes, the zero-based
functions
\begin{equation}
 g_j^0=\delta_j(x_j,\alpha)-f_{j+1}(x_{j+1},\alpha),
 \qquad 1\le j<k,
\label{eq:legacy-7-7}
\end{equation}
are first integrals of one Hilbert derivation
\(\chi_\omega\in\Xi H_k[QRH]\).  The physical splittings are precisely the
integral-projection values \(g_j^0=\lambda_j(\alpha)\).  After adjoining
ordinary passive affine coefficients, the complete closing equation is the
QRH germ.  Write the two zero-based analytic closing-section coordinates as
\begin{equation}
 \Phi_{\rm out}(u,\alpha):=u,\qquad
 \Phi_{\rm in}(u,\alpha)
 :=T_k^{-1}(u,\alpha)-T_k^{-1}(0,\alpha).
\label{eq:legacy-7-7a}
\end{equation}
Then that germ is
\begin{equation}
 F_{A,B,C}=
 A\Phi_{\rm out}(\delta_k(x_k,\alpha),\alpha)
 -B\Phi_{\rm in}(x_1,\alpha)-C.
\label{eq:legacy-7-8}
\end{equation}
The unnormalized physical coefficient triple is
\((A,B,C)=(1,1,\lambda_k)\).  Its positive projective representative is
\begin{equation}
 (A,B,C)=\frac{ (1,1,\lambda_k)}{\sqrt{2+\lambda_k^2}},
\label{eq:legacy-7-8a}
\end{equation}
with
\(\lambda_k=T_k^{-1}(0,\alpha)\), equation \eqref{eq:legacy-7-8} is
equivalent, after multiplication by its positive denominator, to
\(\delta_k(x_k)-T_k^{-1}(x_1)=0\), which is exactly the closing equation
in \eqref{eq:legacy-7-6}.  Thus the inverse connector acts on the incoming coordinate, not
on the outgoing Dulac value.

On a finite refinement of the compact itinerary cell times
\begin{equation}
 {\cal P}=\{(A,B,C):A,B\ge0,\ A^2+B^2+C^2=1\},
\label{eq:legacy-7-9}
\end{equation}
there is one common admissible positive-corner representative and one
physical invariant fiber set on which zeros of \eqref{eq:legacy-7-8} are exactly solutions
of the actual retained H14 itinerary.  Connection and separatrix values,
projective coefficient faces, coordinate-corner faces, and identity closing
fibers remain in this representative.
\end{proposition}

\begin{proof}
The alternating-minor vector field annihilating the \(k-1\) functions in
\eqref{eq:legacy-7-7}, normalized by its QRH unit so that its corner component is
\(x_1\cdots x_k\), defines \(\chi_\omega\).  The decisive derivative is
\begin{equation}
 x_j\delta_j'=x_j^{r_j}
 \{r_j(1+D_j)+x_j\partial_{x_j}D_j\}\in QRH_1,
\label{eq:legacy-7-10}
\end{equation}
and analytic inverse closure applies to the \(f_j\); VB4 is reserved for the
QRH diagonal inverses in the imported reduction.  Hence the normalized
derivation, its first integrals, and \eqref{eq:legacy-7-8} satisfy the recursive QRH
definitions.  The coefficients \((A,B,C)\) are first passive variables in
ordinary open charts and are restricted to \eqref{eq:legacy-7-9} only after membership has
been proved.

For fixed projection values, \eqref{eq:legacy-7-6} recursively gives
\begin{equation}
 x_{j+1}=f_{j+1}^{-1}
 \bigl(\delta_j(x_j,\alpha)-\lambda_j(\alpha)\bigr).
\label{eq:legacy-7-11}
\end{equation}
Monotonicity and the common itinerary inequalities make every nonempty physical
fiber one connected graph and hence one orbit of \(\chi_\omega\).  This
proves the zero/first-hit correspondence required by IVC1.  The corner is
invariant because \(\chi_\omega x_j/x_j\in QRH\); a point on a coordinate
axis itself represents a separatrix value, not a limit cycle.  Compactness of
the certified cell and \eqref{eq:legacy-7-9} now gives finitely many admissible
representatives.  An identity closing fiber contributes no isolated zero,
while the same representative controls its neighboring fibers.
\end{proof}

\begin{theorem}[Uniform zero bound for hyperbolic itineraries]
\label{thm:part-iii-hyperbolic-zero}
On every compact certified analytic H14 itinerary cell satisfying
Propositions~\ref{prop:part-iii-hyperbolic-word} and
\ref{prop:part-iii-qrh-realization}, the complete retained-variable system
has a locally uniform finite number of isolated positive solutions.  This
includes nonzero ratio-one saddles, with or without an analytic first
integral, and all projective coefficient and connection faces in \eqref{eq:legacy-7-9}.
\end{theorem}

\begin{remark}
Uniformity holds after a finite refinement of the compact physical itinerary
cell times the projective coefficient set \eqref{eq:legacy-7-9}, with one admissible
representative on each member.  Connection, separatrix, coefficient, and
identity values remain within that representative.  A zero eigenvalue,
multiple root, node, missing connector, change in the first-contact outcome, lost section
normal, or exit is stopped at its named adjacent regime before QRH is
invoked.
\end{remark}

\begin{proof}
Mourtada's IVC1 theorem applies to every common admissible representative in
Proposition~\ref{prop:part-iii-qrh-realization}.  It gives a finite local
degree for the zero set of \eqref{eq:legacy-7-8} along the physical invariant fibers.  The
exact fiber/first-hit correspondence turns that degree into the required H14
zero count.  Taking the maximum over the finite refinement gives one bound
for the cell.  On \(A=0\) or \(B=0\) the equation is monotone or constant and
has at most one isolated solution; the zero coefficient triple is excluded
by \eqref{eq:legacy-7-9}.  An identity fiber has no isolated member.  No compactness step
crosses a zero-eigenvalue or changed-itinerary face, because those faces were
removed in Proposition~\ref{prop:part-iii-hyperbolic-word}.
\end{proof}

\section{Central passages and the two-saddle-node classification}
\label{sec:part-iii-central}

This section starts with a stopped, exact-once physical itinerary from
Proposition~\ref{prop:exact-once}.  It
does not construct the atlas and it does not infer an orbit from a formal
normal form.  We first treat the unique-central case, then prove an exhaustive
geometric classification of every retained two-central itinerary.  Only after
that classification do we invoke the no-pp clause of DIR; the pp alternatives
are treated by the strict, middle, and root theorems below.

For this section, \(S_h\) denotes the selected nonpersistent horizontal
saddle-node and \(S_v\) the selected upper \(D\)-saddle-node.  The finite
signed endpoint cover writes the upper double root as
\(w=-q\), \(q=\sigma t\), with \(t=|q|\ge0\).  The label \(E_\sigma\)
is the selected regular equatorial arc, \(p_\sigma\) its persistent principal
endpoint, and \(\xi\) is the oriented endpoint-section coordinate for which
\(0<\xi<\xi_h\) lies between the equatorial boundary and the selected hh
separatrix.  These are physical labels from the stopped atlas, not normal-form
objects.

\subsection{A single central block}

Let \(G\) be the sole central vertex selected by a retained itinerary \(\omega\).
The required uniform nondegeneracy condition is
\begin{equation}
 |\lambda_{\rm tr}(G)|,\ |q_G''(0)|\ge m_c,
 \qquad
 \min_{v\in\mathcal V(\omega)\setminus\{G\}}
 \min\{|\lambda_v^s|,|\lambda_v^u|\}\ge m_h>0.
\label{eq:legacy-7-12}
\end{equation}
Thus the local weak equation has a genuine quadratic saddle-node term and
every other singular vertex, including both principal endpoints, is
hyperbolic.  Merely having one internal label would not imply \eqref{eq:legacy-7-12}.

\begin{theorem}[One-central estimate]
\label{thm:part-iii-one-central-zero}
On a compact family of retained exact-once itineraries satisfying \eqref{eq:legacy-7-12}, with an
actual retained passage through the central sector, fixed sections, and the
uniform stopping-set margins, the number of isolated cycles is locally uniformly
finite; in fact the corresponding elementary graphic has cyclicity at most
one.
\end{theorem}

\begin{remark}
The bound is uniform in a full five-parameter neighborhood of a compact
certified family, provided the constants in \eqref{eq:legacy-7-12}, section normals, and
first-hit margins stay uniformly positive.  If the central root splits, the
itinerary becomes all-hyperbolic; a second central block enters the exhaustion
theorem below; and a persistent $B=0$ endpoint enters the mixed regime.
Wrong orientation, stable-center passage, a change in the first-contact outcome, collapse, and
exit are stopped before a return equation is formed.
\end{remark}

\begin{proof}
DIR2002, Theorem~3, supplies a finite-smooth central normal form only for the
actual saddle-node certified in \eqref{eq:legacy-7-12}; it supplies no
connector.  The remaining sections and connectors are the physical
hyperbolic first hits certified in Section~\ref{sec:part-i-atlas}.  The
resulting complete graphic is therefore exactly the one-central,
no-parabolic-connection case of DIR2002, Corollary~3.6, whose unconditional
no-pp clause gives cyclicity one.
The lower bounds in \eqref{eq:legacy-7-12} and the fixed first-contact margins persist on one open
neighborhood in the original parameters.  A finite cover of the compact base
family yields the uniform statement.  If one of these inequalities or
passage conditions is lost, the itinerary enters the adjacent case described
in the preceding remark.
\end{proof}

\subsection{The finite list of two-central configurations}

There are at most two internal central gates.  In signed reciprocal endpoint
coordinate \(u\), the horizontal equation has the prepared form
\begin{equation}
 \widehat X u=-uP(u,\lambda),\qquad
 P(u,\lambda)=U(u,\lambda)(u^2+A_1u+A_0),\quad U\ne0.
\label{eq:legacy-7-13}
\end{equation}
The quadratic factor has at most one multiple root across the two signed
endpoint charts; the persistent root \(u=0\) is an endpoint passage rather
than an internal gate.  The only other central label is the unique upper
\(D\)-collision.  Hence a retained itinerary has zero, one, or precisely these
two internal central gates.

The intersection with the center variety is also explicit.  On the quadratic
center component,
\begin{equation}
 \mu_2=-B,\qquad c=\mu_3=-a,
\label{eq:legacy-7-14}
\end{equation}
and the horizontal and vertical root polynomials
\begin{equation}
 f_Q(x)=1-B+ax+Bx^2,
 \qquad E_Q(w)=-B+aw-(1-B)w^2
\label{eq:legacy-7-15}
\end{equation}
have common discriminant \(a^2-4B(1-B)\).  Their double-root locus is
\begin{equation}
 B=\frac{t^2}{1+t^2},\qquad 0<t<t_0,
 a=\frac{2\epsilon t}{1+t^2},\qquad c=\mu_3=-a,
 \qquad \epsilon\in\{\pm1\}.
\label{eq:legacy-7-16}
\end{equation}
We call this the explicit invariant-center component.  The reversible center component has no
nonpersistent horizontal double root.  The discriminant calculation alone
does not certify a physical lips ensemble.  On the invariant-center component, however, direct
substitution in the H14 field gives the invariant strong line
\begin{equation}
 L_\epsilon=tx-\epsilon t^2z+\epsilon=0,
 \qquad X_\lambda L_\epsilon=K_\epsilon L_\epsilon,
\label{eq:legacy-7-16a}
\end{equation}
and the signed first-hit computation of
Proposition~\ref{prop:complete-physical-lips} supplies the complete pp strip and its
PP boundary through \(p_{-\epsilon}\).  Thus every two-central parameter on
the invariant-center component carries the required physical configuration.  At this
stage the conclusion is only that the residual affine class is empty on this
center component; it is not yet its global elimination.

For completeness, away from this invariant-center component the strong connection is tested physically.
On \(c=d\), write
\begin{equation}
 L=x+\frac{d}{2(1-B)}z+\frac{a}{2B},
\quad
 X_\lambda L=\left(Bx+\frac{a}{2}+\frac{d}{2}z\right)L
              -\kappa_{\rm inv}z,
\quad
 \kappa_{\rm inv}=\frac{4B(1-B)+ad}{4B(1-B)}.
\label{eq:legacy-7-17}
\end{equation}
The symbol \(\kappa_{\rm inv}\) is an invariance defect, not the root-scale
coordinate used later.  Equation \eqref{eq:legacy-7-17} recovers the invariant-center component when
\(\kappa_{\rm inv}=0\), but a general curved connection is accepted only
when its actual landing difference vanishes on the enlarged stopped-itinerary
domain.

\begin{proposition}[Finite successor induction for two-central candidates]
\label{prop:two-central-successors}
Suppose a candidate parabolic interval stops before it becomes a complete pp
component.  Its first stopped event has exactly one of the following
successors, read with the minimal-face priority of
Proposition~\ref{prop:exact-once}.
\begin{enumerate}[label=\textup{(\arabic*)},leftmargin=2.8em]
\item A traversable numbered divider, previous-side cut, corner, or tangency
passes to the pre-existing adjacent retained itinerary.
\item Loss or collision of one selected central gate passes to a hyperbolic,
one-central, no-passage, or persistent-endpoint itinerary.
\item A source, coefficient, or scale equality with both gates retained
passes to the source, mixed, strict, middle, or root regime selected by
\eqref{eq:legacy-5-2}.
\item A landing gap satisfying
\(|\delta_{\rm hh}|\ge\eta>0\) contributes no fixed point in a sufficiently
small neighborhood of the two-central face.
\item A passive side, collar exit, wrong orientation, stable-center limit, or
node core is terminal and has no full-return equation.
\item Coincident dividers or a collapsed interval leave no phase point
between the corresponding labels.
\end{enumerate}
Every continuation reaches, after finitely many such handoffs, a named zero
theorem or a terminal zero-contribution face.
\end{proposition}

\begin{proof}
The stopped first-hit theorem supplies the finite event list, and the
minimal-face rule assigns every simultaneous contact to one physical face.
In case~\textup{(1)}, the adjacent itinerary begins later in the strictly
ordered cut list \eqref{eq:legacy-4-3}; hence the cut-open rank increases.
There are only finitely many ranks and finitely many itineraries.  Cases
\textup{(2)}--\textup{(3)} are the prepared gate classification and the fixed
half-open scale dictionary, respectively.  In case~\textup{(4)}, continuity
of the actual landing difference gives a uniform sign.  The last two cases
are terminal by construction.  Thus a candidate is followed through a finite
handoff rather than discarded merely because its local itinerary label
changes, and an infinite continuation chain is impossible.
\end{proof}

\begin{proposition}[Completion of a physical lips configuration]
\label{prop:complete-physical-lips}
Every retained two-central pp candidate admits a finite refinement.  On each
subfamily carrying a complete physical lips configuration, the following
conditions hold:
\begin{enumerate}[label=\textup{(\roman*)},leftmargin=2.8em]
\item there are two distinct simple physical saddle-nodes;
\item their nonzero transverse eigenvalues have opposite signs;
\item the weak and strong sections are fixed and have positive normals;
\item the actual strong landing difference is zero;
\item there is a nonempty interval of complete pp orbits;
\item the other boundary is a complete PP or BP incidence;
\item no extra root, gate, stopping face, or collar side meets the retained interval;
\item on a PP boundary, \(r_{\rm bd}\notin[1/2,2]\); no ratio condition is
required on a BP boundary.
\end{enumerate}
Let \(\delta_{\rm hh}\) denote the actual strong landing difference on the
fixed physical sections, so condition \textup{(iv)} is
\(\delta_{\rm hh}=0\).  On a closed subcell
\(|\delta_{\rm hh}|\ge\eta>0\), no sequence of isolated fixed points can
approach the two-central face.  The remaining punctured neighborhood of the
connection face is not classified here by a formula: its common physical
product tube is proved in
Lemma~\ref{lem:two-central-landing-product-tube}, and its role as the constant
coefficient of the physical closing equation is proved independently in
Lemma~\ref{lem:part-iii-physical-affine-closing}.  Thus every refined
subfamily is either a complete physical lips configuration satisfying
\textup{(i)}--\textup{(viii)}, a landing-gap-separated subcell with zero
contribution to the local two-central limit, a punctured landing-split
subfamily on the same stopped word, or a named geometric first-hit outcome.

Condition \textup{(v)} means complete orbits whose alpha- and omega-limits are
the two saddle-nodes, not merely a finite section-to-section hit.  Sector
ordering is part of the refinement: if an opposite-sign adjacency has no
parabolic next-section itinerary, it first reaches a stable-center, preceding-side, or
exit face and is not retained.  Consequently every retained no-pp itinerary has
the same-attractivity required by
Theorem~\ref{thm:part-iii-two-central-no-pp}.  If a parabolic next-section interval
is nonempty, its maximal survivor component is followed until its non-hh
endpoint first reaches PP, BP, another gate, a passive side, a collapsed
interval, or the collar.  The first two complete conditions
\textup{(vi)}--\textup{(viii)}; the middle outcomes are already named adjacent
regimes, and Step~5 below eliminates the collar possibility.  The analytic
nonaffinity of the completed second boundary is separated into
Corollary~\ref{cor:two-central-nonaffine-boundary}.
\end{proposition}

\begin{proof}
\medskip\noindent\emph{Step 1: complete the surviving parabolic interval.}\par\noindent
For each ordered sign/outcome label, the actual strong first hit is defined on
an enlarged positive-margin domain and its landing difference
\(\delta_{\rm hh}\) certifies the hh incidence.  On the adjacent parabolic
side, recursive survivor intervals
select finitely many maximal first-hit itineraries.  Near either saddle-node the
local equations have the form
\begin{equation}
 \dot x=x^2U(x),\qquad \dot y=-\gamma(x,y)y,
 \qquad U,\gamma\ge c_0>0.
\label{eq:legacy-7-19}
\end{equation}
The weak travel time diverges while the strong variable contracts
exponentially.  Appending these two semiorbits converts every retained
first-hit interval into a continuum of complete pp orbits.

\medskip\noindent\emph{Step 2: separate a landing coefficient from a landing gap.}\par\noindent
Suppose isolated fixed points approach two internal critical boxes.
The finite list of cyclic itineraries and compact physical sections give a
subsequence with one fixed sector itinerary.  If it has a pp passage, the
complementary limiting arcs are the two selected strong branches, and
continuity of every actual first hit gives
\(\delta_{\rm hh}\to0\).  Hence a closed subcell
\(|\delta_{\rm hh}|\ge\eta\) contains no such fixed point near the
two-central face.  This proves the landing-gap assertion of the proposition.
No persistence of a punctured full-lap domain and no coefficient
normalization is inferred from the equality \(\delta_{\rm hh}=0\); those are
the two separate lemmas cited in the statement.

\medskip\noindent\emph{Step 3: dispatch every truncated candidate.}\par\noindent
The sector alternatives preceding this construction are exhaustive.  With
no parabolic next-section itinerary, cyclic sector order leaves only central
transitions.  An opposite-sign central adjacency necessarily meets the
stable-center divider, the previous-side cut, or the collar exit before the
next gate; these are terminal first-hit outcomes.  Hence the only retained no-pp
alternative has same attractivity.  With a nonempty parabolic itinerary, planar
order makes each maximal survivor set an interval, and its non-hh endpoint is
followed through the finite cut order.  It either supplies the boundary field
below or first hits one of the already numbered terminal or adjacent outcomes.
This is precisely the finite continuation of
Proposition~\ref{prop:two-central-successors}.

\medskip\noindent\emph{Step 4: identify the only possible second boundary.}\par\noindent
The other endpoint is retained only if both physical axis landings and a
positive boundary-arc margin give a complete PP incidence, or if the actual
center-side axis gives a complete BP incidence.  Write the selected upper
double root as \(w=-q\).  On the two-central sheet, after shrinking,
\begin{equation}
 S(0)=-(1-B)q^2(1-2B)^2<0,\qquad S'(R)\le-\tfrac12,
\label{eq:legacy-7-20}
\end{equation}
so \(S(R)<0\) for \(R>0\) and there is no interior boundary saddle.
The only PP saddle is consequently a persistent principal endpoint, with
ratio
\begin{equation}
 r_{\rm bd}=\left|\frac{B}{1-B}\right|
 \quad\hbox{or its reciprocal}.
\label{eq:legacy-7-21}
\end{equation}
After \(|B|<1/4\), this ratio lies outside \([1/2,2]\).

\medskip\noindent\emph{Step 5: exclude every hidden divider and complete the boundary.}\par\noindent
It remains to prove that no unlisted divider truncates the pp interval.  Let
\(w=-q\) denote the selected upper double root.  On that sheet,
\begin{equation}
 m=-(1-B)q^2,\qquad d=2(1-B)q,\qquad
 E(w)=-(1-B)(w+q)^2.
\label{eq:legacy-7-21a}
\end{equation}
Let
\(\sigma\) denote the signed source-arc record selected by the stopped itinerary;
for \(q>0\) it is (+), and reflection treats \(q<0\).  On the
invariant-center component one has
\(\sigma=-\epsilon\).  In physical endpoint coordinates,
\begin{equation}
 \dot\xi=-\rho(\xi,\alpha)\xi(\xi-\xi_h)^2,
 \qquad \dot\eta=\eta,\qquad \rho\ge\rho_0>0,
 \qquad J_0=\{0<\xi<\xi_h\}.
\label{eq:legacy-7-22}
\end{equation}
Every orbit in \(J_0\) has alpha-limit \(S_h\); its two boundary paths are
the selected hh orbit and
\(S_h\to p_\sigma\to E_\sigma\).  The fixed cut order prevents a reset,
extra lap, or opposite source arc, so the only macro itinerary is
\begin{equation}
 S_h\longrightarrow E_\sigma\longrightarrow S_v
\label{eq:legacy-7-22a}
\end{equation}
through a finite ordered list of regular tiles.  Before using order, record
the complete list of physical boundary faces in
Table~\ref{tab:two-central-boundary-faces}.
\begin{table}[H]
\centering
\caption{Physical boundary faces for the retained two-central pp interval.}
\label{tab:two-central-boundary-faces}
\label{eq:legacy-7-22b}
\small
\begin{tabularx}{.92\textwidth}{@{}
  >{\raggedright\arraybackslash}p{.17\textwidth}
  >{\raggedright\arraybackslash}p{.39\textwidth}
  >{\raggedright\arraybackslash}X@{}}
\toprule
Piece&Boundary faces&Possible divider\\
\midrule
Endpoint&\(\eta=\eta_0,\eta_1\), \(\xi=0,\xi_h\)&
\(\xi=0,\xi_h\)\\
\(E_\sigma\) tube&\(\theta=\theta_j,\theta_{j+1}\),
\(n=n_j^-,n_j^+\)&backward exit-corner orbit\\
Upper box&\(I_+,I_-\), \(r=r_*\), \(\partial B_{S_v}\)&
stable or central equilibrium branch\\
\bottomrule
\end{tabularx}
\end{table}
The endpoint row is exhaustive by \eqref{eq:legacy-7-22}, while the middle row consists of
fixed nonsingular flow boxes.  In such a tile,
\begin{equation}
 \frac{dn}{ d\theta}=f_j(\theta,n),\qquad X_\lambda\theta\ge2c_0,
 \qquad n_j^-<n_{e,j}<n_{h,j}<n_j^+,
\label{eq:legacy-7-22c}
\end{equation}
where \(n_{e,j}\) and \(n_{h,j}\) are the equatorial and hh solution graphs.
Let \(C_j\ge1\) be a bi-Lipschitz constant for the fixed physical tile
coordinate, and let \(\eta_{hh}>0\) be one quarter of the minimum stopping-face
distance of the compact selected hh subword in the original physical
metric.  Define the physical collar distances in this tile by
\begin{equation}
\begin{aligned}
 \rho_j^e&=\min_\theta\{
 n_{e,j}(\theta)-n_j^-,\,n_j^+-n_{e,j}(\theta)\}>0,\\
 \rho_j^h&=\min_\theta\{
 n_{h,j}(\theta)-n_j^-,\,n_j^+-n_{h,j}(\theta)\}
 \ge \frac{4\eta_{hh}}{ C_j}>0.
\end{aligned}
\label{eq:legacy-7-22d}
\end{equation}
The first inequality follows from the fixed nonzero widths of the equatorial
tube; the second is precisely the hh stopping-face margin transported by the
bi-Lipschitz chart.  Put
\(\Delta_j=\theta_{j+1}-\theta_j\).
For \(N=n-n_{e,j}\), scalar uniqueness gives the literal difference equation
\begin{equation}
 \frac{dN}{ d\theta}=A_j(\theta,N)N,\qquad |A_j|\le L_j.
\label{eq:legacy-7-22e}
\end{equation}
If \({\cal D}_j=(d_j^-,d_j^+)\) is the interval of entry values reaching the
next cut, Gronwall gives neighborhoods of the two boundary entries of radii
\begin{equation}
 m_j^e=\frac{\rho_j^e}{2}e^{-L_j\Delta_j},\qquad
 m_j^h=\frac{\rho_j^h}{2}e^{-L_j\Delta_j}.
\label{eq:legacy-7-22f}
\end{equation}
Since \({\cal D}_j\) is one interval and its endpoints are exactly the two
backward exit-corner orbits in
Table~\ref{tab:two-central-boundary-faces}, the whole segment between
\(n_{e,j}(\theta_j)\) and \(n_{h,j}(\theta_j)\) lies at positive distance
from every divider.  More explicitly, on the closed order band
\begin{equation}
 K_j=\{(\theta,n):n_{e,j}(\theta)\le n\le n_{h,j}(\theta)\}
\label{eq:legacy-7-22f1}
\end{equation}
one has pointwise
\begin{equation}
 n-n_j^-\ge\rho_j^e,\qquad n_j^+-n\ge\rho_j^h.
\label{eq:legacy-7-22f2}
\end{equation}
Thus neither collar, either collar corner, nor a preceding-side face can be
the first contact.  Since \(X_\lambda\theta>0\), every intermediate graph
reaches the next \(\theta\)-face.  Induction over the finite tile list carries all of
\(J_0\) to the upper entry as
\begin{equation}
 w=\delta,\qquad 0<r<r_h,
\label{eq:legacy-7-22g}
\end{equation}
with equator and hh orbit as endpoints.

Put \(v=-w\).  The exact upper equations are
\begin{equation}
 \dot r=r\{(1+r)v-d-rc\},\qquad
 \dot v=-m-dv+(1-B)v^2+r\{1+(a-c)v+v^2\}.
\label{eq:legacy-7-23}
\end{equation}
There is no unrecorded gate because
\begin{equation}
 S(0)=-(1-B)q^2(1-2B)^2<0,\qquad S'(r)\le-\frac{1}{2}.
\label{eq:legacy-7-23a}
\end{equation}
Before their first target-block entry, the compact equatorial and hh arcs have
positive distance from \(I_-\), \(r=r_*\), and their corners.  Let
\(C_{\rm up}\ge1\) dominate both bi-Lipschitz comparisons from the original
physical metric to the upper \((r,w)\) chart and from that chart to the fixed
target saddle-node chart.  Choose the target block itself disjoint from
\(I_-\cup\{r=r_*\}\).  A common upper-coordinate margin is therefore
\begin{equation}
 \varepsilon_{\rm up}=\min\left\{r_*,\frac{3\delta}{4},
 \frac{2\eta_{hh}}{C_{\rm up}},
 \operatorname{dist}(\overline{B_{S_v}},I_-\cup\{r=r_*\})\right\}>0.
\label{eq:legacy-7-23b}
\end{equation}
Moreover the oriented system is cooperative,
\begin{equation}
 \partial_v\dot r=r(1+r)\ge0,
 \qquad \partial_r\dot v=1+(a-c)v+v^2\ge\frac{1}{2}.
\label{eq:legacy-7-24}
\end{equation}
Every intermediate path therefore remains between the two boundary paths and
keeps the margin \eqref{eq:legacy-7-23b}.  Here the comparison is at a common value of the
positive desingularized upper-chart time: start all solutions on
\(w=\delta\), write \(v=-w\), and denote the
equatorial and hh solutions by \((r_e,v_e)\), \((r_h,v_h)\).  The
quasimonotone system \eqref{eq:legacy-7-24} preserves the positive quadrant, so
\begin{equation}
 r_e(t)\le r(t)\le r_h(t),\qquad
 v_e(t)\le v(t)\le v_h(t)
\label{eq:legacy-7-24b}
\end{equation}
throughout their common maximal interval.  Consequently
\begin{equation}
 r_*-r(t)\ge r_*-r_h(t),\qquad
 w(t)+\delta=\delta-v(t)\ge\delta-v_h(t),
\label{eq:legacy-7-24c}
\end{equation}
and \eqref{eq:legacy-7-23b} excludes the collar, \(I_-\), and their corners for every
intermediate solution, not only at its endpoints.

The two boundary solutions enter the target block
\begin{equation}
 B_{S_v}=\{-X_0\le x\le0,\ 0\le y\le Y_0\},\qquad
 \dot x=x^2U(x),\quad \dot y=-\gamma(x,y)y,
 \qquad U,\gamma\ge c_v>0,
\label{eq:legacy-7-24a}
\end{equation}
with invariant inner axes and strictly inward outer sides, and therefore both
converge to \(S_v\).  This is all that is needed from the target coordinates:
the common-time inequalities \eqref{eq:legacy-7-24b} are in the upper \((r,v)\) chart, and
both their lower and upper bounds converge to the same physical point
\(S_v\).  Coordinate-wise squeezing gives
\((r(t),v(t))\to S_v\) for every intermediate solution.  No monotonicity of
the target coordinate change is assumed.  The whole interval is one pp
component, whose second boundary is
\(S_h\to p_\sigma\to S_v\), not an unrecorded collar limit.  Thus the
geometric boundary is complete.  Its analytic nonaffinity, and hence the
empty residual class, are isolated in the next corollary.
All statements above concern one of the finitely many physical stopped
itineraries; no connected-component finiteness for a landing-zero set is used.
The complete boundary-face and first-contact classification, the representative Gronwall
estimates, and the finite algebraic identities used above are collected in
Appendix~\ref{app:two-central-incidence}.
\end{proof}

\begin{corollary}[Nonaffine second boundary and empty residual class]
\label{cor:two-central-nonaffine-boundary}
For every complete component of
Proposition~\ref{prop:complete-physical-lips}, the non-hh boundary is a
certified PP or BP boundary in the sense of DIR.  Its critical pp transition
is nonaffine.  Consequently the former affine/collar residual satisfies
\begin{equation}
 \mathfrak A_{\rm aff}=\varnothing.
\label{eq:legacy-7-25}
\end{equation}
\end{corollary}

\begin{proof}
In the PP case, Step~4 of the proposition identifies the boundary saddle and
\eqref{eq:legacy-7-21} places its characteristic ratio outside
\([1/2,2]\), hence away from one.  In the BP case no ratio alternative is
required.  Step~5 proves that the entire survivor interval reaches this same
physical boundary and that no collar side or hidden divider replaces it.
The proof of DIR2002, Theorem~3.2, therefore applies to the certified PP/BP
alternative and gives a nonaffine infinity jet of the critical transition.
Analyticity gives a first nonzero derivative of finite order at least two.
An affine residual transition would contradict that jet; an unrecorded collar
residual was already excluded by Step~5.  This proves
\eqref{eq:legacy-7-25}.
\end{proof}

\begin{lemma}[Common physical product tube across the landing split]
\label{lem:two-central-landing-product-tube}
Let $K_{\rm geo}$ be a compact family of complete physical lips
configurations supplied by
Proposition~\ref{prop:complete-physical-lips}.  There is a finite phase cover
by intervals $J_i^-\Subset J_i^+$, one-sided at the hh and PP/BP endpoints,
nested physical product tubes $T_i^-\Subset T_i^+$, open parameter
neighborhoods $W_i$ in the original five parameters, and numbers
$\delta_i>0$ with the following property.  Whenever the two selected central
gates and the signed sector label are retained, $\alpha\in W_i$, and
$|\delta_{\rm hh}(\alpha)|<\delta_i$, every primitive first hit composing the
selected stopped word is defined on $J_i^-$ and remains in $T_i^+$ until its
named target.  The same ordered full-lap word is therefore defined on this
common product tube whether or not $\delta_{\rm hh}$ vanishes.  A nonzero
landing split changes the landing point on the fixed strong section; it does
not insert a root, a gate, a box side, or a collar contact.
\end{lemma}

\begin{proof}
At a base configuration, conditions \textup{(iii)} and \textup{(vii)} of
Proposition~\ref{prop:complete-physical-lips}, the order-band margins
\eqref{eq:legacy-7-22d}--\eqref{eq:legacy-7-22f2}, and the upper margin
\eqref{eq:legacy-7-23b} separate every retained arc from every nontarget
member of the fixed face list $\mathcal F$.  The two boundary arcs and all
intermediate pp arcs lie in the same finite chain of regular tiles and the
same upper target block.  Choose inner and outer tubes inside these positive
margins.  Analytic dependence of a transverse first hit on its entry point
and on all five parameters then preserves each primitive map from the inner
tube into the outer tube.

This construction is uniform on a phase neighborhood.  At the two
compactified phase endpoints use the one-sided section charts already present
in the stopped atlas; a finite phase and base subcover gives the displayed
family of product tubes.  If no $W_i$ and $\delta_i$ worked, a sequence of
entries and parameters would converge to a base orbit but would meet a
nontarget member of $\mathcal F$ before its prescribed target.  Its first
such contacts would converge to a contact of the base orbit with that face,
contradicting the strict stopping-face margin.  The two strong branches may
consequently land at different points of their fixed section without changing
any preceding first-hit domain.  This is a persistence statement for the
physical stopped word; it makes no coefficient division and uses no strict,
middle, root, or zero-scale estimate.
\end{proof}

\begin{lemma}[Landing coefficient in the common physical product tube]
\label{lem:part-iii-physical-affine-closing}
Fix one of the common product tubes from
Lemma~\ref{lem:two-central-landing-product-tube} and a retained
complete-lips itinerary with positive coalescence scale in one of the strict,
middle, or positive-root regimes.  On every sufficiently small punctured
landing-split fiber in that tube, the same ordered primitive first hits form
the physical full lap.  On each fixed parameter fiber,
let \(\widetilde q,\widetilde v\) be the DIR strong-section coordinates and
let \(q,v\) be the direct action coordinates used in the resolved connector.
Then the physical full-return closing equation is, in the direct coordinates,
\begin{equation}
 A_0{\cal T}(q)-B_0q-C_0=0.
\label{eq:part-iii-physical-affine-closing}
\end{equation}
The changes between the two pairs of section coordinates are affine, and
zeros and their multiplicities are preserved.  This assertion is fiberwise
on the positive-scale regimes; no joint DIR chart through the source or exact
mixed face is asserted.  In this decomposition the landing split occurs only
in the undivided constant coefficient $C_0$; it does not alter the product
tube or create another phase map.
\end{lemma}

\begin{proof}
The common-product-tube lemma supplies the domains and the actual first-hit
word before any normal form is introduced.  The coordinate reduction below
is therefore applied to already defined physical maps on each positive-scale
fiber; it is not used to continue a map across a failed first contact.

The section-coordinate reduction in the proof of DIR2002,
Theorem~3.1, makes the two central passages
\(\widetilde q\mapsto M(\alpha)\widetilde q\) and
\(\widetilde v\mapsto m(\alpha)\widetilde v\), with \(Mm\ne0\), and makes
the already certified hh connector
\(\widetilde v\mapsto\widetilde v+\delta(\alpha)\).  If
\(\widetilde{\cal T}\) is the pp transition, composing these four actual
first hits around the exact-once itinerary gives the physical equation
\begin{equation}
 \widetilde{\cal T}(\widetilde q)
 -m^{-1}M^{-1}\widetilde q+m^{-1}\delta=0.
\label{eq:part-iii-dir-closing}
\end{equation}
The direct action coordinates also make the same two strong equations linear.
Consequently, on their common fixed-parameter sections,
\begin{equation}
 \widetilde q=a_hq+b_h,
 \qquad \widetilde v=a_vv+b_v,
 \qquad a_ha_v\ne0,
\label{eq:part-iii-affine-gauges}
\end{equation}
where the four coefficients depend only on the parameter.  Hence
\(\widetilde{\cal T}(\widetilde q)
=a_v{\cal T}((\widetilde q-b_h)/a_h)+b_v\).  Substitution in
\eqref{eq:part-iii-dir-closing}, followed by the affine change of input,
gives \eqref{eq:part-iii-physical-affine-closing}, for example with
\[
 A_0=a_v,\qquad B_0=m^{-1}M^{-1}a_h,\qquad
 C_0=m^{-1}M^{-1}b_h-m^{-1}\delta-b_v,
\]
up to multiplication of the triple by a common nonzero factor.
The connector occurrence of the landing difference is the term
$-m^{-1}\delta$ in $C_0$.  In particular, no factor
$\delta_{\rm hh}$ is divided out and no domain is selected according to
whether that difference vanishes.

The input change has derivative \(a_h\ne0\), and multiplication of the
closing equation by a nonzero parameter factor does not change an isolated
zero or its multiplicity.  In particular,
\begin{equation}
 \widetilde{\cal T}''(\widetilde q)
 =\frac{a_v}{a_h^2}{\cal T}''
   \left(\frac{\widetilde q-b_h}{a_h}\right),
\label{eq:part-iii-curvature-gauge}
\end{equation}
so the curvature/Rolle count in the direct gauges counts precisely the
physical cycles.  On a punctured coefficient chart a common coefficient
factor may be divided out; at its radial apex the original displacement is
identically zero and has no isolated member.  The faces \(t=0\) and
\(\kappa=0\) are not obtained by extending these gauges: they retain their
independent source and exact-mixed assignments.
\end{proof}

\begin{theorem}[Geometric completion and scale routing of two-central itineraries]
\label{thm:part-iii-two-central-exhaustion}
In a sufficiently small full five-parameter neighborhood, every retained
exact-once full-return itinerary on a minimal face and carrying both possible internal
central blocks has the following three ordered conclusions.
\begin{enumerate}[label=\textup{(\roman*)},leftmargin=2.8em]
\item \emph{Physical geometry.}  It is either the same-attractivity no-pp
configuration, a named terminal first-contact outcome, or a complete
physical lips configuration satisfying
Proposition~\ref{prop:complete-physical-lips}\textup{(i)}--\textup{(viii)}.
In the latter case the hh connection, the nonempty pp component, its complete
PP/BP boundary, and the absence of a hidden divider are properties of the
original vector field.
\item \emph{Boundary type.}  Every complete component has the nonaffine
PP/BP transition of
Corollary~\ref{cor:two-central-nonaffine-boundary}; hence no affine or collar
residual remains.
\item \emph{Scale routing.}  Only after \textup{(i)}--\textup{(ii)} have
been established is every retained itinerary assigned exactly once to one of
\begin{equation}
 \begin{gathered}
 \text{two-central no-pp},\quad
 \text{positive-margin strict lips},\\
 \text{middle QBF/QHH},\quad
 \text{positive root merger},\\
 \text{matched source},\quad
 \text{exact mixed persistent endpoint}.
 \end{gathered}
\label{eq:legacy-7-26}
\end{equation}
by the fixed half-open rule \eqref{eq:legacy-5-2}.
\end{enumerate}
No local zero estimate is used to prove these three conclusions.

On the complete candidate cover, a passive side, collar exit,
wrong-orientation sector, node core, or interval collapse is a terminal
outcome before \eqref{eq:legacy-7-26} is entered.  Thus the classification is also exhaustive at
the candidate level: terminal passive/exit outcomes remain in the geometric
classification, while every actual retained itinerary enters exactly one of the
six regimes in \eqref{eq:legacy-7-26}.

A landing-gap-separated subcell has zero contribution to the local
two-central limit before regime selection.  A punctured landing-split fiber
near \(\delta_{\rm hh}=0\) is not another outcome: by
Lemma~\ref{lem:two-central-landing-product-tube} it uses the same physical
product tube, and by
Lemma~\ref{lem:part-iii-physical-affine-closing} its landing value is retained
as the ambient constant coefficient of the selected strict, middle, or root
comparison.
\end{theorem}

\begin{remark}
The classification is uniform on the finite signed gate, sector, and
first-hit cover in all five original parameters, including divider
collisions and both endpoint charts.  A retained itinerary without a complete
pp connection enters the no-pp case.  A pp itinerary enters the complete-lips
case only when the hypotheses of Proposition~\ref{prop:complete-physical-lips} hold;
nearby landing-split fibers remain in the same
ambient coefficient cone, whereas a landing-gap-separated cell has no
approaching two-central fixed point.  In the source/root resolution, $t=0$
enters the source regime, $t>0$, $\varrho_w=0$ the mixed regime,
$\varrho_w\ge\varrho_\#$ the middle regime, and
$0<\varrho_w<\varrho_\#$ the root regime; equality belongs to the middle
regime.  This theorem classifies only itineraries already retained by
Theorem~\ref{thm:finite-words}
and plays no role in constructing the stopped atlas.
\end{remark}

\begin{proof}
\medskip\noindent\emph{Step 1: physical geometry and landing domains.}\par\noindent
Equation \eqref{eq:legacy-7-13} and the unique upper \(D\)-collision give the finite list of gate
configurations.  The sector refinement in
Proposition~\ref{prop:complete-physical-lips} first proves that a retained itinerary with
no parabolic connection has same attractivity; an opposite-sign adjacency
without a parabolic next-section itinerary is already a stable-center,
previous-side, or exit outcome.  The retained no-pp itineraries are therefore
exactly the two-central no-pp case.  Every remaining opposite-sign itinerary has a selected
maximal parabolic survivor interval.  Proposition~\ref{prop:complete-physical-lips}
refines it into a complete physical lips configuration, a
landing-gap-separated cell with zero contribution to the local two-central
limit, or a named geometric first-hit outcome.
Lemma~\ref{lem:two-central-landing-product-tube} places every sufficiently
small punctured landing split in the same stopped-word domain; this conclusion
uses only the physical margins and compactness of the finite first-hit cover.

\medskip\noindent\emph{Step 2: complete boundary and nonaffinity.}\par\noindent
For every complete component, Step~5 of
Proposition~\ref{prop:complete-physical-lips} proves the full PP/BP boundary
before Corollary~\ref{cor:two-central-nonaffine-boundary} invokes the DIR
boundary result.  The corollary gives
\(\mathfrak A_{\rm aff}=\varnothing\).  Neither its hypotheses nor its proof
uses a scale assignment or a zero count.

\medskip\noindent\emph{Step 3: half-open scale routing.}\par\noindent
On the globally fixed nested tubes of Section~\ref{sec:regime-table}, a complete-lips
point outside the inner tube or with \(t\ge t_{\rm str}\) belongs to the strict regime,
including equality.  Only the pointwise region \(0<t<t_{\rm str}\)
proceeds to the middle/root weighted split.

It remains to cover a vanishing source/root margin.  The finite resolution
uses \(t\ge0\) and, for \(t>0\), the weighted middle variables
\begin{equation}
 b_m=\frac{B}{ t^2},\qquad A_m=\frac{a}{ t},\qquad
 \varrho_w=(b_m^2+A_m^4)^{1/4}.
\label{eq:legacy-7-26a}
\end{equation}
The weight-\((2,1)\) blow-up of the punctured \((b_m,A_m)\)-plane has
finitely many compact angular charts.  In a chart that can retain the
horizontal central gate, the leading weak bracket is
\begin{equation}
 b+A\bar R+\bar R^2 .
\label{eq:legacy-7-26b}
\end{equation}
The double-root equations are
\(b+A\bar R+\bar R^2=0\) and \(A+2\bar R=0\).
For a positive retained root, radial normalization therefore puts the
angular pair in finite signed neighborhoods of \((b,A)=(1,-2)\).  On every
other angular chart the same quadratic has
only simple roots, no positive root, or the wrong crossing sign.  The first
two alternatives are respectively hyperbolic/one-central or no-passage;
the last is passive or exit.  At the upper vertex the prepared weak
quadratic has only four outcomes: a retained double root, separated simple
roots, no root in the retained side, or a root with the wrong transverse
sign.  These are, respectively, the present \(D\)-gate, hyperbolic boxes,
one-central/no-passage, and passive/exit outcomes from the finite gate
table.  Thus an opposite
projective chart cannot carry an unlisted two-central pp itinerary.

On the surviving angular charts, \(\varrho_w\ge\varrho_\#\), including
equality, is the fixed middle overlap by definition.  For
\(0<\varrho_w<\varrho_\#\), the weight-\((2,1)\) substitution
\eqref{eq:legacy-5-2b} introduces the positive radial coordinate \(\kappa\)
and the finite root-chart cover.  The face \(\varrho_w=0\) is \(B=a=0\) and is the exact
mixed persistent-endpoint regime.  The face \(t=0\), including their corner,
belongs to the source regime.  These four sets cover the weighted resolution and are
half-open by construction.

Every failed complete-lips hypothesis or first-hit condition is already a
named stopped alternative.  On the surviving common product tubes,
Lemma~\ref{lem:part-iii-physical-affine-closing} retains the landing split as
the undivided constant coefficient of the strict, middle, or root closing
equation; it therefore creates no additional scale row.
Minimal-face priority assigns
the face \(t=0\) to source.  On \(t>0\), it assigns strict first, then
middle (including the overlap equality), root, and finally the
\(\varrho_w=0\) face to mixed.  Hence each retained itinerary reaches exactly one row of
\eqref{eq:legacy-7-26}, while passive, exit, wrong-orientation, node, and collapsed cases
remain terminal geometric outcomes.
\end{proof}

\begin{theorem}[Two-central no-pp estimate]
\label{thm:part-iii-two-central-no-pp}
Let a retained itinerary contain the nonpersistent horizontal and upper \(D\)
simple saddle-nodes, both traversed through their central sectors, with
same attractivity and no complete pp connection.  If all remaining physical
sections, connectors, and vertices are certified, the graphic has cyclicity
one and an ambient locally uniform finite zero bound.
\end{theorem}

\begin{remark}
The theorem is applied only after
Theorem~\ref{thm:part-iii-two-central-exhaustion} has certified the no-pp
alternative.  Its bound is uniform on an open full-parameter neighborhood of
each certified elementary graphic and hence, by a finite cover, on compact
families.  A complete pp component leaves this case; a split gate becomes
hyperbolic or one-central, a persistent endpoint enters the mixed regime, and
wrong-attractivity or failed first-contact outcomes are discarded before a
return equation is formed.
\end{remark}

\begin{proof}
The physical sector classification supplies exactly the same-attractivity,
no-parabolic-connection hypothesis of the unconditional no-pp clause of
DIR2002, Corollary~3.6.  That clause gives cyclicity one.  We do not use its
generic one-pp or two-pp alternatives.  Persistence of the certified
elementary graphic supplies the ambient neighborhood; faces on which its
hypotheses fail have already been assigned to the adjacent regimes.
\end{proof}

\section{Lips configurations with uniform positive margins}
\label{sec:part-iii-strict}

We now consider complete physical lips configurations from
Proposition~\ref{prop:complete-physical-lips} whose geometric and analytic margins
remain positive on a compact base family.  The DIR strong-section coordinates
were introduced in Lemma~\ref{lem:part-iii-physical-affine-closing} only after
the common physical product tube had been certified.  They simplify the
central maps and hh connector but create none of the physical hypotheses.

\begin{theorem}[Positive-margin lips estimate]
\label{thm:part-iii-strict-zero}
Fix one complete physical lips configuration and a compact family
\(K_{\rm str}\) on which the transverse eigenvalues, quadratic weak
coefficients, section normals, stopping-face distances, hh and pp margins,
boundary-arc margin, and, only in the PP case, the PP ratio margin are
positive.  Assume also positive
separation from the source and root-merger faces.  Then the complete
compactified pp phase interval, including its hh and PP/BP endpoints, has a
locally uniform finite zero bound for
\begin{equation}
 A_0(\alpha)R_\alpha(y)-B_0(\alpha)y-C_0(\alpha)=0,
\label{eq:legacy-7-27}
\end{equation}
where
\begin{equation}
 (A_0,B_0,C_0)=\varepsilon(\widehat A,\widehat B,\widehat C),
 \qquad [\widehat A:\widehat B:\widehat C]\in\mathbb{RP}^2
 \quad(\varepsilon\ne0).
\label{eq:legacy-7-27a}
\end{equation}
The projective direction controls the punctured coefficient neighborhood;
$\varepsilon=0$ is the identity coefficient apex.  No certified nonzero
direction gives an identity return.
\end{theorem}

\begin{remark}
The hypotheses describe two actual opposite saddle-nodes, fixed physical
sections, a positive-margin hh itinerary, a nonempty maximal complete pp
strip, its complete PP/BP boundary, isolation, and, in the PP case, \eqref{eq:legacy-7-21}.
The bound is uniform on nested physical tubes and phase intervals in open
neighborhoods of all five original parameters; a finite subcover of
$K_{\rm str}$ gives one constant.  The hh and PP/BP phase endpoints are
included.  Source coalescence, root merger, loss of a section or a
complete-lips hypothesis,
and a change in the first-contact outcome enter their named adjacent regimes rather than being
reached by continuation of this theorem.
\end{remark}

\begin{proof}
The order of choices is important and is
\begin{equation}
 (b,a)\ \longrightarrow\ n(b,a)\ \longrightarrow\
 K\ge n(b,a)+2\ \longrightarrow\
 (T^-\Subset T^+,\ I^-\Subset I^+,\ {\cal U}_{b,a})
 \ \longrightarrow\ \text{Rolle bound}.
\label{eq:strict-quantifier-order}
\end{equation}
Fix first a base-phase pair \((b,a)\).  The complete physical lips component
and its critical transition already exist by
Proposition~\ref{prop:complete-physical-lips}; no normal form is needed to
define them.  Corollary~\ref{cor:two-central-nonaffine-boundary} says that the
critical analytic transition \(R_b\), in the direct section gauge, is
nonaffine.  At an interior phase point, if all derivatives of order at least
two vanished, its Taylor series would be affine near that point and analytic
continuation would make it affine on the connected component.  Therefore
one may first choose a finite \(n=n(b,a)\ge2\) with
\(R_b^{(n)}(a)\ne0\).

Only after this finite order is known do we choose \(K\ge n+2\) and invoke
DIR2002, Theorem~3, for \(C^K\) fibered saddle-node normal forms.  The
physical closing Lemma~\ref{lem:part-iii-physical-affine-closing} then gives
\eqref{eq:legacy-7-27} in those coordinates.  Let \(m_{b,a}>0\) be the
minimum of the physical margins in the statement.  Continuity of the finite
\(C^n\) jet and persistence of those margins now give nested tubes, intervals,
and an open neighborhood in the original five parameters,
\begin{equation}
 T^-_{b,a}\Subset T^+_{b,a},\qquad
 I^-_{b,a}\Subset I^+_{b,a},\qquad
 {\cal U}_{b,a}\subset\mathbb R^5,
\label{eq:legacy-7-28}
\end{equation}
on which every inner-tube orbit follows the same physical itinerary in the
outer tube and, after shrinking,
\begin{equation}
 \inf_{I^+_{b,a}}|R_b^{(n)}|\ge4c_{b,a}>0,
 \qquad
 \|R_\alpha^{(n)}-R_b^{(n)}\|_{C^0(I^+_{b,a})}<2c_{b,a}.
\label{eq:legacy-7-29}
\end{equation}
For \(\varepsilon\ne0\), divide \eqref{eq:legacy-7-27} by \(\varepsilon\).  If
\(\widehat A\ne0\), work only on that projective chart.  The affine terms in
\eqref{eq:legacy-7-27} vanish after \(n\) derivatives, and multiplicity Rolle gives at most
\(n\) zeros on \(I^-_{b,a}\).  On the complementary chart
\(\widehat A=0\), a nonzero affine equation has at most one zero (and a
nonzero constant has
none), uniformly up to its projective boundary.

At the hh endpoint, one-sided analytic continuation of the pp transition
first supplies a finite \(n(b,a)\) exactly as above; the normal-form order and
neighborhood are chosen afterward.  At the natural PP/BP endpoint, the
proof-level argument of DIR2002, Theorem~3.2, first supplies nonaffinity
through, respectively, its hyperbolic power/logarithmic alternative or its
BP length contradiction.  The one-sided analytic expansion then selects a
finite nonzero derivative, and only afterward do we set the finite-smooth
order \(K\).  Thus
both endpoints have their own theorem neighborhoods and are not extrapolated
from an interior jet.  Compactness of the phase interval and of
\(K_{\rm str}\) now selects finitely many neighborhoods from
\eqref{eq:strict-quantifier-order}; the maximum of their finitely many
orders is finite.  If a nonzero projective
direction made \eqref{eq:legacy-7-27} an identity, \(R_b\) would be affine, contradicting the
same proof-level nonaffine-jet argument.
At \(\varepsilon=0\) the original displacement is identically zero and has
no isolated zeros.  Since every punctured fiber is
controlled by the compact projective-direction cover, the same ambient
coefficient cone controls its apex.
\end{proof}

\section{The intermediate coalescing scale}
\label{sec:part-iii-middle}

The strict theorem cannot be compactified through simultaneous source
coalescence.  The middle theorem resolves that loss while the horizontal
nonpersistent-root scale and the upper \(D\)-gate scale remain comparable.
The invariant-center component supplies an exact model, not a perturbation theorem.
Uniformity is proved separately on a buffered finite-phase region (QBF) and
on an unbounded hyperbolic corner (QHH), then extended in a fixed-\(\gamma\)
resolved frame to the entire middle regime.

Every middle and root comparison below uses the coefficient-cone convention
\begin{equation}
 (A_0,B_0,C_0)=\varepsilon_c
   (\widehat A_0,\widehat B_0,\widehat C_0),\qquad
 [\widehat A_0:\widehat B_0:\widehat C_0]\in\mathbb {RP}^2
 \quad(\varepsilon_c\ne0).
\label{eq:legacy-7-29d}
\end{equation}
The projective direction controls the punctured coefficient neighborhood;
\(\varepsilon_c=0\) is the identity apex and contributes no isolated zero.
Statements below about a zero triple always mean this radial apex, not a
point of projective space.

\begin{lemma}[Jets of scalar first-hit maps and their inverses]
\label{lem:scalar-first-hit-jets}
Let \(x'=f(\tau,x,\lambda)\) be a scalar clock equation on a doubled
first-hit tube.  Suppose the clock normal is separated from zero, every
coefficient word of total order at most four has a finite \(L^1(d\tau)\)
majorant, and the clock endpoints, entry value, entry section, and target
section depend mixed \(C^4\)-smoothly on the phase and resolved parameters,
with all tensors through order four uniformly bounded.  Assume also that the
entry and target normal determinants, including the oriented phase
determinant, are separated from zero on the closed doubled tube.  Then the
first-hit map and its inverse have uniform mixed \(C^4\) bounds and fixed
oriented margins \(0<m\le P'\le M\).
\end{lemma}

\begin{proof}
For every mixed derivative word \(I\), differentiation of the scalar flow
gives the triangular recurrence
\begin{equation}
 (D_I X)'=f_xD_I X+{\cal B}_I,
\label{eq:legacy-7-29a}
\end{equation}
where \({\cal B}_I\) is the finite Bell-polynomial sum of coefficient
derivatives and strictly lower solution jets.  Variation of constants and
the \(L^1\) majorants bound these words inductively.  If entry clock \(a\),
target clock \(b\), entry state \(x_0\), and parameters vary in a direction
\(D\), the first derivative is exactly
\begin{equation}
 DP=J\{Dx_0-f(a,x_0)Da\}+Z_D+f(b,P)Db,
\label{eq:legacy-7-29b}
\end{equation}
where \(J=\exp\int f_x\,d\tau\) and \(Z_D\) is the fixed-endpoint parameter
variation.  Differentiating \eqref{eq:legacy-7-29b} three more times gives all moving-
endpoint tensors; the unique top-order solution jet remains linear.

If \(C_j\) bounds the \(j\)-th mixed tensor of \(P\) and \(m\le P'\),
differentiating \(P(P^{-1}(y,\lambda),\lambda)=y\) gives
\begin{equation}
\begin{aligned}
 I_1&=m^{-1}(1+C_1),\\
 I_2&=m^{-1}C_2(1+I_1)^2,\\
 I_3&=m^{-1}\{C_3(1+I_1)^3+3C_2(1+I_1)I_2\},\\
 I_4&=m^{-1}\{C_4(1+I_1)^4+6C_3(1+I_1)^2I_2
       +3C_2I_2^2+4C_2(1+I_1)I_3\}.
\end{aligned}
\label{eq:legacy-7-29c}
\end{equation}
Finally \eqref{eq:legacy-7-29b}, multiplied by the clock normal, is the physical entry-
section determinant.  The two section-normal margins and \(J\) give the
two-sided phase bound.
\end{proof}

\subsection{The invariant-center anchor and its exact transition}

For \(0<t<t_0<1/2\) and \(\epsilon=\pm1\), take the parameters on the
invariant-center component
\begin{equation}
 B=\frac{t^2}{1+t^2},\quad m=-B,
 \quad a=\frac{2\epsilon t}{1+t^2},
 \quad c=d=-a.
\label{eq:legacy-7-30}
\end{equation}
There are two actual simple saddle-nodes
\begin{equation}
 S_h=(-\epsilon/t,-1),\qquad S_v=\{r=0,w=\epsilon t\},
\label{eq:legacy-7-31}
\end{equation}
with opposite transverse signs.  Their strong separatrices lie on
\begin{equation}
 L_\epsilon=tx-\epsilon t^2z+\epsilon=0.
\label{eq:legacy-7-32}
\end{equation}
The component on the side opposite the finite center is a complete pp
continuum; its other boundary passes through the single principal endpoint
\(p_{-\epsilon}\).  Its boundary ratio is \(t^2\) or \(t^{-2}\), hence is
outside \([1/2,2]\) after shrinking \(t_0\).  These assertions follow by
substitution in the physical H14 field and by the first-hit argument of
Proposition~\ref{prop:complete-physical-lips}; no zero count is claimed yet.

Choose analytic strong normalizing coordinates \(q_h,q_v\) on fixed weak
sections at the two saddle-nodes.  Along the exact invariant-center pp
component the
saddle-node first
integrals give
\begin{equation}
 R_t(q_h)=K_tq_h^{t^2},\qquad K_t>0.
\label{eq:legacy-7-33}
\end{equation}
At a reference phase \(q_0\), normalize the value and first derivative:
\begin{equation}
 \frac{R_t(q_0(1+u))-R_t(q_0)}{ q_0R_t'(q_0)}
 =\frac{(1+u)^{t^2}-1}{ t^2},
\label{eq:legacy-7-34}
\end{equation}
with removable value \(\log(1+u)\) at \(t=0\).  Uniformly for
\(0\le t\le1/2\) and \(|u|\le1/2\),
\begin{equation}
 \frac{d^2}{ du^2}\frac{(1+u)^{t^2}-1}{ t^2}
 =(t^2-1)(1+u)^{t^2-2}\le-\frac{1}{3}.
\label{eq:legacy-7-35}
\end{equation}
Thus the exact model affine comparison has at most two zeros on each
normalizing interval.  The rest of the section establishes that the relevant
concavity survives on neighboring physical itineraries.

\subsection{Buffered finite phase and the resolved connector}

First use relative parameters
\begin{equation}
 B=t^2b,\qquad m=t^2M,\qquad
 a=tA,\qquad c=tC,\qquad d=tD,
\label{eq:legacy-7-36}
\end{equation}
with normalized coefficients in a fixed compact neighborhood of either
signed invariant-center branch.  On the finite-phase region
\(0<R=r/t^2\le R_L\), require the same physical first-hit itinerary on the larger
buffer \(R\le R_L^+\).  In coherent strong coordinates the transition
factors as
\begin{equation}
 {\cal T}(q)=H\left(t^2\log\frac{q_*}{ q}\right),
 \qquad \left|\frac{H''}{ H'}\right|\le K_L.
\label{eq:legacy-7-37}
\end{equation}
The map \(H\) is the composition of the residual physical regular factors,
not an unspecified germ.  On the doubled QBF tube every such factor has a
scalar clock with denominator and entry/target section normals separated
from zero.  All coefficient words through order four are bounded on fixed
clock intervals.  Lemma~\ref{lem:scalar-first-hit-jets} and the finite
Fa\`a di Bruno composition recurrence therefore bound \(H,H^{-1}\) in
\(C^4\) and bound \(H'\) away from zero.  This proves the displayed
\(H''/H'\) estimate from the same physical buffers used to define the itinerary.
Consequently, for small \(t\),
\begin{equation}
 \frac{q{\cal T}''(q)}{{\cal T}'(q)}
 =-1-t^2\frac{H''}{ H'}\le-\frac{1}{2}.
\label{eq:legacy-7-38}
\end{equation}
Two applications of multiplicity Rolle show that an affine comparison has
at most two isolated zeros on each QBF component.

Large \(R\) is not sent to the source theorem.  On every buffered QHH
through component the physical connector is decomposed, in the order met by
the orbit, as
\begin{equation}
 G=P_4\circ P_3\circ P_2\circ P_1\circ P_0.
\label{eq:legacy-7-39}
\end{equation}
Here \(P_0\) is the horizontal moving-entry bridge, \(P_1\) the fixed
reciprocal flight, \(P_2\) the complete long weighted flight, \(P_3\) the
fixed core flight, and \(P_4\) the moving-entry vertical tail.  Moving
sections are part of these definitions.  In the signed coordinate
\(\xi=-\epsilon tx\), with \(s=t^2z\), the long factor is resolved by
\(h=s-t^2=\rho^2\) and \(t=\rho T\).  Put
\begin{equation}
 \widehat A=-\epsilon A,\qquad
 \widehat C=-\epsilon C,\qquad
 \widehat D=-\epsilon D.
\label{eq:legacy-7-39a0}
\end{equation}
Its exact weighted field is
\begin{equation}
\begin{aligned}
 \xi'&=\rho^2B_w,& \rho'&=\frac{\rho}{2}A_w,&
 T'&=-\frac{T}{2}A_w,\\
 B_w&=-1+T^2\xi(b\xi+\widehat A)+M\rho^2,&
 A_w&=\xi(1+T^2)+\rho^2(T^2\widehat C+\widehat D).
\end{aligned}
\label{eq:legacy-7-39a}
\end{equation}
The resolved norm of a section map is
\begin{equation}
 \|P\|_{C^4_{\rm res}}=
 \max_{a+|\nu|\le4}\sup
 |\partial_x^a{\mathcal D}^{\nu}P|,\qquad
 {\mathcal D}\in\{\mathcal E,\partial_b,\partial_M,
 \partial_{\widehat A},\partial_{\widehat C},
 \partial_{\widehat D}\},
\label{eq:legacy-7-39b}
\end{equation}
where \(\mathcal E\) is the Euler lift on the current section; ordinary
unweighted \(\partial_t\) is not in this norm.  On the long chart the
commuting interior frame is
\begin{equation}
 \partial_\xi,\qquad
 {\cal V}_2=\rho\partial_\rho-T\partial_T,\qquad
 {\cal E}_2=T\partial_T|_\rho,\qquad
 \partial_{b,M,\widehat A,\widehat C,\widehat D}.
\label{eq:legacy-7-39c}
\end{equation}
At its moving endpoints the exact lifts are
\begin{equation}
 {\cal E}_-={\cal E}_2+{\cal V}_2,
 \qquad {\cal E}_+={\cal E}_2-T_+^2{\cal V}_2,
\label{eq:legacy-7-40}
\end{equation}
and every resolved derivative word through order four obeys
\begin{equation}
 \left|{\cal W}\left(2\rho\frac{B_w}{ A_w}\right)\right|
 \le C_{\cal W}\rho,\qquad
 \int_{\rho_-}^{\rho_+}C_{\cal W}\rho\,d\rho
 \le \frac{C_{\cal W}s_-}{2}.
\label{eq:legacy-7-41}
\end{equation}
This integrable estimate, rather than compact convergence in unscaled
coordinates, controls the long flight and all moving-endpoint terms.

The other four factors have the following exact clocks.  Put
\(\sigma=-\epsilon\).  The horizontal bridge uses \(z\) as clock:
\begin{equation}
\begin{aligned}
 \dot R&=-\sigma t^2R
 \{b+\sigma AR+R^2[1-z+t^2M(z-1)^2]\}=:N_0,\\
 \dot z&=\sigma z+t^2R\{D(z-1)^2+C(z-1)\}=:Q_0,
 \qquad \frac{dR}{ dz}=\frac{N_0}{ Q_0}.
\end{aligned}
\label{eq:legacy-7-41a}
\end{equation}
It starts on the moving strong-section curve
\((R_0(p),z_0(p))\), ends at fixed \(z=z_a\), and outputs \(R^{-1}\).
On the doubled bridge \(R_-\le R\le R_+\), \(|Q_0|\ge d_0\), and
\begin{equation}
 \chi_0=R_{0,p}-\frac{N_0}{ Q_0}z_{0,p}
\label{eq:legacy-7-41b}
\end{equation}
is the physical entry-section determinant divided by \(Q_0\).  Hence
\(0<\chi_{0,-}\le|\chi_0|\le\chi_{0,+}\), and every resolved coefficient
word in \(N_0/Q_0\) has an \(L^1(dz)\) majorant \(C_jt^2\).

On fixed reciprocal levels \(v_a\le v\le v_b\),
\begin{equation}
 \frac{d\xi}{ dv}=t^2\frac{-v+\xi(b\xi+\widehat A)+t^2Mv^2
  }{ \xi(1+v)+t^2(\widehat Cv+\widehat Dv^2)},
 \qquad Q_1\ge d_1>0.
\label{eq:legacy-7-41c}
\end{equation}
Both endpoints are fixed and every coefficient word is
\(L^1(dv)\)-bounded by \(C_jt^2\).  For the compact core put
\(U=\xi+s-1\).  Its exact clock is
\begin{equation}
 \frac{dU}{ ds}=\frac{N_3}{ Q_3},\qquad s_-\le s\le s_+,\qquad Q_3\ge c_s>0,
\label{eq:legacy-7-41d}
\end{equation}
where
\begin{equation}
\begin{aligned}
 N_3={}&sU+(M+\widehat D-1)s^2\\
 &+t^2\{1+b(U-s+1)^2+\widehat A(U-s+1)
       +(\widehat C-2\widehat D-2M)s\}\\
 &+t^4(M+\widehat D-\widehat C),\\
 Q_3={}&s\{U+1+(\widehat D-1)s\}
       +t^2(\widehat C-2\widehat D)s
       +t^4(\widehat D-\widehat C).
\end{aligned}
\label{eq:legacy-7-41e}
\end{equation}
This fixed clock interval gives finite \(L^1(ds)\) bounds for every
fourth-order quotient word.

At \(s=s_+\), the exact overlap with the vertical box is
\begin{equation}
 R_{v,0}=(s_+-t^2)^{-1},\qquad
 W_0=-\epsilon(U-s_++1)(s_+-t^2)^{-1}.
\label{eq:legacy-7-41f}
\end{equation}
In the direct vertical coordinates supplied by the graph transform, write
\(u_0(U)\), \(\ell_0(U)=\log|v_0(U)|\).  The final clock and target are
\begin{equation}
 \frac{du}{ d\ell}=F_v(u),\qquad
 \ell_0(U)\longrightarrow\ell_*:=\log v_v^*.
\label{eq:legacy-7-41g}
\end{equation}
The entry bracket
\(\chi_4=u_{0,U}-F_v(u_0)\ell_{0,U}\) is the transported physical
\(s=s_+\) section-normal determinant.  The vertical coordinate Jacobian,
the normal eigenvalue, \(Q_3\), and \(|v_0|\) all have fixed two-sided
margins, and therefore
\begin{equation}
 0<\chi_{4,-}\le|\chi_4|\le\chi_{4,+}.
\label{eq:legacy-7-41h}
\end{equation}
All coefficient words in \(F_v\) have finite \(L^1(d\ell)\) bounds on the
fixed vertical clock interval.

For \(P_2\), \eqref{eq:legacy-7-41} is the corresponding \(L^1\) estimate in logarithmic
clock \(d\log\rho\), with profile \(C_j\rho^2\).  Thus the five profiles,
in their physical order, are
\begin{equation}
 C_jt^2\,dz,\quad C_jt^2\,dv,\quad
 C_j\rho^2\,d\log\rho,\quad C_j\,ds,\quad C_j\,d\ell .
\label{eq:legacy-7-41i}
\end{equation}
They are uniformly integrable.  Equations \eqref{eq:legacy-7-40}, \eqref{eq:legacy-7-41b}, and
\eqref{eq:legacy-7-41f}--\eqref{eq:legacy-7-41h} include every moving endpoint and phase-entry term.
Reorient the intermediate local phase coordinates coherently, flipping a
shared source/target pair whenever necessary, so that physical order gives
\(P_j'>0\) for all five factors.  These flips do not change a first-hit
itinerary or an absolute section-normal margin.  The absolute derivative bounds
above may therefore be multiplied with one common sign.
Applying Lemma~\ref{lem:scalar-first-hit-jets} to the five rows, composing
by Fa\`a di Bruno, and using \eqref{eq:legacy-7-29c} now gives
\begin{equation}
 \|G\|_{C^4_{\rm res}}+\|G^{-1}\|_{C^4_{\rm res}}\le C_G,
 \qquad 0<g_0\le G'\le g_1.
\label{eq:legacy-7-42}
\end{equation}
These are lifted Euler bounds through \(t=0\); no ordinary unweighted
\(\partial_t^j\) estimate is asserted.  The exact local clocks can be chosen
as
\begin{equation}
 \dot p=\varsigma_ht^2F_h(p),\quad \dot q=\varsigma_hq,
 \qquad
 \dot u=\varsigma_vF_v(u),\quad \dot v=\varsigma_vv,
\label{eq:legacy-7-43}
\end{equation}
with
\begin{equation}
 F_v''\ge c_v,\qquad |F_v'(u)|\le L_v(|u|+\eta),
 \qquad u=G(p).
\label{eq:legacy-7-44}
\end{equation}

Let \({\cal T}(Q_h(p))=Q_v(G(p))\).  The curvature numerator is
\begin{equation}
\begin{split}
 {\cal K}={}&t^2F_hG'-F_v(G)
 +t^2F_hF_v(G)\frac{G''}{ G'}\\
 &-t^2F_hF_v'(G)G'+t^2F_v(G)F_h'.
\end{split}
\label{eq:legacy-7-45}
\end{equation}
It satisfies
\begin{equation}
 \frac{d}{ dp}\log\frac{d{\cal T}}{ dQ_h}
 =\frac{\cal K}{ t^2F_hF_v(G)},
\quad
 {\cal K}''=-F_v''(G)(G')^2-F_v'(G)G''+t^2{\cal R}_K,
 \quad |{\cal R}_K|\le C_K.
\label{eq:legacy-7-46}
\end{equation}
Choose the inner vertical interval first, then the parameter radius, and
finally \(t_0\), so that
\begin{equation}
 |F_v'(G)G''|\le\frac{c_vg_0^2}{4},\qquad
 t_0^2C_K\le\frac{c_vg_0^2}{4}.
\label{eq:legacy-7-47}
\end{equation}
Equations \eqref{eq:legacy-7-44}--\eqref{eq:legacy-7-47} imply
\begin{equation}
 {\cal K}''\le-\frac{1}{2}c_vg_0^2<0.
\label{eq:legacy-7-48}
\end{equation}
Hence \({\cal K}\), and therefore \({\cal T}''\), has at most two zeros
with multiplicity.  By
Lemma~\ref{lem:part-iii-physical-affine-closing}, the middle affine comparison
is the physical full-return closing equation
\begin{equation}
 A_0{\cal T}(q)-B_0q-C_0=0.
\label{eq:legacy-7-48a}
\end{equation}
If \(A_0\ne0\), two more Rolle steps give at
most four zeros.  If \(A_0=0\) and \(B_0\ne0\), there is at most one; if
\(A_0=B_0=0\) and \(C_0\ne0\), there is none.  The zero triple is an
identity apex with no isolated zeros.  A nonzero projective direction cannot
give an identity, since with \(A_0\ne0\) that would make \({\cal T}\)
affine, contrary to \eqref{eq:legacy-7-48}.  The horizontal and vertical weak brackets are
prepared quadratics on the fixed signed cell.  Each contributes at most two
divider roots.  Four ordered divider points, together with the two side
endpoints, cut the input interval into at most five nonempty intervals;
scalar uniqueness makes the preimage of every through interval an interval.
Hence there are at most five through components, so \(20\) is a valid QHH
aggregate bound per signed local itinerary; the number four is per component.

The overlap is assigned once.  With \(R_L=2R_*\) and \(R_L^+=3R_*\),
\begin{equation}
 \text{QBF: }0<R\le\frac{3R_*}{2}\text{, including equality};
 \qquad
 \text{QHH: }R>\frac{3R_*}{2}.
\label{eq:legacy-7-49}
\end{equation}

\subsection{The full fixed-gamma middle regime}

The theorem below uses
\begin{equation}
 B=t^2b,\qquad m=t^2M,\qquad a=tA,\qquad d=tD,\qquad c=\gamma,
\label{eq:legacy-7-50}
\end{equation}
where \((b,M,A,D)\) stays in a fixed compact signed QL neighborhood and
\(|\gamma|\le\gamma_0\).  There is no bound on \(\gamma/t\).  One cannot
deduce this case by putting \(e=t^2\) in \eqref{eq:legacy-7-36}, since
\(t\gamma=\gamma\sqrt e\) is not \(C^1\) in \(e\) at zero.  Instead the
oriented horizontal field, with all outer factors retained, is
\begin{equation}
 R'=t^2f(R,N,t,\beta_\gamma),\qquad
 N'=\sigma N+t\gamma_sR(N-R)+t^2G_0(R,N,t,\beta_\gamma),
\label{eq:legacy-7-50a}
\end{equation}
where \(G_0\) is polynomial with bounded derivatives and \(\gamma_s\) is
the fixed signed copy of \(\gamma\).  Writing \(N=t\eta\) gives the regular
invariance equation
\begin{equation}
 \eta+t^2f(R,t\eta,t,\beta_\gamma)\eta_R
 =\gamma_sR(t\eta-R)+tG_0(R,t\eta,t,\beta_\gamma).
\label{eq:legacy-7-50b}
\end{equation}
The time-one graph transform on a fixed doubled box has limiting weak
linear part \(e^{-1}\).  Shrink \(t_0,\gamma_0\) so its phase-jet
contraction is at most \(3/4\).  When the fixed-point equation is
differentiated with
\begin{equation}
 \partial_R,\quad t\partial_t|_\gamma,\quad
 \partial_b, \partial_M, \partial_A, \partial_D, \partial_\gamma.
\label{eq:legacy-7-51}
\end{equation}
the only highest graph jet is multiplied by \(I-D\mathfrak T\), whose
inverse has norm at most \(4\); all right-hand terms contain bounded
coefficient jets and lower graph jets.  Induction through order six gives
the required entry and normalizer jets and the physical graph
\(N=t\eta=O(t|\gamma|+t^2)\).  The vertical oriented field is likewise
\begin{equation}
\begin{aligned}
 R_v'&=-R_v\{W+D+t^2R_vW+t\gamma_sR_v\},\\
 W'&=M-DW-(1-t^2b)W^2-R_v
 +t^2R_v(AW-W^2)-t\gamma_sR_vW.
\end{aligned}
\label{eq:legacy-7-51a}
\end{equation}
Its weak polynomial at \(R_v=0\) is independent of \(\gamma\); the same
separated-normal graph, foliation, and fiber recurrences therefore retain
\eqref{eq:legacy-7-44} and the moving-entry margin \eqref{eq:legacy-7-41h}.
In the weighted long chart the denominator is
\begin{equation}
 A_w^\gamma=\xi(1+T^2)+\rho T\gamma+\rho^2D.
\label{eq:legacy-7-52}
\end{equation}
A \(\partial_\gamma\)-derivative contributes \(\rho T=t\le C\rho\), so
\eqref{eq:legacy-7-41} and hence the five-factor closure remain valid in this frame.  The
finite core differs from the relative chart by bounded \(t\gamma\)-terms.
Consequently each of the five rows \eqref{eq:legacy-7-41a}--\eqref{eq:legacy-7-41i} satisfies the
hypotheses of Lemma~\ref{lem:scalar-first-hit-jets} in the fixed-\(\gamma\)
frame; this proves \eqref{eq:legacy-7-42} without an ordinary \(e=t^2\) derivative.

Finally, the chart must contain every actual middle hh base, not merely a
formal neighborhood of the invariant-center component.  Construct first the
nearby upper-double-root
sheet, before imposing the hh landing.  Put
\begin{equation}
 B=bt^2,\qquad a=-tC,\qquad c=\tau+tC .
\label{eq:legacy-7-52a0}
\end{equation}
At \(t=0\) the three horizontal saddle-node equations are the regular IFT
system
\begin{equation}
 1+bX^2-CX=0,\qquad XZ_h-\tau=0,\qquad X(2bX-C)=0.
\label{eq:legacy-7-52a1}
\end{equation}
On the opposite-attractivity branch
\(X=b^{-1/2}\), \(C=2\sqrt b\), \(Z_h=\tau\sqrt b\), and the IFT
determinant is nonzero.  Hence \(b\) is a genuine local coordinate on this
two-gate sheet; all remaining endpoint coefficients and entries vary
analytically with \((b,t,\tau)\).

It remains to prove that every source-accumulating hh branch enters this
local sheet.  In the horizontal overlap put
\begin{equation}
 U=tx+t^2z-1+t\tau,\qquad {\cal Z}=t^2z.
\label{eq:legacy-7-52a1a}
\end{equation}
On each fixed compact \((U,{\cal Z})\)-rectangle the exact scaled field and
all derivatives needed below converge as \(t\to0\) to
\begin{equation}
 U'=U{\cal Z},\qquad {\cal Z}'={\cal Z}(U+{\cal Z}+1),
\label{eq:legacy-7-52a1b}
\end{equation}
and the horizontal strong point converges to
\((b^{-1/2}-1,0)\).  At the upper vertex, with
\(r=t^2R\), \(w=tW\), and \(\eta=W+1\), the limiting field is
\begin{equation}
 R'=-R(1+\eta),\qquad \eta'=-R-\eta^2,
\label{eq:legacy-7-52a1c}
\end{equation}
whose strong graph is \(\eta=R\).  On the common fixed section \(R=R_0\),
the overlap has \({\cal Z}=t^2+R_0^{-1}\), and its vertical strong landing
converges to \(U=0\).

Choose a fixed first-exit rectangle for \eqref{eq:legacy-7-52a1b} with vertical sides
\(U=\pm K\), then shrink the physical doubled tubes so that the same sides
remain transverse for the exact fields.  A retained hh branch cannot leave
through either side and later re-enter without first taking the named
preceding-side stopping face.  Hence every source-accumulating retained hh sequence has
a compact limiting branch ending on \(U=0\).  Since \(U=0\) is invariant for
\eqref{eq:legacy-7-52a1b}, uniqueness forces that whole limiting branch to be this axis.
Its horizontal endpoint therefore satisfies
\(b^{-1/2}-1=0\), so \(b\to1\).  Thus all such hh branches eventually lie
inside the common IFT neighborhood above; the local landing computation below
is exhaustive, not merely local near a selected base.

We now display the finite-\(t\) system behind that IFT.  On the selected upper
double-root sheet,
\begin{equation}
 m=-t^2(1-bt^2),\qquad d=2t(1-bt^2).
\label{eq:legacy-7-52a1d}
\end{equation}
In the positive horizontal endpoint chart \eqref{eq:legacy-3-6}, put
\(r=t/X\), \(z=tZ\), and set
\begin{equation}
\begin{aligned}
 K_t(Z)&=1-tZ-t^2(1-bt^2)(tZ-1)^2,\\
 F_1&=bX^2-CX+K_t(Z),\\
 F_2&=XZ+2t(1-bt^2)(tZ-1)^2+(\tau+tC)(tZ-1),\\
 F_3&=\det D_{(X,Z)}(F_1,F_2).
\end{aligned}
\label{eq:legacy-7-52a1e}
\end{equation}
The exact horizontal saddle-node equations are
\(F_1=F_2=F_3=0\).  Indeed, \(F_1=F_2=0\) are exactly the two equilibrium
equations after multiplication by nonvanishing factors, and the determinant
condition is preserved by the analytic change \((r,z)=(t/X,tZ)\).  At
\(t=0\), \eqref{eq:legacy-7-52a1e} is precisely \eqref{eq:legacy-7-52a1}.  Thus the finite-\(t\) equations,
not only their limiting system, supply the analytic three-variable IFT.

Put \(D_t=1+t^2\).  At \(b_0=D_t^{-1}\), direct differentiation of
\eqref{eq:legacy-7-52a1e}, with \(x_h=r_h^{-1}\), gives
\begin{equation}
 \frac{d}{ dB}(x_h+tz_h)=-\frac{D_t^3}{2t^3},
 \qquad \partial_bU(p_h)=-\frac{D_t^3}{2}.
\label{eq:legacy-7-52a1f}
\end{equation}
This is the derivative of the moving physical strong endpoint.  Along the
exact invariant line \(U=0\), the connector equations reduce to
\begin{equation}
 {\cal Z}'=\frac{1-t^2}{ D_t}(1+{\cal Z})({\cal Z}-{\cal Z}_h),
 \qquad
 \partial_UU'=\frac{1-t^2}{ D_t}({\cal Z}-{\cal Z}_h).
\label{eq:legacy-7-52a1g}
\end{equation}
Along the moving two-gate sheet \(C=C(b,t,\tau)\), the total parameter
derivative is
\(\mathcal B_b=\partial_b+(\partial_bC)\partial_C\).  Before imposing the
hh value, direct substitution in the exact scaled field gives, after removal
of its nonvanishing orbital factor,
\begin{equation}
 U'\big|_{U=0}
 =-t(-2{\cal Z}+1+t^2-t\tau)\,\Xi,
 \qquad \Xi=C+b(t^2+t\tau-1)-1.
\label{eq:legacy-7-52a1g0}
\end{equation}
At \(b_0\), the endpoint IFT values displayed in \eqref{eq:legacy-7-52a1j} give
\[
 \Xi=0,\qquad
 \mathcal B_b\Xi
 =(1-t^2-t\tau)+(t^2+t\tau-1)=0.
\]
The derivative of the omitted orbital factor is multiplied by \(\Xi=0\).
Thus it is this explicit total cancellation, not a fixed-\(C\) invariance
claim, that makes the forcing vanish:
\begin{equation}
 \mathcal B_b U'\big|_{U=0,b=b_0}=0.
\label{eq:legacy-7-52a1g1}
\end{equation}
Thus the normal variation \(u=\mathcal B_b U\), including the
\(\partial_bC\) contribution, satisfies the homogeneous equation
\begin{equation}
 \frac{du}{ d{\cal Z}}=\frac{u}{1+{\cal Z}},
 \qquad
 u({\cal Z})=u({\cal Z}_0)\frac{1+{\cal Z}}{1+{\cal Z}_0}.
\label{eq:legacy-7-52a1h}
\end{equation}

For the vertical branch write its compactified strong graph as
\(\eta=RH_v(R;b,t,\tau)\).  Its exact endpoint slope is
\begin{equation}
 h_v:=H_v(0;b,t,\tau)
 =\frac{2Ct^2-t^2+t\tau-1}{2bt^2-1}.
\label{eq:legacy-7-52a1i}
\end{equation}
At \(b=b_0\), the same IFT gives
\begin{equation}
 C_0=\frac{2-t\tau}{ D_t},\qquad
 \partial_bC=1-t^2-t\tau,\qquad
 h_v=1-t\tau-t^2,\qquad \partial_bh_v=0.
\label{eq:legacy-7-52a1j}
\end{equation}
The overlap identities
\({\cal Z}=t^2+R^{-1}\) and
\(U=\eta/R-1+t\tau+t^2\) show that the vertical normal variation has finite
endpoint value zero.  Every nonzero solution of \eqref{eq:legacy-7-52a1h} is a nonzero
multiple of \(1+{\cal Z}\) and is unbounded as \(R\downarrow0\); therefore
the admissible vertical variation is identically zero.

Let \({\cal Z}_h=t^2z_h\) be the scaled horizontal strong-entry level, and
fix the vertical comparison section \(R=R_0\), so that
\(M_t=t^2+R_0^{-1}\).  Denote by
\(\Delta(b,t,\tau)\) the horizontal strong landing minus the vertical
strong landing on this same physical section.  At
\(b_0=(1+t^2)^{-1}\), the line
\begin{equation}
 x+tz-t^{-1}+\tau=0
\label{eq:legacy-7-52a2}
\end{equation}
is invariant and \(\Delta=0\).  The finite-\(t\) endpoint calculation
\eqref{eq:legacy-7-52a1e}--\eqref{eq:legacy-7-52a1j}, followed by transport along the invariant line, gives
\begin{equation}
\begin{aligned}
 \partial_bU_h(M_t)\big|_{b_0}
 &=-\frac{(1+t^2)^3}{2}\frac{1+M_t}{1+{\cal Z}_h},\\
 \partial_bU_v(M_t)\big|_{b_0}&=0.
\end{aligned}
\label{eq:legacy-7-52a3}
\end{equation}

Subtracting the two rows yields
\begin{equation}
 \partial_b\Delta(b_0,t,\tau)
 =-\frac{(1+t^2)^3}{2}\frac{1+M_t}{1+{\cal Z}_h}.
\label{eq:legacy-7-53}
\end{equation}
On the two fixed clock intervals the first two \(b\)-variations satisfy
\begin{equation}
 v'=\Phi_Uv+\Phi_b,\qquad
 w'=\Phi_Uw+\Phi_{UU}v^2+2\Phi_{Ub}v+\Phi_{bb}.
\label{eq:legacy-7-53a}
\end{equation}
Their entry jets and all coefficients are bounded, while the vertical graph
normal eigenvalue is separated from zero.  Gronwall and the bounded overlap
therefore give \(|\partial_b^2\Delta|\le C_\Delta\).  On the compact doubled
sections, \((1+M_t)/(1+{\cal Z}_h)\) has a fixed positive lower bound, so
\eqref{eq:legacy-7-53} gives \(-\partial_b\Delta(b_0,t,\tau)\ge c_\Delta>0\).  Choose the
common \(b\)-radius smaller than \(c_\Delta/(2C_\Delta)\).  Then
\(\partial_b\Delta<-c_\Delta/2\) throughout the box, and the mean-value
theorem proves \(\Delta=0\) if and only if \(b=b_0\).

Substitution of this unique hh value in the two-gate sheet gives the actual
base family
\begin{equation}
 B=\frac{t^2}{1+t^2},\quad m=-\frac{t^2}{1+t^2},\quad
 d=\frac{2t}{1+t^2},\quad
 a=\frac{-2t+t^2\tau}{1+t^2},\quad
 c=\gamma=\frac{2t+\tau}{1+t^2}.
\label{eq:legacy-7-53b}
\end{equation}
Thus \(\tau=a+c\) is the exact unscaled center departure.  The
fixed-\(\gamma\) chart
contains the entire source-accumulating middle base locus, including
unbounded \((a+c)/t\).

\begin{theorem}[Uniform estimate on the intermediate scale]
\label{thm:part-iii-middle-zero}
On every retained complete-lips itinerary in the middle chart
\eqref{eq:legacy-7-50}, for
\(0<t<t_0\), every actual affine pp comparison has at most two isolated
zeros on each QBF component and at most four on each QHH component.  The
estimates are uniform on the full compact normalized base and all
\(|\gamma|\le\gamma_0\), with no restriction on \(\gamma/t\).
\end{theorem}

\begin{remark}
The hypotheses comprise a complete physical lips configuration from
Proposition~\ref{prop:complete-physical-lips}, comparable horizontal and upper-gate
scales, the parameters \eqref{eq:legacy-7-50}, and doubled QBF/QHH physical
buffers.  The bounds are uniform for $0<t<t_0$ on a fixed compact
$(b,M,A,D)$-tube, for all allowed $\gamma$, and on the half-open phase split
\eqref{eq:legacy-7-49}.  On coefficient faces we use the cone \eqref{eq:legacy-7-29d}: for
$\varepsilon_c\ne0$, curvature applies when $\widehat A_0\ne0$, the case
$\widehat A_0=0$, $\widehat B_0\ne0$ is affine, and
$\widehat A_0=\widehat B_0=0$, $\widehat C_0\ne0$ has no zero.  The apex
$\varepsilon_c=0$ is an identity and contributes no isolated zero.
A collapsed through interval contributes no isolated point, while neighboring
nonempty intervals retain the same estimates.  Failure of a denominator,
root, clock, side, section normal, target-section landing normal, or first-hit
condition stops at its named boundary outcome; this is loss of the target first-hit tube,
not an hh landing split.  The outer endpoint-root merger is treated in the
next section.
\end{remark}

\begin{proof}
The QBF estimate is \eqref{eq:legacy-7-38}.  Equations \eqref{eq:legacy-7-39}--\eqref{eq:legacy-7-48} prove the QHH estimate
on every actual through component; doubled boxes ensure that the analysis is
performed before a competing stopping face is reached.  The fixed-\(\gamma\) graph and
flow recurrences \eqref{eq:legacy-7-51}--\eqref{eq:legacy-7-52} reproduce these estimates uniformly without
ordinary \(e\)-smoothness.  Equation \eqref{eq:legacy-7-53} supplies physical base
exhaustiveness, and Proposition~\ref{prop:complete-physical-lips} supplies the
complete pp strip and boundary.  The half-open assignment \eqref{eq:legacy-7-49} covers all
phases once.  Finite prepared root alternatives give finitely many components,
so the component bounds sum to a uniform itinerary bound.
\end{proof}

\section{The positive root-merger scale}
\label{sec:part-iii-root}

The final positive-scale regime is the merger of the nonpersistent horizontal
double root with the persistent \(B=a=0\) triple endpoint, while the upper
\(D\)-gate remains on its normalized scale.  The object is again a complete
physical lips itinerary, not merely a weighted polynomial chart.  We obtain an
explicit bound of \(24\) per signed angular itinerary.  The exact source and
mixed faces remain the independent estimates of
Theorems~\ref{thm:part-ii-source-zero} and~\ref{thm:part-ii-mixed-zero}.

Fix \(\sigma\in\{\pm1\}\) and write
\begin{equation}
 B=t^2\kappa^2b,\qquad m=t^2M,\qquad
 a=\sigma t\kappa A,\qquad d=\sigma tD,\qquad c=\sigma\gamma,
\label{eq:legacy-7-54}
\end{equation}
where \((b,A,M,D)\) lies in a compact neighborhood of
\((1,-2,-1,2)\),
\begin{equation}
 0<t\le t_0,\qquad 0<\kappa\le\kappa_0,\qquad
 |\gamma|\le\gamma_0.
\label{eq:legacy-7-55}
\end{equation}
No condition is imposed on \(\gamma/t\), \(t/\kappa\), or
\(\kappa/t\).  Off the double-root sheet, \(\kappa\) is the radial
variable of the weight-\((2,1)\) blow-up of \((B/t^2,a/t)\); its angular
coordinates \((b,A)\) are not simultaneously zero.  Thus all five original
parameters remain fixed along each physical orbit.

At the horizontal endpoint put
\begin{equation}
 u=\sigma t\kappa\bar R,\qquad \bar N=\bar Rz.
\label{eq:legacy-7-56}
\end{equation}
The prepared weak bracket is
\begin{equation}
 H=b+A\bar R+\bar R^2-\bar R\bar N
       +t^2M(\bar N-\bar R)^2,
\label{eq:legacy-7-57}
\end{equation}
so the exact weak factor is
\begin{equation}
 e=(t\kappa)^2>0.
\label{eq:legacy-7-58}
\end{equation}
The exact oriented endpoint field is
\begin{equation}
\begin{aligned}
 \bar R'&=-\sigma t^2\kappa^2\bar R H,\\
 \bar N'&=\sigma\{\bar N+\kappa t\gamma\bar R(\bar N-\bar R)
 +\kappa t^2D(\bar N-\bar R)^2-t^2\kappa^2\bar NH\}.
\end{aligned}
\label{eq:legacy-7-58a}
\end{equation}
Its time-one graph transform on a fixed doubled box has contraction at most
\(3/4\).  The invariant graph is
\begin{equation}
 \bar N=\kappa\eta(\bar R,t,\kappa,b,M,A,D,\gamma),
\label{eq:legacy-7-58b}
\end{equation}
and differentiating its fixed-point equation in
\begin{equation}
 \partial_{\bar R},\quad t\partial_t,\quad\kappa\partial_\kappa,\quad
 \partial_b,\partial_M,\partial_A,\partial_D,\partial_\gamma
\label{eq:legacy-7-58c}
\end{equation}
again isolates the highest graph jet under an operator with inverse norm at
most \(4\).  Thus its mixed jets through order six are bounded.  On this
graph,
\begin{equation}
 H_{\rm gr}=b+A\bar R+(1+t^2M)\bar R^2
 -\kappa(1+2t^2M)\bar R\eta+\kappa^2t^2M\eta^2,
\label{eq:legacy-7-58d}
\end{equation}
so \(\partial_{\bar R}^2(H_{\rm gr}-b-A\bar R)>c_h>0\).  The graph,
foliation, and finite fiber-linearization recurrences have denominators
\begin{equation}
 |k\varsigma_h-jeF_h'(p)|\ge\tfrac12,
 \qquad k\ge1,\quad0\le j\le6,
\label{eq:legacy-7-58e}
\end{equation}
and supply the direct weak coordinates used below, with bounded inverses.
The physical connector has five factors
\begin{equation}
 G_{\rm rt}=P_4^{\rm rt}\circ P_3^{\rm out}\circ
 P_2^{\rm rt}\circ P_1^{\rm rt}\circ P_0^{\rm rt}.
\label{eq:legacy-7-59}
\end{equation}
Here \(u=1/x=\sigma t\kappa\bar R\) and
\(\xi=\sigma tx=\sigma t/u=X/\kappa\) is the signed scaled horizontal
coordinate; \(s=t^2z\) is the physical overlap coordinate.  The first bridge uses
\(X=\kappa\xi=1/\bar R\), the second fixed reciprocal levels, and the long
overlap uses \(h=s-t^2=\rho^2\), \(t=\rho T\).
The first three factor clocks are exact.  From the moving horizontal strong
section to fixed \(z=z_a>1\),
\begin{equation}
 \frac{d\bar R}{ dz}
 =-\frac{e\bar RH}{
 z+\kappa t\gamma\bar R(z-1)+\kappa t^2D\bar R(z-1)^2},
 \qquad Q_0\ge d_0>0.
\label{eq:legacy-7-59a}
\end{equation}
The moving-entry phase bracket is again \eqref{eq:legacy-7-29b}, and after multiplication
by \(Q_0\) it is the physical section-normal determinant.  The direct
normalizer and fixed strong-section normal therefore give
\(0<m_0\le|(P_0^{\rm rt})'|\le M_0\).  On fixed reciprocal levels,
\begin{equation}
 \frac{dX}{ dv}=e\,\frac{-v+X(bX+A)+t^2Mv^2}{
 X(1+v)+\kappa t\gamma v+\kappa t^2Dv^2},
 \qquad Q_1\ge d_1>0.
\label{eq:legacy-7-59b}
\end{equation}
Both endpoints are fixed.  On the long overlap,
\begin{equation}
 \frac{dX}{ d\rho}=2\kappa^2\rho\,
 \frac{-1+T^2X(bX+A)+M\rho^2}{
 X(1+T^2)+\kappa\rho T\gamma+\kappa\rho^2D},
 \qquad A_\kappa\ge a_\kappa>0.
\label{eq:legacy-7-59c}
\end{equation}
Its endpoints and Euler lifts are
\begin{equation}
\begin{gathered}
 \rho_-=t\sqrt{v_b},\qquad \rho_+=\sqrt{s_--t^2},\\
 {\cal E}^{t}_-={\cal E}_{t,2}
   +\rho\partial_\rho-T\partial_T,\qquad
 {\cal E}^{t}_+={\cal E}_{t,2}
 -\frac{t^2}{ s_--t^2}(\rho\partial_\rho-T\partial_T).
\end{gathered}
\label{eq:legacy-7-59d}
\end{equation}
All \(\kappa\)- and \(\gamma\)-derivatives of these geometric endpoints
vanish.  Resolved coefficient words in \eqref{eq:legacy-7-59a}, \eqref{eq:legacy-7-59b}, and \eqref{eq:legacy-7-59c} have
the respective \(L^1\) profiles
\begin{equation}
 C_je\,dz,\qquad C_je\,dv,\qquad C_j\kappa^2\rho\,d\rho.
\label{eq:legacy-7-59e}
\end{equation}
Thus Lemma~\ref{lem:scalar-first-hit-jets} gives direct and inverse
\(C^4_{\rm res}\) bounds and two-sided phase margins, uniformly at every
imbalance allowed by \eqref{eq:legacy-7-55}, for the three maps
\[
 P_0^{\rm rt},\qquad P_1^{\rm rt},\qquad P_2^{\rm rt}.
\]
The genuinely phase-dependent outer transfer is written
\begin{equation}
 S=\kappa s,\qquad Y=\kappa(\xi+s),\qquad
 X=Y-S,\qquad \kappa=SK,
\label{eq:legacy-7-60}
\end{equation}
and the final factor lands in the upper variables
\begin{equation}
 R_v=(s-t^2)^{-1}=\kappa\bar R_v,\qquad
 W=\frac{\sigma\xi}{ s-t^2}.
\label{eq:legacy-7-61}
\end{equation}
These coordinates preserve the essential product \(SK\); the proof never
estimates the outer factor on an inadmissible independent \((S,K)\)-box.
More precisely, with \(X=Y-S\),
\begin{equation}
\begin{aligned}
 \widetilde P={}&X+(M+D)S
 +SK\{-1+t\gamma-2t^2(M+D)\}\\
 &+SK^2\{t^2(1+bX^2+AX)+t^4(M+D)-t^3\gamma\},\\
 \widetilde Q={}&X+DS+SK(t\gamma-2t^2D)
 +SK^2(t^4D-t^3\gamma),
\end{aligned}
\label{eq:legacy-7-60a}
\end{equation}
and the exact desingularized outer field is
\begin{equation}
 Y'=S\widetilde P,\qquad S'=S\widetilde Q,\qquad
 K'=-K\widetilde Q,\qquad (SK)'=0.
\label{eq:legacy-7-60b}
\end{equation}
The doubled physical wedge is chosen so that
\(\widetilde Q\ge q_0>0\); hence \(S\) is a valid clock and the physical
quotient is
\begin{equation}
 \frac{dY}{ dS}=\frac{\widetilde P}{\widetilde Q}.
\label{eq:legacy-7-60c}
\end{equation}
Thus the outer map is a first hit on the fixed original-\(\kappa\) leaf, not
a map on an independent \((S,K)\)-rectangle.

Here is the promised first-exit and landing estimate.  The physical wedge is
\begin{equation}
 {\cal W}_3=\{0\le S\le S_*,\ 0\le K\le K_*:=s_-^{-1},\
               0\le SK=\kappa\le\kappa_0\}.
\label{eq:legacy-7-60c1}
\end{equation}
Thus every derivative descendant containing \(K\) retains either \(SK\) or
\(SK^2=(SK)K\); these products are bounded on \({\cal W}_3\).  At the
reference face \(\kappa=0\), \(M=-1\), \(D=2\), the quotient is
\begin{equation}
 \frac{dY}{ dS}=\frac{Y}{ Y+S},\qquad
 S=Y\log\frac{Y}{ X_0},\qquad 0<X_-\le X_0\le X_+,
\label{eq:legacy-7-60c2}
\end{equation}
and \(\widetilde Q=Y+S\ge X_->0\).  Choose the finite target \(S_*\) so
this compact reference family enters the interior of the fixed vertical
weak buffer, and choose a doubled tube around it with distance
\(\delta_*>0\) from every nontarget side.

On that tube the explicit polynomials \eqref{eq:legacy-7-60a} give constants \(C_*,L_*\)
such that
\begin{equation}
 \left|\frac{\widetilde P}{\widetilde Q}-\frac{Y}{ Y+S}\right|
 \le C_*\eta_*,\qquad
 \left|\partial_Y\frac{\widetilde P}{\widetilde Q}\right|\le L_*,
\label{eq:legacy-7-60c3}
\end{equation}
where \(\eta_*\) is the radius of the angular, \(t\), \(\kappa\), and
\(\gamma\) neighborhood.  Fix those radii so that
\begin{equation}
 C_*\eta_*S_*e^{L_*S_*}<\delta_*/2,\qquad
 \widetilde Q\ge X_-/2.
\label{eq:legacy-7-60c4}
\end{equation}
Gronwall then keeps every flight in the doubled wedge until its first hit of
\(S=S_*\), and its endpoint remains inside the vertical buffer.  Therefore
no nontarget side is crossed and the landing is a genuine first hit.  The
phase variation is
\(\exp\int_0^{S_*}\partial_Y(\widetilde P/\widetilde Q)\,dS\), so it has
fixed positive upper and lower bounds.  At \(S=0\) the exceptional
\(K\)-segment is the identity in \(Y\); at \(K=0\) \eqref{eq:legacy-7-60c2} is the limiting
scalar flight.  No derivative across these two faces is introduced.

At the target, the exact upper entry is
\begin{equation}
 \bar R_{v,0}=\frac{1}{ S_*-\kappa t^2},\qquad
 W_0=\frac{\sigma X_*}{ S_*-\kappa t^2}.
\label{eq:legacy-7-60c5}
\end{equation}
The outer landing derivative is positive by \eqref{eq:legacy-7-60c3}--\eqref{eq:legacy-7-60c4}, so we may use
the landing value \(X_*\) itself as the local phase coordinate.  Its target
phase tangent is therefore
\begin{equation}
 \eta_Y:=\partial_{X_*}(\bar R_{v,0},W_0)
 =\left(0,\frac{\sigma}{ S_*-\kappa t^2}\right).
\label{eq:legacy-7-60c5a}
\end{equation}
With \(D_s=\sigma D\), \(A_s=\sigma A\), and
\(\gamma_s=\sigma\gamma\), the vertical field is
\begin{equation}
\begin{aligned}
 \bar R_v'&=-\bar R_v
 \{W+D_s+\kappa t^2\bar R_vW+\kappa t\gamma_s\bar R_v\},\\
 W'&=M-D_sW-(1-t^2\kappa^2b)W^2-\kappa\bar R_v\\
 &\quad+\kappa t^2\bar R_v(\kappa A_sW-W^2)
       -\kappa t\gamma_s\bar R_vW .
\end{aligned}
\label{eq:legacy-7-60c6}
\end{equation}
At \(\kappa=0\) its weak polynomial is \(M-D_sW-W^2\), while the transverse
normal \(W_c+D_s\) is separated from zero on the \(D\)-collision cell.
The same \(3/4\) graph-transform and finite fiber recurrences therefore
give direct coordinates
\begin{equation}
 u'=\varsigma_vF_v(u),\qquad \bar v'=\varsigma_v\bar v,\qquad
 F_v''\ge c_v>0,\qquad |F_v'(u)|\le L_v(|u|+\eta_v).
\label{eq:legacy-7-60c7}
\end{equation}
If \(Q_v\) denotes the full brace in the first row of \eqref{eq:legacy-7-60c6}, the phase
tangent of \eqref{eq:legacy-7-60c5} has exact physical determinant
\begin{equation}
 \det(\eta_Y,X_v)=\frac{\sigma Q_v}{(S_*-\kappa t^2)^2}.
\label{eq:legacy-7-60c8}
\end{equation}
The direct-coordinate Jacobian, \(Q_v\), the target normal
\(\bar v_v^*\), and \(S_*-\kappa t^2\) have fixed two-sided margins.
Formula \eqref{eq:legacy-7-29b} therefore gives
\begin{equation}
 0<m_4\le |(P_4^{\rm rt})'|\le M_4,
 \qquad P_4^{\rm rt},(P_4^{\rm rt})^{-1}\in C^4_{\rm res}.
\label{eq:legacy-7-60c9}
\end{equation}

For root factors, \(C^4_{\rm res}\) means all phase derivatives and all words
of total length at most four in
\begin{equation}
 t\partial_t,\quad\kappa\partial_\kappa,\quad
 \partial_b,\partial_M,\partial_A,\partial_D,\partial_\gamma.
\label{eq:legacy-7-60d}
\end{equation}
On the long chart this is represented by
\begin{equation}
 \partial_X,\quad \rho\partial_\rho-T\partial_T,\quad
 T\partial_T|_\rho,\quad\kappa\partial_\kappa,
 \quad\partial_{b,M,A,D,\gamma},
\label{eq:legacy-7-60e}
\end{equation}
and on the outer wedge by
\begin{equation}
 \partial_Y,\quad S\partial_S-K\partial_K,\quad
 K\partial_K,\quad t\partial_t,\quad
 \partial_{b,M,A,D,\gamma}.
\label{eq:legacy-7-60f}
\end{equation}
The overlap matrices and moving-endpoint lifts are bounded on the doubled
wedges.  Every nonzero derivative descendant containing \(K\) retains the
factor \(SK=\kappa\), which is the closure needed for \eqref{eq:legacy-7-62}.

For the outer factor the exact entry and target data are
\begin{equation}
 S_{\rm in}=\kappa s_-,\quad K_{\rm in}=s_-^{-1},\quad
 Y_{\rm in}=X+\kappa s_-,\qquad
 S_{\rm out}=S_*,\quad K_{\rm out}=\frac{\kappa}{ S_*}.
\label{eq:legacy-7-61a}
\end{equation}
All resolved endpoint words are bounded.  On \eqref{eq:legacy-7-60c1}, every fourth-order
word in \(\widetilde P/\widetilde Q\) is bounded in \(L^1(dS)\) because
\(\widetilde Q\ge X_-/2\) and the only \(K\)-terms occur through the bounded
products \(SK,SK^2\).  Lemma~\ref{lem:scalar-first-hit-jets}, \eqref{eq:legacy-7-60c4},
and \eqref{eq:legacy-7-61a} therefore give direct and inverse \(C^4_{\rm res}\) bounds and
a two-sided phase margin for \(P_3^{\rm out}\), including its two limiting
faces.  Equation \eqref{eq:legacy-7-60c9} gives the same conclusion for \(P_4^{\rm rt}\).

Reorient the five local phase coordinates coherently along the retained
physical itinerary.  Flipping a shared source/target pair preserves every
absolute determinant estimate and makes each factor orientation preserving.
Thus the two-sided absolute margins above imply positive derivatives for all
five factors simultaneously, rather than five unrelated signs.

Thus all five factors satisfy the scalar first-hit lemma.  Their coefficient
profiles are \(O(e)\), \(O(e)\), \(O(\kappa^2\rho)\), \(O(1)\), and
\(O(1)\) on the five displayed finite or integrable clocks.  Fa\`a di
Bruno composition through order four and the inverse recurrence \eqref{eq:legacy-7-29c}
give
\begin{equation}
 \|G_{\rm rt}\|_{C^4_{\rm res}}
 +\|G_{\rm rt}^{-1}\|_{C^4_{\rm res}}\le C_G,
 \qquad 0<g_0\le G_{\rm rt}'\le g_1.
\label{eq:legacy-7-62}
\end{equation}
Here \(g_0=\prod_{j=0}^4m_j\) and \(g_1=\prod_{j=0}^4M_j\).
The bound is uniform on the compact angular base and under all imbalances
allowed in \eqref{eq:legacy-7-55}.  The finite symbolic calculation expands derivative words only
for the displayed polynomial numerator/denominator factors; the physical
order \eqref{eq:legacy-7-59}, first-contact margins, and through-component construction are
proved analytically here.

Let \(F_h,F_v\) be the prepared horizontal and vertical weak polynomials in
the direct coordinates, with \(F_v''\ge c_v>0\), and set \(G=G_{\rm rt}\).
The transition obeys
\begin{equation}
 \frac{Q_h'}{ Q_h}=\frac{1}{ eF_h},\qquad
 \frac{Q_v'}{ Q_v}=\frac{1}{ F_v},\qquad
 {\cal T}(Q_h(p))=Q_v(G(p)).
\label{eq:legacy-7-63}
\end{equation}
Define
\begin{equation}
\begin{split}
 {\cal K}_{\rm rt}={}&eF_hG'-F_v(G)
 +eF_hF_v(G)\frac{G''}{ G'}\\
 &-eF_hF_v'(G)G'+eF_v(G)F_h'.
\end{split}
\label{eq:legacy-7-64}
\end{equation}
Then
\begin{equation}
 \frac{d}{ dp}\log\frac{d{\cal T}}{ dQ_h}
 =\frac{{\cal K}_{\rm rt}}{eF_hF_v(G)},
\label{eq:legacy-7-65}
\end{equation}
and exact differentiation gives
\begin{equation}
 {\cal K}_{\rm rt}''
 =-F_v''(G)(G')^2-F_v'(G)G''+e{\cal R}_{\rm rt},
 \qquad |{\cal R}_{\rm rt}|\le C_K.
\label{eq:legacy-7-66}
\end{equation}
Choose the fixed vertical inner interval first so that
\(|F_v'(G)G''|\le c_vg_0^2/4\), and then shrink \(t_0,\kappa_0\) until
\((t_0\kappa_0)^2C_K\le c_vg_0^2/4\).  It follows that
\begin{equation}
 {\cal K}_{\rm rt}''\le-\frac{1}{2}c_vg_0^2<0.
\label{eq:legacy-7-67}
\end{equation}
Thus \({\cal T}''\) has at most two zeros with multiplicity.  By
Lemma~\ref{lem:part-iii-physical-affine-closing}, the affine comparison is
the physical full-return closing equation; two further Rolle steps give at most
four zeros on each through component.  Strict convexity of each prepared
horizontal and vertical weak bracket gives at most two roots.  Their inverse
images contribute at most four internal divider points.  The persistent side
contributes at most one
additional divider; scalar uniqueness preserves their order and makes each
surviving preimage an interval.  Hence at most five divider points cut a
signed angular itinerary into at most six through components.

\begin{theorem}[Uniform estimate at the positive root-merger scale]
\label{thm:part-iii-root-zero}
For every retained complete-lips physical root-scale itinerary
\eqref{eq:legacy-7-54}--\eqref{eq:legacy-7-55}, with the
doubled five-factor first-hit buffers used above, every comparison
\begin{equation}
 A_0{\cal T}(q)-B_0q-C_0=0
\label{eq:legacy-7-68}
\end{equation}
has at most four isolated zeros, counted with multiplicity, on each connected
through component.  Consequently there are at most
\begin{equation}
 6\cdot4=24
\label{eq:legacy-7-69}
\end{equation}
isolated zeros per signed angular itinerary.  The bound is ambient under positive
\(t\), positive \(\kappa\) split-root, landing-split, no-root, coefficient,
and identity specializations in the same stopped itinerary.
\end{theorem}

\begin{remark}
The bound is uniform on the compact angular base for all $0<t\le t_0$,
$0<\kappa\le\kappa_0$, and $|\gamma|\le\gamma_0$, without a hidden bound on
any ratio of these variables.  On $\varepsilon_c\ne0$ we use the projective
direction in \eqref{eq:legacy-7-29d}.  If $\widehat A_0\ne0$, \eqref{eq:legacy-7-67} gives strict curvature;
if $\widehat A_0=0$, $\widehat B_0\ne0$, \eqref{eq:legacy-7-68} has at most one isolated
zero; and if $\widehat A_0=\widehat B_0=0$,
$\widehat C_0\ne0$, it has none.  The apex $\varepsilon_c=0$ is an identity
with no isolated zeros, while \eqref{eq:legacy-7-67} excludes an identity in a nonzero
projective direction.  At the boundary, loss of a denominator, section normal, weak box,
target-section landing normal, or first hit stops at the named adjacent
regime.  A collapsed interval or changed itinerary is not continued through
\eqref{eq:legacy-7-59}.
\end{remark}

\begin{proof}
Equations \eqref{eq:legacy-7-56}--\eqref{eq:legacy-7-62} construct the five factors on the same actual
physical itinerary and provide the uniform direct and inverse phase bounds.
Equations \eqref{eq:legacy-7-63}--\eqref{eq:legacy-7-67} give strict curvature.  Multiplicity Rolle and the
six-component prepared-root count yield \eqref{eq:legacy-7-69}.  Coefficient and identity
faces are handled in the ambient comparison as stated.  Every physical loss
of the five-factor itinerary is stopped before the analytic estimate is invoked,
so no non-through component is counted as a root-scale cycle.
\end{proof}

The zero-scale theorem-validity handoff is the first two rows of
\eqref{eq:legacy-5-2}.  On the remaining positive root chart the root theorem
is valid for \(0<\kappa\le\kappa_0\), although its assigned interval is only
\(0<\kappa<\kappa_\#(\mathrm{angle})\); the outer equality and overlap are
owned by the middle theorem.
At \(\kappa=0\), \eqref{eq:legacy-7-54} gives exactly \(B=a=0\), where the persistent
triple endpoint changes the physical transition.  At \(t=0\), the return is
the matched noncompact source problem.  Neither face follows by continuity
from a positive-\(e\) curvature estimate; they are the independent mixed and
source estimates proved in Sections~\ref{sec:part-ii-matched-proof}
and~\ref{sec:part-ii-mixed}.

For clarity, the positive-scale regime partition is fixed before any zero
theorem is applied.  Use the nested coalescing tubes and the fixed value
$t_{\rm str}$ chosen in Section~\ref{sec:regime-table}.  After named first-hit
losses are removed, a
complete-lips point outside the inner tube or with $t\ge t_{\rm str}$ is in
the strict regime,
including equality.  On the pointwise region $0<t<t_{\rm str}$ use the middle variables
\(b_m=B/t^2\), \(A_m=a/t\) and the weighted radius
\begin{equation}
 \varrho_w=(b_m^2+A_m^4)^{1/4}.
\label{eq:legacy-7-70a}
\end{equation}
On a root chart \(b_m=\kappa^2b\), \(A_m=\sigma\kappa A\), hence
\begin{equation}
 \varrho_w=\kappa(b^2+A^4)^{1/4}.
\label{eq:legacy-7-70b}
\end{equation}
The angular factor is bounded above and away from zero.  Choose
\(\varrho_\#\) inside the common doubled domains.  Assign
\(0<t<t_{\rm str}\), \(\varrho_w\ge\varrho_\#\), including equality, to
the middle theorem, and assign \(0<t<t_{\rm str}\),
\(0<\varrho_w<\varrho_\#\) to the root theorem.  Equivalently, on each member of the
finite root-chart cover the root regime occupies a half-open interval
\(0<\kappa<\kappa_\#(\text{angle})\).  The root theorem itself is valid
through the outer value \(0<\kappa\le\kappa_0\), and the fixed regime cutoff
satisfies \(\kappa_\#(\text{angle})\le\kappa_0\); its equality face is used
by the earlier middle regime.  Assign \(\kappa=0\) to the mixed theorem and \(t=0\) to
the source theorem.  Opposite
projective charts terminate, by their first physical contact, in the strict,
hyperbolic, one-central, passive, or exit row.  On a common physical overlap
the earlier regime in this order contains the boundary.  Every cutoff is fixed
before a theorem neighborhood is selected.  Thus the partition is by finite
resolved faces and first-hit outcomes, not by zeros of a landing function or by the
size of a subsequently chosen closure.

\begin{theorem}[Two-central physical completion and zero-theorem interface]
\label{thm:two-central-physical-completion}
Let $K_2$ be a compact family of retained exact-once itinerary lifts carrying
both possible internal central gates.  After a finite refinement by the
signed gate, sector, and first-contact data, every lift belongs to exactly one
of the following theorem domains or to a terminal zero-contribution outcome:
the two-central no-pp domain, positive-margin strict lips, middle QBF/QHH,
positive root merger, matched source, or exact mixed persistent endpoint.
For every nonterminal owner face the following implications hold in this order:
\begin{equation}
 \begin{gathered}
 \text{complete physical first-hit geometry}\\[-2pt]
 \Downarrow\\[-2pt]
 \text{complete nonaffine PP/BP boundary}\\[-2pt]
 \Downarrow\\[-2pt]
 \text{half-open scale assignment}\\[-2pt]
 \Downarrow\\[-2pt]
 \text{the corresponding local zero estimate}.
 \end{gathered}
\label{eq:two-central-interface-chain}
\end{equation}
More precisely, at a complete-lips owner with
$\delta_{\rm hh}=0$, the first implication certifies two distinct simple
saddle-nodes of opposite attractivity, the actual hh incidence, a nonempty
complete pp component, its complete second boundary, and the absence of an
intervening physical stopping face.  A neighboring fiber with a small
nonzero landing split need not itself carry the hh connection: it stays on
the same physical product tube and enters the closing equation only through
the undivided constant coefficient.  A landing gap separated from zero has
no fixed point approaching the two-central face.  The second implication removes
the affine/collar residual, and the third uses exactly the equality ownership
in \eqref{eq:legacy-5-2}.  Consequently no retained two-central lift is left
outside the finite family of proved zero estimates.

At the exact mixed and source faces, the conclusion means the face theorem
together with the finitely many incident punctured estimates; it does not
assert that one positive-scale formula extends continuously across a
zero-scale face.  Their assembly into an open original-parameter
neighborhood is carried out later in
Proposition~\ref{prop:ambient-edge-neighborhoods}.
\end{theorem}

\begin{proof}
\medskip\noindent\emph{1. Physical completion.}\par\noindent
The finite gate count \eqref{eq:legacy-7-13},
Proposition~\ref{prop:two-central-successors}, and
Proposition~\ref{prop:complete-physical-lips} give the no-pp, complete pp,
landing-gap, and named first-contact alternatives without invoking a zero
theorem.  On a complete component, the order-band Gronwall buffers and upper
cooperative squeeze prove the actual hh/pp/PP-or-BP incidence and exclude a
hidden divider.  Lemma~\ref{lem:two-central-landing-product-tube} then gives
one common physical domain across a small landing split, while
Lemma~\ref{lem:part-iii-physical-affine-closing} shows that its only occurrence
in the affine comparison is $C_0$.  The compactness argument in Step~2 of
Proposition~\ref{prop:complete-physical-lips} excludes a landing gap bounded
away from zero.

\medskip\noindent\emph{2. Boundary nonaffinity.}\par\noindent
Only after the complete second boundary has been identified does
Corollary~\ref{cor:two-central-nonaffine-boundary} apply the DIR PP/BP
boundary result.  It gives a nonzero derivative of finite order at least two
and proves $\mathfrak A_{\rm aff}=\varnothing$.  No scale theorem or zero
count enters this step.

\medskip\noindent\emph{3. Scale assignment.}\par\noindent
Theorem~\ref{thm:part-iii-two-central-exhaustion} and the fixed cutoffs
$t=t_{\rm str}$ and $\varrho_w=\varrho_\#$ assign strict first, then middle
including its equality, then positive root, with $t=0$ owned by source and
$t>0$, $\kappa=0$ owned by mixed.  These are physical and resolved-coordinate
faces fixed before any theorem neighborhood is chosen.

\medskip\noindent\emph{4. Zero estimates.}\par\noindent
The no-pp domain is bounded by
Theorem~\ref{thm:part-iii-two-central-no-pp}; the strict, middle, and root
domains by Theorems~\ref{thm:part-iii-strict-zero},
\ref{thm:part-iii-middle-zero}, and
\ref{thm:part-iii-root-zero}; and the two zero-scale faces by the independent
Theorems~\ref{thm:part-ii-source-zero} and
\ref{thm:part-ii-mixed-zero}.  The source bound is two per normalized-action
cell, the middle theorem gives its QBF/QHH component bounds, the root bound is
$24$ per signed angular itinerary, and the mixed Dulac component has at most
one isolated orbit.  Finite signed and component refinements turn these into
a finite incident bound at every lift.  This proves the theorem in the
one-way order \eqref{eq:two-central-interface-chain}.
\end{proof}

\section{Global regime exhaustion}
\label{sec:part-iii-package}

All local zero estimates and the two-central geometric classification are now
available.  Table~\ref{tab:regimes} has one proved zero estimate for each
retained physical regime, and
Theorem~\ref{thm:two-central-physical-completion} has assembled the complete
two-central interface with no residual class.  The remaining issue is no longer another zero
estimate: every first loss must descend to an already named regime, and that
descent must terminate.  The next proposition establishes this compatibility;
Section~\ref{sec:specialization} then treats proper boundary and identity
strata, and Section~\ref{sec:part-i-assembly} takes the finite compact cover.

\begin{proposition}[Compatibility of the positive-scale estimates]
\label{prop:part-iii-handoff}
For every exact-once itinerary in a positive-scale regime, the
regime table selects exactly one analytic estimate.  The selected
theorem includes its stated coefficient, identity, and phase-boundary
specializations, and sends every remaining first loss to a strictly lower
node of the finite specialization graph.  At a proper face, that face theorem
together with its finitely many incident theorem neighborhoods supplies an
open zero-bound neighborhood in all five original parameters, as made
explicit in Proposition~\ref{prop:ambient-edge-neighborhoods}.  The only
zero-scale alternatives are the source and mixed estimates of
Theorems~\ref{thm:part-ii-source-zero} and~\ref{thm:part-ii-mixed-zero}.
\end{proposition}

\begin{proof}
For zero central blocks, Table~\ref{tab:regimes} selects the compact analytic theorem or
Theorem~\ref{thm:part-iii-hyperbolic-zero}.  For one central block, \eqref{eq:legacy-7-12}
chooses Theorem~\ref{thm:part-iii-one-central-zero}.  For two blocks,
Theorem~\ref{thm:two-central-physical-completion} first removes all geometric
ambiguity, certifies the landing coefficient on its common product tube, and
selects the no-pp, strict, middle, root, source, or mixed alternative.  The strict/middle/root domains are disjoint by the fixed cutoff
$t=t_{\rm str}$, the weighted regime
partition \eqref{eq:legacy-7-70a}--\eqref{eq:legacy-7-70b}, and the half-open
assignment rule \eqref{eq:legacy-5-2}.  Each theorem statement records its
coefficient and identity policy before a boundary descent is started.  In
particular, the exact mixed-face Dulac estimate is not asserted to be open in
$(B,a)$ by itself: its positive-root, source, and split first-contact
neighbors form the finite incident family in
Table~\ref{tab:specialization-edges}.  Every other first loss is one of the
faces already present in the stopped first-hit complex, so the minimal-face
rank decreases.  Finiteness of that rank proves termination.
This argument imports the theorems only after exact-once reduction and hence
does not feed back into construction of the physical itineraries.
\end{proof}

\begin{proposition}[Exhaustive and disjoint regime assignment]
\label{prop:regime-assignment}
Every retained itinerary of Theorem~\ref{thm:finite-words} belongs to exactly
one relative-interior row of Table~\ref{tab:regimes}.  After the fixed strict
region, including \(t=t_{\rm str}\), has been assigned, the inner
coalescing resolution \(0<t<t_{\rm str}\) uses the weighted radius
\eqref{eq:legacy-5-2a} and the half-open priority
\eqref{eq:legacy-5-2}.
Equality in the last line belongs to the middle case.  Minimal resolved faces
and these half-open rules count every physical lift once.
\end{proposition}

\begin{proof}
First separate source itineraries and bounded analytic itineraries.  Among the
remaining itineraries, Proposition~\ref{prop:finite-singular-alphabet}
distinguishes zero, one, or two internal central blocks.  The zero- and
one-central cases go to the hyperbolic or one-central row.
Theorem~\ref{thm:two-central-physical-completion} routes the two-central case
to no-pp or to a complete physical lips configuration, proves that no
residual affine class remains, and records the applicable zero theorem.
Proposition~\ref{prop:part-iii-handoff} then assigns every positive-scale lips
configuration to strict, middle, or root and assigns the only zero-scale
possibilities to the independent source and mixed theorems.

The cutoff \(t=t_{\rm str}\), the weighted radius
\(\varrho_w=\varrho_\#\), and the two zero-scale faces are owned exactly as in
\eqref{eq:legacy-5-2}; hence no diagonal sequence is omitted and no equality
is counted twice.  Terminal first-contact outcomes were removed before this
classification, while chart overlaps represent the same physical first hit.
This proves exhaustiveness and disjointness.
\end{proof}

\section{Boundary and identity strata}
\label{sec:specialization}

The preceding classification first applies to relative interiors.  We now
follow every proper specialization in the finite incidence complex.  A
strictly decreasing integer complexity proves termination, and the half-open
priority assigns a terminal analytic regime or a genuine zero-contribution
face.

For this purpose an \emph{ambient zero-bound neighborhood} of a physical lift
$\zeta$ means an open neighborhood $W_\zeta$ of its image in the original
five-parameter space, a physical phase neighborhood $J_\zeta$, and an integer
$N_\zeta<\infty$ such that the sum of the isolated zeros represented by all
minimal-face itinerary labels incident to $\zeta$ is at most $N_\zeta$ on
$J_\zeta\times W_\zeta$.  The word \emph{ambient} refers to openness in the
original parameters, not merely relative openness on an exceptional divisor
or coefficient face.  An individual face theorem need not supply this by
itself; finitely many incident theorem neighborhoods may be summed.

The regime of a relative-interior itinerary does not by itself determine every proper
face.  We therefore form a finite specialization graph.  Its vertices record
the box, divider ordering, first-hit, primitive, fixed regime-scale status,
and radial--projective coefficient labels.  A proper arrow occurs when
consecutive dividers coincide; a transverse stopping face becomes a named
singular-sector outcome; one of the fixed source/strict/middle/root scale inequalities reaches
its assigned equality face; a nonzero normalized coefficient reaches a
projective face; or the coefficient radius reaches its all-zero apex.
Overlap identifications are the same physical map and are not arrows.

Let \(n_{\rm sc}(v)\) be the number of active regime-scale inequalities at
\(v\) which have not yet been specialized to a named equality, and let
\(n_{\rm rad}(v)\in\{0,1\}\) be one off the all-zero coefficient apex and
zero at that apex.  These labels are finite because the stopped atlas and all
half-open cutoffs were fixed in advance.

For a vertex $v$, define
\begin{equation}
\begin{aligned}
 \mathfrak c(v)= {}&
 \bigl(\text{number of nonempty labelled intervals}\bigr)\\
 &+\bigl(\text{number of unspecialized stopping faces}\bigr)\\
 &+n_{\rm sc}(v)\\
 &+\bigl(\text{number of nonzero projective coefficients}\bigr)\\
 &+n_{\rm rad}(v).
\end{aligned}
\label{eq:legacy-5-4}
\end{equation}

Termination and regime assignment use different data.  The integer $\mathfrak c$
proves only that specialization cannot continue indefinitely.  The terminal regime is
fixed independently by the following canonical priority.  A physical lift is
first placed on the unique minimal resolved face containing it.  On that face,
its first stopped event precedes every remote description; a collapse face
precedes its parent itinerary; source-threshold equality uses the post-threshold
primitive vector and the explicit compact/hyperbolic/central/mixed/terminal
reclassification above; and a persistent $B=0$ endpoint uses the mixed row
before a generic hyperbolic or one-central closure row.  Certified lips faces use the
strict/middle/root/source/mixed half-open rules, including \eqref{eq:legacy-5-2}, and the
internal QBF/QHH boundary belongs to the fixed earlier side.  Coefficient
faces use their unique minimal projective face, while an identity is retained
under the ambient theorem that controls its neighboring fibers.  Chart
overlaps identify the same physical lift and do not enter this priority.

For a physical lift \(\zeta\), package these order-independent data into its
canonical signature
\begin{equation}
 \mathfrak s(\zeta)=
 \bigl(F_{\min}(\zeta),E_{\rm first}(\zeta),
       S_{\rm ho}(\zeta),R_{\rm coef}(\zeta),
       P_{\min}(\zeta)\bigr).
\label{eq:specialization-signature}
\end{equation}
Here \(F_{\min}\) is the minimal resolved face, \(E_{\rm first}\) the first
physical stopped event, \(S_{\rm ho}\) the source, strict, middle, root, or
mixed half-open scale status (or \emph{not applicable}), \(R_{\rm coef}\) records
whether the coefficient radius is zero, and \(P_{\min}\) is the minimal
projective coefficient face when that radius is nonzero.  The canonical
priority above defines the single-valued map
\begin{equation}
 \operatorname{Assign}\colon
 \{\text{canonical signatures}\}\longrightarrow
 \{\text{theorem regimes or terminal labels}\}.
\label{eq:specialization-assignment}
\end{equation}

The owner of a boundary point and the estimate controlling its neighborhood
are deliberately separate.  The owner column below prevents double counting;
the ambient-control column may sum estimates from several incident sides.
Thus, for example, the mixed face is owned by the mixed theorem although an
open neighborhood of that face also uses the uniform positive-root estimate.

\begin{lemma}[First changed specialization slot]
\label{lem:first-specialization-slot}
Let a sequence in one relative-interior vertex converge to a proper physical
lift.  After passage to a subsequence, its first changed datum is a divider or
stopping-face event, a fixed scale equality, a projective coefficient face,
or the radial coefficient apex.  If several changes occur simultaneously,
their common minimal face carries the combined canonical signature
\eqref{eq:specialization-signature}.  There is no additional limiting datum.
\end{lemma}

\begin{proof}
The face list, divider orderings, scale cutoffs, and coefficient charts are
finite.  Bounded flight-time limits reach a member of $\mathcal F$, while an
unbounded flight enters one of the finite singular sectors by
Proposition~\ref{prop:stopped-first-hit}.  All remaining data are precisely
the three finite coordinate statuses in
\eqref{eq:specialization-signature}.  Definition~\ref{def:first-contact-owner}
assigns simultaneous changes to their common minimal face.
\end{proof}

Table~\ref{tab:specialization-edges} is therefore the complete finite edge
list.  Chart overlap is included only to record that it is an identification,
not an arrow.

\begin{table}[p]
\centering
\footnotesize
\setlength{\tabcolsep}{3pt}
\caption{Specialization edges, physical owners, and ambient zero-bound neighborhoods.}
\label{tab:specialization-edges}
\vspace{3pt}
\begin{tabular}{@{}
 >{\raggedright\arraybackslash}p{.18\textwidth}
 >{\raggedright\arraybackslash}p{.25\textwidth}
 >{\raggedright\arraybackslash}p{.23\textwidth}
 >{\raggedright\arraybackslash}p{.27\textwidth}@{}}
\toprule
Boundary trigger&First physical event and target&Physical owner at the face&Ambient control in the original parameters\\
\midrule
Resolved chart overlap&The same orbit reaches the same member of $\mathcal F$;&
physical identification by \eqref{eq:legacy-4-1}&The same theorem
representative is used on both charts; the lift is counted once\\
Consecutive dividers coincide&The intervening section interval collapses; a
regular boundary point uses its common Poincar\'e germ&The minimal
corner/divider face of Definition~\ref{def:first-contact-owner}&The empty
interval contributes zero; finitely many incident half-open labels are
restrictions of one ambient germ by Proposition~\ref{prop:exact-once}\\
Loss of a denominator, clock, section normal, target landing normal, or
through condition&The orbit reaches its first named root, gate, previous side,
node, stable-center side, wrong-orientation sector, or collar face&The
corresponding member of $\mathcal A_{\rm stop}$; a surviving passage is
reclassified by Table~\ref{tab:regimes}&Propositions~\ref{prop:fixed-skeleton}
and~\ref{prop:stopped-first-hit}, followed by the target-row theorem; terminal
faces contribute zero\\
First effective source threshold becomes an equality&The post-threshold
primitive is evaluated before any remote landing description&Source,
compact, hyperbolic, one-central, mixed, or a terminal first-contact label&The
source bound, the appropriate target-row bound, and the finite first-loss
cover in the common physical tube\\
One or two central passages split, or the pp condition changes&The first gate
split or the actual absence/presence of a complete pp strip is recorded&
Hyperbolic, one-central, two-central no-pp, or complete-lips routing&
Theorems~\ref{thm:part-iii-hyperbolic-zero},
\ref{thm:part-iii-one-central-zero},
\ref{thm:part-iii-two-central-no-pp}, and
\ref{thm:part-iii-two-central-exhaustion} on their certified physical sides\\
$t$ reaches $t_{\rm str}$ with positive lips margins&No new contact; this is
the fixed outer boundary of the coalescing tube&Strict, including equality&
Theorem~\ref{thm:part-iii-strict-zero}; loss of a physical margin is owned by
one of the preceding first-contact rows\\
$0<t<t_{\rm str}$ and $\varrho_w$ reaches $\varrho_\#$&No new contact; the
doubled middle/root tubes represent the same itinerary&Middle, including
equality&Theorem~\ref{thm:part-iii-middle-zero}; the root theorem is valid on
the common doubled overlap\\
$t>0$ and $\kappa\downarrow0$&The endpoint root becomes the persistent
$B=a=0$ transition&Exact mixed&Theorem~\ref{thm:part-ii-mixed-zero} on the
face, Theorem~\ref{thm:part-iii-root-zero} uniformly for
$0<\kappa\le\kappa_0$, and the finitely many split/no-passage first outcomes\\
$t\downarrow0$&The coalescing itinerary becomes the matched noncompact source
return&Source&Theorem~\ref{thm:part-ii-source-zero} on the face together with
the uniform positive-$t$ middle, root, mixed, and first-contact bounds on the
finite incident cover\\
A nonzero projective coefficient vanishes while the coefficient radius stays
positive&The same physical return equation reaches its minimal projective
face&That coefficient face, not an identity unless the full displacement
vanishes&The ambient coefficient clauses of the compact, QRH, strict, middle,
root, source, or mixed owner; no division is made before the face is assigned\\
The coefficient radius reaches zero&The unnormalized displacement is read at
its radial apex&Identity if it vanishes identically; otherwise the minimal
analytic coefficient face&The apex has no isolated member; compactness of the
projective direction set gives finitely many punctured ambient theorem
neighborhoods controlling nearby fibers\\
A phase endpoint is reached&The actual hh, PP/BP, lower-gate, separatrix, or
collar first contact is used&The incident theorem endpoint or a terminal
label&The endpoint clauses of the source, strict, middle, root, hyperbolic,
and central theorems; at a regular divider use the common germ of
Proposition~\ref{prop:exact-once}\\
\bottomrule
\end{tabular}
\end{table}

\begin{proposition}[Ambient neighborhoods at every specialization edge]
\label{prop:ambient-edge-neighborhoods}
Every physical lift in the closed resolved carrier has an ambient zero-bound
neighborhood in the sense above.  On a relative interior it is supplied by
the theorem in Table~\ref{tab:regimes}.  On a proper face it is supplied by
the finite family of incident estimates in the corresponding row of
Table~\ref{tab:specialization-edges}.  In particular, the mixed face, the
source face, and the coefficient apex are not treated by continuity from one
punctured formula: their own face conclusions and all incident punctured
estimates are used together.
\end{proposition}

\begin{proof}
Interior points are exactly the hypotheses of the applicable local zero
theorem.  At a first-contact or divider edge, the finite face list
$\mathcal F$, the finite outcome alphabet
\eqref{eq:stopped-outcome-alphabet}, and the ambient-germ clause of
Proposition~\ref{prop:exact-once} leave only the finitely many incident labels
listed in the first five rows of Table~\ref{tab:specialization-edges}.  Their
target theorems have uniform constants up to the stated first loss; summing
those constants gives the required $N_\zeta$.

At $\varrho_w=\varrho_\#$ the doubled middle/root domains overlap, and at
$\kappa=0$ the root bound is uniform for every positive $\kappa$ while the
mixed theorem controls the equality face.  At $t=0$ the source theorem
controls the equality face and the positive-$t$ middle/root estimates are
uniform on their compact normalized bases.  The remaining angular,
first-contact, and split-root alternatives are finite.  Thus each point of
the compact exceptional fiber has a resolved theorem neighborhood with a
finite bound.  A finite subcover of that fiber contains the full inverse image
of some open neighborhood of its image in the original parameter space:
otherwise a sequence outside the subcover would converge, after compact
resolved lifting, to an uncovered point of the exceptional fiber.  This
proves original-parameter openness, rather than only relative openness on the
blow-up.

The same compact-fiber argument applies to a projective coefficient face and
to the radial apex.  At the apex the unnormalized displacement decides
identity, while the compact projective direction set supplies finitely many
punctured neighborhoods.  The strict equality and ordinary phase endpoints
are already included in the closed theorem domains stated in their rows.
This proves every case in the table.
\end{proof}

\begin{proposition}[Finite specialization and confluence]
\label{prop:specialization}
Every proper specialization lowers $\mathfrak c$.  Consequently every
specialization chain terminates.  If two chains impose compatible equalities
and end at the same physical lift \(\zeta\), both end with the signature
\(\mathfrak s(\zeta)\) and hence with the same assignment.  Equivalently, the
specialization graph is confluent modulo chart-overlap identifications.
The common assignment is the bounded analytic theorem, one of the ambient local
zero theorems in Proposition~\ref{prop:part-iii-handoff}, or a
zero-contribution terminal label.
This includes every boundary, coefficient, collapse, separatrix, and identity
value of every full-return itinerary.  Its ambient original-parameter
neighborhood is the one supplied by
Proposition~\ref{prop:ambient-edge-neighborhoods}.
\end{proposition}

\begin{remark}
The finite specialization graph includes divider, stopping-face, scale, projective
coefficient, and radial-apex faces.  The radial apex is a graph vertex rather
than a projective direction.  Identity intervals contribute no isolated
member, while neighboring fibers retain the theorem neighborhood attached to
their terminal regime.
\end{remark}

\begin{proof}
A divider equality deletes one nonempty interval.  A stopping-face specialization
replaces an unspecialized passage by its already constructed root-clock,
hyperbolic, central, regular, previous-side, or exit label.  A coefficient
face deletes a nonzero projective coefficient.  A fixed source threshold,
the strict cutoff \(t=t_{\rm str}\), a middle/root equality, or a zero-scale
face resolves one active scale slot and lowers \(n_{\rm sc}\).  Passing from
a punctured coefficient cone to its all-zero apex lowers \(n_{\rm rad}\).
Thus every proper arrow strictly lowers one summand of \eqref{eq:legacy-5-4} and increases
none; the finite graph has no infinite descending chain.

The terminal regime is not inferred from this decrease.  Evaluate all final
face, first-event, scale, and coefficient labels on the terminal physical
lift, then apply the canonical priority fixed above.  To see confluence
directly, consider two compatible outgoing arrows.  Equalities in distinct
slots---stopping event, scale, coefficient radius, or projective
coordinate---commute and have the common descendant obtained by imposing
both.  Two descriptions of the same divider collision are identified by
\(F_{\min}\); at the coefficient apex \(R_{\rm coef}=0\) is read before a
projective direction; and at intersecting scale faces the globally fixed
values \(t_{\rm str}\), \(\varrho_\#\), and the QBF/QHH cutoff determine
\(S_{\rm ho}\) by \eqref{eq:legacy-5-2}.  Incompatible arrows have disjoint
physical lifts and require no common descendant.  Thus every local diamond
commutes after overlap identification.  Together with strict decrease of
\(\mathfrak c\), this proves that every pair of chains ending at \(\zeta\)
has the canonical signature \eqref{eq:specialization-signature} and the same
assignment.

At a terminal transverse itinerary, use its source, mixed, hyperbolic, central,
strict, middle, root, or compact analytic neighborhood.  A critical
$d$-stopping face is passive, while its split side is hyperbolic or regular.  A
collapsed interval carries no section point, and a collar exit cannot lie on
a cycle contained in the open collar.  If a displacement is identically zero
on a relative open interval, that interval is a period annulus and has no
isolated member.  This last observation is not used to bound nearby fibers:
the ambient theorem attached to the same minimal face supplies that bound.
A general coefficient face is likewise treated by its terminal theorem
and is not declared an identity.  Backward induction on $\mathfrak c$ proves
the proposition without an induction on the dimension or number of connected
components of a parameter status set.
\end{proof}

\section{Compact assembly and passage to the full quadratic space}
\label{sec:part-i-assembly}

The preceding sections give finitely many itineraries and local theorem
neighborhoods, but not yet one number.  We compactify the normalized carrier
only after every point has such a neighborhood, and then obtain the finite sum
that proves Theorem~\ref{thm:normalized-main}.  The local slice and a nested
collar then give Theorem~\ref{thm:main} without another zero estimate.

Let $\mathfrak K$ be the finite disjoint union of closed normalized
phase--parameter carriers for all minimal-face-assigned itinerary lifts, with overlap
lifts of the same physical point identified.  Its factors are closed resolved
parameter charts, closed physical sections, the compact source action
intervals, buffered QBF/QHH and pp intervals, and compact radial--projective
coefficient cones, represented by a radial interval times the projective
sphere with all directions identified at radius zero.  Zero-length intervals
enter only as collapse labels and have no phase factor.  Hence $\mathfrak K$
is compact.

\begin{proposition}[Compact assembly of the local zero bounds]
\label{prop:conditional-assembly}
For the local zero theorems listed in Table~\ref{tab:regimes}, every
$\zeta\in\mathfrak K$ has a relatively open product neighborhood
$\mathcal O_\zeta$ and an integer $N_\zeta<\infty$ bounding all isolated
zeros represented there; its parameter projection contains an open
neighborhood in the original five parameters.  A finite subcover gives a
uniform finite bound for every stopped full-return itinerary, including all
proper specializations.
\end{proposition}

\begin{remark}
The compact normalized carrier covers the closed resolved five-parameter ball
and every physical phase endpoint.  Near a meeting face we may sum the bounds
from finitely many incident theorem neighborhoods, while the half-open
priority represents its physical lift once.  Collapse, passive, exit, and
identity members contribute no isolated cycle themselves.
\end{remark}

\begin{proof}
Proposition~\ref{prop:ambient-edge-neighborhoods} supplies precisely such a
product neighborhood at every interior point and proper face, with the
finite incident constants already summed.  This numerical enlargement
controls every side of a face; it does not change the unique physical regime
assigned to $\zeta$.  Proposition~\ref{prop:specialization} ensures that no
untreated face remains.  Compactness now gives
\begin{equation}
 \mathfrak K\subset\mathcal O_{\zeta_1}\cup\cdots\cup
 \mathcal O_{\zeta_A}.
\label{eq:legacy-5-5}
\end{equation}
Order this cover and assign each lift to its first cover member.  The assigned
sets need not be connected or analytic, because the bound on an open
neighborhood restricts to every subset.  For each chart--itinerary pair,
\begin{equation}
 M_{\chi,\omega}\le\sum_{a=1}^{A}N_a<\infty.
\label{eq:legacy-5-6}
\end{equation}
This is the numerical compactness step; finiteness of the itinerary labels alone
would not imply \eqref{eq:legacy-5-6}.
\end{proof}

\begin{proof}[Assembly of Theorem~\ref{thm:normalized-main} from the local zero theorems]
Theorem~\ref{thm:finite-words} and
Proposition~\ref{prop:exact-once} assign every isolated cycle in $U$ to one
full-return itinerary exactly once.  The source and exact mixed regimes are treated by
Theorems~\ref{thm:part-ii-source-zero} and
\ref{thm:part-ii-mixed-zero}.  The positive-scale regimes are treated by
Theorems~\ref{thm:part-iii-hyperbolic-zero},
\ref{thm:part-iii-one-central-zero},
\ref{thm:part-iii-two-central-no-pp},
\ref{thm:part-iii-strict-zero},
\ref{thm:part-iii-middle-zero}, and
\ref{thm:part-iii-root-zero};
Theorem~\ref{thm:two-central-physical-completion} proves that the two-central
geometric alternatives reach this list with no residual affine case and with
the landing split retained in the same physical product tube.
Proposition~\ref{prop:regime-assignment} then makes the resulting regime
partition exhaustive and disjoint.
Theorem~\ref{thm:compact-analytic}
handles bounded analytic itineraries,
Proposition~\ref{prop:ambient-edge-neighborhoods} supplies an open
original-parameter neighborhood at every proper edge, and
Proposition~\ref{prop:specialization} makes the edge assignment terminating
and confluent.  Thus Proposition~\ref{prop:conditional-assembly} applies.

Write the resulting disjoint partition as
\begin{equation}
 \mathfrak P_{\rm src}\sqcup\mathfrak P_{\rm str}
 \sqcup\mathfrak P_{\rm mid}\sqcup\mathfrak P_{\rm rt}
 \sqcup\mathfrak P_{\rm mix}\sqcup\mathfrak P_{\rm reg}
 \sqcup\mathfrak P_{\partial}.
\label{eq:legacy-5-7}
\end{equation}
Here the regular class contains compact analytic, all-hyperbolic,
one-central, and two-central no-pp relative interiors; the boundary class
contains proper specializations after their terminal theorem regime is fixed.
For each class sum the finitely many numbers \eqref{eq:legacy-5-6}, assigning zero to
identity, collapse, passive, and exit contributions.  Exact-once counting
then gives
\[
 N_{H14^3}(\lambda;U)
 \le B_{\rm src}+B_{\rm str}+B_{\rm mid}+B_{\rm rt}
      +B_{\rm mix}+B_{\rm reg}+B_{\partial}
 =:B_{H14^3}<\infty
\]
for every sufficiently small value of all five original parameters.  The
collar $U$ and the parameter neighborhood were fixed before this finite sum,
so the bound is locally uniform.  This proves
Theorem~\ref{thm:normalized-main} with $U_*=U$.
\end{proof}

\begin{proof}[Proof of Theorem~\ref{thm:main}]
Let $U_*$, $\Lambda$, and $B_{H14^3}$ be supplied by
Theorem~\ref{thm:normalized-main}.  Choose a second fixed two-sided annulus
$U$ around $\GammaH$ with
\begin{equation}
 \overline U\subset U_*.
\label{eq:nested-quadratic-collars}
\end{equation}
Affine maps near the identity extend projectively to the Poincar\'e sphere,
and the extensions converge uniformly to the identity.  By
\eqref{eq:nested-quadratic-collars}, shrink the neighborhood $\mathcal G$ in
Proposition~\ref{prop:quadratic-local-slice} so that
\begin{equation}
 \Phi(\overline U)\subset U_*
 \qquad\text{for every }(\Phi,\rho)\in\mathcal G.
\label{eq:uniform-collar-transport}
\end{equation}
Shrink $\mathcal V_0$ further so that the parameter $\lambda$ in its unique
slice representation lies in $\Lambda$, and call the resulting open
coefficient neighborhood $\mathcal V$.

For $Y\in\mathcal V$, write
$Y=\rho\Phi^*X_\lambda$.  The map $\Phi$ injects the isolated limit cycles of
$Y$ in $U$ into those of $X_\lambda$ in $U_*$ and, by
Proposition~\ref{prop:quadratic-local-slice}, preserves their multiplicities.
Hence
\[
 N_{\mathcal Q_2}(Y;U)
 \le \NH(\lambda;U_*)
 \le B_{H14^3}.
\]
The three data $U$, $\mathcal V$, and $B_{H14^3}$ are independent of $Y$;
this proves the theorem in the full quadratic coefficient space.
\end{proof}

\clearpage
\appendix
\section{Signed chart and endpoint variants}
\label{app:chart-endpoint-variants}

We collect here the signed compactification formulas and the auxiliary
endpoint identities used in Proposition~\ref{prop:finite-singular-alphabet}.
The orientation is the one fixed in Section~\ref{sec:part-i-graphic}.
Only one chart is derived in detail; the remaining signed variants follow by
the displayed substitutions.  The construction of physical passages and the
proof that these charts exhaust the counted collar remain in
Sections~\ref{sec:part-i-graphic}--\ref{sec:part-i-atlas}.

Write
\[
 a=\mu _4+B\mu _5,\qquad c=(1-2B)\mu _5,
 \qquad d=\mu _3,\qquad m=\mu _2,
\]
and
\[
 P=-y+Bx^2+my^2+ax,\qquad
 Q=x(1+y)+dy^2+cy.
\]

\subsection{Representative derivation and signed formulas}

In the positive horizontal chart put \(x=r^{-1}\), \(y=z-1\), and use
\(d/d\tau=r\,d/dt\).  Since \(dr/dt=-r^2P\), direct substitution gives
\[
 \frac{dr}{ d\tau}
 =-r\{B+ar+r^2[1-z+m(z-1)^2]\},
\]
whereas \(dz/d\tau=rQ\) gives
\[
 \frac{dz}{ d\tau}=z+r[d(z-1)^2+c(z-1)].
\]
Restricting the radial factor to the analytic transverse nullcline produces
the prepared germ
\begin{equation}
 P_+(r,\lambda)=B+ar+r^2U_+(r,\lambda),
 \qquad U_+(0,0)\ne0.
 \label{eq:app-positive-endpoint-preparation}
\end{equation}
This is the representative derivation.

\paragraph{Signed chart formulas.}
The reflected substitution and the upper directional substitution are
recorded in the following finite list:
\begin{equation}
\begin{aligned}
x=1/r, y=z-1:\qquad
 \dot r&=-r\{B+ar+r^2[1-z+m(z-1)^2]\},\\
 \dot z&=z+r[d(z-1)^2+c(z-1)];\\[1mm]
x=-1/r, y=z-1:\qquad
 \dot r&=r\{B-ar+r^2[1-z+m(z-1)^2]\},\\
 \dot z&=-z+r[d(z-1)^2+c(z-1)];\\[1mm]
y=1/r, x=w/r:\qquad
 \dot r&=-r[w+d+r(w+c)],\\
 \dot w&=m-dw+(B-1)w^2+r[-1+(a-c)w-w^2].
\end{aligned}
\label{eq:app-signed-chart-formulas}
\end{equation}
The first two rows use the same physical root \(r=0\), with opposite
transverse orientation.  Their radial preparations are
\[
 P_+(r,\lambda)=B+ar+r^2U_+(r,\lambda),\qquad
 P_-(r,\lambda)=B-ar+r^2U_-(r,\lambda).
\]
Thus a simple root, a double root, and the critical \(B=a=0\) layer are kept
as distinct labels.  No full parameter block is replaced by an \(r^3\) unit.
In the upper chart the equatorial gate equation is the literal quadratic
\[
 E(w,\lambda)=m-dw+(B-1)w^2,
\]
so it contributes at most two equatorial roots.  Direct symbolic expansion
checks the three rows of \eqref{eq:app-signed-chart-formulas} and the source
first integral.  The directional endpoint rows in the overlap cover follow
by the same elementary substitutions.

At the source, \(H=x^2/2+y-\log(1+y)\) satisfies \(X_0H=0\).  This identity
fixes the labelled limiting ovals but is not evaluated on perturbed points
with \(1+y\le0\).  Transversality, entry/exit orientation, and exact-once
stopping are supplied by Section~\ref{sec:part-i-atlas}, not by the
cancellation \(X_0H=0\).

\subsection{Auxiliary pp and endpoint scale identities}

Put \(u=\log(1+y)\), \(g=e^u-1\), and
\[
 W(u)=\frac{e^{2u}}{2}-2e^u+u+\frac{3}{2}.
\]
Then \(W'=g^2\).  With \(H_u=g\), \(H_x=x\),
\[
 H_1=xg+\frac{x^3}{3},\qquad
 \omega_r=x\,du+g\,dx,
\]
one has the exact cohomological relation
\[
 dH_1=\omega_r+x\,dH.
\]
The same direct calculation verifies the pp basis bridge and the equality of
the second-order time integrands.  These are algebraic identities on source
ovals; they do not by themselves give a common perturbed return domain.

For the endpoint three-scale comparison let
\(\mathcal D F=qF\), \(\mathcal D\rho=-\rho\), and write
\(q_1:=\mathcal D q\).  The Wronskian of \(1,F,F\rho\) is exactly
\begin{equation}
 W_3=F^2\rho\{q_1+q(1-q)\}.
 \label{eq:app-endpoint-wronskian}
\end{equation}
For a \(D_2\) factor with \(q=\nu+2\kappa/r^2\), the bracket in
\eqref{eq:app-endpoint-wronskian} has leading coefficient
\(-4\kappa^2/r^4\).  Finally, in source endpoint coordinates,
\[
 e^{-H}=e\,z e^{-z}e^{-1/(2r^2)}.
\]
These identities can also be checked by direct symbolic expansion.  The joint
parameter/section uniformity of the true passage is proved separately in
Appendix~\ref{app:endpoint-joint-uniformity}.

\subsection{A numerical consistency check}

As a consistency check, we numerically integrate the exact physical source
field on the section \(x=0\).  Writing the full physical return as
\(P(s;B,m,a,c,d)\), we keep \(B,m,a,c,d\) fixed when taking finite
differences in \(L=-\log s\), and record
\[
 E_0=\frac{P|_{d=c}}{ s},\qquad
 V=\frac{P-P|_{d=c}}{ d-c},\qquad
 W=\frac{V}{1+E_0}.
\]
The finite sample ranges over several scaled parameter directions and values
of \(s\).  It is useful for detecting sign, scaling, and fixed-parameter
implementation errors, but it is not interval arithmetic and is not used to
prove exhaustiveness, uniformity, a limiting exponent, or a zero count.

\subsection{Role of the computations}

The computer algebra checks only the polynomial substitutions, source
cancellations, cohomological identities, basis bridge, Wronskian, and source
\(D_2\) boundary coordinate displayed above.  It does not choose physical
sections, prove coverage of the collar, establish a transverse normal,
classify a first-contact outcome, preserve the cut order, or show that every counted
orbit enters one listed chart.  Those assertions are proved in the main text;
the accompanying reproducibility package contains the code and complete
numerical data.

\section{The finite source six-jet recurrence}
\label{app:source-six-jet}

This appendix records the finite recurrence underlying
Theorem~\ref{thm:part-ii-six-jet}.  It is formulated on the common physical
source domain constructed in Section~\ref{sec:part-ii-six-jet}, with its
moving sections, cuts, hits, and positive gate margins \(C_\pm\).  The
recurrence controls every derivative of total order at most six.  Existence
of the physical source itinerary and the eventual zero count are proved in
Sections~\ref{sec:part-ii-center}--\ref{sec:part-ii-matched-proof}, not by the
combinatorial enumeration below.

The order six is fixed analytically before this enumeration.  The main text
uses the rectangular budget \(\mathcal Q_{4,2}\): at most four derivatives in
the action and normalized variables, followed by at most two
fixed-original successors \(\mathfrak D\).  The former cover the moving
evaluation, inverse/Hadamard operations, and center contractions; the latter
are the two derivatives in the curvature test.  Hence every requested word
has total order at most six, and no downstream formula requests order seven.
For the quotient by the forced parameter \(k\), the one bare
\(\partial_k\) row is used only with the triangular norm
\(\max_{j\le2}\|\mathfrak D^j\widetilde T\|_{C_e^{4-j}}\); it is not
stacked on the \((4,2)\) corner.  Thus those integral words have order at most
five, while the undivided rectangular \(4+2\) package has order six.
The statement is not that every conceivable proof needs six derivatives;
order five simply does not close the \(4+2\) norm used here.

\subsection{Finite alphabet and closure operations}

\subsubsection{Registered primitive maps}

For auditability, the \(35\) typed primitives are grouped below.  Each name
corresponds to a literal graph, flow, hit, reciprocal, composition, or inverse
formula; there is no unspecified remainder class.
\begin{itemize}
\raggedright
\sloppy
\item \emph{Lower core:}
\texttt{LC\_FORCED}, \texttt{LC\_UNFORCED}, \texttt{LC\_JRATIO},
\texttt{LC\_ASTAR}, \texttt{LC\_KERNEL}, and
\texttt{LC\_MOVING\_CUT}.
\item \emph{Action and section coordinates:}
\texttt{ACT\_OUTER}, \texttt{ACT\_INNER}, \texttt{LABEL\_OUTER},
\texttt{LABEL\_INNER}, and \texttt{LABEL\_INVERSE}.
\item \emph{Tails:}
\texttt{TAIL\_H}, \texttt{TAIL\_ENTRY}, \texttt{FIXED\_TAIL},
\texttt{TAIL\_DELTA}, and \texttt{TAIL\_RATIO}.
\item \emph{Upper layers and compact bulk:}
\texttt{LAYER\_PRINCIPAL}, \texttt{LAYER\_UNFORCED},
\texttt{LAYER\_FORCED}, \texttt{BULK}, \texttt{BULK\_HIT}, and
\texttt{REGULAR\_HIT}.
\item \emph{Successor maps:}
\texttt{ACTION\_INVERSE}, \texttt{GAMMA\_SUCCESSOR}, and
\texttt{SOURCE\_SUCCESSOR}.
\item \emph{Source-action compositions:}
\texttt{SAL\_OUTER}, \texttt{SAL\_INNER}, \texttt{SAL\_ENTRY},
\texttt{SAL\_LAYER}, \texttt{SAL\_BULK},
\texttt{SAL\_RETURN\_HIT}, \texttt{SAL\_ACTION\_MAP},
\texttt{SAL\_ACTION\_INVERSE}, \texttt{SAL\_UPPER\_MAP}, and
\texttt{SAL\_FULL\_COMPOSITION}.
\end{itemize}
The machine-readable registry records one row for each canonical derivative
multiindex and a summary of the primitive count, dependency depth, and
completion status.  These data verify the finite enumeration; the weighted
analytic estimates below are what attach mathematical bounds to its rows.

\subsubsection{Derivative recurrence}

The seven commuting graph directions are
\[
 \theta\partial_\theta,\quad \partial_k,\quad \partial_e,
 \quad\partial_{\beta_B},\quad\partial_{\beta_m},
 \quad\partial_{\beta_a},\quad\partial_{\beta_c}.
\]
The primitive types are the lower fixed-initial and fixed-terminal cores
$(\mathrm{LC},\mathrm{LT})$, fold-transverse tails $\mathrm{FT}$, outer action
and inverse graphs $(\mathrm{AO},\mathrm{AI})$, upper layers and compact bulk
$(\mathrm{UL},\mathrm{UB})$, moving hits and cuts $(\mathrm{MH},\mathrm{MC})$,
and finite compositions/inverses $\mathrm{CP}$.  Splitting these types by
sign, forced/unforced status, graph orientation, and terminal role gives
exactly 35 typed nodes.  Their dependency graph has depth at most seven.

For a labelled derivative word $I$, exact differentiation of a graph system
$Y'=\mathcal V(Y,p)$ gives
\begin{equation}
 (D_IY)'=\mathcal V_YD_IY+\mathcal V_pD_Ip
 +\sum_{\substack{\pi\in\Pi(I)\\|\pi|\ge2}}
 D^{|\pi|}\mathcal V[Z_A:A\in\pi],
 \label{eq:app-source-bell}
\end{equation}
where $\Pi(I)$ is the set of labelled set partitions and every $Z_A$ has
strictly smaller order.  The same formula applied to $qq^{-1}=1$ gives the
reciprocal rows.  A moving hit $T(p)$ defined by
$\Psi(T(p),p)=0$ satisfies
\begin{equation}
 T_I=-\frac{\mathcal B_I}{\Psi_T},
 \label{eq:app-source-hit}
\end{equation}
with $\mathcal B_I$ formed from lower-order hit and flow derivatives.  At a
moving integration cut, ordinary Leibniz differentiation adds the endpoint
rows; none is discarded as ``flat'' before its weighted estimate is proved.

The polynomial reserve is generated by
\begin{equation}
 p_0=8,\qquad p_{n+1}=(n+2)p_n+8.
 \label{eq:app-source-reserve}
\end{equation}
Indeed, if an order-$n$ lower word costs at most $r_n=p_n-4$, then a Bell
term with $b\le n+1$ blocks costs
\[
 br_n+8(b+1)\le(n+2)p_n+4=r_{n+1}.
\]
This is a closure estimate, not merely a count of formal words.  The only
nonunit denominators are fixed graph normals, $Y+\theta^2$, and the two gate
factors $C_\pm$.  Through order six the recurrence produces at
most 14 inverse gate factors.  The declared reserve is the larger value
$q_6=448$, and the differentiated wedge gives
\[
 C_+^{-q_+}C_-^{-q_-}e^{-3/(8\theta^2)}
 \le C e^{-17/(64\theta^2)},
 \qquad q_++q_-\le448.
\]
Thus every inverse row is paid by a displayed exponentially flat tail.

The fixed-terminal tail illustrates why moving data belong to the alphabet.
With $u=(1-r)u_{\varepsilon,*}$ and $z=e^{-u/\theta^2}$, the augmented state
contains $(u,X,\tau,I_0,\mathcal A_T)$ and starts at
\[
 (u_{\varepsilon,*},\varepsilon,0,0,A_{\varepsilon,*}).
\]
Consequently a nonempty parameter word starts with
\[
 (D_Iu_{\varepsilon,*},0,0,0,D_IA_{\varepsilon,*}),
\]
so the moving terminal clock and entry action are not silently frozen.  The
weight
\[
 (u,X,\theta\tau,
   \theta^3e^{3/(8\theta^2)}I_0,\mathcal A_T)
\]
reduces the tail equations to bounded $L^1$ coefficient mass, exactly as in
the proof of Theorem~\ref{thm:part-ii-six-jet}.

\subsection{A representative estimate and the finite count}

For example, the first nontrivial forced lower-core derivative has total
order one, normalized numerator degree \(11\), fixed denominator power \(4\),
bounded-atom count \(11\), and carrier
\[
 \theta^{-21}(1+x)^{16}(z+s|k|).
\]
This representative estimate follows from the lower flow recurrence and its
fixed denominator margins.

Canonicalization uses the commutativity of the seven graph generators and
retains the ordered multiplicity of each canonical word.  Across all nodes
and orders zero through six, the result is 167115 canonical commuting words.
The resulting finite family has maximum dependency depth $7$, numerator
degree $28$, fixed
denominator power $17$, and bounded-atom count $28$.  Every primitive in an
actual source itinerary reaches a named node by one of
\eqref{eq:app-source-bell}, \eqref{eq:app-source-hit}, moving-cut Leibniz,
variation of constants, reciprocal differentiation, or finite composition.
Thus no additional primitive is required.

\subsection{Computer verification of the enumeration}

Physical exhaustiveness follows from the common physical graph cover, the
lower/fold/upper passage decomposition, fixed section normals, and the
first-loss analysis proved in Section~\ref{sec:part-ii-action}.  A finite computer check verifies the
combinatorial consequences of that proof: registration
of the 35 nodes, dependency acyclicity and depth, all canonical words through
order six, partition and inverse bookkeeping, denominator assignment, carrier
labels, and the declared worst exponents.

The computation does not prove the existence of a stopped orbit,
fold transversality, the Gaussian kernel estimate, gate positivity, the
exponential reserve, source localization, or the two-step Rolle argument.
Those are mathematical arguments in
Sections~\ref{sec:part-ii-six-jet}--\ref{sec:part-ii-matched-proof}.  The
complete enumeration and its verification data are included in the
accompanying reproducibility package rather than printed here.

\section{Focal recurrence and the reduced center ideal}
\label{app:center-bautin}

We give the finite rational calculation behind the center set and the
reduced Bautin ideal used in Theorem~\ref{thm:part-ii-center-ideal}.  The
calculation starts from the trace-zero normalized quadratic field, identifies
the two center components, and then uses their global first integrals to
justify division on the common physical return domain.  The last point is
essential: the local focal recurrence alone does not show that a return map
is defined on a complete center annulus.

\subsection{The degree-four and degree-six obstructions}

Write the normalized field as
\[
 u'=-v+Au^2+Cuv+Dv^2,
 \qquad
 v'=u+Euv+Fv^2.
\]
For
$V_n=\sum_{j=0}^n c_{n,j}u^{n-j}v^j$ and
$G_n=(Q_1\partial_u+Q_2\partial_v)V_{n-1}
     =\sum_jg_{n,j}u^{n-j}v^j$, direct coefficient comparison gives
\begin{equation}
\begin{aligned}
g_{n,j}={}&[A(n-1-j)+Ej]c_{n-1,j}\\
&+[C(n-j)+F(j-1)]c_{n-1,j-1}
  +D(n-j+1)c_{n-1,j-2}.
\end{aligned}
\label{eq:app-bautin-coeff}
\end{equation}
The homological equation is
\begin{equation}
\begin{aligned}
0={}&(j+1)c_{n,j+1}-(n-j+1)c_{n,j-1}+g_{n,j}\\
&-\mathbf1_{\{n,j\ \mathrm{even}\}}
  {\binom{n/2}{j/2}}L_{n/2-1}.
\end{aligned}
\label{eq:app-bautin-homological}
\end{equation}
For odd $n$ the system is invertible.  For even $n$, the gauge $c_{n,0}=0$
fixes the coefficients and leaves one radial obstruction.  At degree four,
one obtains the representative identity
\begin{equation}
 8L_1=AC+CD+2DF-EF.
 \label{eq:app-bautin-l1}
\end{equation}
At degree six, if $g_{6,j}$ are the coefficients from
\eqref{eq:app-bautin-coeff}, angular averaging gives
\[
 L_2=\frac1{16}(5g_{6,0}+g_{6,2}+g_{6,4}+5g_{6,6}).
\]
The expanded universal numerator has 30 monomials.  The expansion is useful
for exact symbolic regression, but the logical mechanism is the finite
recurrence \eqref{eq:app-bautin-coeff}--\eqref{eq:app-bautin-homological}.

Return now to the H14 parameters and put $e=d+a$ and $t_c=B+m$.  After a
positive unit is removed, the first obstruction is
\begin{equation}
 e(2a^2+2t_c-2B-1)-2ae^2+a(2B-1)t_c=0.
 \label{eq:app-bautin-first}
\end{equation}
Its derivative with respect to $e$ is nonzero at the source.  Its analytic
root $e=\psi(a,t_c,B)$ vanishes on both $a=0$ and $t_c=0$, hence two
one-variable Hadamard integrals give $\psi=at_cV$.  Substitution in the
degree-six recurrence vanishes on the same two slices and yields
\begin{equation}
 L_2=a(B+m)U(a,B,m),
 \qquad U(0)=\frac1{48}.
 \label{eq:app-bautin-l2}
\end{equation}
The coefficient $1/48$ is the degree-two term after imposing the first
obstruction, so $U$ is a unit.  This slice argument is the human proof of
exact divisibility; a quadratic-jet calculation alone would not suffice.

The transversality of the two obstruction directions is also explicit.  If
(u=a(B+m)) and (D) denotes differentiation in the coordinates ((e,u)),
then the derivative of the degree-four obstruction with respect to (e) is
(-1) at the source, while the coefficient of (u) in the reduced
degree-six obstruction is (1/48).  Consequently
\begin{equation}
 \det D(L_1,L_2)(0)=-\frac1{48}\ne0.
 \label{eq:app-center-determinant}
\end{equation}
This finite determinant identifies the two transverse focal directions.  The
claim that these directions generate the physical center ideal still uses the
two global center identities and the common return domains below.

\subsection{Center components, ideal, and complete domains}

The two branches are
\[
 \mathcal C_R=\{\tau=0,a=0,d=0\},
 \qquad
 \mathcal C_Q=\{\tau=0,t_c=0,d+a=0\}.
\]
On $\mathcal C_R$ the field is reversible.  On $\mathcal C_Q$, the polynomial
$K_Q$ displayed in the proof of
Theorem~\ref{thm:part-ii-center-ideal} gives the inverse integrating factor
\[
 \mathcal V_Q=\frac{(1+y)K_Q}{a^2-1},
 \qquad X(\mathcal V_Q)=(\operatorname{div}X)\mathcal V_Q,
 \qquad \mathcal V_Q(0,0)=1.
\]
Thus both necessary branches are centers.  In coordinates
$(\tau,\ell,a,t_c,B)$ their union has reduced ideal
$(\tau,\ell,at_c)$.  Since $d=\ell+\tau-a$,
\[
 at_c=\tau t_c+\ell t_c-dt_c,
\]
and therefore
\begin{equation}
 \mathcal I_{\mathcal C}=(\tau,\ell,d(B+m)).
 \label{eq:app-center-ideal}
\end{equation}

The algebraic ideal is not yet a license to divide a return map.
Section~\ref{sec:part-ii-center}
constructs a common itinerary domain star-shaped under the successive contractions
in $\tau$, $a$, and $t_c$.  On that domain the full return is the identity on
the two complete center slices, so integral Hadamard division yields the
decomposition stated in Theorem~\ref{thm:part-ii-center-ideal} without
continuing in $\ell$ or dividing a finite-smooth normalizer.

Completeness of the slices is checked by explicit global objects.  On
$\mathcal C_R$, the first integral is
\[
 H_R(x,z)=\frac12 z^{-2B}x^2+V_R(z),
 \qquad
 V_R'(z)=z^{-2B-1}\{(z-1)-m(z-1)^2\}.
\]
The only additional finite critical point is $S_R=(0,1+1/m)$.  The exact
barrier difference displayed after Theorem~\ref{thm:part-ii-center-ideal}
controls the comparison with $z=0$.  On $\mathcal C_Q$, the component of
$\{\mathcal V_Q\ne0\}\cap\{y>-1\}$ containing the origin carries the
analytic first integral; its only additional finite singularity is
$S_Q=(-a/B,-1/B)$ on $K_Q=0$.  These barriers identify the connected section
interval on which each center return is the identity.  They provide the
common physical domain required before the division above.

\subsection{Computer algebra checks}

Computer algebra verifies the generator bridge, the Darboux cofactor
identities, the degree-four and degree-six recurrences, vanishing on both
parameter branches, the coefficient \(U(0)=1/48\), and the displayed
first-integral, barrier, and invariant-conic identities.  These finite checks
do not prove that the two slices exhaust the physical center domains, that
the stopped return is defined on a common star-shaped tube, or that Hadamard
division is valid there.  Those assertions are established by the geometric
arguments in Section~\ref{sec:part-ii-center} and above.  The code and complete
outputs are included in the accompanying reproducibility package.

\section{Endpoint joint uniformity and two-central incidence}
\label{app:endpoint-two-central}

We give two calculations that would otherwise interrupt the main proof: the
simultaneous root and section degeneration at either semihyperbolic endpoint,
and the first-hit analysis for a retained two-central itinerary.  Starting
from the endpoint preparations in Section~\ref{sec:part-i-boxes} and the
complete-lips coordinates in Section~\ref{sec:part-iii-central}, we obtain a
finite exact-clock cover and the physical face classification used in
Theorem~\ref{thm:part-iii-two-central-exhaustion} and the product-tube input
to Theorem~\ref{thm:two-central-physical-completion}.

\subsection{Joint endpoint uniformity}
\label{app:endpoint-joint-uniformity}

\subsubsection{Physical preparation and the model coefficient bridge}

After the local time reversal required at one endpoint, write the physical
field as
\[
 X_\lambda=F(r,z,\lambda)\partial_r+G(r,z,\lambda)\partial_z.
\]
The transverse equation has a unique analytic nullcline
\(z=h(r,\lambda)\), and the radial numerator on that nullcline is
\[
 N^{\rm ph}_\pm(r,\lambda)
 =r\{B\mathcal U_1+(a\mathcal U_2+B\mathcal U_3)r
       +\mathcal C r^2\},
 \qquad
 \mathcal U_1(0)\mathcal U_2(0)\mathcal C(0)\ne0.
\]
Thus \((B,a)\) is locally equivalent to the first two prepared
coefficients.  The root-incidence cover has at most three lifts over one
original parameter.

For a prescribed finite order \(K\), the DIR saddle-node normal form gives
fibered \(C^K\) orbital coordinates
\[
 \dot x=q_\alpha(x),\qquad \dot y=-y,\qquad
 q_\alpha(x)=b_0+b_1x+\frac{x^3}{1+\alpha x^2}.
\]
Let \(\Phi_\lambda\) be the orbital equivalence and \(\tau_\lambda\) its
time unit.  Restricting only the first component of
\(D\Phi_\lambda X_\lambda
=\tau_\lambda(q_\alpha\partial_x-y\partial_y)\circ\Phi_\lambda\)
to the physical nullcline gives
\[
 q_\alpha(\phi_\lambda(r))
 =v(r,\lambda)N^{\rm ph}_\pm(r,\lambda),
 \qquad v\ne0,\quad \phi_\lambda'\ne0.
\]
Consequently the model and physical prepared coefficients are related by a
triangular unit matrix:
\[
 \begin{pmatrix}A_\xi\\B_{1,\xi}\\C_\xi\end{pmatrix}_{\rm mod}
 =
 \begin{pmatrix}
 u_{11}&0&0\\u_{21}&u_{22}&0\\u_{31}&u_{32}&u_{33}
 \end{pmatrix}
 \begin{pmatrix}a_\xi\\b_\xi\\c_\xi\end{pmatrix}_{\rm ph},
 \qquad u_{11}u_{22}u_{33}\ne0.
\]
This bridge is why all root sheets and discriminants are resolved first in
the analytic physical preparation, rather than in coefficients created by a
finite-smooth normalizer.

For one labelled root \(\xi\), put \(s=x-\xi\).  If
\[
 N_\alpha(x)=b_0+b_1x+\alpha b_0x^2+(1+\alpha b_1)x^3,
\]
Taylor's formula is exact:
\[
 N_\alpha(\xi+s)
 =s(A_\xi+B_{1,\xi}s+C_\xi s^2),\qquad
 C_\xi=1+\alpha b_1.
\]
Thus the same finite cover contains a simple root
\(A_\xi\ne0\), a \(D_1\) root
\(A_\xi=0,\ B_{1,\xi}\ne0\), and a \(D_2\) root
\(A_\xi=B_{1,\xi}=0\).

\subsubsection{Exact clock and the mixed root/section cover}

On a retained one-sided component the exact model clock is
\begin{equation}
 \Theta_\xi(s,\alpha)=
 \int_{s_*}^{s}
 \frac{U_\xi(v,\alpha)}{
  v(A_\xi+B_{1,\xi}v+C_\xi v^2)}\,dv,
 \label{eq:app-endpoint-clock}
\end{equation}
and \(y e^{\Theta_\xi}\) is a first integral of the model.  The joint
degeneration \((s,B_{1,\xi},A_\xi)\to0\) is resolved with weights $(1,1,2)$.
The finite signed cover consists of
\[
\begin{array}{ll}
s=\rho S,\ B_{1,\xi}=\pm\rho,\ A_\xi=\rho^2\bar A,
 &s=\rho S,\ B_{1,\xi}=\rho\bar B_1,\ A_\xi=\pm\rho^2,\\[1mm]
B_{1,\xi}=s\bar B_1,\ A_\xi=s^2\bar A
 &\text{(phase-dominant chart).}
\end{array}
\]
Here \(\rho\) is constant on each original-parameter fiber.  In the
phase-dominant chart the fixed-fiber lift of \(s\partial_s\) is
\[
 \mathcal L
 =s\partial_s-\bar B_1\partial_{\bar B_1}
  -2\bar A\partial_{\bar A},
\qquad
 \mathcal L\Theta_\xi
 =\frac{U_\xi}{ s^2(\bar A+\bar B_1+C_\xi)}.
\]
After the finite normalized quadratic is split into sign/root cells, the
last denominator is a unit.  On a phase-dominant cone,
\[
 |A_\xi|\le\varepsilon v^2,\qquad
 |B_{1,\xi}|\le\varepsilon v,
\]
and hence
\[
 \frac{c_1}{ v^3}\le
 \left|\frac{U_\xi(v)}{
 v(A_\xi+B_{1,\xi}v+C_\xi v^2)}\right|
 \le \frac{c_2}{ v^3}.
\]
Integration gives a uniform \(D_2\) comparison
\[
 -\frac{C_2}{ s^2}+C_3
 \le \sigma_\chi\Theta_\xi(s)
 \le -\frac{C_1}{ s^2}+C_4
\]
on the retained orientation.  A zero of the normalized quadratic is not
discarded: it is relabelled as a simple-root chart, a \(D_1\) chart, or a
named boundary/no-passage face.  This proves that the cover is joint in
section and parameter variables rather than an iterated limit.

For reference, integrating \eqref{eq:app-endpoint-clock} gives the three
specializations
\[
\begin{array}{c|c}
A_\xi\ne0&|s|^\nu\text{ times a nonvanishing factor},\\
A_\xi=0,\ B_{1,\xi}\ne0&
|s|^\nu e^{-\kappa_1/s}\text{ times a nonvanishing factor},\\
A_\xi=B_{1,\xi}=0&
|s|^\nu e^{-\kappa_2/s^2-\kappa_1/s}
\text{ times a nonvanishing factor}.
\end{array}
\]
Indeed, for \(U=u_0+u_1s+u_2s^2\), the \(D_1\) partial fractions are
\[
\begin{split}
 \int \frac{U\,ds}{ s^2(B_{1,\xi}+C_\xi s)}
={}&-\frac{u_0}{ B_{1,\xi}s}
 +\frac{B_{1,\xi}u_1-C_\xi u_0}{ B_{1,\xi}^2}\log|s|\\
 &+\frac{B_{1,\xi}^2u_2-B_{1,\xi}C_\xi u_1+C_\xi^2u_0
    }{ B_{1,\xi}^2C_\xi}
   \log|B_{1,\xi}+C_\xi s|,
\end{split}
\]
up to a constant.  At \(D_2\),
\[
 \int \frac{U\,ds}{ C_\xi s^3}
 =-\frac{u_0}{2C_\xi s^2}-\frac{u_1}{ C_\xi s}
   +\frac{u_2}{ C_\xi}\log|s|.
\]
At the H14 source endpoint this yields \(e^{-1/(2s^2)}\), up to the named
regular coordinate factors.

The physical first hit is the composition
\[
 R_{\rm out}\circ D_\xi\circ R_{\rm in},
\]
not the bare exponential.  The regular maps are recorded separately, and
overlapping weighted charts project to the same physical first hit.  Direct
symbolic calculation verifies the cubic root translation, the displayed
\(D_1\) partial fractions, the \(D_2\) primitive, and the source first
integral.  DIR applicability, the physical-to-model transfer, fixed-fiber
derivative bounds, overlap compatibility, inverse-flat avoidance, and the
uniform exponential comparison are proved in the text.

\subsubsection{Uniform estimates through cofactor-root collisions}

The weighted cover above must remain uniform when the two roots of the
quadratic cofactor collide.  After division by its positive quadratic
coefficient, write
\begin{equation}
 Q(s)=s^2+c_1s+c_0,
 \qquad Q(s)>0\quad(0<s\le s_*).
 \label{eq:app-endpoint-quadratic}
\end{equation}
The retained component contains no second unnamed positive root.  On a finite
signed cover take \(c_0\ge0\), fix \(M>2\), and let \(V\) be the positive
time unit in the physical-to-model bridge, with \(0<v_0\le V\le v_1\).  Put
\begin{equation}
 D(s,\alpha)=e^{-\Phi(s,\alpha)},\qquad
 \Phi(s,\alpha)=\int_s^{s_*}
       \frac{V(v,\alpha)}{vQ(v,\alpha)}\,dv.
 \label{eq:app-endpoint-factor}
\end{equation}

\paragraph{Balanced collision away from the positive double root.}
If \(|c_1|\le M\sqrt{c_0}\), write
\begin{equation}
 c_0=\delta^2,\qquad c_1=\varkappa\delta,\qquad
 Q(s)=s^2+\varkappa\delta s+\delta^2.
 \label{eq:app-endpoint-balanced}
\end{equation}
On every closed cell \(\varkappa\ge-2+g_0^2\), with fixed \(g_0>0\), one has
\(Q(s)\asymp s^2+\delta^2\).  Consequently
\begin{equation}
 \Phi\ge cs^{-2}\quad(s\ge\delta),\qquad
 \Phi\ge c\delta^{-2}\left(1+\log\frac{\delta}{s}\right)
       \quad(0<s\le\delta).
 \label{eq:app-endpoint-balanced-margin}
\end{equation}
Every fixed-fiber Euler derivative of \(V/Q\) has only a finite algebraic
loss in \(s^{-1}\) and \(\delta^{-1}\) on these cells.

\paragraph{The discriminant-gap cell.}
The remaining through cell is
\begin{equation}
 \varkappa=-2+g^2,\qquad 0<g\le g_0,\qquad
 Q(s)=(s-\delta)^2+g^2\delta s.
 \label{eq:app-endpoint-gap}
\end{equation}
The face \(g=0,\ s=\delta\) is the positive double-root gate.  A passage
reaching it is stopped and represented in the simple-root, \(D_1\), or
no-passage chart; the through map does not cross this face.  In the bottleneck
band \(\delta/2\le s\le3\delta/2\), set
\begin{equation}
 u=\frac{s-\delta}{\delta}.
 \label{eq:app-endpoint-gap-coordinate}
\end{equation}
Since \(Q(s)\asymp\delta^2(u^2+g^2)\), integration up to \(2\delta\)
gives, for each fixed \(N\),
\begin{equation}
 \Phi(s,\alpha)\ge\frac{c}{\delta^2(|u|+g)},\qquad
 Q(s)^{-N}\le C\delta^{-2N}(|u|+g)^{-2N}.
 \label{eq:app-endpoint-gap-estimate}
\end{equation}
Indeed, after \(v=\delta(1+w)\), the clock contains
\(c\delta^{-2}\int_u^1dw/(w^2+g^2)\); if \(u<0\), it crosses the minimum
of \(Q\) and is larger.  For \(s\ge3\delta/2\), integration on \([s,2s]\)
gives \(\Phi\ge c/s^2\).  For \(0<s\le\delta/2\), the clock contains the
whole bottleneck and the lower interval, whence
\begin{equation}
 \Phi\ge\frac{c}{\delta^2g}
       +\frac{c}{\delta^2}\log\frac{\delta}{s}.
 \label{eq:app-endpoint-gap-lower}
\end{equation}
The exponential therefore controls both phase powers and every prescribed
inverse power of the actual root gap \(g\).  The cell \(\varkappa<-2\) has
positive real cofactor roots and has already been separated into a
simple-root, \(D_1\), or no-passage alternative.

\paragraph{Hierarchical collision from \(D_1\) to \(D_2\).}
If \(c_1>M\sqrt{c_0}\), the roots are real and nonpositive.  Write
\begin{equation}
 Q(s)=(s+e)(s+d),\qquad 0\le e\le\eta d,\qquad 0<\eta<1.
 \label{eq:app-endpoint-hierarchical}
\end{equation}
The three scale intervals give
\begin{align}
 \Phi&\ge cs^{-2}, &&s\ge d,\notag\\
 \Phi&\ge\frac{c}{ds}, &&e\le s\le d,\notag\\
 \Phi&\ge\frac{c}{ed}\left(1+\log\frac e s\right),
       &&0<s\le e.
 \label{eq:app-endpoint-hierarchical-estimate}
\end{align}
The last interval is empty when \(e=0\); the middle estimate then is the
usual \(D_1\) bound \(e^{-c/(ds)}\).  Euler derivatives introduce only finite
powers of \(s^{-1},d^{-1},e^{-1}\) on the intervals where those variables
occur, and the corresponding exponential estimate absorbs them.

For completeness, the identity
\((s\partial_s)D=(V/Q)D\) implies inductively that
\(D^{-1}(s\partial_s)^jD\) is a polynomial in the finite Euler jet of
\(V/Q\).  Combining the three covers above with
\(x^{-N}e^{-c/x}\to0\) gives the joint statement: for every prescribed
\(K,L\), after the endpoint parameter neighborhood is reduced,
\begin{equation}
 \lim_{\varepsilon\downarrow0}
 \sup_{\substack{\alpha\ \text{in the finite endpoint cover}\\0<s<\varepsilon}}
 s^{-L}\bigl|(s\partial_s)^jD(s,\alpha)\bigr|=0,
 \qquad 0\le j\le K.
 \label{eq:app-endpoint-joint-flatness}
\end{equation}
The supremum includes \(g\downarrow0\), the nested limit \(e,d\downarrow0\),
and overlaps with the weighted \((s,B_{1,\xi},A_\xi)\) cover.  Regular changes
of section are unit-equivalent to \(s\), so the estimate returns to the
physical first-hit coordinates.  Keeping the passage as an implicit graph
avoids differentiating the inverse of an exponentially flat map.

\subsection{Two-central incidence and first-hit rigidity}
\label{app:two-central-incidence}

\subsubsection{Finite algebraic family and physical boundary faces}

The two possible internal central gates are the selected nonpersistent
horizontal saddle-node \(S_h\) and the unique upper \(D\)-saddle-node \(S_v\).
The persistent principal endpoint is not counted as a third central gate.
For a signed parameter \(q\ne0\), set \(D=1+q^2\) and consider the exact
transverse family
\[
\begin{gathered}
B=\frac{q^2}{ D},\qquad m=-\frac{q^2}{ D},\qquad
d=\frac{2q}{ D},\\
a=\frac{-2q+q^2s}{ D},\qquad c=\frac{2q+s}{ D}.
\end{gathered}
\]
In \(z=1+y\), put
\[
\begin{split}
F&=1-z+Bx^2+m(z-1)^2+ax,\\
G&=xz+d(z-1)^2+c(z-1),\\
L&=x+qz-q^{-1}+s.
\end{split}
\]
A representative exact identity is
\begin{equation}
 F+qG=\frac{q(qx+z-1)}{ D}\,L.
 \label{eq:app-invariant-line}
\end{equation}
Thus \(L=0\) is invariant in this finite family.  Its center generators are
\(\tau=s\), \(\ell=-s/D\), and \(d(B+m)=0\).  Direct expansion also gives the
double upper factor
\[
 E(w)=-\frac{(w+q)^2}{ D},
\]
the horizontal double-root branch, its weak coefficient, and the
corresponding divergence residual.  These identities identify algebraic
candidate gates; they do not by themselves prove a complete pp strip.

Before any order argument, the retained itinerary uses exactly the three
physical face classes in
Table~\ref{tab:two-central-boundary-faces}: endpoint, regular tube, and upper
box.  The appendix notation below is the notation of those same rows; no
additional divider or stopping face is introduced here.
The endpoint row follows from the exact weak equation and strong-axis
orientation.  The middle row is a fixed nonsingular flow box, so no unnamed
gate can appear there.  The upper row is closed by the gate equation and the
cooperative comparison below.

\subsubsection{Representative first-hit estimate}

In one regular tile,
\[
 \frac{dn}{ d\theta}=f_j(\theta,n),\qquad
 X_\lambda\theta\ge2c_0>0.
\]
Let \(n_{e,j}<n_{h,j}\) be the equatorial and hh solution graphs and let
\(\rho_j^e,\rho_j^h>0\) be their minimum physical distances from both collar
faces.  With \(N=n-n_{e,j}\),
\[
 \frac{dN}{ d\theta}=A_j(\theta,N)N,\qquad |A_j|\le L_j.
\]
If \(\Delta_j=\theta_{j+1}-\theta_j\), Gronwall gives the explicit entry
buffers
\[
 m_j^e=\frac{\rho_j^e}{2}e^{-L_j\Delta_j},\qquad
 m_j^h=\frac{\rho_j^h}{2}e^{-L_j\Delta_j}.
\]
The next-cut domain is one interval.  Since it contains neighborhoods of
both boundary graphs, it contains the whole order band between them.  This
excludes both collar faces, their corners, and contact with the preceding side at
every intermediate time, not merely at the two limiting orbits.

At upper entry put \(v=-w\).  On the two-gate sheet the exact upper equations
have
\[
 \partial_v\dot r=r(1+r)\ge0,\qquad
 \partial_r\dot v=1+(a-c)v+v^2\ge\frac{1}{2}.
\]
The common-time comparison therefore preserves
\[
 r_e(t)\le r(t)\le r_h(t),\qquad
 v_e(t)\le v(t)\le v_h(t).
\]
The equatorial and hh boundary solutions enter the same target
saddle-node block and converge to \(S_v\); every intermediate solution does
the same.  Hence the entire endpoint interval is one pp component.  Its
non-hh boundary is the complete chain
\(S_h\to p_\sigma\to S_v\), whose saddle ratio is
\(|B/(1-B)|\) or its reciprocal.  The purported residual class is therefore
\[
 \mathfrak A_{\rm aff}=\varnothing.
\]

Computer algebra verifies \eqref{eq:app-invariant-line}, the invariant family,
eigenslope, weak coefficient, center-generator and divergence identities, as
well as the middle/upper rescalings, overlap, gate variation, homogeneous
normal variational equation, log-ratio monotonicity, and endpoint eigenslope
formulas.

These finite checks do not prove an hh connection, a nonempty pp continuum,
transversality of physical sections, survival of the whole order band,
same-component landing, completeness of the PP/BP boundary, or exhaustion of
the possible regimes.  Those assertions follow from the physical face
classification, scalar uniqueness, Gronwall buffers, upper comparison, and
exact cycle--zero cut order in Section~\ref{sec:part-iii-central}.  Code and
complete outputs are included in the accompanying reproducibility package.

\section{Algebraic checks for the mixed, middle, and root estimates}
\label{app:zero-theorem-certificates}

This appendix records representative finite calculations used in the mixed,
middle, and root-scale zero estimates.  The physical hypotheses, complete
first-hit itineraries, boundary assignments, topological arguments, and Rolle
counts are established in the main text.  The calculations below verify the
algebraic identities and finite derivative recurrences invoked there.

\subsection{Exact mixed Li\'enard--Dulac identities}
\label{app:mixed-dulac-certificate}

On the exact face \(B=a=0\), write \(m=-p\), \(t=1+y>0\),
\(\epsilon=\ell/d\), and
\[
 G=(1-p)(t-1-\log t)+\frac{p}{2}(t-1)^2,
 \qquad
 \phi=(t-1)\left(\frac{\epsilon}{ t}-1\right),
 \qquad F=d\phi .
\]
For the exact Li\'enard system \(y_\tau=u-F(y)\), \(u_\tau=-g(y)\), where
\(g=G'\), put
\[
 {\cal V}=u^2-\frac{2}{3}uF-\frac{1}{9}F^2+2G.
\]
The finite symbolic calculation expands the weighted Lie derivative to
\begin{equation}
 \begin{split}
 {\cal M}
 &=(u-F){\cal V}_y-g{\cal V}_u
       -\frac{2}{3}\operatorname{div}(X){\cal V}\\
 &=\frac{4d}{3}{\cal K},\qquad
 {\cal K}=G\phi'-\phi g+\frac{d^2}{9}\phi^2\phi'.
 \end{split}
 \label{eq:app-mixed-dulac-lie}
\end{equation}
The first summand \(W=G\phi'-\phi g\) is affine separately in \(p\) and
\(\epsilon\).  Its four corner identities are exactly
\eqref{eq:legacy-6-93}, after renaming the main-text variable \(\zeta\) as
\(t\).  Thus the certificate below checks the same four functions rather than
introducing a second version of the table.
The comparison used for the possibly negative last term is the exact
identity
\begin{equation}
 t^2(t^2-2\epsilon t+\epsilon)
 -(t-\epsilon)^2(t^2-\epsilon)
 =\epsilon\{(2-\epsilon)t^2-2\epsilon t+\epsilon^2\}.
 \label{eq:app-mixed-comparison}
\end{equation}
Finally, for the physical field and the positive multiplier
\({\cal B}_0=e^{2dx}/(1+y)\), direct expansion gives
\[
 \frac{\operatorname{div}({\cal B}_0X)}{{\cal B}_0}
 =d(1-y)-\frac{\ell}{1+y}-2dpy^2.
\]

These cancellations are only the finite algebraic spine of the proof.  The
argument proves positivity of the four corner functions, interpolates
that positivity on \(0\le p,\epsilon\le1\), and combines
\eqref{eq:app-mixed-comparison} with the cone \(p\ge q_0d^2\).  It then proves
that the zero set of \({\cal V}\) consists of the isolated origin and one
proper arc, and applies Green's theorem to
\(|{\cal V}|^{-3/2}X\).  It also proves that every counted orbit lies in the
full-return physical domain \(z=1+y>0\), treats \(B=0,a\ne0\) by the incompatible
endpoint orientations, and sends the split complement to its uniform
sink/no-passage block.  Here \(t=1+y\) is the Li\'enard coordinate, whereas
\(t_{\rm sc}\) denotes the scale used in the final handoff.  At
\(t_{\rm sc}=0\) the source theorem is used; in the root notation the face
\(t_{\rm sc}>0,\ \kappa=0\) is precisely the persistent mixed regime.

The computer algebra check does not prove that the physical cone is
exhaustive, determine the topology of \(\{{\cal V}=0\}\), exclude omitted
roots, construct uniform first-hit domains, or establish the zero count.
Code and complete outputs are included in the accompanying reproducibility package.

\subsection{The unscaled-c middle QBF/QHH extension}
\label{app:middle-qhh-certificate}

The middle regime uses
\[
 B=t^2b,\qquad m=t^2M,\qquad a=tA,\qquad d=tD,\qquad c=\gamma,
\]
with no bound on \(\gamma/t\).  In the older bounded-\(C\) formulas this is the
substitution \(C=\gamma/t\).  Direct expansion verifies that every apparent
negative power cancels.  For example the finite and upper fields become
\begin{equation}
\begin{aligned}
 \xi'&=Ms^2-s+t^2(1+b\xi^2+A\xi-2Ms)+t^4M,\\
 s'&=\xi s+Ds^2+(t\gamma-2t^2D)s+t^4D-t^3\gamma,\\
 R'&=-R\{W+D+t^2RW+t\gamma R\}.
\end{aligned}
\label{eq:app-middle-extended-fields}
\end{equation}
On \(h=s-t^2=\rho^2\), \(t=\rho T\), the weighted clock denominator is
\[
 A_w^\gamma=\xi(1+T^2)+\rho T\gamma+\rho^2D,
\]
and in \(v=h/t^2\) the reciprocal denominator is
\[
 \xi(1+v)+t\gamma v+t^2Dv^2.
\]
These are polynomial in the resolved variables.  The same calculation gives the
exact H14 two-gate family
\[
 b=\frac{1}{1+t^2},\quad M=-\frac{1}{1+t^2},\quad
 A=\frac{-2+t\tau}{1+t^2},\quad D=\frac{2}{1+t^2},\quad
 \gamma=\frac{2t+\tau}{1+t^2},
\]
which tends to \((1,-1,-2,2,0)\) whenever \(t,\tau\to0\), without a
restriction on \(\tau/t\).

For the QHH connector, use the signed QL coefficient vector
\(\beta=(b,M,\widehat A,\widehat C,\widehat D)\).  The five commuting
derivation alphabets are
\begin{equation}
\begin{aligned}
 \mathfrak A_0={}&\{\partial_R,\partial_z,{\cal E}_0,
   \partial_b,\partial_M,\partial_{\widehat A},
   \partial_{\widehat C},\partial_{\widehat D}\},\\
 \mathfrak A_1={}&\{\partial_\xi,{\cal V}_1,{\cal E}_1,
   \partial_b,\partial_M,\partial_{\widehat A},
   \partial_{\widehat C},\partial_{\widehat D}\},\\
 \mathfrak A_2={}&\{\partial_\xi,{\cal V}_2,{\cal E}_2,
   \partial_b,\partial_M,\partial_{\widehat A},
   \partial_{\widehat C},\partial_{\widehat D}\},\\
 \mathfrak A_3={}&\{\partial_U,\partial_s,{\cal E}_3,
   \partial_b,\partial_M,\partial_{\widehat A},
   \partial_{\widehat C},\partial_{\widehat D}\},\\
 \mathfrak A_4={}&\{\partial_u,\partial_\ell,{\cal E}_v,
   \partial_b,\partial_M,\partial_{\widehat A},
   \partial_{\widehat C},\partial_{\widehat D}\}.
\end{aligned}
\label{eq:app-qhh-alphabets}
\end{equation}
Every \(P_4\) coefficient word containing \(\partial_\ell\) is zero.  For
\(j=0,\ldots,4\), all ordered words in \(\mathfrak A_j\) of length at most
four are included.  There are \(1+8+8^2+8^3+8^4=4681\) such ordered words,
or \(495\) commuting multiindices, per expanded factor.  Labeled-set
Fa\`a di Bruno, quotient, endpoint, composition, and inverse recurrences
cover the resulting section-map jets.  Equations
\eqref{eq:legacy-7-39}--\eqref{eq:legacy-7-42} give the physical factor
order and show why these alphabets are exhaustive.

The proof does substantially more.  It constructs the horizontal and
vertical graph transforms in the fixed-\(\gamma\) resolved frame, checks the
moving endpoint terms, and composes the five physical factors
\(P_0,\ldots,P_4\) on doubled first-hit buffers.  It proves
\[
 \|G\|_{C^4_{\rm res}}+\|G^{-1}\|_{C^4_{\rm res}}\le C_G,
 \qquad 0<g_0\le G'\le g_1,
\]
and, on every QHH component, obtains
\[
 {\cal K}''=-F_v''(G)(G')^2-F_v'(G)G''+t^2{\cal R}_K
 \le-\frac{1}{2}c_vg_0^2<0.
\]
The five commuting resolved alphabets and the labeled-set Fa\`a di Bruno,
quotient, moving-endpoint, composition, and inverse recurrences are displayed
in Section~\ref{sec:part-iii-middle}.  They form a finite, formula-defined proof of the
five-factor \(C^4\) closure.  No standalone generated QHH derivative table is
used as evidence for that closure.

The substitution check alone does not construct the finite QBF/QHH phase
split, prove the complete through components and denominator/section-normal
margins, or obtain the two-zero QBF and four-zero-per-QHH-component Rolle
bounds.

The half-open QBF/QHH overlap is assigned once.  Loss of a root, landing,
coefficient direction, denominator, side, or first hit stops at the
strict-lips, root-scale, mixed-face, passive, or exit regime specified in
Table~\ref{tab:regimes}.  The face \(t=0\) belongs to the source theorem, while the exact
persistent face \(t>0,\kappa=0\) belongs to the mixed theorem.  No
positive-margin theorem is continued across either face.

The finite algebraic check does not prove the graph-transform contraction,
the five-factor physical order, complete through components, coefficient
margins, phase connectors, or the curvature/Rolle count.  Those arguments
appear in Section~\ref{sec:part-iii-middle}; code and complete outputs are
included in the accompanying reproducibility package.

\subsection{The positive root-scale derivative recurrence}
\label{app:root-scale-certificate}

In a signed root chart the physical parameters are
\[
 B=t^2\kappa^2b,\qquad m=t^2M,\qquad
 a=\sigma t\kappa A,\qquad d=\sigma tD,\qquad c=\sigma\gamma.
\]
We first check the exact reciprocal endpoint substitution.  With
\(u=\sigma t\kappa\bar R\), \(\bar N=\bar Rz\), and
\[
 H=b+A\bar R+\bar R^2-\bar R\bar N+t^2M(\bar N-\bar R)^2,
\]
the two components are exactly
\[
\begin{aligned}
 \bar R'&=-\sigma t^2\kappa^2\bar R H,\\
 \bar N'&=\sigma\{\bar N+\kappa t\gamma\bar R(\bar N-\bar R)
 +\kappa t^2D(\bar N-\bar R)^2-t^2\kappa^2\bar N H\}.
\end{aligned}
\]
It also checks the outer transfer
\[
 S=\kappa s,\qquad Y=\kappa(\xi+s),\qquad X=Y-S,\qquad \kappa=SK.
\]
For the displayed polynomials \(\widetilde P,\widetilde Q\), the exact
desingularized field satisfies
\[
 Y'=S\widetilde P,\qquad S'=S\widetilde Q,\qquad
 K'=-K\widetilde Q,\qquad (SK)'=0.
\]
Thus the substitution preserves the original \(\kappa\)-fiber; \(S\) and
\(K\) are not independent physical parameters.

We next expand the eight numerator/denominator polynomials for the first four
polynomial factors and exhaust all ordered derivative words
of length at most four.  Table~\ref{tab:root-derivative-enumeration}
summarizes this finite certificate.
\begin{table}[H]
\centering
\caption{Finite derivative enumeration for the four polynomial root factors.}
\label{tab:root-derivative-enumeration}
\scriptsize
\begin{tabular}{@{}crrrrl@{}}
\toprule
Factor&Monomials&Letters&Words/monomial&Pairs&Multipliers, orders \(0\)--\(4\)\\
\midrule
\(N_0\)&7&9&7381&51667&1, 4, 16, 64, 256\\
\(Q_0\)&6&9&7381&44286&1, 2, 4, 8, 16\\
\(N_1\)&4&9&7381&29524&1, 4, 16, 64, 256\\
\(Q_1\)&4&9&7381&29524&1, 2, 4, 8, 16\\
\(N_2\)&4&11&16105&64420&1, 3, 9, 27, 81\\
\(Q_2\)&4&11&16105&64420&1, 2, 4, 8, 16\\
\(N_3\)&17&9&7381&125477&1, 4, 16, 64, 256\\
\(Q_3\)&7&9&7381&51667&1, 4, 16, 64, 256\\
\bottomrule
\end{tabular}
\end{table}
The total is \(460985\) monomial--ordered-itinerary pairs, including zero
descendants.  The outer rows are first rewritten with
\(SK=\varkappa\) and \(SK^2=\varkappa K\); hence every surviving descendant
containing \(K\) retains the essential \(\varkappa\)-factor.  A representative
quotient recurrence, obtained from \(Qf=N\), is
\[
 c_0=\frac{n_0}{ d},\qquad
 c_k=\frac{n_k+\sum_{j=1}^k \binom{k}{j}q_jc_{k-j}}{ d},
 \qquad 1\le k\le4,
\]
where \(d>0\) is the denominator margin proved in
Section~\ref{sec:part-iii-root}.

The proof constructs all five physical factors in their actual order,
proves the fixed-product wedge, first-exit landing, denominator and section
normal margins, the full vertical determinant, and the direct/inverse
\(C^4_{\rm res}\) bounds.  The four polynomial factors covered by the
\(460985\)-pair enumeration are \(P_0^{\rm rt},P_1^{\rm rt},P_2^{\rm rt}\), and
\(P_3^{\rm out}\).  The fifth physical factor \(P_4^{\rm rt}\) is handled by
the separate formula-defined graph, normal, entry, vertical-flight, output,
and inverse rows in the proof; it is not silently counted in the table.
The argument then proves the at-most-six through components and the
four-zero affine Rolle bound on each.

The theorem-validity handoff is the one fixed in
\eqref{eq:legacy-5-2}: \(t=0\) belongs to source, \(t>0,\kappa=0\) to the
mixed persistent-\(D\) theorem, and the positive-root theorem is valid on
\(t>0,\ 0<\kappa\le\kappa_0\).  Its owner interval stops at
\(\kappa_\#(\mathrm{angle})\), whose equality belongs to middle.  Every
other lost denominator, section, landing, or first hit stops at its named
adjacent regime.

The finite enumeration does not prove existence of a retained complete-lips itinerary,
the physical five-factor order, the fixed wedge, landing, vertical
determinant, through-component topology, boundary assignment, or Rolle
counting.  Those are proved in Section~\ref{sec:part-iii-root}; code and
complete outputs are included in the accompanying reproducibility package.

\subsection{Domains of the finite identities}

The identities above are used only on the physical domains for which the
main text proves all section, denominator, and first-hit margins.  We record
those domains explicitly to separate algebraic verification from analytic
applicability.

For the mixed Li\'enard--Dulac identities, the full-return domain is
\(z=1+y>0\), with
\[
 0\le p,\epsilon\le1,\qquad d>0,\qquad p\ge q_0d^2.
\]
The endpoint orientations and the topology of the relevant Dulac oval are
established before the multiplier is used.

For the intermediate-scale identities, \(|\gamma|\le\gamma_0\), with no
bound on \(\gamma/t\).  Every clock denominator, entry normal, and target
normal has a fixed positive margin on the doubled through component.
Coefficient faces remain in the same compact normalized base and are treated
by the displayed recurrence or by the identity alternative.  In particular,
the exact two-gate family converges without any restriction on \(\tau/t\).

For the positive root-merger recurrence,
\[
 t>0,\qquad 0<\kappa\le\kappa_0,
\]
and the normalized coefficients \(b,M,A,D,\gamma\) range over the compact
root carrier.  Positivity of the clock denominators and section-normal
determinants is an analytic hypothesis proved for the physical five-factor
itinerary.  The faces \(t=0\) and \(\kappa=0\) are not obtained by continuing
the positive-root calculation: they are treated by the source and
persistent-endpoint theorems, respectively.

\section{Computer-assisted calculations and data availability}
\label{sec:reproducibility}

Several finite algebraic calculations were checked by computer.  For the
source estimate, the calculation enumerates the \(35\) primitive derivative
types and all \(167115\) canonical commuting words of total order at most
six.  Appendix~\ref{app:mixed-dulac-certificate} records the exact
Li\'enard--Dulac identities used on the mixed face.  For the middle QHH
regime, the calculation checks the finite derivative alphabets
and the polynomial substitutions recorded in
Appendix~\ref{app:middle-qhh-certificate}.  For the root-scale regime, it
expands the first four polynomial factors and checks the \(460985\)
monomial--derivative pairs described in
Appendix~\ref{app:root-scale-certificate}.  Further symbolic calculations
verify the compactification formulas, center ideal, endpoint identities, and
mixed Li\'enard--Dulac cancellations in the preceding appendices.

Table~\ref{tab:certificate-interface} gives the audit boundary for every
machine-assisted part.  A successful run proves only the finite predicate in
the second column.  In particular, none of the scripts constructs an orbit or
promotes a symbolic chart identity to a physical incidence statement.

\begingroup
\footnotesize
\setlength{\tabcolsep}{3pt}
\refstepcounter{table}
\label{tab:certificate-interface}
\begin{center}
\textup{Table~\thetable: Claim--certificate--human-proof map.}\par
\smallskip
\begin{tabularx}{\textwidth}{@{}
 >{\raggedright\arraybackslash}p{.22\textwidth}
 >{\raggedright\arraybackslash}p{.25\textwidth}
 >{\raggedright\arraybackslash}p{.19\textwidth}
 >{\raggedright\arraybackslash}p{.27\textwidth}@{}}
\toprule
Mathematical use&Finite machine predicate&Replay unit&Human proof not supplied
by the run\\
\midrule
Signed compactification and endpoint preparation,
Propositions~\ref{prop:finite-singular-alphabet} and
\ref{prop:endpoint-joint-uniformity}&
Exact chart substitutions, endpoint-scale identities, and the four displayed
joint-endpoint cancellations&
{\ttfamily chart-signed-\linebreak compactification},
{\ttfamily chart-pp-scale}, and
{\ttfamily endpoint-joint-\linebreak uniformity}&
Chart coverage, section transversality, sign conventions, first-hit
exhaustiveness, and joint parameter--section estimates\\
\addlinespace
Source six-jet closure,
Theorem~\ref{thm:part-ii-six-jet} and
Appendix~\ref{app:source-six-jet}&
Closure of \(35\) primitive types and enumeration of \(167115\) canonical
words through total order six&
{\ttfamily source-six-jet-\linebreak owner-dag}&
Existence of the physical word, common moving-boundary domains, analytic
majorants, coefficient handoffs, and the Rolle argument\\
\bottomrule
\end{tabularx}
\end{center}

\pagebreak
\begin{center}
\textup{Table~\thetable\ (continued)}\par
\smallskip
\begin{tabularx}{\textwidth}{@{}
 >{\raggedright\arraybackslash}p{.22\textwidth}
 >{\raggedright\arraybackslash}p{.25\textwidth}
 >{\raggedright\arraybackslash}p{.19\textwidth}
 >{\raggedright\arraybackslash}p{.27\textwidth}@{}}
\toprule
Mathematical use&Finite machine predicate&Replay unit&Human proof not supplied
by the run\\
\midrule
Center division,
Theorem~\ref{thm:part-ii-center-ideal} and
Appendix~\ref{app:center-bautin}&
Focal recurrences, Darboux identities, generator bridges, and center-domain
polynomial identities&
{\ttfamily bautin-center-\linebreak basis},
{\ttfamily bautin-recurrence},
{\ttfamily bautin-center-\linebreak bautin}, and
{\ttfamily bautin-global-\linebreak domains}&
Completeness of both physical center return domains and legality of division
only after full-word composition\\
\addlinespace
Mixed Li\'enard--Dulac estimate,
Lemma~\ref{lem:part-ii-mixed-dulac} and
Appendix~\ref{app:mixed-dulac-certificate}&
Exact full-ratio substitutions and the identities
\eqref{eq:legacy-6-91}--\eqref{eq:legacy-6-95}&
{\ttfamily mixed-dulac-\linebreak full-ratio}&
The topology of \eqref{eq:mixed-zero-arc}, the line-punctured Dulac argument,
the physical cone, and boundary assignment\\
\addlinespace
Two-central completion,
Proposition~\ref{prop:complete-physical-lips},
Lemma~\ref{lem:two-central-landing-product-tube},
Corollary~\ref{cor:two-central-nonaffine-boundary}, and
Theorem~\ref{thm:two-central-physical-completion}&
Invariant-line, eigenslope, overlap, first-variation, and first-port
identities&
{\ttfamily open-aff-\linebreak invariant-line} and
{\ttfamily open-aff-first-\linebreak port}&
Actual hh/pp/PP-or-BP incidence, survival of the order band, cooperative
squeezing, boundary completeness, and nonaffinity transfer\\
\addlinespace
Middle QBF/QHH estimate,
Theorem~\ref{thm:part-iii-middle-zero} and
Appendix~\ref{app:middle-qhh-certificate}&
The unscaled-\(c\) substitutions, weighted overlaps, and common compact
base-point limit&
{\ttfamily middle-qhh-c-\linebreak dominant}&
Five-factor physical order, complete derivative recurrence, through-component
exhaustiveness, flow estimates, and Rolle counting\\
\addlinespace
Positive root estimate,
Theorem~\ref{thm:part-iii-root-zero} and
Appendix~\ref{app:root-scale-certificate}&
Four-factor substitutions, operator identities, and the \(460985\)
monomial--word enumeration&
{\ttfamily root-scale-triple-\linebreak merger} and
{\ttfamily root-scale-\linebreak rtm27-rtm47}&
The fifth physical factor, positive denominators, component topology,
boundary handoffs, and the final curvature/Rolle count\\
\bottomrule
\end{tabularx}
\end{center}
\endgroup

The versioned certificate archive is deposited at Zenodo
\cite{LuH14Data2026}.  Version~2.0.0 has DOI
\href{https://doi.org/10.5281/zenodo.21542256}
{\nolinkurl{10.5281/zenodo.21542256}}.  The replay-bundle fingerprints are
\begin{quote}
\footnotesize\ttfamily
\noindent SHA-256:\quad
2d29fe9dad3d4e748268660adc3fe026abd99c5615c6c44c006a429fa404842b\\
\noindent MD5:\quad 30edab70ad304a0411ff3024ffa50677
\end{quote}
The authoritative manifest is
\nolinkurl{reproduce/MANIFEST.md}; the per-case commands, script hashes,
canonical-output hashes, artifact assertions, and hashes of this TeX source
and the unified reader PDF are in
\nolinkurl{reproduce/expected_outputs.json}.  The row-level map from current
theorem or appendix labels to replay cases and output hashes is
\nolinkurl{reproduce/claim_certificate_map.csv}.

In the locked CPython~3.12.5 environment described by
\nolinkurl{reproduce/environment.lock}, the complete noninteractive replay is
\begin{verbatim}
H14_REPRO_PYTHON=/path/to/cpython-3.12.5/bin/python \
  ./reproduce/reproduce_all.sh
\end{verbatim}
and terminates with
\texttt{reproduction passed: 15/15 cases}.  The launcher checks the
interpreter and package versions, every script hash, the frozen blueprint
hash, the hash-locked requirements file, the unified-PDF and source bindings,
the canonical stdout hashes, and the two generated source artifacts.  A
separate standard-library implementation,
\nolinkurl{reproduce/independent_audit.py}, rereads the case logs, verifies the
claim map, and independently checks the source TSV invariants and the
\(460985\)-pair root total.  The continuous replay records the exact Git commit
and tree with both logs.  Neither checker imports a mathematical conclusion
from the other.  The mathematical statements, their domains, and the human
arguments needed to interpret the finite checks are all contained in this
single PDF; raw generated tables are electronic data rather than a second
paper.

\section*{Declaration of AI use}

The author used OpenAI Codex in preparing the manuscript and takes full
responsibility for its content.

\end{document}